\documentclass[a4paper, 11pt, oneside]{preprint}
\usepackage[a4paper,innermargin=1.4in,outermargin=1.4in,bottom=1.5in,marginparwidth=1.5in,marginparsep=3mm]{geometry}

\usepackage[utf8]{inputenc}

\usepackage[T1]{fontenc}
\usepackage[english]{babel}
\usepackage{amsmath}
\usepackage{amsfonts}
\usepackage{amssymb}
\usepackage{bm}
\usepackage{amsthm}
\usepackage{graphicx}
\usepackage{array}
\usepackage[dvipsnames]{xcolor}
\usepackage{mathrsfs}
\usepackage{hyperref}
\usepackage{esint} %%% for \fint
\usepackage{tikz}
\usepackage{upgreek}
\usepackage{enumitem}
\usepackage{mhequ}

\definecolor{darkgreen}{rgb}{0.1,0.7,0.1}

\usepackage{amsthm}

\theoremstyle{plain}
\newtheorem{theorem}{Theorem}[section]
\newtheorem{lemma}[theorem]{Lemma}
\newtheorem{proposition}[theorem]{Proposition}
\newtheorem{corollary}[theorem]{Corollary}

\theoremstyle{remark}
\newtheorem{remark}[theorem]{Remark}

\theoremstyle{definition}
\newtheorem{definition}[theorem]{Definition}
\newtheorem{assumption}[theorem]{Assumption}

\newtheorem{fact}[theorem]{Fact}

\usepackage{amsmath}
\usepackage{amssymb}
\usepackage{amsfonts}
\usepackage{mathrsfs}

\newcommand{\bn}[1]{{[\kern-0.5ex] #1 [\kern-0.5ex]}}

\newcommand{\eps}{\varepsilon}

\newcommand{\R}{\mathbf{R}}
\newcommand{\N}{\mathbf{N}}
\newcommand{\Z}{\mathbf{Z}}

\newcommand{\T}{\mathbf{T}}
\newcommand{\PPi}{\bm{\Pi}}
\renewcommand{\P}{\operatorname{\mathbb{P}}}
\newcommand{\E}{\operatorname{\mathbb{E}}}

\newcommand{\dd}{\operatorname{d\!}}

\def\un{\mathbf{1}}

\def\rm{\mathrm{m}}

\def\can{\mathrm{can}}
\def\adj{\mathrm{adj}}

\def\${|\!|\!|}

\newcommand{\crochet}[1]{\left\langle #1 \right\rangle}
\newcommand{\norm}[1]{\left\| #1 \right\|}

\newcommand{\bB}{\mathbf{B}}

\newcommand{\bI}{\mathbf{I}}

\newcommand{\bK}{\mathbf{K}}

\newcommand{\bM}{\mathbf{M}}

\newcommand{\bX}{\mathbf{X}}

\newcommand{\cC}{\mathcal{C}}
\newcommand{\cD}{\mathcal{D}}
\newcommand{\cE}{\mathcal{E}}
\newcommand{\cF}{\mathcal{F}}
\newcommand{\cG}{\mathcal{G}}

\newcommand{\cI}{\mathcal{I}}
\newcommand{\cJ}{\mathcal{J}}
\newcommand{\cK}{\mathcal{K}}
\newcommand{\cL}{\mathcal{L}}
\newcommand{\cM}{\mathcal{M}}
\newcommand{\cN}{\mathcal{N}}

\newcommand{\cP}{\mathcal{P}}
\newcommand{\cQ}{\mathcal{Q}}
\newcommand{\cR}{\mathcal{R}}
\newcommand{\cS}{\mathcal{S}}
\newcommand{\cT}{\mathcal{T}}

\newcommand{\cV}{\mathcal{V}}

\newcommand{\cZ}{\mathcal{Z}}

\newcommand{\ccR}{\mathscr{R}}

\newcommand{\ccT}{\mathscr{T}}

\newcommand{\poly}{\mathrm{pol}}
\newcommand{\RHS}{\mathrm{RHS}}
\newcommand{\sol}{\mathrm{sol}}

\newcommand{\origin}{\tikz \node[root] at (0, 0) {};}
\newcommand{\kernel}{\tikz[baseline=-0.1cm] \draw[kernel] (0,0) to (1,0);}
\newcommand{\kernelr}{\tikz[baseline=-0.1cm] \draw[kernel1] (0,0) to (1,0);}
\newcommand{\kernelrr}{\tikz[baseline=-0.1cm] \draw[kernel2] (0,0) to (1,0);}
\newcommand{\dkernel}{\tikz[baseline=-0.1cm] \draw[dkernel] (0,0) to (1,0);}
\newcommand{\dkernelr}{\tikz[baseline=-0.1cm] \draw[dkernel1] (0,0) to (1,0);}
\newcommand{\ddkernel}{\tikz[baseline=-0.1cm] \draw[ddkernel] (0,0) to (1,0);}

\newcommand{\kernelBig}{\tikz[baseline=-0.1cm] \draw[kernelBig] (0,0) to (1,0);}
\newcommand{\BigG}{\tikz[baseline=-0.1cm] \draw[BigG] (0,0) to (1,0);}
\newcommand{\erho}{\tikz[baseline=-0.1cm] \draw[rho] (0,0) to (1,0);}
\newcommand{\drho}{\tikz[baseline=-0.1cm] {\draw[drho] (0,0) to (1,0);}}
\newcommand{\ddrho}{\tikz[baseline=-0.1cm] {\draw[ddrho] (0,0) to (1,0);}}
\newcommand{\dkernelrr}{\tikz[baseline=-0.1cm] \draw[dkernel2] (0,0) to (1,0);}

\newcommand{\testfcn}{\tikz[baseline=-0.1cm] \draw[testfcn] (1,0) to (0,0);}
\newcommand{\dtestfcn}{\tikz[baseline=-0.1cm] \draw[dtestfcn] (1,0) to (0,0);}
\newcommand{\testfcnx}{\tikz[baseline=-0.1cm] \draw[testfcnx] (1,0) to (0,0);}
\newcommand{\multx}{\tikz[baseline=-0.1cm] \draw[multx] (1,0) to (0,0);}

\usetikzlibrary{snakes}
\usetikzlibrary{decorations}
\usetikzlibrary{decorations.markings}
\usetikzlibrary{positioning}
\usetikzlibrary{shapes}
\usetikzlibrary{arrows}
\usetikzlibrary{arrows.meta} 

\colorlet{symbols}{black!50}
\colorlet{testcolor}{green!60!black}
\definecolor{connection}{rgb}{0.7,0.1,0.1}
\definecolor{lblue}{rgb}{0.1,0.5,1}
\definecolor{airforceblue}{rgb}{0.36, 0.54, 0.66} % For multx

\makeatletter
\pgfdeclareshape{crosscircle}
{
	\inheritsavedanchors[from=circle] % this is nearly a circle
	\inheritanchorborder[from=circle]
	\inheritanchor[from=circle]{north}
	\inheritanchor[from=circle]{north west}
	\inheritanchor[from=circle]{north east}
	\inheritanchor[from=circle]{center}
	\inheritanchor[from=circle]{west}
	\inheritanchor[from=circle]{east}
	\inheritanchor[from=circle]{mid}
	\inheritanchor[from=circle]{mid west}
	\inheritanchor[from=circle]{mid east}
	\inheritanchor[from=circle]{base}
	\inheritanchor[from=circle]{base west}
	\inheritanchor[from=circle]{base east}
	\inheritanchor[from=circle]{south}
	\inheritanchor[from=circle]{south west}
	\inheritanchor[from=circle]{south east}
	\inheritbackgroundpath[from=circle]
	\foregroundpath{
		\centerpoint%
		\pgf@xc=\pgf@x%
		\pgf@yc=\pgf@y%
		\pgfutil@tempdima=\radius%
		\pgfmathsetlength{\pgf@xb}{\pgfkeysvalueof{/pgf/outer xsep}}%  
		\pgfmathsetlength{\pgf@yb}{\pgfkeysvalueof{/pgf/outer ysep}}%  
		\ifdim\pgf@xb<\pgf@yb%
		\advance\pgfutil@tempdima by-\pgf@yb%
		\else%
		\advance\pgfutil@tempdima by-\pgf@xb%
		\fi%
		\pgfpathmoveto{\pgfpointadd{\pgfqpoint{\pgf@xc}{\pgf@yc}}{\pgfqpoint{-0.707107\pgfutil@tempdima}{-0.707107\pgfutil@tempdima}}}
		\pgfpathlineto{\pgfpointadd{\pgfqpoint{\pgf@xc}{\pgf@yc}}{\pgfqpoint{0.707107\pgfutil@tempdima}{0.707107\pgfutil@tempdima}}}
		\pgfpathmoveto{\pgfpointadd{\pgfqpoint{\pgf@xc}{\pgf@yc}}{\pgfqpoint{-0.707107\pgfutil@tempdima}{0.707107\pgfutil@tempdima}}}
		\pgfpathlineto{\pgfpointadd{\pgfqpoint{\pgf@xc}{\pgf@yc}}{\pgfqpoint{0.707107\pgfutil@tempdima}{-0.707107\pgfutil@tempdima}}}
	}
}

\pgfdeclareshape{cross2}
{
	\inheritsavedanchors[from=rectangle]
	\inheritanchorborder[from=rectangle]
	\inheritanchor[from=rectangle]{center}
	\inheritanchor[from=rectangle]{north}
	\inheritanchor[from=rectangle]{south}
	\inheritanchor[from=rectangle]{east}
	\inheritanchor[from=rectangle]{west}
	\inheritanchor[from=rectangle]{north east}
	\inheritanchor[from=rectangle]{north west}
	\inheritanchor[from=rectangle]{south east}
	\inheritanchor[from=rectangle]{south west}
	
	\inheritbackgroundpath[from=rectangle]
	
	\behindforegroundpath{%
		\pgfextractx{\pgf@xa}{\southwest}%
		\pgfextracty{\pgf@ya}{\southwest}%
		\pgfextractx{\pgf@xb}{\northeast}%
		\pgfextracty{\pgf@yb}{\northeast}%
		
		\pgfsetlinewidth{0.5pt} % Thicker lines
		
		\pgfpathmoveto{\pgfpoint{0.95\pgf@xa}{0.95\pgf@ya}}%
		\pgfpathlineto{\pgfpoint{0.95\pgf@xb}{0.95\pgf@yb}}%
		
		\pgfpathmoveto{\pgfpoint{0.95\pgf@xa}{0.95\pgf@yb}}%
		\pgfpathlineto{\pgfpoint{0.95\pgf@xb}{0.95\pgf@ya}}%
		
		\pgfusepath{stroke}
	}
}
\makeatother

\def\drawx{\draw[-,solid] (-3pt,-3pt) -- (3pt,3pt);\draw[-,solid] (-3pt,3pt) -- (3pt,-3pt);}
\tikzset{
	xi/.style={very thin,circle,fill=white,draw=black,inner sep=0pt,minimum size=1.1mm},
	zeta1/.style={very thin,circle,fill=black!18,draw=black,inner sep=0pt,minimum size=1.1mm},
	zeta2/.style={very thin,circle,fill=black,draw=black,inner sep=0pt,minimum size=1.1mm},
	eta/.style={thin,rectangle,fill=red!20,draw=red,inner sep=0pt,minimum size=1.1mm},
	aeta/.style={very thin,rectangle,fill=white,draw=black,inner sep=0pt,minimum size=1.1mm},
	etab/.style={thin,rectangle,fill=red!20,draw=red,inner sep=0pt,minimum size=1.6mm},
	etabx/.style={cross2,fill=red!20,draw=red,inner sep=0pt,minimum size=1.6mm},
	aetabx/.style={cross2,fill=white,draw=black,inner sep=0pt,minimum size=1.6mm},
	aetab/.style={very thin,rectangle,fill=white,draw=black,inner sep=0pt,minimum size=1.6mm},
	xix/.style={crosscircle,fill=white,draw=black,inner sep=0pt,minimum size=1.2mm},
	xibx/.style={crosscircle,fill=white,draw=black,inner sep=0pt,minimum size=1.6mm},
	xib/.style={very thin,circle,fill=white,draw=black,inner sep=0pt,minimum size=1.6mm},
	zeta2b/.style={very thin,circle,fill=black,draw=black,inner sep=0pt,minimum size=1.6mm},	
	zeta1b/.style={very thin,circle,fill=black!18,draw=black,inner sep=0pt,minimum size=1.6mm},	
	xibx/.style={crosscircle,fill=white,draw=black,inner sep=0pt,minimum size=1.6mm},
	not/.style={thin,circle,fill=black,draw=black,inner sep=0pt,minimum size=0.3mm},
	>=stealth,
	root/.style={circle,fill=testcolor,inner sep=0pt, minimum size=2mm},
	sroot/.style={circle,fill=testcolor,inner sep=0pt, minimum size=1.3mm},
	dot/.style={circle,fill=black,inner sep=0pt, minimum size=1mm},
	var/.style={circle,fill=black!10,draw=black,inner sep=0pt, minimum size=2mm},
	dotred/.style={circle,fill=black!50,inner sep=0pt, minimum size=2mm},
	generic/.style={semithick,shorten >=1pt,shorten <=1pt},
	dist/.style={ultra thick,draw=testcolor,shorten >=1pt,shorten <=1pt},
	testfcn/.style={ultra thick,testcolor,shorten >=1pt,shorten <=1pt,<-},
	testfcnx/.style={ultra thick,testcolor,shorten >=1pt,shorten <=1pt,<-,
		postaction={decorate,decoration={markings,mark=at position 0.6 with {\drawx}}}},
	dtestfcn/.style={ultra thick,testcolor,shorten >=1pt,shorten <=1pt,<-,
		postaction={decorate,decoration={markings,mark=at position 0.6 with {\draw[-, black] (0,-0.1) -- (0,0.1);}}}},
	kprime/.style={semithick,shorten >=1pt,shorten <=1pt,densely dashed,->},
	kprimex/.style={semithick,shorten >=1pt,shorten <=1pt,densely dashed,->,
		postaction={decorate,decoration={markings,mark=at position 0.4 with {\drawx}}}},
	kernel/.style={semithick,shorten >=1pt,shorten <=1pt,->},
	multx/.style={ultra thick, airforceblue, <-, shorten >=1pt,shorten <=1pt,
		},
	kernelx/.style={semithick,shorten >=1pt,shorten <=1pt,->,
		postaction={decorate,decoration={markings,mark=at position 0.4 with {\drawx}}}},
	Keps/.style={densely dashed,semithick,shorten >=1pt,shorten <=1pt,->},
	dKeps/.style={densely dashed,semithick,shorten >=1pt,shorten <=1pt,->,postaction={decorate,decoration={markings,mark=at position 0.45 with {\draw[-] (0,-0.1) -- (0,0.1);}}}},
	ddKeps/.style={densely dashed,semithick,shorten >=1pt,shorten <=1pt,->,postaction={decorate,decoration={markings,mark=at position 0.45 with {\draw[-] (0.05,-0.1) -- (0.05,0.1);\draw[-] (-0.05,-0.1) -- (-0.05,0.1);}}}},
	dkernel/.style={->,semithick,shorten >=1pt,shorten <=1pt,postaction={decorate,decoration={markings,mark=at position 0.5 with {\draw[-] (0,-0.1) -- (0,0.1);}}}},
	dkernelx/.style={->,semithick,shorten >=1pt,shorten <=1pt,
		postaction={decorate,decoration={markings,mark=at position 0.4 with {\draw[-] (0,-0.15) -- (0,0.15);\drawx}}}},
	ddkernel/.style={->,semithick,shorten >=1pt,shorten <=1pt,postaction={decorate,decoration={markings,mark=at position 0.45 with {\draw[-] (0.05,-0.1) -- (0.05,0.1);\draw[-] (-0.05,-0.1) -- (-0.05,0.1);}}}},
	kernel1/.style={->,semithick, Cerulean, shorten >=1pt,shorten <=1pt},
	Keps1/.style={->,densely dashed,semithick, Cerulean, shorten >=1pt,shorten <=1pt},
	dkernel1/.style={->,semithick, Cerulean, shorten >=1pt,shorten <=1pt,postaction={decorate,decoration={markings,mark=at position 0.45 with {\draw[-, black] (0,-0.1) -- (0,0.1);}}}},
	dkernel2/.style={->,semithick, Red, shorten >=1pt,shorten <=1pt,postaction={decorate,decoration={markings,mark=at position 0.45 with {\draw[-, black] (0,-0.1) -- (0,0.1);}}}},
	dKeps1/.style={->,densely dashed,semithick,Cerulean,shorten >=1pt,shorten <=1pt,postaction={decorate,decoration={markings,mark=at position 0.45 with {\draw[-, black] (0,-0.1) -- (0,0.1);}}}},
	ddkernel1/.style={->,semithick, Cerulean, shorten >=1pt,shorten <=1pt,postaction={decorate,decoration={markings,mark=at position 0.45 with {\draw[-, black] (0.05,-0.1) -- (0.05,0.1);\draw[-, black] (-0.05,-0.1) -- (-0.05,0.1);}}}},
	BigG/.style={semithick,shorten >=1pt,shorten <=1pt,decorate, decoration={coil,aspect=0.7,amplitude=1.9pt,segment length = 3pt,pre length=2pt,post length=2pt}},
	kernel2/.style={->,semithick, Red, shorten >=1pt,shorten <=1pt},
	kernelBig/.style={->, semithick,shorten >=1pt,shorten <=1pt,decorate, decoration={zigzag,amplitude=1pt,segment length = 3pt,pre length=2pt,post length=5pt}},
	rho/.style={dotted,semithick,shorten >=1pt,shorten <=1pt},
	drho/.style={->,dotted,semithick,shorten >=1pt,shorten <=1pt,postaction={decorate,decoration={markings,mark=at position 0.3 with {\draw[-, solid] (0,-0.1) -- (0,0.1);}}}},
	ddrho/.style={-,dotted,semithick,shorten >=1pt,shorten <=1pt,postaction={decorate,decoration={markings,mark=at position 0.5 with {\draw[-, solid] (0.05,-0.1) -- (0.05,0.1);\draw[-,solid] (-0.05,-0.1) -- (-0.05,0.1);}}}},
	ddrho-shift/.style={-,dotted,semithick,shorten >=1pt,shorten <=1pt,postaction={decorate,decoration={markings,mark=at position 0.3 with {\draw[-, solid] (0.05,-0.1) -- (0.05,0.1);\draw[-,solid] (-0.05,-0.1) -- (-0.05,0.1);}}}}, % To prevent overlap of decorations
	renorm/.style={shape=circle,fill=white,inner sep=1pt},
	labl/.style={shape=rectangle,fill=white,inner sep=1pt},
}

\makeatletter
\def\DeclareSymbol#1#2#3{\expandafter\gdef\csname MH@symb@#1\endcsname{\tikz[baseline=#2,scale=0.15,draw=symbols,line join=round]{#3}}\expandafter\gdef\csname MH@symb@#1s\endcsname{\scalebox{0.7}{\tikz[baseline=#2,scale=0.15,draw=symbols,line join=round]{#3}}}}
\def\<#1>{\csname MH@symb@#1\endcsname}
\makeatother

\DeclareSymbol{Xi}{-2.8}{\node[xib] {};}

\DeclareSymbol{XiX}{-2.8}{\node[xibx] {};}

\DeclareSymbol{Eta}{-2.8}{\node[etab] {};}

\DeclareSymbol{Zeta1}{-2.8}{\node[zeta1b] {};}

\DeclareSymbol{Zeta2}{-2.8}{\node[zeta2b] {};}

\DeclareSymbol{AEta}{-2.8}{\node[aetab] {};}

\DeclareSymbol{XAEta}{-2.8}{\node[aetabx] {};}

\DeclareSymbol{XXi}{-2.8}{\node[xibx] {};}

\DeclareSymbol{XEta}{-2.8}{\node[etabx] {};}

\DeclareSymbol{IXi}{0}{\draw (0,0) node {} -- (0,1.5) node[xi] {};}

\DeclareSymbol{IXXi}{0}{\draw (0,0) node {} -- (0,1.5) node[xix] {};}

\DeclareSymbol{IEta}{0}{\draw[red] (0,0) node {} -- (0,1.5) node[eta] {};}

\DeclareSymbol{IAEta}{0}{\draw (0,0) node {} -- (0,1.5) node[aeta] {};}

\DeclareSymbol{EtaIEta}{0}{\draw[red] (0,0) node[eta] {} -- (1.2,1.2) node[eta] {};}

\DeclareSymbol{AEtaIAEta}{0}{\draw (0,0) node[aeta] {} -- (1.2,1.2) node[aeta] {};}

\DeclareSymbol{IXi2}{0}{\draw (-1,1) node[xi]{}-- (0,0) node {} -- (1,1) node[xi] {};}

\DeclareSymbol{IXi3}{0}{\draw (-1,1) node[xi]{}-- (0,0) node {} -- (1,1) node[xi] {};\draw (0,0)--(0,1.5) node[xi]{};}

\DeclareSymbol{XiIXi}{0}{\draw (0,0) node[xi] {} -- (1,1) node[xi] {};}

\DeclareSymbol{XiIZeta2}{0}{\draw (0,0) node[zeta1] {} -- (1,1) node[zeta2] {};}

\DeclareSymbol{Zeta2IXi}{0}{\draw (0,0) node[zeta2] {} -- (1,1) node[zeta1] {};}

\DeclareSymbol{XiIXXi}{0}{\draw (0,0) node[xi] {} -- (1,1) node[xix] {};}

\DeclareSymbol{XXiIXi}{0}{\draw (0,0) node[xix] {} -- (1,1) node[xi] {};}

\DeclareSymbol{IXiIXi}{0}{\draw (1,-1) node  {} -- (0,0) node[xi] {} -- (1,1) node[xi] {};}

\DeclareSymbol{IXiIXiNEW}{0}{\draw (1,-1) node  {} -- (0,0) node[zeta2] {} -- (1,1) node[zeta1] {};}

\DeclareSymbol{XiIXiIXi}{-2}{\draw (1,-1) node [xi] {} -- (0,0) node[xi] {} -- (1,1) node[xi] {};}

\DeclareSymbol{dontneed}{-2}{\draw (1,-1) node [xi] {} -- (0,0) node {} -- (1,1) node[xi] {};}

\DeclareSymbol{XiIXiIXiNEW}{-2}{\draw (1,-1) node [zeta1] {} -- (0,0) node[zeta2] {} -- (1,1) node[zeta1] {};}

\DeclareSymbol{IXiIXiIXi}{0}{\draw (0,-2) node {}-- (1,-1) node [xi] {} -- (0,0) node[xi] {} -- (1,1) node[xi] {};}

\DeclareSymbol{XiIXi2}{0}{\draw (-1,1) node [xi] {} -- (0,0) node[xi] {} -- (1,1) node[xi] {};}

\DeclareSymbol{XiIXi2NEW}{0}{\draw (-1,1) node [zeta1] {} -- (0,0) node[zeta2] {} -- (1,1) node[zeta1] {};}

\DeclareSymbol{IXiIXi2}{0}{\draw(0,-1.5) node {}--(0,0);\draw (-1,1) node [xi] {} -- (0,0) node[xi] {} -- (1,1) node[xi] {};}

\DeclareSymbol{XiIXiIXi2}{-3}{\draw(0,-1.5) node[xi]{}--(0,0);\draw (-1,1) node [xi] {} -- (0,0) node[xi] {} -- (1,1) node[xi] {};}

\DeclareSymbol{XiIXiIXi2NEW}{-3}{\draw(0,-1.5) node[zeta1]{}--(0,0);\draw (-1,1) node [zeta1] {} -- (0,0) node[zeta2] {} -- (1,1) node[zeta1] {};}

\DeclareSymbol{XiIXi3}{0}{\draw (0,0)--(0,1.5) node[xi]{}; \draw (-1,1) node [xi] {} -- (0,0) node[xi] {} -- (1,1) node[xi] {};}

\DeclareSymbol{XiIXi3NEW}{0}{\draw (0,0)--(0,1.5) node[zeta1]{}; \draw (-1,1) node [zeta1] {} -- (0,0) node[zeta2] {} -- (1,1) node[zeta1] {};}

\DeclareSymbol{XiIXiIXiIXi}{-4}{\draw (0,-2) node [xi]{}-- (1,-1) node [xi] {} -- (0,0) node[xi] {} -- (1,1) node[xi] {};}

\DeclareSymbol{lastone}{-2}{\draw(2,0) node[xi]{}-- (1,-1) node [xi] {} -- (0,0) node[xi] {} -- (1,1) node[xi] {};}

\DeclareSymbol{XiIIXi}{-2}{\draw (1,-1) node [xi] {} -- (0,0) node {} -- (1,1) node[xi] {};}

\hypersetup{
	citecolor=blue,
	colorlinks=true, %colorise les liens
	breaklinks=true, %permet le retour a la ligne dans les liens trop longs
	urlcolor= blue, %couleur des hyperliens
	linkcolor= black, %couleur des liens internes
	bookmarksopen=true, %si les signets Acrobat sont crees,
	pdftitle={Variance gPAM2d}, %informations apparaissant dans
}

\begin{document}
\title{Variance renormalisation in regularity structures -- the case of $2d$ gPAM}
\author{Máté Gerencsér \inst{1} and Yueh-Sheng Hsu \inst{2}
\institute{TU Wien, Austria. \email{mate.gerencser@tuwien.ac.at} \and TU Wien, Austria. \email{yueh-sheng.hsu@tuwien.ac.at}}}

\date{\today}
\maketitle

\begin{abstract}
We consider the variance renormalisation of a singular SPDE for which a Da Prato-Debussche trick is not applicable.
The example taken is the $2$-dimensional generalised parabolic Anderson model (gPAM), driven by a much rougher than white noise, necessitating both a multiplicative and an additive renormalisation.
To handle the discrepancy between the regularity structures of the approximate and the limiting equations, we consider models that lift $0$ noises to nontrivial models, in analogy with ``pure area'' from rough paths.
The convergence to such a model is shown for the BPHZ model over the vanishing noise via graphical computations.
\end{abstract}

\tableofcontents

\section{Introduction}
The theory of regularity structures, introduced in \cite{H0} and developed to large generality in \cite{CH, BHZ, BCCH}, provides a renormalised local solution theory to a large class of singular stochastic PDEs: equations driven by very rough random fields, where due to the low regularity, products on the right-hand side of the equation have no classically well-defined meaning.
The roughness of the noise has to satisfy two essential constraints, the first and more well-known is scaling subcriticality.
For example, consider the well-known case of the KPZ equation, formally given as
\begin{equ}\label{eq:KPZ}
\partial_t h=\Delta h+|\nabla h|^2-\infty+\xi.
\end{equ}
Given a noise $\xi$ with regularity\footnote{Understood here as the index of self-similarity under parabolic rescaling: $\xi(\lambda^2 t,\lambda x)\overset{\mathrm{law}}{=}\lambda^\alpha \xi(t,x)$.} $\alpha$, if one (formally) rescales the equation in a way that leaves the linear part invariant, $h^\lambda(t,x)=\lambda^{-\alpha-2}h(\lambda^2 t, \lambda x)$, then $h^\lambda$ solves \eqref{eq:KPZ} with another noise $\tilde\xi$ following the same law as $\xi$ but the nonlinearity multiplied by $\lambda^{2+\alpha}$. This suggests that the scaling critical exponent is $\alpha_c=-2$, and for $\alpha>\alpha_c$ the solution on small scales should look like the Gaussian process solving the linear equation.
For critical ($\alpha=\alpha_c$) or supercritical noise ($\alpha<\alpha_c$) the solution theory of regularity structures (or of paracontrolled distributions \cite{GIP}) do not apply.
We remark that while $\alpha_c$ does not depend on the dimension, the criticality of the  choice of space-time white noise does: as it has regularity $(-d-2)/2$, for \eqref{eq:KPZ}  it is subcritical in spatial dimension $d=1$, critical for $d=2$, and supercritical for $d\geq 3$.

There is another exponent $\alpha_v$ (that does depend on the dimension) that ensures that the variance of all of the stochastic objects required for the theory of regularity structures and constructed in \cite{CH} remains finite.
The condition $\alpha>\alpha_v$ is another crucial assumption, as an infinite variance can not be removed by an additive counterterm.
For example, for \eqref{eq:KPZ} in dimension $d=1$, one has $\alpha_v=-7/4$, which is precisely the barrier of the theory instead of $\alpha>\alpha_c=-2$, c.f. \cite{Hoshino}.

Recently \cite{Hai25, MateFabio} considered the regime $\alpha< \alpha_v$ in the case of \eqref{eq:KPZ} in dimension $1$.
The discussion in \cite[Sec.~1.2]{Hai25} hints that if $\alpha_c<\alpha_v$, then the subcritical toolbox applies for \emph{all} $\alpha$, but in the regime $\alpha\leq\alpha_v$ one can expect a different type of result: when replacing the noise by its smooth approximation in the equation, then to tame the exploding variance, a \emph{multiplicative} renormalisation is required and a nontrivial limit \emph{in law} is obtained.
This is formulated in a general prediction in \cite{Hai25}. The results of \cite{Hai25, MateFabio} are about \eqref{eq:KPZ} and the proofs depend on the Da Prato-Debussche trick of subtracting the linear solution and treating the (more regular) remainder equation.
It should be mentioned that \cite{Hai25} also treats a similar problem for stochastic ordinary differential equation driven by fractional Brownian motions with Hurst parameter $H\leq 1/4$ where the noise is not additive.

The goal of the present paper is to confirm the prediction of \cite{Hai25} for a ``more nonlinear'' SPDE, where the Da Prato-Debussche trick is not applicable, and therefore several of the issues that are treated in an ad hoc manner in \cite{Hai25, MateFabio} have to be handled within the theory of regularity structures.
The standard ``guinea pig'' singular SPDE is the the two-dimensional generalised parabolic Anderson model ($2d$ gPAM), see e.g. \cite{CannizzaroFrizGassiat, FrizKlose, GWP1, GWP2}. We follow this tradition here.

For the $2d$ gPAM one has $\alpha_c=-2$ and $\alpha_v=-3/2$, and so one has indeed $\alpha_c<\alpha_v$. To get a noise below $\alpha_v$, we take a directional derivative of a spatial white noise $\xi$. The equation then formally reads as
\begin{equ}\label{eq:gPAM}
	\partial_t u = \Delta u + g(u) \partial_{x_1}\xi, \quad \text{on } \T^2,
\end{equ}
where the nonlinearity $g: \R \to \R$ is assumed to be regular enough.

Let $\rho:\R^2\to\R$ be smooth, radial, nonnegative, supported on the ball of radius $1$, and integrate to $1$. Then set $\rho_\eps(x)=\eps^{-2}\rho(\eps^{-1}x)$  and define $\xi_\eps = \rho_\eps * \xi$. In line with the above discussion, when replacing $\xi$ by $\xi_\eps$ in \eqref{eq:gPAM}, both an additive and a multiplicative renormalisation are required to get a limit. The convergence takes place in law, and the limiting equation is another $2d$ gPAM, but with a different nonlinearity and driven by a white noise (without the derivative).
This is the main result of the paper.
\begin{assumption}\label{asn:data}
One has $g\in C^5_b$ and $\psi\in C^\theta(\T^2)$ for some $\theta\in(0,1/4)$.
\end{assumption}
\begin{theorem}\label{thm:main}
Let Assumption \ref{asn:data} hold.
For any $\eps>0$ let $u_\eps$ be the solution to
\begin{equ}\label{eq:gPAM_regular}
	\partial_t u_\eps = \Delta u_\eps + g(u_\eps) \eps^{\frac12} \partial_{x_1} \xi_\eps-C_\eps g'g(u_\eps)-\bar C_\eps (g')^3g(u_\eps)-\hat C_\eps g''g'g^2(u_\eps)
\end{equ}
on $\T^2$ with initial condition $\psi$.
There exists a sequence of constants $C_\eps,\bar C_\eps, \hat C_\eps$ and a constant $c_\rho>0$ such that $u_\eps$  converges in law as $\eps\to 0$ in $C([0,1]\times\T^2)$ to the renormalized solution $u$ of the SPDE
	\begin{equ}\label{eq:gPAM_limit}
		\partial_t u = \Delta u + c_\rho g'g(u) \eta
	\end{equ}
	on $\T^2$ with initial condition $\psi$, 	where $\eta$ is a  spatial white noise.
\end{theorem}

\begin{remark}
Recall from \cite{H0} that the solution $u$ of \eqref{eq:gPAM_limit} can also be characerised via smooth approximations with counterterms of the same form as the last two nonlinearities in \eqref{eq:gPAM_regular}: if $\eta_\eps=\rho_\eps\ast\eta$, then $u$ is the limit in probability of the solutions $v_\eps$ of the renormalised equations
\begin{equ}
\partial_t v_\eps = \Delta v_\eps + c_\rho g'g(v_\eps)\eta_\eps-\tilde C_\eps\big( (g')^3g(v_\eps)+ g''g'g^2(v_\eps)\big),
\end{equ}
with an appropriate choice of $\tilde C_\eps$.
\end{remark}

\begin{remark}\label{rem:independence}
A small modification of the proof also shows that the pair $(\xi,u_\eps)$ converges in law to $(\xi,u)$, where \eqref{eq:gPAM_limit} is driven by an independent white noise $\eta$.
\end{remark}

The rest of the paper is structured as follows. In Section \ref{sec:pure_area_model} we set up the regularity structure framework, divide the strategy into a deterministic and probabilistic statement, prove the former (Lemma \ref{lem:equivalence}), formulate precisely the latter (Theorem \ref{thm:convergence}), and give the proof of the main result modulo these two statements. In Section \ref{sec:conv_tree} we prove Theorem \ref{thm:convergence},
a statement on the convergence of a sequence of BPHZ models. This convergence is fairly different from many others in the singular SPDE literature, since the limit is \emph{not} a BPHZ model.

\textbf{Acknowledgments} 
The authors thank Rhys Steele for discussions on the topic of the article.
The authors are funded by the European Union (ERC, SPDE, 101117125). Views and opinions expressed
are however those of the author(s) only and do not necessarily reflect those of the European Union
or the European Research Council Executive Agency. Neither the European Union nor the granting
authority can be held responsible for them.

\section{The deterministic step}\label{sec:pure_area_model}
The proof of Theorem \ref{thm:main} can be divided in two conceptually distinct steps. We start with the one that in the logical order of the proof comes second, since this is much shorter, and it also gives some hopefully illuminating context for the lengthy calculations in Section \ref{sec:conv_tree}.

We start with the following observation. The noise in \eqref{eq:gPAM_regular} is uniformly in $\eps>0$ bounded in $C^{-3/2-\kappa}(\T^2)$ for any $\kappa>0$, so it is natural to describe the solutions $u^\eps$ via the regularity structure associated to the $1$-dimensional (nonlinear) SHE or the $3$-dimensional gPAM, see \cite{HP15}.
Take $\kappa\in(0,1/100)$. Consider the regularity structure $\mathscr{T}=(\mathcal{T},G,A)$ of \cite{HP15} with noise regularity $-3/2-\kappa$, with the only difference\footnote{We emphasise that while the definition of the regularity structure carries through without any problem in the $2$-dimensional setting, and so does the definition of the BPHZ models above smooth noises, the bounds on the BPHZ models do not, since \cite[Def.~2.28]{CH} fails.} that the abstract polynomial $\mathbf{X}$ has now two instances $ \mathbf{X}_1$, $\mathbf{X}_2$, representing the polynomials $x\mapsto x_1$, $x\mapsto x_2$, respectively.
See below in Section \ref{sec:reg-str} for a detailed definition.

In contrast, the regularity structure used to describe the limiting equation \eqref{eq:gPAM_limit} is that of the $2$-dimensional gPAM, with noise regularity $-1-2\kappa$, see \cite{H0}, henceforth denoted by $\mathscr{T}'=(\mathcal{T}',G',A')$. 
See below in Section \ref{sec:reg-str} for a detailed definition.
Therefore, it appears that when taking limits, the regularity structure changes. To treat this issue, we separate the questions of taking limit and changing structure. We show that the BPHZ lifts $(\hat\Pi^\eps,\hat\Gamma^\eps)$ of $\eps^{\frac12}\partial_1\xi^\eps$ converge to a model, still for $\mathscr{T}$. This limiting model turns out to be a nontrivial lift of a trivial noise\footnote{In particular, the naive way of writing the equation in the limit would be 
$(\partial_t-\Delta)u=g(u)\cdot 0,$ which is of course fairly misleading.}. Analogous objects in the rough path context appear under the terminology ``pure area paths''. As far as we are aware, in the regularity structure context they have not yet been studied, in our setup they arise naturally.

The other task is then show that equations driven by such ``pure area models'' over $\mathscr{T}$ can be related to \emph{different} equations driven by models over $\mathscr{T}'$. This is the purpose of this section.
This argument is purely deterministic.

\begin{remark}
In \cite{Hai25}, the change of regularity structures can be effectively avoided thanks to the Da Prato-Debussche trick.
\end{remark}

\subsection{Regularity structure setup}\label{sec:reg-str}
In our context the setup of the relevant regularity structures can be given without most of the general formalism of \cite{CH, BHZ, BCCH}, following \cite{H0} with some ad hoc simplifications adapted to the concrete setting of \eqref{eq:gPAM}.

Define a set of abstract symbols as follows. First let $T_\poly=\{\mathbf{1}, \mathbf{X}_1,\mathbf{X}_2\}.$
Then define the sets $V^{\RHS}_n$, $V_n^{\sol}$, indexed by $n\in\N$, inductively by setting $V^{\RHS}_0 = V^{\sol}_0 = \emptyset$ and for $n \ge 1$
\begin{equs}
	V_{n+1}^{\RHS} &=V_{n}^{\RHS} \bigcup\left\{\tau_1\cdots\tau_k\Xi:  k\ge 0, \tau_j \in V_n^{\sol}\right\},\\
	V_{n+1}^{\sol} &= T_\poly \bigcup \left\{\cI(\tau): \tau \in V_{n+1}^{\RHS}\right\}.
\end{equs}
The product above is set to be commutative and associative with unit element $\mathbf{1}$.
Define $\widetilde{V}^{\RHS} = \bigcup_{n\ge0} V_n^{\RHS}$ and $\widetilde{V}^{\sol} = \bigcup_{n\ge0} V_n^{\sol}$. 
For any symbol $\tau$ included above define its homogeneity $|\tau|$ inductively by the rules
\begin{equ}
|\Xi|=-3/2-\kappa,\,|\mathbf{1}|=0,\,|\mathbf{X}_1|=|\mathbf{X}_2|=1,\quad|\cI(\tau)|=|\tau|+2,\,\,|\tau\tau'|=|\tau|+|\tau'|.
\end{equ}
Set $V^{\RHS}=\{\tau\in\widetilde{V}^{\RHS}:\,|\tau|<\kappa\}$, $V^{\sol}=\{\tau\in\widetilde{V}^{\sol}:\,|\tau|<3/2+2\kappa\}$, and finally $T=V^{\RHS}\cup V^{\sol}$.

Let $\cT$ be the graded vector space that is the set of linear combinations of $T$ and $A$ be the set of possible homogeneities $A=\{|\tau|:\,\tau\in T\}$. We also equip $\cT$ with an inner product $\langle\cdot,\cdot\rangle$ by setting $T$ as its orthonormal basis.  Finally, we take a subgroup $G$ of the group of linear maps from $\cT$ to $\cT$ whose exact description is not necessary for our purposes, it suffices to remark that for all $\Gamma\in G$ and $\tau\in T$, $\Gamma\tau-\tau\in\mathrm{span}\{\tau'\in T:\,|\tau'|<|\tau|\}$. This completes the definition of the regularity structure $\ccT=(\cT,G,A)$.

Elements of $T\setminus T_{\poly}$ will also be denoted graphically by rooted trees with vertices possibly decorated by one of three decorations $\<Xi>,\,\<XXi>_1,\,\<XXi>_2$. The symbol $\Xi$ is denoted by the single vertex graph $\<Xi>$, $\mathbf{X}_i\Xi$ is denoted by $\<XXi>_i$, and then one follows two type of inductive steps: the graph associated to $\cI(\tau)$ is obtained by adding a new undecorated vertex as a root and connecting it to root of the graph of $\tau$,
and the graph associated to $\tau\tau'$ is obtained by joining the two roots of the graphs of $\tau$ and $\tau'$. It is easy to see that if $\tau\tau'\in T$, at most one of the roots of $\tau$ or $\tau'$ has a decoration, which is inherited by the root of the product.
For example $ \cI(\mathbf{X}_i\Xi)\Xi=\<XiIXXi>_i$ and $\cI(\Xi)\cI(\Xi)\Xi=\<XiIXi2>$.
It is easy to check that this inductive procedure gives a graphical representation for all elements of $T\setminus T_{\poly}$ (but not all possible symbols, where for example higher powers of $\mathbf{X}_i$ can appear). A full list of symbols\footnote{We only need and so only use the symbols that are required to set up the abstract fixed point problem. We do not need symbols that are only used in the action of the renormalisation group, for example $\<dontneed>$.} can be found in \eqref{eq:area-model}.

The regularity structure $\ccT'=(\cT',G',A')$ is defined analogously, except that the symbol $\Xi$ is replaced by $\Xi'$ (in graphical notation: $\<AEta>$) that is assigned homogeneity $-1-2\kappa$.

Next recall the concept of models. A model $(\Pi,\Gamma)$ is a pair of maps
\begin{equ}
\Pi:\R^2\to L(\cT,\cS'),\qquad\Gamma:(\R^2)^2\to G,
\end{equ}
such that they satisfy for all $x,y,z\in\R^2$ the algebraic identities $\Pi_y=\Pi_x\Gamma_{xy}$, $\Gamma_{xy}\Gamma_{yz}=\Gamma_{xz}$, and the following analytic bounds. Let $\Phi$ be the set of smooth functions on $\R^2$ supported on the unit ball and being bounded by $1$ along with their first and second partial derivatives.
For $\lambda>0$ and $x\in\R^2$ denote $\varphi_x^\lambda(y)=\lambda^{-2}\varphi(\lambda^{-1}(y-x))$.
One then requires that the maps $\Pi$ and $\Gamma$ satisfy the bounds
\begin{equ}\label{eq:model-norm}
|(\Pi_x\tau)\varphi_x^\lambda|\lesssim \lambda^{|\tau|},\qquad|\langle\Gamma_{xy}\tau,\tau'\rangle|\lesssim|x-y|^{|\tau|-|\tau'|}.
\end{equ}
uniformly in $\tau,\tau'\in T$, $\lambda\in(0,1]$, $\varphi\in\Phi$, and $x,y$ over compacts. 
The functions $\Pi$ and $\Gamma$ are extended to $\R\times \R^2$ and $(\R\times \R^2)^2$, respectively, by setting $\Pi_{(t,x)}\equiv\Pi_x$ and $\Gamma_{(t,x)(s,y)}\equiv \Gamma_{xy}$. The optimal constant in \eqref{eq:model-norm} defines a norm $\|(\Pi,\Gamma)\|$ on the space of pairs $(\Pi,\Gamma)$ (whether satisfying the algebraic identities or not).
This norm depends on the choice of $\kappa$, which is often suppressed, but when we want to emphasise this dependence, we write $\|(\Pi,\Gamma)\|_\kappa$.
This norm induces a metric on the space of models by setting $\|(\Pi,\Gamma);(\bar\Pi,\bar\Gamma)\|_\kappa=\|(\Pi-\bar\Pi,\Gamma-\bar\Gamma)\|_\kappa$. The space of models is denoted by $\cM_\kappa$.

We now proceed with the abstract
reformulation of equations of the form \eqref{eq:gPAM}.
More precisely, one aims to understand \eqref{eq:gPAM} via its Duhamel formulation as a fixed point problem and lift this formulation to an abstract fixed point problem 
on an appropriate space of $\cT$-valued functions.
Given a model $(\Pi,\Gamma)$, a subspace $\bar\cT\subset\cT$, and two exponents $\eta\leq\gamma$, the set of modelled distributions, denoted by $\cD^{\gamma,\eta}(\bar\cT,\Gamma)$, is defined as the set of functions $f:(0,1]\times\T^2\to\cT$ such that the norm
\begin{equ}
\max_{|\tau|<\gamma}\sup_{z}\frac{|\langle f(z),\tau\rangle|}{t^{(\eta-|\tau|)\wedge0}}+\max_{|\tau|<\gamma}\sup_{z, \bar z}\frac{|\langle f(z)-\Gamma_{z\bar z}f(\bar z),\tau\rangle|}{|z-\bar z|^{\gamma-|\tau|}
t^{\eta-\gamma}}
\end{equ}
is finite. Here and below $|z-\bar z|$ denotes the parabolic distance $\sqrt{|t-\bar t|}+|x-\bar x|$ and in the suprema $z=(t,x)$ runs over $(0,1]\times \T^2$, and $z=(t,x),\bar z=(\bar t,\bar x)$ run over $\{z,\bar z\in (0,1]\times \T^2:\,0<2|z-\bar z|<t\}$.
We will be interested in two particular cases: the right-hand side of the equation will be an element of $\cD_{\RHS}:=\cD^{\kappa,-3/2-\kappa}(\cV^{\RHS},\Gamma)$, while the solution itself will be an element of $\cD_{\sol}:=\cD^{3/2+2\kappa,\theta}(\cV^{\sol},\Gamma)$.
When differentiating between modelled distributions on the different regularity structures $\ccT$ and $\ccT'$, we write e.g. $\cD_\sol$ and $\cD_\sol'$.
Two modelled distributions $f\in\cD^{\gamma,\eta}(\bar\cT,\Gamma)$ and $\tilde f\in\cD^{\gamma,\eta}(\bar\cT,\tilde \Gamma)$ with respect to two different models $(\Pi,\Gamma)$ and $(\tilde \Pi,\tilde \Gamma)$ (but the same regularity structure) are compared with the metric
\begin{equ}\label{eq:norm-modelled}
\max_{|\tau|<\gamma}\sup_{z}\frac{|\langle f(z)-\tilde f(z),\tau\rangle|}{t^{(\eta-|\tau|)\wedge0}}+\max_{|\tau|<\gamma}\sup_{z,\bar z}\frac{|\langle f(z)-\Gamma_{z\bar z}f(\bar z)-\tilde f(z)+\tilde\Gamma_{z,\bar z}\tilde f(\bar z),\tau\rangle|}{|z-\bar z|^{\gamma-|\tau|}
t^{\eta-\gamma}}.
\end{equ}
Localised version of this metric can be taken by restricting the $t,\bar t$ variables to run over $(0,T]$ with some $T\in(0,1)$.

The next step is the definition of the right-hand side of the equation, given an element $U\in\cD_{\sol}$. Set $\bar u=\langle U,\mathbf{1}\rangle$ and $\widetilde{U}=U-u\mathbf{1}$. Denote for $\alpha\in\R$ by $\cQ_\alpha$ the projection on the subspace spanned by basis symbols with homogeneity less than $\alpha$.
Note that for any $\ell>1$, $\widetilde{U}^\ell$ does not take values in $\cT$, but $\cQ_{3/2+2\kappa}\widetilde {U}^\ell$ does.
Then set
\begin{equ}
\widehat{h}(U)\,\<Xi>=h(\bar u(z))\,\<Xi>+\sum_{\ell\geq 1} \partial^\ell h(\bar u(z))\frac{\cQ_{3/2+2\kappa}\widetilde {U}^\ell}{\ell!}\,\<Xi>
\end{equ}
and analogously for $\widehat{h}(U)\,\<AEta>$. Note that the summands are nonzero only for $\ell=1,2,3$ in $\ccT$ and only for $\ell=1$ in $\ccT'$.
For sufficiently regular $h$, the output belongs to $\cD_{\RHS}$.

Turning to the other part of the Duhamel formulation, let us define the abstract counterpart of the convolution with the heat kernel $P$.
Consider a truncation $\bar K$ of the heat kernel as in \cite[Sec.~5]{H0}, for our purposes it suffices to recall that $\bar K$ has support contained in $(-1/2,1/2)^3$ (in particular, can also be viewed as a function on $\R\times\T^2$) and it differs from the heat kernel by a smooth function. Set $K(x)=\int_\R\bar K(t,x)\,dt$ for $x\in\R^2$. 
A model is called admissible if it satisfies
\begin{equ}\label{eq:model-poly}
(\Pi_x\mathbf{1})(y)=1,\,\Gamma_{xy}\mathbf{1}=\mathbf{1},\quad
(\Pi_x\mathbf{X}_i)(y)=y_i-x_i,\,\Gamma_{xy}\mathbf{X}_i=\mathbf{X}_i+(x_i-y_i)\mathbf{1},
\end{equ}
the identity $\Gamma_{xy}(\tau\bar\tau)=(\Gamma_{xy}\tau)(\Gamma_{xy}\bar\tau)$,
as well as
\begin{equs}
(\Pi_x\cI\tau)&=K\ast\Pi_x\tau-\Pi_x(\cJ(x)\tau)
\\&:=K\ast\Pi_x\tau-\sum_{|k|<|\tau|+2}\int\mathbf{X}^{k}D^kK(x-y)(\Pi_x\tau)(\dd y),
\end{equs}
the last integral interpreted in the distributional sense. Here and below in all sum over $k$ the possible elements are $(0,0)$, $(1,0)$, $(0,1)$, with the conventions that $|k|=k_1+k_2$, $\mathbf{X}^{(0,0)}=\mathbf{1}$, $\mathbf{X}^{(1,0)}=\mathbf{X}_1$, $\mathbf{X}^{(0,1)}=\mathbf{X}_2$, $D^{(0,0)}=\mathrm{id}$, $D^{(1,0)}=\partial_{x_1}$, $D^{(0,1)}=\partial_{x_2}$.

The pointwise operators $\cI$ and $\cJ$ are used in conjunction with a nonlocal operator to encode the abstract counterpart of convolution with $\bar K$.
First recall an operator of utmost importance that turns abstract, $\cT$-valued functions into genuine distributions, referred to as reconstruction and denoted by $\cR$. It is a continuous map $\cR:\cD^{\gamma,\eta}(\Gamma)\to \cS'((0,1]\times\T^2)$.
Although the exact construction or characterisation of $\cR$ is not important for us, it will be useful to note that it can be written in the form of the form
\begin{equ}\label{eq:reconstruction-def}
\cR f=\lim_{n\to\infty}\sum_{z\in \Lambda_n}\big(\Pi_zf(z)\big)(\psi_{z,n})\psi_{z,n},
\end{equ}
with a certain sequence of finite sets $(\Lambda_n)_{n\in\N}$ and sequence of functions $(\psi_{z,n})_{n\in\N,z\in\Lambda_n}$ regular enough for \eqref{eq:reconstruction-def} to make sense.
Set, for $f\in \cD^{\gamma,\eta}(\Gamma)$ with $\gamma>0$, $\gamma\geq \eta$, and $\eta>-2$,
\begin{equ}\label{eq:N-def}
(\cN f)(z)=\sum_{k}\bX^k\int D^k\bar K(z-z')\big(\cR f-\Pi_zf(z)\big)(\dd z').
\end{equ}
Then defining the map $\bar\cK$ of $\cD^{\gamma,\eta}(\Gamma)$ by
\begin{equ}
\bar\cK f(z)=\cQ_{3/2+2\kappa}\cI (f(z))+\cJ(z)f(z)+(\cN f)(z), 
\end{equ}
it maps $\cD_{\RHS}$ to $\cD_{\sol}$
and it satisfies $\cR\bar \cK f=\bar K\ast \cR f$, where in the latter space-time convolution we extend $\cR f$ to $(-\infty,1]\times \T^2$ by setting it to be $0$ on negative times.

For a real-valued function $h$ that is continuously differentiable on $(0,1]\times \T^2$, denote its lift to a $\cT$-valued function by
\begin{equ}
\cL h(z)=\sum_{k}\mathbf{X}^k D^kh(z).
\end{equ}
Recall that $P$ denoting the heat kernel, $Z:=P-\bar K$ is smooth and so is its spatial periodisation $\tilde Z$.
Therefore we define
\begin{equ}\label{eq:abstract-Z}
\cZ f=\cL\big(\tilde Z\ast\cR f\big).
\end{equ}
The contribution of the initial condition $\psi$ is simply given by
\begin{equ}
\Psi=\cL\Big((t,x)\mapsto \big(P(t,\cdot)\ast \psi\big)(x)\Big),
\end{equ}
the convolution this time understood only in space.
With all of these ingredients the abstract counterpart of \eqref{eq:gPAM} reads, with the notation $\cP f =\bar\cK f+\cZ f +\Psi$,
\begin{equ}\label{eq:abstract-gPAM-old}
U=\mathcal{P}\big(\hat g(U)\,\<Xi>\big).
\end{equ}
To summarise the main results of \cite{H0} for \eqref{eq:abstract-gPAM-old}, one has:
\begin{theorem}\label{eq:gPAM-Hairer-theorem}
For any admissible model $(\Pi,\Gamma)$ there exists a $T\in(0,1]$ such that there exists a unique solution $U$ of \eqref{eq:abstract-gPAM-old} on $[0,T)\times \T^2$. The solution is a continuous function of the model in the sense that if a sequence of admissible models $\big((\Pi^\eps,\Gamma^\eps)\big)_{\eps\in(0,1]}$ converge to $(\Pi,\Gamma)$ in the norm given by \eqref{eq:model-norm}, then for any $\delta\in(0,T)$, $U^\eps$ converges to $U$ on $[0,T-\delta]$ in $\cD_{\mathrm{sol}}$.
Finally, if  $T<1$, then $\lim_{t\nearrow T}\|(\cR U)(t,\cdot)\|_{C^{1/4}(\T^2)}=\infty$.
\end{theorem}

\subsection{Pure area noise}
We shall use graphical notation not only for the abstract symbols, but also for distributions (more on this in Section \ref{sec:conv_tree}).
In particular, for a given distribution $\<Eta>$ on $\R^2$, we further denote the distributions 
\begin{equ}
	\<IEta>=K\ast\<Eta>,\qquad\<XEta>_i(\varphi)=\<Eta>\big(x\mapsto x_i\varphi(x)\big).
\end{equ}
Here and throughout this section $i=1,2$.
\begin{assumption}\label{asn:area-model}
	We are given two distributions $\<Eta>$ and $\<EtaIEta>$ on $\T^2$ such that the linear extension of following maps (complemented by \eqref{eq:model-poly} on $T_{\poly}$) defines an admissible model $(\Pi',\Gamma')$ over $\mathscr{T}'$:
\begin{equs}[eq:final-model]
\Pi_x'\<AEta>&=\<Eta>
\qquad&\Gamma_{xy}'\<AEta>&=\<AEta>
\\
\Pi_x'\<XAEta>_i&=\<XEta>_i-x_i\,\<Eta>
\qquad&\Gamma_{xy}'\<XAEta>_i&=\<XAEta>_i+(x_i-y_i)\,\<AEta>
\\
\Pi_x'\<IAEta>&=\<IEta>-\<IEta>(x)
\qquad&\Gamma_{xy}'\<IAEta>&=\<IAEta>+\big(\<IEta>(x)-\<IEta>(y)\big)\,\mathbf{1}
\\
\Pi_x'\<AEtaIAEta>&=\<EtaIEta>-\<IEta>(x)\,\<Eta>
\qquad&\Gamma_{xy}'\<AEtaIAEta>&=\<AEtaIAEta>+\big(\<IEta>(x)-\<IEta>(y)\big)\,\<AEta>
\end{equs}
 \end{assumption}
 Then we define a model over the regularity structure $\mathscr{T}$ as the linear extension of the following maps (complemented by \eqref{eq:model-poly} on $T_\poly$): 
\begin{equs}[eq:area-model]
\Pi_x \<Xi>&=0
\qquad& \Gamma_{xy}\<Xi>&=\<Xi>
\\
\Pi_x\<XXi>_i&=0
\qquad& \Gamma_{xy}\<XXi>_i&=\<XXi>_i+(x_i-y_i)\,\<Xi>
\\
\Pi_x\<IXi>&=0
\qquad& \Gamma_{xy}\<IXi>&=\<IXi>
\\
\Pi_x\<IXXi>_i&=0
\qquad& \Gamma_{xy}\<IXXi>_i&=\<IXXi>_i+(x_i-y_i)\,\<IXi>
\\
\Pi_x\<XiIXi>&=\<Eta>
\qquad& \Gamma_{xy}\<XiIXi>&=\<XiIXi>
\\
\Pi_x \<XiIXXi>_i&=\<XEta>_i-x_i\,\<Eta>
\qquad& \Gamma_{xy}\<XiIXXi>_i&=\<XiIXXi>_i+(x_i-y_i)\,\<XiIXi>
\\
\Pi_x \<XXiIXi>_i&=\<XEta>_i-x_i\,\<Eta>
\qquad& \Gamma_{xy}\<XXiIXi>_i&=\<XXiIXi>_i+(x_i-y_i)\,\<XiIXi>
\\
\Pi_x\<IXiIXi>&=\<IEta>-\<IEta>(x)
\qquad& \Gamma_{xy}\<IXiIXi>&=\<IXiIXi>+\big(\<IEta>(x)-\<IEta>(y)\big)\,\mathbf{1}
\\
\Pi_x\<XiIXiIXi>&=0
\qquad& \Gamma_{xy}\<XiIXiIXi>&=\<XiIXiIXi>+\big(\<IEta>(x)-\<IEta>(y)\big)\,\<Xi>
\\
\Pi_x\<XiIXi2>&=0
\qquad& \Gamma_{xy}\<XiIXi2>&=\<XiIXi2>
\\
\Pi_x\<IXiIXiIXi>&=0
\qquad& \Gamma_{xy}\<IXiIXiIXi>&=\<IXiIXiIXi>+\big(\<IEta>(x)-\<IEta>(y)\big)\,\<IXi>
\\
\Pi_x\<IXiIXi2>&=0
\qquad& \Gamma_{xy}\<IXiIXi2>&=\<IXiIXi2>
\\
\Pi_x\<XiIXiIXiIXi>&=\<EtaIEta>-\<IEta>(x)\,\<Eta>
\qquad& \Gamma_{xy}\<XiIXiIXiIXi>&=\<XiIXiIXiIXi>+\big(\<IEta>(x)-\<IEta>(y)\big)\,\<XiIXi>
\\
\Pi_x \<XiIXiIXi2>&=0
\qquad&\Gamma_{xy}\<XiIXiIXi2>&=\<XiIXiIXi2>
\\
\Pi_x\<XiIXi3>&=0
\qquad& \Gamma_{xy}\<XiIXi3>&=\<XiIXi3>
\\
\Pi_x\<lastone>&=\<EtaIEta>-\<IEta>(x)\,\<Eta>
\qquad& \Gamma_{xy}\<lastone>&=\<lastone>+\big(\<IEta>(x)-\<IEta>(y)\big)\,\<XiIXi>
\end{equs}
\begin{remark}
In the case of the KPZ equation in \cite{Hai25}, all symbols that ``survive'' in the limit
belong to the minimal set $T[\tau]$ that contains the symbol $\tau$ that gives rise to the new noise, the polynomials, and is closed under integration and multiplication (as long as the corresponding product belongs to $T$ to begin with). This is not the case here.
Indeed, $\<XiIXXi>,\,\<XiIXiIXiIXi>\notin T[\<XiIXi>]$ but $\Pi_x\<XiIXXi>,\Pi_x\<XiIXiIXiIXi>\neq 0$.
\end{remark}
It is elementary to check the following.
\begin{proposition}
Under Assumption \ref{asn:area-model}, \eqref{eq:area-model} defines an admissible model for $\mathscr{T}$.
\end{proposition}

Let the linear map $\iota:\mathcal{T}\to \mathcal{T}'$ be defined on the basis elements as
\begin{equ}
\iota \mathbf{1}=\mathbf{1},\,\,
\iota \mathbf{X}_i=\mathbf{X}_i,\,\,
\iota \<XiIXi>=\<AEta>, \,\,
\iota \<IXiIXi>=\<IAEta>,\,\,
\iota \<XXiIXi>_i=\iota\<XiIXXi>_i=\<XAEta>_i,\,\,
\iota \<XiIXiIXiIXi>=\iota\<lastone>=\<AEtaIAEta>,
\end{equ}
and $\iota \tau=0$ for all other elements of $T$.

\begin{lemma}\label{lem:equivalence}
Under Assumptions \ref{asn:data} and \ref{asn:area-model}, if $U\in\cD_{\sol}(\Gamma)$ is a solution to \eqref{eq:abstract-gPAM-old}
with respect to the model $(\Pi,\Gamma)$ defined in \eqref{eq:area-model},
then $V=\iota U$ belongs to $\cD_{\sol}'(\Gamma')$ and is a solution to
\begin{equ}\label{eq:abstract-gPAM-new}
V=\mathcal{P}\Big(\widehat{gg'}(V)\,\<AEta>\Big)
\end{equ}
with respect to the model $(\Pi',\Gamma')$.
\end{lemma}
\begin{proof}
Denote $u=\langle U,\mathbf{1}\rangle=\langle V,\mathbf{1}\rangle=\cR U=\cR' V$ and $\nabla u=\langle U,\mathbf{X}\rangle=\langle V,\mathbf{X}\rangle$
(here $\nabla$ is a formal notation and not a real derivative). Exactly as in \cite[Eq.~(4.9),(4.10)]{HP15} one has that $U$ is of the form
\begin{equs}[eq:U-expansion]
U&=u\mathbf{1}+g(u)\,\<IXi>+g'g(u)\,\<IXiIXi>+\nabla u \mathbf{X}
\\
&\quad+(g')^2g(u)\,\<IXiIXiIXi>+\frac12 g''g(u)\,\<IXiIXi2>+g'(u)\nabla u\,\<IXXi>
\end{equs}
and the right-hand side of the equation has the expansion
\begin{equs}[eq:hat-g]
\widehat{g}(U)\,\<Xi>&=g(u) \,\<Xi> + gg'(u) \,\<XiIXi> + (g')^2 g(u)\, \<XiIXiIXi>+ \frac12 g'' g^2(u) \<XiIXi2>  + g'(u)\nabla u\,\<XXi>
\\
&\quad+ (g')^3 g(u)\, \<XiIXiIXiIXi> + \frac12 g''g' g^2(u)\, \<XiIXiIXi2> + \frac16 g''' g^3(u)\, \<XiIXi3> + g'' g' g^2(u)\, \<lastone>
\\
&\quad +(g')^2(u)\nabla u\, \<XiIXXi>+g''g(u)\nabla u\,\<XXiIXi>.
\end{equs}
Therefore one has
\begin{equ}\label{eq:V-expansion}
V=u\mathbf{1}+g'g(u)\,\<IAEta>+\nabla u\mathbf{X}.
\end{equ}
Since $V(z)-\Gamma'_{z \bar z}V(\bar z)=\iota\big(U(z)-\Gamma_{z \bar z}U(\bar z)\big)$, $V\in\cD_{\sol}'(\Gamma') $ follows.
Also, \eqref{eq:V-expansion} implies that for any $h\in C^2$,
\begin{equs}[eq:hat-h]
\widehat{h}(V)\,\<AEta>&=h(u)\,\<AEta>+(h'gg')(u)\,\<AEtaIAEta>+h'(u)\nabla u\,\<XAEta>.
\end{equs}
With the choice $h=g'g$, one sees that the equality \eqref{eq:abstract-gPAM-new} is satisfied on the non-polynomial part (which in this case has only one symbol, $\<IAEta>$). 

As for the polynomial part,
the contributions come from the operators $\cJ$, $\cN$, $\cZ$, and $\Psi$, so we compare them one by one.
The contribution of $\Psi$ is obviously the same for both \eqref{eq:abstract-gPAM-old} and \eqref{eq:abstract-gPAM-new}.
Denote $F=\widehat{g}(U)\,\<Xi>$ and $F'=\widehat{g'g}(V)\,\<AEta>$. Comparing \eqref{eq:hat-g} and \eqref{eq:hat-h} with the choice $h=g'g$, one sees $\iota F=F'$.
Since furthermore for all $\tau\in\mathcal{T}$ and all $x$ one has $\Pi_x\tau=\Pi_x'\iota\tau$, this implies $\cJ(x)F(x)=\cJ'(x)F'(x)$, as well as $\cR F=\cR' F'$, owing to \eqref{eq:reconstruction-def}.
This in turn  implies by \eqref{eq:N-def} that $\cN F=\cN' F'$ and by \eqref{eq:abstract-Z} that $\cZ F=\cZ'F'$. The proof is finished.
\end{proof}

A converse counterpart of Lemma \ref{lem:equivalence} will be useful to make the convergence in Theorem \ref{thm:main} global in time.
\begin{proposition}\label{prop:converse}
Let Assumptions \ref{asn:data} and \ref{asn:area-model} hold. Let $V\in \cD_{\sol}(\Gamma')$ be the solution to \eqref{eq:abstract-gPAM-new} and denote $u=\langle V,\mathbf{1}\rangle$ and $\nabla u=\langle V,\mathbf{X}\rangle$. Then the function $U$ defined by the equality \eqref{eq:U-expansion} belongs to $\cD_\sol(\Gamma)$ and solves \eqref{eq:abstract-gPAM-old}.
\end{proposition}
\begin{proof}
Bounding the first term in \eqref{eq:norm-modelled} for $U$ is straightforward, so we focus on the second one.
Since $\langle U(z)-\Gamma_{z \bar z}U(\bar z),\tau\rangle=\langle V(z)-\Gamma_{z \bar z}V(\bar z),\tau\rangle$ for $\tau\in\{\mathbf{1},\mathbf{X}_i\}$, in these cases the estimates for $U$ follow from the ones for $V$. Next consider $\tau=\<IXi>$.
Take $z,\bar z\in(0,1]\times\T$ such that $t>2|z-\bar z|$.
Denote $\tilde u_{z \bar z}:=\langle\Gamma_{z \bar z}V(\bar z),\mathbf{1}\rangle=u( \bar z)+\nabla u( \bar z)(x- \bar x)+gg'(u( \bar z))\big(\<IEta>(x)-\<IEta>( \bar x)\big)$ and note the bounds
\begin{equ}\label{eq:tiny1}
|u(z)-\tilde u_{z \bar z}|\lesssim t^{\theta-3/2-2\kappa}|z-\bar z|^{3/2+2\kappa},
\end{equ}
and 
\begin{equ}\label{eq:tiny2}
|u(\bar z)-\tilde u_{z \bar z}|\lesssim t^{\theta-1}|x-\bar x|+|x-\bar x|^{1-2\kappa}.
\end{equ}
Then 
\begin{equs}
\big|\langle U(z)-\Gamma_{z\bar z}U(\bar z),\<IXi>\rangle\big|&=\big|g(u(z))-g(u(\bar z))-g'(u(\bar z))\nabla u(\bar z)(z-\bar z)
\\
&\qquad-g'(u(\bar z))gg'(u(\bar z))\big(\<IEta>(z)-\<IEta>(\bar z)\big)+g(\tilde u_{z\bar z})-g(\tilde u_{z\bar z})
\big|
\\
&=\big|g(u(z))-g(\tilde u_{z\bar z})+g(\tilde u_{z\bar z})-g(u(\bar z))-g'(u(\bar z))(\tilde u_{z \bar z}-u(\bar z))\big|
\\
&\lesssim |g|_{C^2}\big(|u(z)-\tilde u_{z\bar z}|+|\tilde u_{z\bar z}-u(\bar z)|^2\big).
\end{equs}
Owing to \eqref{eq:tiny1} and \eqref{eq:tiny2}, both terms are bounded by the required order $t^{\theta-3/2-2\kappa}|z-\bar z|^{1+3\kappa}$.
For all the symbols $\tau\in\{\<IXiIXi>,\,\<IXiIXiIXi>,\,\<IXiIXi2>\}$, one simply has
\begin{equs}
\big|\langle U(z)-\Gamma_{z \bar z}U(\bar z),\tau\rangle\big|=\big|\langle U(z)-U(\bar z),\tau\rangle\big|&\lesssim (1+|g|_{C^3}^3)|u(z)-u(\bar z)|
\\
&\lesssim t^{\theta-(1-2\kappa)}|z-\bar z|^{1-2\kappa},
\end{equs}
which is better than the required order $t^{\theta-3/2-2\kappa}|z-\bar z|^{3/2+2\kappa-|\tau|}$ since in this case $|\tau|\geq 1-2\kappa$.
Finally, the bound in the case $\tau=\<IXXi>_i$ requires
$|g(u(z))\partial_iu(z)-g(u(\bar z))\partial_iu(\bar z)|\lesssim t^{\theta-3/2-2\kappa}|z-\bar z|^{3\kappa}$, which follows from \eqref{eq:tiny1}, \eqref{eq:tiny2}, and the bounds
\begin{equ}
|\partial_i u(z)|\lesssim t^{\theta-1},\quad
|\partial_i u(z)-\partial_i u(\bar z)|\lesssim t^{\theta-3/2-2\kappa}|z-\bar z|^{1/2+2\kappa},
\end{equ}
keeping in mind that $|z-\bar z|\leq t$ by assumption.
\end{proof}

\subsection{Proof of Theorem \ref{thm:main}}
Let $(\hat\Pi^\eps,\hat \Gamma^\eps)$ be the BPHZ model over the noise $\eps^{\frac12}\partial_{x_1}\xi_\eps$, which we take essentially from \cite{HP15}, modulo the obvious modifications for symbols involving polynomials, and the changes of kernels (but with the same regularising effects) due to the different dimensions.
The probabilistic result establishing the convergence of these models and proved in the rest of the paper, is the following.
\begin{theorem}\label{thm:convergence}
Let $\kappa\in(0,1/100)$ and let $c_\rho>0$ be the constant defined in \eqref{eq:c-whitenoise}. Let $\<Eta>$ be $c_\rho$ times a $2$-dimensional spatial white noise and let $\<EtaIEta>$ be the Wick product of $\<Eta>$ and $\<IEta>$. Then they almost surely satisfy Assumption \ref{asn:area-model} and $(\hat \Pi^\eps,\hat \Gamma^\eps)\to(\Pi,\Gamma)$ in law in $\cM_\kappa$, where $(\Pi, \Gamma)$ is as in \eqref{eq:area-model}.
\end{theorem}
Given Theorem \ref{thm:convergence}, we can write the proof of the main result.
\begin{proof}[Proof of Theorem \ref{thm:main}]
By Theorem \ref{thm:convergence} and Skorohod's theorem, there exists a probability space $(\bar \Omega,\bar \cF,\bar{\mathbb{P}})$ and random models $(\bar\Pi^\eps,\bar\Gamma^\eps)$, $(\bar\Pi, \bar\Gamma)$ on it, such that $(\bar\Pi^\eps,\bar\Gamma^\eps)\to(\bar\Pi,\bar\Gamma)$ for all $\bar \omega\in\bar \Omega$ and  $(\bar\Pi^\eps,\bar\Gamma^\eps)\overset{\mathrm{law}}{=}(\hat \Pi^\eps,\hat \Gamma^\eps)$,  $(\bar\Pi,\bar\Gamma)\overset{\mathrm{law}}{=}(\Pi,\Gamma)$. 
In particular, for the corresponding solutions of \eqref{eq:abstract-gPAM-old}, $\bar U^\eps\overset{\mathrm{law}}{=}U^\eps$ and thus $\bar u^\eps:=\cR\bar U^\eps\overset{\mathrm{law}}{=}\cR U^\eps=u^\eps$, where the last equality follows from \cite[Prop.~4.4]{HP15}.
By Theorem \ref{thm:convergence} we have $\bar U^\eps\to\bar U$ for all $\bar \omega$ up to time $T-\delta$, where $T$ is the maximal lifetime of $\bar U$. By Lemma \ref{lem:equivalence}, $\cR \bar U=\cR' \bar V\overset{\mathrm{law}}{=} u$. Finally, from Proposition \ref{prop:converse}, the lifetime of $\bar U$ equals the lifetime of $\bar V$, but the latter is $1$ owing to \cite{GWP1}, finishing the proof.
\end{proof}

\section{The probabilistic step}\label{sec:conv_tree}
The aim of this section is to prove Theorem \ref{thm:convergence}.
First we recall the following tightness result for random models.
\begin{lemma}\label{lem:tightness}
	Let $\kappa\in(0,1/100)$, $Z^{\eps} = (\Pi^{\eps}, \Gamma^{\eps})$ be a sequence of random models defined on a Gaussian probability space for the regularity structure $\ccT = (\cT, G, A)$ and suppose that there exists $K\in\N$ such that  for all $\tau\in T$, $\varphi\in\Phi$, $x\in\T^2$, $\eps\in(0,1]$, $\Pi_x^{\eps}\varphi$ belongs to the $K$-th inhomogeneous Wiener chaos.
	Assume that there exist $\kappa'>0$ and $C > 0$ such that $Z^{\eps}$ satisfies the following bound for all $\eps\in(0,1]$:
	\begin{equ}[eq:tightness_model]
		\sup_{\tau \in V^{\RHS}} \sup_{\varphi \in \Phi} \sup_{x \in \T^2} \sup_{\lambda \in (0, 1]} \E\left(\frac{\Pi^{\eps}_x \tau (\varphi^\lambda_x)}{\lambda^{|\tau| + \kappa'}}\right)^2 \leq C\;.
	\end{equ}
	Then the sequence of laws of $(Z^{\eps})_{\eps>0}$ is tight on $\cM_\kappa$ and thus there exists a subsequence which converges in law in $\cM_\kappa$ to a random model $Z$.
\end{lemma}
\begin{proof}
	By \cite[Thm.~10.7]{H0}, one deduces from \eqref{eq:tightness_model} that for any $p\geq 2$ there is some constant $C' > 0$ and some $\kappa'' < \kappa$ such that
	\begin{equ}
		\E\norm{Z^{\eps}}_{\kappa''}^p \leq C'
	\end{equ}
	uniformly in $\eps\in(0,1]$.
	Combining this with the fact that for $\kappa_1 < \kappa_2$, the space $\cM_{\kappa_1}$ embeds compactly into the space $\cM_{\kappa_2}$, the claim is proved: tightness follows in a standard way by Markov's inequality and the compactness of the set $\{\norm{Z}_{\kappa_1} \leq M\}$ in $\cM_{\kappa_2}$ for any $M>0$.
\end{proof}

To prove Theorem \ref{thm:convergence}, in view of Lemma \ref{lem:tightness} and the translational invariance in law, we consider the second moment of the random field $\hat \Pi^{\eps}_0 \tau$ for each tree $\tau \in V^{\RHS}$. More precisely, our proof is structured in three steps:
\begin{enumerate}
	\item For each $\tau \in V^{\RHS}$, we verify that the bound
	\begin{equ}\label{eq:second_moment_model}
		\sup_{\varphi \in \Phi} \sup_{\lambda \in (0, 1]} \E \left(\frac{\zeta (\varphi_0^\lambda)}{\lambda^{\beta}}\right)^2 \lesssim 1
	\end{equ}
	is satisfied with the choice $\zeta=\hat\Pi^{\eps}_0 \tau$ and $\beta=|\tau|+\kappa'$,
	uniformly in $\eps \in (0,1]$.
	Lemma \ref{lem:tightness} then implies the existence of a subsequential limit in law $\hat Z$. Then it remains to uniquely identify that the law of $\hat Z$ is as the limit stated in Theorem \ref{thm:convergence}.
	
	\item For $\tau \in V^{\RHS}\backslash\{\<XiIXi>, \<XiIXXi>_i, \<XXiIXi>_j, \<XiIXiIXiIXi>, \<lastone>\}$, a fairly short argument shows that $\Pi^\eps_0 \tau(\varphi)$ converges to the zero in probability. This is the content of Section \ref{sec:vanishing_trees} below, which also gives \eqref{eq:second_moment_model} for these symbols.
	
	\item For non-vanishing symbols $\tau \in \{\<XiIXi>, \<XiIXXi>_i, \<XXiIXi>_j, \<XiIXiIXiIXi>, \<lastone>\}$, we identify their limit with the pure area model $\hat \Pi$ defined in \eqref{eq:area-model}. This is done by
	\begin{itemize}
		\item using the fourth moment theorem \cite{NP4} to show $\hat \Pi_0^{\eps} \<XiIXi>$ converges to a new white noise (Section \ref{sec:new_noise}), this step is very similar to \cite{Hai25, MateFabio};
		\item using a criterion for second-order homogeneous Wiener chaos (Lemma \ref{lem:criterion_second_chaos}) and tools on bounding generalised convolutions described by graphs \cite{HQ18} to characterise the limiting laws of $\hat \Pi_0^{\eps} \<XiIXiIXiIXi>$, $\hat \Pi_0^{\eps} \<lastone>$ (Section \ref{sec:critical-graphs}).
	\end{itemize}
\end{enumerate}

\subsection{The graphical framework}
Since the random variables $(\hat \Pi_0^{\eps} \tau)(\varphi)$ belong to the inhomogenous Wiener chaos of order equal to the number of noise nodes in $\tau$, we shall follow the paradigm of \cite{HP15} and use graphical notations in the sequel. First of all, let the node \tikz[baseline=-0.1cm] \node[root] at (0,0) {}; denote the origin $(0,0)$, a node \tikz[baseline=-0.1cm] \node[var] at (0,0) {}; denotes a variable to be integrated against the spatial white noise $\xi$, and a node \tikz[baseline=-0.1cm] \draw (0, 0) to node[dot]{} (0,0); denotes a variable to be integrated out over $\R^2$. Apart from the above nodes, we also introduce the following types of edges (where we use $x = (x_1, x_2), y=(y_1, y_2)$ to denote respectively the coordinates of the starting and the end point of an edge):
\begin{itemize}
	\item \tikz[baseline=-0.1cm] \draw[kernel] (0,0) to (1,0); represents the kernel $K(y-x)$.
	\item \tikz[baseline=-0.1cm] \draw[kernel1] (0,0) to (1,0); represents the recentred kernel $K^{(1)}(y, x) := K(y-x) - K(-x)$.
	\item \tikz[baseline=-0.1cm] \draw[kernel2] (0,0) to (1,0); represents the recentred kernel $K^{(2)}(y, x) := K(y-x) - K(-x) - y \cdot \nabla K(-x)$.
	\item \tikz[baseline=-0.1cm] \draw[dkernel] (0,0) to node[below]{\tiny $j$} (1,0); represents the differentiated kernel $(\partial_j K)(y-x)$; when $j = 1$, we sometimes omit the subscript $j$.
	\item \tikz[baseline=-0.1cm] \draw[dkernel1] (0,0) to node[below]{\tiny\color{black} $j$} (1,0); represents the recentred differentiated kernel $-\partial_{x_j} K^{(1)}(y, x) = (\partial_j K)(y-x) - (\partial_j K)(-x)$; when $j = 1$, we sometimes omit the subscript $j$.
	\item \tikz[baseline=-0.1cm] \draw[dkernel2] (0,0) to node[below]{\tiny\color{black} $j$} (1,0); represents the recentred differentiated kernel $-\partial_{x_j} K^{(2)}(y, x) = (\partial_j K)(y-x) - (\partial_j K)(-x) - y \cdot \nabla (\partial_j K)(-x)$; when $j = 1$, we sometimes omit the subscript $j$.
	\item \tikz[baseline=-0.1cm] \draw[ddkernel] (0,0) to (1,0); represents the two-time differentiated kernel $(\partial_1^2 K)(y-x)$.
	\item \tikz[baseline=-0.1cm] \draw[multx] (1,0) to node[below]{\tiny $j$} (0,0);, \tikz[baseline=-0.1cm] \draw[kernelx] (0,0) to node[below]{\tiny $j$} (1,0); and
	\tikz[baseline=-0.1cm] \draw[dkernelx] (0,0) to node[below]{\tiny $k,j$} (1,0); represent the kernels $y_j - x_j, -(y_j - x_j) K(y - x)$ and $-(y_j - x_j) \partial_k K(y - x)$, respectively. 
	\item \tikz[baseline=-0.1cm] \draw[kernelBig] (0,0) to (1,0); represents the renormalised kernel 
	\begin{equ}[eq:kernel_renorm]
		K(y - x) \partial_1^2 \rho_\eps^{*2}(y - x) - \delta_0(y-x) \int K(z) \partial_1^2 \rho_\eps^{*2}(z) \dd z
	\end{equ}
	\item \tikz[baseline=-0.1cm] \draw[BigG] (0,0) to (1,0); represents the kernel $G_\eps$ defined in \eqref{eq:G} below.
	\item \tikz[baseline=-0.1cm] \draw[rho] (0,0) to (1,0);  represents a mollifier function, which will be either $\rho_\eps$ or $\rho_\eps\ast\rho_\eps$.
	\item \tikz[baseline=-0.1cm] \draw[drho] (0,0) to (1,0); and \tikz[baseline=-0.1cm] \draw[ddrho] (0,0) to (1,0); represent the first and second partial derivatives $\partial_{x_1}$ and $\partial_{x_1}^2$ of a mollifier \tikz[baseline=-0.1cm] \draw[rho] (0,0) to (1,0);, respectively.
	\item \tikz[baseline=-0.1cm] \draw[testfcn] (1,0) to (0,0);,  \tikz[baseline=-0.1cm] \draw[dtestfcn] (1,0) to node[below]{\tiny\color{black} $j$} (0,0); and \tikz[baseline=-0.1cm] \draw[testfcnx] (1,0) to node[below]{\tiny\color{black} $k$} (0,0);
	represent generic test functions $\varphi^\lambda(y-x)$ at scale $\lambda \in (0, 1]$, its $j$-th partial derivative $\partial_j \varphi^\lambda(y - x)$ and the product $(y_k - x_k)\varphi^\lambda(y - x)$, with $\varphi \in \Phi$.  Note that $\tikz[baseline=-0.1cm] \draw[dtestfcn] (1,0) to node[below]{\tiny\color{black} $j$} (0,0); = \lambda^{-1} \tikz[baseline=-0.1cm] \draw[testfcn] (1,0) to (0,0);$ and $\tikz[baseline=-0.1cm] \draw[testfcnx] (1,0) to node[below]{\tiny\color{black} $k$} (0,0); = \lambda \tikz[baseline=-0.1cm] \draw[testfcn] (1,0) to (0,0);$, where the equalities are understood up to change of test functions.
\end{itemize}

{\color{black}In addition, we often encounter kernels of the form $\eps^\gamma K_e$ with some $\gamma>0$ and a kernel $K_e$ from the list above. Such additional $\eps$ powers will be simply written next to the corresponding edges.} 

To represent an element in a Wiener chaos, we select a set of nodes and connect them with edges of the above types: this amounts to multiply all kernels represented by the edges, with arguments given by the coordinates of the two nodes attached to their end points, and then integrate out the node variables accordingly to their types. As an example, the stochastic integral $K\ast\xi (\varphi) = \iint K(x - y) \varphi(x) \dd x \xi(\dd y)$ is represented graphically by
$\begin{tikzpicture}
	\node at (-1, 0) [var] (left) {};
	\node at (0, 0)	[dot] (center) {};
	\node at (1, 0) [root] (right) {};
	\draw[kernel] (left) to (center);
	\draw[testfcn] (center) to (right);
\end{tikzpicture}$. In the sequel, we call a graph formed in this manner a \emph{stochastic graph}. In particular, an element in the $n$-th homogenous chaos is represented by a stochastic graph with exactly $n$ \tikz \node[var] at (0, 0) {}; nodes.

In order to prove the tightness, one has to bound second moments of stochastic graphs. By Wick's product formula, the computation of moments can be performed by taking two copies of the corresponding stochastic graph, forming pairs of the noise nodes $\tikz \node[var] at (0, 0) {};$ in each copy, and contract each pair to a plain \tikz[baseline=-0.1cm] \node[dot] at (0,0) {}; node. This procedure results in what we call the \emph{Feynman graphs/diagrams}: for us, a \emph{Feynman graph} is a connected simple graph $(V, E)$ whose vertex set $V$ consist of nodes of types $\origin$ and \tikz[baseline=-0.1cm] \node[dot] at (0,0) {}; and edge set $E$ of types prescribed above. Given a test function $\varphi$, every Feynman diagram $G = (V, E)$ is associated to an integral of the form
\begin{equ}\label{eq:integral_graph}
	\bI(G, \varphi^\lambda) := \int_{(\R^2)^{V\setminus \{\tikz \node[sroot] at (0, 0) {};\}}} \prod_{e \in E \setminus E_*} K_e(x_{e_+}, x_{e_-}) \prod_{v \in V_\ast\setminus\{\origin\}} \varphi^\lambda(x_v) \dd x\;.
\end{equ}
Here, $K_e$ denotes the kernel associated to the edge $e = (e_-, e_+)$ and $x_{e_-}, x_{e_+}$ denote the coordinates of the starting and the end point of $e$; the subset $V_\ast \subset V$ consists of the origin $\origin$ as well as the vertices associated to the variables that are going to be tested against the test functions; the $E_* \subset E$ is the subset of edges of test function type (see \eqref{eq:edge-classes}). As an example, the second moment of $K*\xi (\varphi^\lambda)$ is given by
\begin{equ}
	\E\left(K*\xi (\varphi^\lambda)\right)^2 = \bI \left(
	\begin{tikzpicture}[scale=0.5, baseline=0.4cm]
		\node[root] at (0, 0) (root) {};
		\node[dot] at (-1, 1) (left) {};
		\node[dot] at (1, 1) (right) {};
		\node[dot] at (0, 2) (top) {};
		
		\draw[testfcn] (left) to (root);
		\draw[testfcn] (right) to (root);
		\draw[kernel] (top) to (left);
		\draw[kernel] (top) to (right);
	\end{tikzpicture}\;, \varphi^\lambda\right)\;.
\end{equ}
When the context is clear, we sometimes will omit $\bI$ and instead use the graph itself to mean its associated integral. 

Note that since some edge types correspond to kernels that are distributions, \eqref{eq:integral_graph} has to be interpreted in a distributional sense.
 As it stands, it is not clear if the integral \eqref{eq:integral_graph} makes sense. The conditions under which it does and gives a bound in terms of $\lambda$ uniformly for $\eps$ will be identified by a well-known result due to Hairer and Quastel \cite[Thm.~A.3]{HQ18}. To state this result, we will associate each Feynman diagram to a labelled graph, with each edge $e \in E$ labelled by a couple $(a_e, r_e)$ with $a_e \in \R, r_e \in \Z$.
\begin{definition}\label{def:faithful}
	We say that the label $(a_e, r_e) \in \R \times \Z$ is \emph{faithful} to the edge $e$, if the following relations hold
	\begin{itemize}
		\item When the edge $e$ is labelled with $(a_e, 0)$, the kernel $K_e$ associated to $e$ is a function smooth everywhere except at the origin $(0,0)$ and satisfies
		\begin{equ}
			\norm{K_e}_{a_e,m} := \sup_{|k| < m} \sup_{0 < |x| \le 1} |x|^{a_e+|k|} |D^k K_e(x)| < \infty,
		\end{equ}
		for some non-negative integer $m$, in which case we say that $K_e$ is of singularity $a_e$ at the origin.
		\item When $e$ has label $(a_e, r_e)$ with $r_e > 0$, then there is a kernel $J_e$ of singularity $a_e$ such that
		\begin{equ}
			K_e(x_{e_+}, x_{e_-}) = J_e(x_{e_+} - x_{e_-}) - \sum_{|k| < r_e} {x_{e_+}^k \over k!} D^k J_e(-x_{e_-})\;,
		\end{equ}
		in which case we say that $K_e$ is the recentred version of $J_e$ up to level $r_e$.
		\item When $e$ has label $(a_e, r_e)$ with $r_e < 0$, then there is a kernel $J_e$ of singularity $a_e$ and  constants $I_{e, k}$ indexed by the multi-indices $k, |k| < |r_e|$, such that the associated kernel $K_e$ (not necessarily integrable) is interpreted as a distribution on $\R^2 \times \R^2$ defined by
		\begin{equ}
			K_e(\psi) = \frac12 \int_{\R^2} \ccR J_e\left(\psi\left(\frac{z + \cdot}{2}, \frac{z - \cdot}{2}\right)\right) \dd z 
		\end{equ}
		for all smooth test functions $\psi$ on $\R^2 \times \R^2$, where 
		\begin{equ}
			\ccR J_e (\varphi) = \int J_e(x) \Bigl(\varphi(x) - \sum_{|k| < |r_e|} {x^k \over k!} D^k\varphi(0)\Bigr) \dd x + \sum_{|k| < |r_e|} {I_{e,k} \over k!} D^k\varphi(0)\;,
		\end{equ}
		for all smooth test functions $\varphi$ on $\R^2$. In this case, we say that the kernel $K_e$ is the renormalised version of $J_e$ up to order $|r_e|$. The above definition is devised so that if $J_e$ is a integrable kernel with the property $\int |J_e(x)||x|^k < \infty$ for $|k| < r_e$ and if $\psi(x_1, x_2) = \varphi(x_1) \varphi(x_2)$, then with the choice $I_{e, k} = \int J_e(x) x^k$ one has
		\begin{equ}
			K_e(\psi) = \iint J_e(x_{e_+} - x_{x_-})\varphi(x_{e_+}) \varphi(x_{e_-}) \dd x_{e_+} \dd x_{e_-}\;.
		\end{equ}
	\end{itemize}
	{\color{black}We say a labelling of a Feynman graph $G$ is faithful, if the labels of all edges are faithful. We say a labelling of a Feynman graph of the form $\eps^\gamma G$ is faithful if there exists a Feynman graph $G'$ with the same vertex and edge set as $G$ but possibly different edge types: 
	each edge kernel $K_e'$ is of the form $K_e'=\eps^{\gamma_e}K_e$, with $K_e$ being the kernel of the corresponding edge in $G$, $\gamma_e\geq 0$ and $\sum \gamma_e=\gamma$, such that the labelling is faithful to $G'$.}
\end{definition}

Let us list out one possible faithful labelling of available edge types:
\begin{equs}[eq:labels]
	\;&\tikz[baseline=-0.6cm] \draw[kernel] (0,0) -- (0,-1); \to \tikz[baseline=-0.6cm] \draw (0,0) to node[labl] {\tiny $\bar\kappa$,0} (0,-1); \;,\qquad
	\tikz[baseline=-0.6cm] \draw[kernel1] (0,0) -- (0,-1); \to \tikz[baseline=-0.6cm] \draw (0,0) to node[labl] {\tiny $\bar\kappa$,1} (0,-1); \;,\qquad
	\tikz[baseline=-0.6cm] \draw[kernel2] (0,0) -- (0,-1); \to \tikz[baseline=-0.6cm] \draw (0,0) to node[labl] {\tiny $\bar\kappa$,2} (0,-1); \;,\qquad
	\tikz[baseline=-0.6cm] \draw[dkernel] (0,0) -- (0,-1); \to \tikz[baseline=-0.6cm] \draw (0,0) to node[labl] {\tiny 1,0} (0,-1); \;,\qquad
	\tikz[baseline=-0.6cm] \draw[dkernel1] (0,0) -- (0,-1); \to \tikz[baseline=-0.6cm] \draw (0,0) to node[labl] {\tiny 1,1} (0,-1); \;, \qquad
	\tikz[baseline=-0.6cm] \draw[dkernel2] (0,0) -- (0,-1); \to \tikz[baseline=-0.6cm] \draw (0,0) to node[labl] {\tiny 1,2} (0,-1); \;, \\
	&\tikz[baseline=-0.6cm] \draw[ddkernel] (0,0) -- (0,-1); \to \tikz[baseline=-0.6cm] \draw (0,0) to node[labl] {\tiny 2,-1} (0,-1); \;,\qquad
	\tikz[baseline=-0.6cm] \draw[multx] (0,-1) -- (0,0); \to \tikz[baseline=-0.6cm] \draw (0,0) to node[labl] {\tiny -1,0} (0,-1); \;,\qquad
	\tikz[baseline=-0.6cm] \draw[kernelx] (0,0) -- (0,-1); \to \tikz[baseline=-0.6cm] \draw (0,0) to node[labl] {\tiny -1+$\bar\kappa$,0} (0,-1); \;,\qquad
	\tikz[baseline=-0.6cm] \draw[dkernelx] (0,0) -- (0,-1); \to \tikz[baseline=-0.6cm] \draw (0,0) to node[labl] {\tiny 0,0} (0,-1); \;,\\
	&\tikz[baseline=-0.6cm]{\draw[rho] (0, 0) to (0,-1);} \to \tikz[baseline=-0.6cm] \draw (0,0) to node[labl] {\tiny 2,-1} (0,-1); \;,\qquad
	\tikz[baseline=-0.6cm]{\draw[drho] (0, 0) to (0,-1);} \to \tikz[baseline=-0.6cm] \draw (0,0) to node[labl] {\tiny 3,-2} (0,-1); \;,\qquad
	\tikz[baseline=-0.6cm]{\draw[ddrho] (0, 0) to (0,-1);} \to \tikz[baseline=-0.6cm] \draw (0,0) to node[labl] {\tiny 4,-3} (0,-1); \;, \qquad
	\tikz[baseline=-0.6cm] \draw[kernelBig] (0,0) -- (0,-1); \to \tikz[baseline=-0.6cm] \draw (0,0) to node[labl] {\tiny 4+$\bar\kappa$,-2} (0,-1); \;, \qquad
\end{equs}
where we choose $\bar\kappa$ to be sufficiently small, say $\bar\kappa \in (0, \kappa/10)$ is sufficient.
Moreover, the test function edge \tikz[baseline=-0.1cm] \draw[testfcn] (1,0) to (0,0); is always attributed a label $(0,0)$ and \tikz[baseline=-0.1cm] \draw[testfcnx] (1,0) to node[below]{\tiny\color{black} $j$} (0,0); a label $(-1,0)$. Let us comment on the choice of labels for edges with $r_e<0$:
\begin{itemize}
	\item For $e = \ddkernel$, we have fixed $r_e = -1$ with $I_{e, (0, 0)} = \int \partial_1^2 K = 0$.
	\item For $e = \erho$, we fixed $r_e = -1$ with $I_{e, (0, 0)} = \int \rho_\eps = 1$.
	\item For $e = \drho$, we fixed $r_e = -2$ with $I_{e, (0, 0)} = \int \partial_1\rho_\eps = 0$ and $I_{e, k} = \int x^k \partial_1 \rho_\eps = \un_{k = (1, 0)}$.
	\item For $e = \ddrho$, we fixed $r_e = -3$ with $I_{e, k} = 2\un_{k = (2, 0)}$ for all $|k| \leq 2$, using the observation $\int x^k \partial_1^2\rho_\eps^{*2}(x) \dd x = 0$ for all $|k| \leq 2, k \neq (2, 0)$ and $\int x_1^2 \partial_1^2\rho_\eps^{*2}(x) \dd x = 2$.
	\item For $e = \kernelBig$, we fixed $r_e = -2$ using the definition of the associated distribution \eqref{eq:kernel_renorm}, thus $I_{e, (0, 0)} = 0$, as well as the fact that $K$ and $\rho_\eps$ are even functions, so that and $I_{e, k} = \int x^k K(x) \partial_1^2 \rho_\eps^{*2}(x) \dd x = 0$ for $|k| = 1$.
\end{itemize}

Now note that as our driving noise in \eqref{eq:gPAM_regular} is $\eps^\frac12 \partial_1 \xi_\eps$, every stochastic graph with $n$ noise nodes will be accompanied by a prefactor $\eps^{n \over 2}$, which implies that Feynman graphs resulting from the squaring process would have $\eps^n$ in disposal. Furthermore, let us observe that the singularity index $a_e$ for the mollifier can be improved by multiplying with an $\eps$-prefactor: for $\gamma \ge 0$ and any multiindex $k$,
\begin{equ}
	|\eps^{\gamma} D^k \rho_\eps(x)| \lesssim (|x| + \eps)^{-2-|k| + \gamma}\;.
\end{equ}
Therefore, for $\gamma>0$ the following label assignment is also faithful:
\begin{equ}[eq:labels_eps]
	\eps^\gamma \; \tikz[baseline=-0.6cm]{\draw[rho] (0, 0) to (0,-1);} \to \tikz[baseline=-0.6cm] \draw (0,0) to node[labl] {\tiny 2-$\gamma$,0} (0,-1); \;,\qquad
	\eps^\gamma \; \tikz[baseline=-0.6cm]{\draw[drho] (0, 0) to (0,-1);} \to \tikz[baseline=-0.6cm] \draw (0,0) to node[labl] {\tiny 3-$\gamma$,-1} (0,-1); \;,\qquad
	\eps^\gamma\; \tikz[baseline=-0.6cm]{\draw[ddrho] (0, 0) to (0,-1);} \to \tikz[baseline=-0.6cm] \draw (0,0) to node[labl] {\tiny 4-$\gamma$,-2} (0,-1); \;, \qquad
	\eps^\gamma\; \tikz[baseline=-0.6cm] \draw[kernelBig] (0,0) -- (0,-1); \to \tikz[baseline=-0.6cm] \draw (0,0) to node[labl] {\tiny 4+$\bar\kappa$-$\gamma$,-2} (0,-1); \;.
\end{equ}
In other words, we have a rather great freedom to distribute the available $\eps$-powers to the mollifier-type edges in a Feynman graph.  We emphasise that unlike in previous works that we are aware of, the distribution of the $\eps$ powers will be chosen \emph{after} the Wick pairing, and different pairings will enforce different choices of which mollifier we want to make more or less regular.

Given a subset of vertices $\bar V \subset V$, we denote by $|\bar V|$ its cardinality and we define the following edge sets
\begin{equs}
	E^\uparrow(\bar V) &= \{e  \in E\,:\, e\cap \bar V = e_- \;, \; r_e > 0\}\;,\\
	E^\downarrow(\bar V) &= \{e  \in E\,:\, e\cap \bar V = e_+ \;, \; r_e > 0\}\;,\\
	E_0(\bar V) &= \{e  \in E\,:\, e\cap \bar V = e\}\;,\\
	E(\bar V) &= \{e  \in E\,:\, e\cap \bar V \neq \emptyset\}\;.
\end{equs}
Let us also denote by $V_\ast$ the set of vertices that are endpoints of a test function edge, i.e., $\testfcn, \testfcnx$, or $\dtestfcn$. Now we can state the result \cite[Thm.~A.3]{HQ18}.
\begin{theorem}[\cite{HQ18}]\label{thm:power_counting}
	Suppose that the labelled graph $G = (V,E)$ satisfies the following properties:
	\begin{itemize}\itemsep0em
		\item[0.] No edge with $r_e \ne 0$ connects two elements in $V_\ast$ and $\origin \in e \implies r_e = 0$; no more than one edge with negative renormalization $r_e < 0$ may emerge from the same vertex.
		\item[1.] For every edge $e \in E$, one has $a_e + r_e \wedge 0 < 2$.
		\item[2.] For every subset
		$\bar V \subset V \setminus \{\origin\}$ with $|\bar V|\geq 3$ 	, one has 
		\begin{equ}[eq:power_counting-2]
			\sum_{e \in E_0(\bar V)} a_e < 2(|\bar V| - 1)\;.
		\end{equ}
		\item[3.] For every subset
		$\bar V \subset V$ with $|\bar V|\geq 2$ and containing $\origin$, one has 
		\begin{equ}[eq:power_counting-3]
			\sum_{e \in E_0(\bar V)} a_e + \sum_{e \in E^\uparrow(\bar V)}(a_e + r_e - 1) - \sum_{e \in E^\downarrow(\bar V)} r_e < 2(|\bar V| - 1)\;.
		\end{equ}
		\item[4.] For every non-empty subset $\bar V \subset V \setminus V_\ast$,
		one has the bounds
		\begin{equ}[eq:power_counting-4]
			\sum_{e \in E(\bar V)\setminus E^\downarrow(\bar V)} a_e  + \sum_{e \in E^\uparrow(\bar V)} r_e- \sum_{e \in E^\downarrow(\bar V)} (r_e-1) > 2|\bar V| \;.
		\end{equ}
	\end{itemize}
	Then, there exist $m > 0$ such that for all $M$ there exists a constant $C$ depending on $M$ and $|V|$, such that if $\|K_e\|_{a_e,m}, I_{e,k}\leq M$ then one has 
	\begin{equ}[eq:bound_power_counting]
		\bI(G, \varphi^\lambda) \leq C  \lambda^\alpha
	\end{equ}
	for all $\lambda \in (0,1]$ and $\varphi\in\Phi$, where $\alpha = 2|V \setminus V_\ast| - \sum_{e \in E} a_e$.
\end{theorem}

\begin{remark}
	The conditions \eqref{eq:power_counting-2}-\eqref{eq:power_counting-3}-\eqref{eq:power_counting-4} are imposed on different families of subgraphs, as a shorthand we call these families \emph{relevant} for the given condition.
\end{remark}
\begin{remark}
	In the literature often a slightly different version of the bound \eqref{eq:bound_power_counting} appears, with the right-hand side being multiplied by $\prod_{e\in E} \|K_e\|_{a_e,m}$. This form is not quite correct, such linear dependence on the kernel norms only hold for edges with $r_e\geq 0$ or with $r_e<0$ and $I_{e,k}\equiv 0$. A simple counter-example is 
	\begin{equ}
		\begin{tikzpicture}[baseline=0.3cm, scale=0.5]
			\node at (0, 0) [root] (root) {};
			\node at (1, 1.73) [dot] (dot1) {};
			\node at (-1, 1.73) [dot] (dot2) {};
			\draw[testfcn] (dot1) to (root);
			\draw[rho] (dot1) to (dot2);
			\draw[kernel] (dot2) to (root);
		\end{tikzpicture}
		\;=\;
		\begin{tikzpicture}[baseline=0.3cm, scale=0.5]
			\node at (0, 0) [root] (root) {};
			\node at (1, 1.73) [dot] (dot1) {};
			\node at (-1, 1.73) [dot] (dot2) {};
			\draw[dist] (dot1) to (root);
			\draw[generic] (dot1) to node[labl]{\tiny2+$\bar\kappa$,-1} (dot2);
			\draw[generic] (dot2) to node[labl]{\tiny$\bar\kappa$,0} (root);
		\end{tikzpicture}\;,
	\end{equ}
	where we have labelled the mollifier by $(2+\bar\kappa, -1)$ with $\bar \kappa > 0$ and $I_{e, (0,0)} = 1$. One can check Theorem \ref{thm:power_counting} does apply, and that $\norm{\rho_\eps}_{2+\bar\kappa, m} \lesssim \eps^{\bar\kappa}$ for all $m$. However, it is of course not true that the limit of the graph, which is $\int \varphi(x) K(-x) \dd x$, vanishes.
\end{remark}

In the sequel, Theorem \ref{thm:power_counting} will serve as our principal tool to check the convergence of Feynman graphs.
In other words, the question is whether we can distribute the $\eps$'s in a way that the assumptions of Theorem \ref{thm:power_counting} be satisfied, and a small power of $\eps$ is left over. The answer turns out to be positive for the majority of Feynman graphs we will encounter. These Feynman graphs will converge to $0$ as $\eps \to 0$. On the other hand, those graphs which fail the conditions of Theorem \ref{thm:power_counting} even after $\eps$-distribution can have non-trivial limit. We call these graphs the \emph{critical graphs}, and they will contribute to the non-vanishing terms in \eqref{eq:area-model}. 

In Section \ref{sec:new_noise}, we will demonstrate how we use the graphical tools introduced in this section to treat the most basic trees. In Section \ref{sec:vanishing_trees}, we show that for the symbols that do not give rise to any critical graph, the argument can be greatly simplified.
In Section \ref{sec:graph_argument}, we move on to the more difficult cases and establish how the conditions of Theorem \ref{thm:power_counting} can be verified. Finally in Section \ref{sec:critical-graphs}, we will investigate the bounds on the critical graphs as well as their limits. All this will finally lead to the proof of Theorem \ref{thm:convergence} in Section \ref{sec:conclusion}.

\subsection{The new white noise}\label{sec:new_noise}

Let us first work out the convergence of $\<XiIXi>$ using the tools presented above. Notice that the BPHZ model $\hat\Pi^{\eps}$ gives

\begin{equ}[eq:XiIXi]
	\hat\Pi^{\eps}_0 \<XiIXi> (\varphi^\lambda) = \; \eps \left(
	\begin{tikzpicture}[scale=0.35,baseline=0.3cm]
		\node at (0,-1)  [root] (root) {};
		\node at (-2,3)  [dot] (left2) {};
		\node at (-2,1)  [dot] (left1) {};
		\node at (0,1) [var] (variable1) {};
		\node at (0,3) [var] (variable2) {};
		
		\draw[testfcn] (left1) to  (root);
		
		\draw[drho] (variable2) to (left2);
		\draw[drho] (variable1) to (left1);
		\draw[kernel1] (left2) to (left1);
	\end{tikzpicture}\;
	- \;
	\begin{tikzpicture}[scale=0.35,baseline=0cm]
		\node at (0,-1)  [root] (root) {};
		\node at (-1,1)  [dot] (left) {};
		\node at (1,1)  [dot] (right) {};
		
		\draw[testfcn] (left) to  (root);
		
		\draw[ddrho] (left) to (right); 
		\draw[kernel] (right) to (root);
	\end{tikzpicture}\right)\;.
\end{equ}

\subsubsection*{Tightness bound}
Observe that the second term can be made to satisfy the assumptions of Theorem \ref{thm:power_counting} by a simple distribution of $\eps$-s. Indeed, by writing
\begin{equ}
	\eps \;
	\begin{tikzpicture}[scale=0.3,baseline=0.1cm]
		\node at (0,-2)  [root] (root) {};
		\node at (-2,2)  [dot] (left) {};
		\node at (2,2)  [dot] (right) {};
		
		\draw[testfcn] (left) to  (root);
		
		\draw[ddrho] (left) to (right); 
		\draw[kernel] (right) to (root);
	\end{tikzpicture}
	= \eps^{\bar\kappa}\;
	\begin{tikzpicture}[scale=0.3,baseline=0.1cm]
		\node at (0,-2)  [root] (root) {};
		\node at (-2,2)  [dot] (left) {};
		\node at (2,2)  [dot] (right) {};
		
		\draw[testfcn] (left) to  (root);
		\draw[generic] (left) to node[labl,pos=0.5] {\tiny 3+$\bar\kappa$,-2} (right);  
		\draw[generic] (right) to node[labl,pos=0.5] {\tiny $\bar\kappa$,0} (root);
	\end{tikzpicture}
	\lesssim \eps^{\bar\kappa} \lambda^{-1- 2\bar\kappa},
\end{equ} 
we deduce that the distribution $\zeta_\eps$ represented by this graph satisfies the tightness bound \eqref{eq:second_moment_model} with $\beta= |\<XiIXi>|+2\kappa-2\bar \kappa$ uniformly in $\eps\in(0,1]$, and that $\zeta_\eps(\varphi)\to 0$ for all $\varphi$.
As for the first term in \eqref{eq:XiIXi}, computing its  second moment gives 
\begin{equ}\label{eq:dumbbell_var}
	\E\left(\eps \;\;
	\begin{tikzpicture}[scale=0.35,baseline=0.3cm]
		\node at (0,-1)  [root] (root) {};
		\node at (-2,3)  [dot] (left2) {};
		\node at (-2,1)  [dot] (left1) {};
		\node at (0,1) [var] (variable1) {};
		\node at (0,3) [var] (variable2) {};
		
		\draw[testfcn] (left1) to  (root);
		
		\draw[drho] (variable2) to (left2);
		\draw[drho] (variable1) to (left1);
		\draw[kernel1] (left2) to (left1);
	\end{tikzpicture}\;\right)^2
	= \eps^2 \left(\begin{tikzpicture}[scale=0.35,baseline=0.3cm]
		\node at (0,-1)  [root] (root) {};
		\node at (-2,1)  [dot] (left) {};
		\node at (-2,3)  [dot] (left1) {};
		\node at (2,1)  [dot] (right) {};
		\node at (2,3)  [dot] (right1) {};
		
		\draw[testfcn] (left) to  (root);
		\draw[testfcn] (right) to  (root);
		
		\draw[kernel1] (left1) to (left);
		\draw[ddrho] (left1) to (right1);
		\draw[ddrho] (left) to (right); 
		\draw[kernel1] (right1) to (right);
	\end{tikzpicture}\;
	+ \;
	\begin{tikzpicture}[scale=0.35,baseline=0.3cm]
		\node at (0,-1)  [root] (root) {};
		\node at (-2,1)  [dot] (left) {};
		\node at (-2,3)  [dot] (left1) {};
		\node at (2,1)  [dot] (right) {};
		\node at (2,3)  [dot] (right1) {};
		
		\draw[testfcn] (left) to  (root);
		\draw[testfcn] (right) to  (root);
		
		\draw[kernel1] (left1) to (left);
		\draw[ddrho-shift] (left1) to (right);
		\draw[ddrho-shift] (right1) to (left); 
		\draw[kernel1] (right1) to (right);
	\end{tikzpicture}\right)\,.
\end{equ}
Here, we see that both Feynman graphs fail the condition \eqref{eq:power_counting-2} in all possible ways of distributing $\eps$-s. These are the first examples of the aforementioned critical graphs. However, by performing integration by parts (henceforth abbreviated as IBP) once at each node in the graph, the sum can be decomposed into
\begin{gather}
	\eps^{2} \left(
	\begin{tikzpicture}[scale=0.3,baseline=0.3cm]
		\node at (0,-1)  [root] (root) {};
		\node at (-2,1)  [dot] (left) {};
		\node at (-2,3)  [dot] (left1) {};
		\node at (2,1)  [dot] (right) {};
		\node at (2,3)  [dot] (right1) {};		
		\draw[testfcn] (left) to  (root);
		\draw[testfcn] (right) to  (root);		
		\draw[ddkernel] (left1) to (left);
		\draw[rho] (left1) to (right1);
		\draw[rho] (left) to (right); 
		\draw[ddkernel] (right1) to (right);
	\end{tikzpicture}
	\;+\;
	\begin{tikzpicture}[scale=0.3,baseline=0.3cm]
		\node at (0,-1)  [root] (root) {};
		\node at (-2,1)  [dot] (left) {};
		\node at (-2,3)  [dot] (left1) {};
		\node at (2,1)  [dot] (right) {};
		\node at (2,3)  [dot] (right1) {};		
		\draw[testfcn] (left) to  (root);
		\draw[testfcn] (right) to  (root);		
		\draw[ddkernel] (left1) to (left);
		\draw[rho] (left1) to (right);
		\draw[rho] (left) to (right1); 
		\draw[ddkernel] (right1) to (right);
	\end{tikzpicture}
	\right)\label{eq:first-IBP1}
	\\
	+\eps^2\left(
	\;\;2\;
	\begin{tikzpicture}[scale=0.3,baseline=0.3cm]
		\node at (0,-1)  [root] (root) {};
		\node at (-2,1)  [dot] (left) {};
		\node at (-2,3)  [dot] (left1) {};
		\node at (2,1)  [dot] (right) {};
		\node at (2,3)  [dot] (right1) {};		
		\draw[testfcn] (left) to  (root);
		\draw[dtestfcn] (right) to  (root);	
		\draw[ddkernel] (left1) to (left);
		\draw[rho] (left1) to (right1);
		\draw[rho] (left) to (right); 
		\draw[dkernel1] (right1) to (right);
	\end{tikzpicture}
	\;+\;2\;
	\begin{tikzpicture}[scale=0.3,baseline=0.3cm]
		\node at (0,-1)  [root] (root) {};
		\node at (-2,1)  [dot] (left) {};
		\node at (-2,3)  [dot] (left1) {};
		\node at (2,1)  [dot] (right) {};
		\node at (2,3)  [dot] (right1) {};		
		\draw[testfcn] (left) to  (root);
		\draw[dtestfcn] (right) to  (root);		
		\draw[ddkernel] (left1) to (left);
		\draw[rho] (left1) to (right);
		\draw[rho] (left) to (right1); 
		\draw[dkernel1] (right1) to (right);
	\end{tikzpicture}
	\;+\;
	\begin{tikzpicture}[scale=0.3,baseline=0.3cm]
		\node at (0,-1)  [root] (root) {};
		\node at (-2,1)  [dot] (left) {};
		\node at (-2,3)  [dot] (left1) {};
		\node at (2,1)  [dot] (right) {};
		\node at (2,3)  [dot] (right1) {};		
		\draw[dtestfcn] (left) to  (root);
		\draw[dtestfcn] (right) to  (root);		
		\draw[dkernel1] (left1) to (left);
		\draw[rho] (left1) to (right1);
		\draw[rho] (left) to (right); 
		\draw[dkernel1] (right1) to (right);
	\end{tikzpicture}
	+
	\begin{tikzpicture}[scale=0.3,baseline=0.3cm]
		\node at (0,-1)  [root] (root) {};
		\node at (-2,1)  [dot] (left) {};
		\node at (-2,3)  [dot] (left1) {};
		\node at (2,1)  [dot] (right) {};
		\node at (2,3)  [dot] (right1) {};		
		\draw[dtestfcn] (left) to  (root);
		\draw[dtestfcn] (right) to  (root);		
		\draw[dkernel1] (left1) to (left);
		\draw[rho] (left1) to (right);
		\draw[rho] (left) to (right1); 
		\draw[dkernel1] (right1) to (right);
	\end{tikzpicture}\right)\label{eq:first-IBP2}
\end{gather}
First note that each instance of $\dtestfcn$ can be replaced by $\testfcn$ at the cost of a $\lambda^{-1}$ prefactor.
We claim that the graphs resulting from \eqref{eq:first-IBP2} (after the aforementioned replacement) together with a prefactor $\eps^{2-2\bar \kappa}$, with $\bar\kappa>0$ sufficiently small, satisfy the conditions of Theorem \ref{thm:power_counting}.
Indeed, distributing $\eps^{1-\bar \kappa}$ to every mollifier edge, \eqref{eq:first-IBP2} can be rewritten as the labelled graphs
\begin{equs}
	\;\;2\eps^{2\bar\kappa}\lambda^{-1}\;
	\begin{tikzpicture}[scale=0.3,baseline=0.3cm]
		\node at (0,-1)  [root] (root) {};
		\node at (-2,1)  [dot] (left) {};
		\node at (-2,3)  [dot] (left1) {};
		\node at (2,1)  [dot] (right) {};
		\node at (2,3)  [dot] (right1) {};		
		\draw[dist] (left) to  (root);
		\draw[dist] (right) to  (root);	
		\draw[generic] (left1) to node[labl]{\tiny 2,-1} (left);
		\draw[generic] (left1) to node[labl]{\tiny 1+$\bar\kappa$,0} (right1);
		\draw[generic] (left) to node[labl]{\tiny 1+$\bar\kappa$,0} (right); 
		\draw[->] (right1) to node[labl]{\tiny 1,1} (right);
	\end{tikzpicture}
	\;+2\eps^{2\bar\kappa}\lambda^{-1}\;
	\begin{tikzpicture}[scale=0.3,baseline=0.3cm]
		\node at (0,-1)  [root] (root) {};
		\node at (-2,1)  [dot] (left) {};
		\node at (-2,3)  [dot] (left1) {};
		\node at (2,1)  [dot] (right) {};
		\node at (2,3)  [dot] (right1) {};		
		\draw[dist] (left) to  (root);
		\draw[dist] (right) to (root);		
		\draw[generic] (left1) to node[labl]{\tiny 2,-1} (left);
		\draw[generic] (left1) to node[labl, pos=0.45, above]{\tiny 1+$\bar\kappa$,0} (right);
		\draw[generic] (left) to node[labl, pos=0.45, below]{\tiny 1+$\bar\kappa$,0} (right1); 
		\draw[->] (right1) to node[labl]{\tiny 1,1} (right);
	\end{tikzpicture}
	\;+\eps^{2\bar\kappa}\lambda^{-2}\;
	\begin{tikzpicture}[scale=0.3,baseline=0.3cm]
		\node at (0,-1)  [root] (root) {};
		\node at (-2,1)  [dot] (left) {};
		\node at (-2,3)  [dot] (left1) {};
		\node at (2,1)  [dot] (right) {};
		\node at (2,3)  [dot] (right1) {};		
		\draw[dist] (left) to  (root);
		\draw[dist] (right) to  (root);		
		\draw[->] (left1) to node[labl]{\tiny 1,1} (left);
		\draw[generic] (left1) to node[labl]{\tiny 1+$\bar\kappa$,0} (right1);
		\draw[generic] (left) to node[labl]{\tiny 1+$\bar\kappa$,0} (right); 
		\draw[->] (right1) to node[labl]{\tiny 1,1} (right);
	\end{tikzpicture}
	\;+\eps^{2\bar\kappa}\lambda^{-2}\;
	\begin{tikzpicture}[scale=0.3,baseline=0.3cm]
		\node at (0,-1)  [root] (root) {};
		\node at (-2,1)  [dot] (left) {};
		\node at (-2,3)  [dot] (left1) {};
		\node at (2,1)  [dot] (right) {};
		\node at (2,3)  [dot] (right1) {};		
		\draw[dist] (left) to  (root);
		\draw[dist] (right) to  (root);		
		\draw[->] (left1) to node[labl]{\tiny 1,1} (left);
		\draw[generic] (left1) to node[labl, pos=0.45, above]{\tiny 1+$\bar\kappa$,0} (right);
		\draw[generic] (left) to node[labl, pos=0.45, below]{\tiny 1+$\bar\kappa$,0} (right1); 
		\draw[->] (right1) to node[labl]{\tiny 1,1} (right);
	\end{tikzpicture}\;.
\end{equs}
One can check that all assumptions of Theorem \ref{thm:power_counting} are comfortably satisfied (in fact at this stage even using the singular kernel convolution rules from \cite[Sec.~10.3]{H0} would suffice).
As a consequence, \eqref{eq:first-IBP2} is bounded by a quantity of order $\eps^{2\bar\kappa} \lambda^{-2-2\bar\kappa}$, as desired.

On the other hand, the Feynman diagrams in \eqref{eq:first-IBP1}, even with the full prefactor $\eps^2$, fail the assumption \eqref{eq:power_counting-2}.
Taking $\bar V = V \setminus \{ \origin\}$, the best we can get being an equality, regardless of how the $\eps$ factors are distributed to the mollifier edges.
Although the integration by parts only seems to transform Feynman graphs failing to satisfy the assumptions to other failing graphs, we are actually in a better position. Indeed, the sum of the above critical graphs can be written as the integral
\begin{equs}
	\iint &G_\eps(x - y) \varphi^\lambda(x) \varphi^\lambda(y) \dd x \dd y\;, \label{eq:critical_integral}\\
	\intertext{where}
	G_\eps :=\; &\eps^{2} (\partial_{x_1}^2 K * \rho_\eps)^{*2} \cdot \rho_\eps^{*2} + \eps^{2} (\partial_{x_1}^2 K *\rho_\eps^{*2})^2\;.\label{eq:G}
\end{equs}
Using properties of functions with prescribed singularities \cite[Lem.~10.14,~10.17]{H0} and the fact that the kernels $\partial_{x_1} K$ and $\eps^{\bar\kappa} \partial_{x_1}^2 (K*\rho_\eps)$ have respectively singularity of order $-1$ and $-2+\bar\kappa$ at the origin, one can check that (writing $K_\eps$ as a shorthand for $K\ast\rho_\eps$)
\begin{equs}
	\eps^{2} |(\partial_{x_1}^2 K * \rho_\eps)^{*2}(x)| &= \eps^{2-\bar\kappa} |\partial_{x_1} (\partial_{x_1} K_\eps \ast \eps^{\bar\kappa} \partial_{x_1}^2 K_\eps)(x)| \lesssim \eps^{2-\bar\kappa} (|x| + \eps)^{-2+\bar\kappa}\;,\\
	|\partial_{x_1}^2 K * \rho_\eps^{\ast 2}(x)| &= |\partial_{x_1} (\partial_{x_1} K * \rho_\eps^{\ast 2})(x)| \lesssim (|x|+\eps)^{-2}\;.
\end{equs}
In particular, with the support property of $\rho_\eps$, one deduces the bound
\begin{equ}[eq:G_bound]
	|G_\eps(x)| \lesssim \eps^2 (|x| + \eps)^{-4}
\end{equ}
uniformly for $\eps$.
Therefore, $\|G_\eps\|_{L^1}\lesssim 1$.
 As a consequence, the integral \eqref{eq:critical_integral} is then upper bounded by $C \lambda^{-2} \norm{G_\eps}_{L^1} \norm{\varphi}_{L^1} \norm{\varphi}_{L^{\infty}}$ uniformly in $\eps\in(0,1]$. Combining with the previous bounds, we therefore conclude
\begin{equ}
	\sup_{\eps \in (0, 1]} \sup_{\lambda \in (0, 1]} \E \left(\frac{\hat\Pi^{\eps}_0 \<XiIXi> (\varphi_0^\lambda)}{\lambda^{-1-\kappa}}\right)^2 \lesssim 1\;,
\end{equ}
for we have chosen $10\bar\kappa<\kappa$.

\begin{remark}\label{rem:more_eps_power}
For later use we remark that if one replaces $\eps$ in \eqref{eq:XiIXi} by $\eps^{1+\kappa'}$ for some $\kappa'>0$, then the above argument can be repeated and in fact simplified, because with $\eps^{2+2\kappa'}$ in place of $\eps^2$, both graphs in \eqref{eq:dumbbell_var} can be made to satisfy the conditions of Theorem \ref{thm:power_counting}. This results in the uniform in $\eps$ bound
\begin{equ}
	\sup_{\eps \in (0, 1]} \sup_{\lambda \in (0, 1]} \E \left|\frac{\eps^{\kappa'}\hat\Pi^{\eps}_0 \<XiIXi> (\varphi_0^\lambda)}{\lambda^{-1-\kappa+\kappa'}}\right|^p \lesssim 1\;,
\end{equ}
for any $p\geq 1$, using hypercontractivity in moving from $p=2$ to general $p$.
\end{remark}

\subsubsection{Limit identification}
Let us assume for the moment that the tightness bound \eqref{eq:second_moment_model} holds for all $\tau \in V^\RHS$. Lemma \ref{lem:tightness} then implies the existence of a limiting in law random model $(\hat \Pi,\hat \Gamma)$ along a subsequence denoted still by $\hat\Pi^{\eps}$. We now would like to identify the law of $\hat \Pi_0 \<XiIXi>$.
For this purpose $\lambda$ plays no role and thus we fix $\lambda=1$ and drop it from the notation.

We start by characterising the covariance function of the random field. By the bounds in the previous subsection (with $\lambda = 1$ fixed), one knows
\begin{equ}\label{eq:G-intermediate}
	\lim_{\eps \to 0}\left(\E (\hat\Pi^{\eps}_0 \<XiIXi> (\varphi))^2 - \iint G_\eps(x - x') \varphi(x) \varphi(x') \dd x \dd x'\right) = 0\;
\end{equ}
with $G_\eps$ given by \eqref{eq:G}. The following lemma will be useful in the sequel.
\begin{lemma}\label{lem:G}
	The sequence of function $(G_\eps)_{\eps > 0}$ satisfies the approximation of unity property:
	\begin{enumerate}
		\item $\sup_{\eps>0} \int |G_\eps| < \infty$.
		\item There exists a constant $c_\rho^2 > 0$ depending only on the mollifier $\rho$, such that $\lim_{\eps\to0} \int G_\eps = c_\rho^2$.
		\item For every $\delta > 0$, one has $\lim_{\eps\to0} \int_{|x| > \delta} |G_\eps(x)| \dd x = 0$.
	\end{enumerate}
	In particular, $G_\eps * f \to c_\rho^2 f$ in $L^2$ for all $f \in L^2$, and $G_\eps$ converges to $c_\rho^2 \delta_0$ in the sense of distributions.
\end{lemma}
\begin{proof}
	Note that the properties 1 and 3 readily follow from the bound \eqref{eq:G_bound}. 
	To prove the property 2, recall $\hat K(x)=\frac{-\log|x|}{2\pi}$ satisfies that $R = K-\hat K$ is smooth and vanishes in a neighbourhood of the origin. Define
	\begin{equ}[eq:G_hat]
		\hat G_\eps = \eps^{2} \Big((\partial_{x_1}^2 \hat K_\eps)^{*2} \rho_\eps^{*2} +  (\partial_{x_1}^2 \hat K_{\eps,\eps})^2\Big)\;,
	\end{equ}
	where $\hat K_\eps = \hat K \ast \rho_\eps$ and $\hat K_{\eps,\eps}=\hat K\ast\rho_\eps\ast\rho_\eps$. Using the scaling relation $\partial_{x_1}^2\hat K(\eps x)=\eps^{-2}\partial_{x_1}^2\hat K(x)$,
	one has 
	\begin{equ}\label{eq:c-whitenoise}
		\int_{\R^2} \hat G_\eps(y)\dd y=\int_{\R^2} \hat G_1(y)\dd y =: c_\rho^2\;.
	\end{equ}
	The above integral is absolutely convergent since $\hat G_\eps$ also verifies the bound \eqref{eq:G_bound}. To see that the value of the integral is indeed strictly positive (thus justifying the notation $c_\rho^2$), it suffices to notice that the second term of \eqref{eq:G_hat} (with $\eps=1$) is clearly positive and that, by passing to Fourier domain (with the normalisation convention $\cF f(s)=\int e^{-is\cdot x}f(x)\dd x$), the first term is integrated to
	\begin{equ}
		\int \cF\big((\partial_{x_1}^2 \hat K_1)^{\ast 2}\big)(s) \cF(\rho^{*2})(-s) \frac{\dd s}{(2\pi)^2} = \frac{1}{4\pi^2}\int \frac{s_1^4}{|s|^4} \big|\cF \rho(s)\big|^4 \dd s > 0\;. 
	\end{equ}
	
	One can write $\int G_\eps - c_\rho^2 = \int (G_\eps - \hat G_\eps)$ with
	\begin{equ}
		G_\eps - \hat G_\eps = \eps^2 \Big(\big(2\partial_{x_1}^2 \hat K_\eps * \partial_{x_1}^2R_\eps + (\partial_{x_1}^2R_\eps)^{*2}\big) \cdot \rho_\eps^{*2} +  2 \partial_{x_1}^2 \hat K_{\eps,\eps} \partial_{x_1}^2 R_{\eps,\eps} + (\partial_{x_1}^2R_{\eps,\eps})^2)\Big)\;.
	\end{equ}
	Since $R$ is bounded and equals to $\hat K$ when away from the origin, one easily sees that $\eps^{-2}(G_\eps-\hat G_\eps)$ is integrable uniformly in $\eps$. It follows that $\left|\int G_\eps - c_\rho^2\right| = O(\eps^2)$, proving the desired assertion. The convergence of $G_\eps \ast f$ in $L^2$ and that of $G_\eps$ in the sense of distribution follow by standard argument.
\end{proof}

From \eqref{eq:G-intermediate} and Lemma \ref{lem:G} and polarisation we deduce that the law of the subsequential limiting model $\hat \Pi$ satisfies  $\E\big(\hat\Pi_0 \<XiIXi> (\varphi)\hat\Pi_0 \<XiIXi> (\psi)\big) = c_\rho^2 \crochet{\varphi, \psi}_{L^2}$, which we recognise as the covariance function of $c_\rho$ times the white noise.
To show that the limiting law is indeed that, it remains to prove Gaussianity of $\hat \Pi_0\<XiIXi>$.

To do so, we aim to apply the Nualart-Peccati fourth moment theorem \cite{NP4}, which requires us to show that the fourth cumulant
\begin{equs}
	\kappa_4(\hat\Pi^{\eps}_0 \<XiIXi> (\varphi)) \to 0, \quad \text{as } \eps \to 0\;,
\end{equs}
where we recall $\kappa_4(X) = \E[(X- \E X)^4] - 3\E[(X - \E X)^2]^2$ for any random variable $X$.
By Wick's theorem, $\kappa_4(\hat\Pi^{\eps}_0 \<XiIXi> (\varphi))$ is given by the sum of Feynman graphs generated as follows. We take four copies of
\begin{equ}
	\hat\Pi^{\eps}_0 \<XiIXi> (\varphi) - \E\hat\Pi^{\eps}_0 \<XiIXi> (\varphi) = \eps \;
	\begin{tikzpicture}[scale=0.35,baseline=0.3cm]
		\node at (0,-1)  [root] (root) {};
		\node at (-2,3)  [dot] (left2) {};
		\node at (-2,1)  [dot] (left1) {};
		\node at (0,1) [var] (variable1) {};
		\node at (0,3) [var] (variable2) {};
		
		\draw[testfcn] (left1) to  (root);
		
		\draw[drho] (variable2) to (left2);
		\draw[drho] (variable1) to (left1);
		\draw[kernel1] (left2) to (left1);
	\end{tikzpicture}
\end{equ}
and consider the parings of noise nodes $\tikz \node[var] at (0, 0) {};$  and contracting them to  a plain \tikz[baseline=-0.1cm] \node[dot] at (0,0) {}; node, such that the resulting graph is connected.
There are in total four possible distinct resulting Feynman graphs, which are
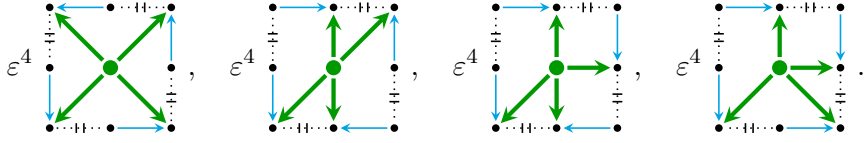
\begin{equs}
	\eps^4\;
	\begin{tikzpicture}[scale=0.4,baseline=-0.1cm]
		\node at (-2,-2)  [dot] (sw) {};
		\node at (-2,2) [dot] (nw) {};
		\node at (2,-2) [dot] (se) {};
		\node at (2,2)  [dot] (ne) {};
		\node at (-2, 0) [dot] (w) {};
		\node at (0, 2) [dot] (n) {};
		\node at (2, 0) [dot] (e) {};
		\node at (0, -2) [dot] (s) {};
		
		\node at (0, 0) [root] (root) {};
		
		\draw[testfcn] (nw) to  (root);
		\draw[testfcn] (ne) to  (root);
		\draw[testfcn] (se) to  (root);
		\draw[testfcn] (sw) to  (root);
		
		\draw[ddrho] (nw) to (w);
		\draw[kernel1] (w) to (sw);
		\draw[ddrho] (sw) to (s);
		\draw[kernel1] (s) to (se);
		\draw[ddrho] (se) to (e);
		\draw[kernel1] (e) to (ne);
		\draw[ddrho] (ne) to (n);
		\draw[kernel1] (n) to (nw);
	\end{tikzpicture}\;,
	\quad \eps^4\;
	\begin{tikzpicture}[scale=0.4,baseline=-0.1cm]
		\node at (-2,-2)  [dot] (sw) {};
		\node at (-2,2) [dot] (nw) {};
		\node at (2,-2) [dot] (se) {};
		\node at (2,2)  [dot] (ne) {};
		\node at (-2, 0) [dot] (w) {};
		\node at (0, 2) [dot] (n) {};
		\node at (2, 0) [dot] (e) {};
		\node at (0, -2) [dot] (s) {};
		
		\node at (0, 0) [root] (root) {};
		
		\draw[testfcn] (n) to  (root);
		\draw[testfcn] (ne) to  (root);
		\draw[testfcn] (s) to  (root);
		\draw[testfcn] (sw) to  (root);
		
		\draw[ddrho] (w) to (nw);
		\draw[kernel1] (nw) to (n);
		\draw[ddrho] (n) to (ne);
		\draw[kernel1] (e) to (ne);
		\draw[ddrho] (e) to (se);
		\draw[kernel1] (se) to (s);
		\draw[ddrho] (s) to (sw);
		\draw[kernel1] (w) to (sw);
	\end{tikzpicture}\;,
	\quad \eps^4\;
	\begin{tikzpicture}[scale=0.4,baseline=-0.1cm]
		\node at (-2,-2)  [dot] (sw) {};
		\node at (-2,2) [dot] (nw) {};
		\node at (2,-2) [dot] (se) {};
		\node at (2,2)  [dot] (ne) {};
		\node at (-2, 0) [dot] (w) {};
		\node at (0, 2) [dot] (n) {};
		\node at (2, 0) [dot] (e) {};
		\node at (0, -2) [dot] (s) {};
		
		\node at (0, 0) [root] (root) {};
		
		\draw[testfcn] (n) to  (root);
		\draw[testfcn] (e) to  (root);
		\draw[testfcn] (s) to  (root);
		\draw[testfcn] (sw) to  (root);
		
		\draw[ddrho] (w) to (nw);
		\draw[kernel1] (nw) to (n);
		\draw[ddrho] (n) to (ne);
		\draw[kernel1] (ne) to (e);
		\draw[ddrho] (e) to (se);
		\draw[kernel1] (se) to (s);
		\draw[ddrho] (s) to (sw);
		\draw[kernel1] (w) to (sw);
	\end{tikzpicture}\;,
	\quad \eps^4\;
	\begin{tikzpicture}[scale=0.4,baseline=-0.1cm]
		\node at (-2,-2)  [dot] (sw) {};
		\node at (-2,2) [dot] (nw) {};
		\node at (2,-2) [dot] (se) {};
		\node at (2,2)  [dot] (ne) {};
		\node at (-2, 0) [dot] (w) {};
		\node at (0, 2) [dot] (n) {};
		\node at (2, 0) [dot] (e) {};
		\node at (0, -2) [dot] (s) {};
		
		\node at (0, 0) [root] (root) {};
		
		\draw[testfcn] (n) to  (root);
		\draw[testfcn] (e) to  (root);
		\draw[testfcn] (se) to  (root);
		\draw[testfcn] (sw) to  (root);
		
		\draw[ddrho] (w) to (nw);
		\draw[kernel1] (nw) to (n);
		\draw[ddrho] (n) to (ne);
		\draw[kernel1] (ne) to (e);
		\draw[ddrho] (e) to (se);
		\draw[kernel1] (s) to (se);
		\draw[ddrho] (s) to (sw);
		\draw[kernel1] (w) to (sw);
	\end{tikzpicture}\;.
\end{equs}
Distributing $\eps$ to each mollifier edge, these Feynman graphs under the faithful labelling given by \eqref{eq:labels} and \eqref{eq:labels_eps} with $\gamma = 1$ all fail condition 2 of Theorem \ref{thm:power_counting} due to the presence of the paths $\tikz[baseline=-0.1cm]{\draw[ddrho] (0,0) to (1,0); \draw[kernel1] (1,0) to (2,0); \draw[ddrho] (2,0) to (3,0);}$. However, IBP will once again save us: by performing one IBP at each node not attached to the test functions, we turn every \tikz[baseline=-0.1cm] \draw[kernel1] (0,0) to (1,0); into \tikz[baseline=-0.1cm] \draw[dkernel1] (0,0) to (1,0); and \tikz[baseline=-0.1cm] \draw[ddrho] (0,0) to (1,0); into \tikz[baseline=-0.1cm] \draw[drho] (0,0) to (1,0);, and then distributing $\eps^{1-\bar\kappa}$ to each $\drho$ edge results in labelled graphs verifying all assumptions of Theorem \ref{thm:power_counting}. Indeed, take the first graph for example, one has
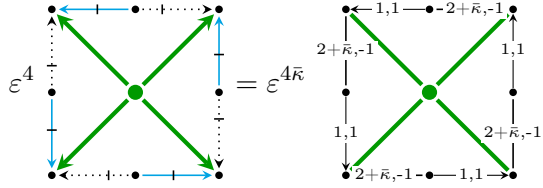
\begin{equ}
	\eps^{4}\;
	\begin{tikzpicture}[scale=0.55,baseline=0cm]
		\node at (-2,-2)  [dot] (sw) {};
		\node at (-2,2) [dot] (nw) {};
		\node at (2,-2) [dot] (se) {};
		\node at (2,2)  [dot] (ne) {};
		\node at (-2, 0) [dot] (w) {};
		\node at (0, 2) [dot] (n) {};
		\node at (2, 0) [dot] (e) {};
		\node at (0, -2) [dot] (s) {};
		
		\node at (0, 0) [root] (root) {};
		
		\draw[testfcn] (nw) to  (root);
		\draw[testfcn] (ne) to  (root);
		\draw[testfcn] (se) to  (root);
		\draw[testfcn] (sw) to  (root);
		
		\draw[drho] (w) to (nw);
		\draw[dkernel1] (w) to (sw);
		\draw[drho] (s) to (sw);
		\draw[dkernel1] (s) to (se);
		\draw[drho] (e) to (se);
		\draw[dkernel1] (e) to (ne);
		\draw[drho] (n) to (ne);
		\draw[dkernel1] (n) to (nw);
	\end{tikzpicture}
	= \eps^{4\bar\kappa}\;
	\begin{tikzpicture}[scale=0.55,baseline=0cm]
		\node at (-2,-2)  [dot] (sw) {};
		\node at (-2,2) [dot] (nw) {};
		\node at (2,-2) [dot] (se) {};
		\node at (2,2)  [dot] (ne) {};
		\node at (-2, 0) [dot] (w) {};
		\node at (0, 2) [dot] (n) {};
		\node at (2, 0) [dot] (e) {};
		\node at (0, -2) [dot] (s) {};
		
		\node at (0, 0) [root] (root) {};
		
		\draw[dist] (nw) to  (root);
		\draw[dist] (ne) to  (root);
		\draw[dist] (se) to  (root);
		\draw[dist] (sw) to  (root);
		
		\draw[generic] 	(nw) to node[labl,pos=0.5] {\tiny 2+$\bar\kappa$,-1} (w);
		\draw[->] 	(w) to node[labl,pos=0.5] {\tiny 1,1} (sw);
		\draw[generic] 		(sw) to node[labl,pos=0.45] {\tiny 2+$\bar\kappa$,-1} (s);
		\draw[->] 	(s) to node[labl,pos=0.5] {\tiny 1,1} (se);
		\draw[generic] 		(se) to node[labl,pos=0.5] {\tiny 2+$\bar\kappa$,-1} (e);
		\draw[->] 	(e) to node[labl,pos=0.5] {\tiny 1,1} (ne);
		\draw[generic] 		(ne) to node[labl,pos=0.45] {\tiny 2+$\bar\kappa$,-1} (n);
		\draw[->] 	(n) to node[labl,pos=0.5] {\tiny 1,1} (nw);
	\end{tikzpicture}
\end{equ}
which indeed satisfies all four conditions of Theorem \ref{thm:power_counting} and is thus bounded by $C\eps^{4\bar\kappa}$ with some constant $C$. The same holds for the other three graphs, from which it follows that $\kappa_4(\hat\Pi^{\eps}_0 \<XiIXi> (\varphi)) \to 0$.
We can conclude that for any limiting model $(\hat \Pi,\hat \Gamma)$, the law of $\hat \Pi_0\<XiIXi>$ is $c_\rho$ times a white noise.

\begin{remark}
To prove the slightly stronger version alluded to in Remark \ref{rem:independence}, it suffices to slightly modify the above computation to show that given two test functions $\varphi, \psi$, the random variables $\xi_\eps(\varphi) + \hat\Pi^{\eps}_0 \<XiIXi> (\psi)$ has vanishing fourth cumulants as $\eps\to 0$. Indeed, this shows joint Gaussianity in the limit, and independence then follows from $ \E\big(\hat \Pi_0^{(\eps)} \<XiIXi> (\psi) \cdot \xi_\eps(\varphi)\big) = 0$.
\end{remark}

\subsubsection{Two simple trees associated to the new white noise}
Let us treat additionally the trees $\<XXiIXi>_i$ and $\<XiIXXi>_j$, $i,j \in \{1, 2\}$, which are directly related to $\<XiIXi>$ and have homogeneity $-2\kappa$.

For $\<XXiIXi>$, the required bound for tightness and the limit follow from the observation that, for every $\varphi \in \Phi$, there exists $\tilde \varphi \in \Phi$ such that
\begin{equ}
	\crochet{\hat \Pi^{\eps}_0 \<XXiIXi>_i, \varphi^\lambda} = \crochet{\hat \Pi^{\eps}_0 \<XiIXi>, x_i\varphi^\lambda} = \crochet{\hat \Pi^{\eps}_0 \<XiIXi>, \lambda {\tilde \varphi}^\lambda}\;.
\end{equ}
The previous discussion on $\<XiIXi>$ then allows us to conclude.

As for $\<XiIXXi>_j$, one can compute by definition of $\hat \Pi^{\eps}$
\begin{equ}[eq:XiIXXi]
	\hat \Pi^{\eps}_0 \<XiIXXi>_j (\varphi^\lambda) =
	\eps \left(
	\begin{tikzpicture}[scale=0.35,baseline=0.3cm]
		\node at (0,-1)  [root] (root) {};
		\node at (-1,1)  [dot] (left) {};
		\node at (1,1)  [dot] (left1) {};
		\node at (-1,3) [var] (variable1) {};
		\node at (1,3) [var] (variable2) {};
		
		\draw[testfcn] (left) to  (root);
		
		\draw[kernel2] (left1) to (left);
		\draw[drho] (variable2) to (left1); 
		\draw[drho] (variable1) to (left); 
		\draw[multx] (left1) to node[right]{\tiny $j$} (root); 
	\end{tikzpicture}
	\; + \;
	\begin{tikzpicture}[scale=0.35,baseline=0cm]
		\node at (0,-1)  [root] (root) {};
		\node at (-1,1)  [dot] (left) {};
		\node at (1,1) [dot] (right) {};
		
		\draw[testfcn] (left) to  (root);
		\draw[kernelBig] (right) to (left);
		\draw[multx] (right) to node[right]{\tiny $j$} (root);
	\end{tikzpicture}
	\; - \;
	\begin{tikzpicture}[scale=0.35,baseline=0cm]
		\node at (0,-1)  [root] (root) {};
		\node at (-1,1)  [dot] (left) {};
		\node at (1,1) [dot] (right) {};
		
		\draw[testfcn] (left) to  (root);
		
		\draw[kernelx] (right) to node[right]{\tiny $j$} (root);
		\draw[ddrho] (left) to (right);
	\end{tikzpicture}
	\; - \sum_{k = 1, 2} \;
	\begin{tikzpicture}[scale=0.35,baseline=0cm]
		\node at (0,-1)  [root] (root) {};
		\node at (-1,1)  [dot] (left) {};
		\node at (1,1) [dot] (right) {};
		
		\draw[testfcnx] (left) to node[left]{\tiny\color{black} $k$}  (root);
		
		\draw[dkernelx] (right) to node[right]{\tiny $k, j$} (root);
		\draw[ddrho] (left) to (right);
	\end{tikzpicture}
	\right)\;.
\end{equ}
Again by distributing $\eps^{1-\bar\kappa}$ to $\kernelBig$ or $\ddrho$ edges, the Feynman graphs corresponding to the three deterministic terms in \eqref{eq:XiIXXi} are of order $\eps^{\bar\kappa} \lambda^{-2\bar\kappa}$ by Theorem \ref{thm:power_counting}, and in particular vanish in the limit as $\eps \to 0$ with $\lambda = 1$ fixed. Concerning the remaining stochastic graph, we can further decompose by IBP
\begin{equ}
	\begin{tikzpicture}[scale=0.35,baseline=0.3cm]
		\node at (0,-1)  [root] (root) {};
		\node at (-1,1)  [dot] (left) {};
		\node at (1,1)  [dot] (left1) {};
		\node at (-1,3) [var] (variable1) {};
		\node at (1,3) [var] (variable2) {};
		
		\draw[testfcn] (left) to  (root);
		
		\draw[kernel2] (left1) to (left);
		\draw[drho] (variable2) to (left1); 
		\draw[drho] (variable1) to (left); 
		\draw[multx] (left1) to node[right]{\tiny $j$} (root); 
	\end{tikzpicture}
	\;= -\;
	\begin{tikzpicture}[scale=0.35,baseline=0.3cm]
		\node at (0,-1)  [root] (root) {};
		\node at (-1,1)  [dot] (left) {};
		\node at (1,1)  [dot] (left1) {};
		\node at (-1,3) [var] (variable1) {};
		\node at (1,3) [var] (variable2) {};
		
		\draw[testfcn] (left) to  (root);
		
		\draw[ddkernel] (left1) to (left);
		\draw[rho] (variable2) to (left1); 
		\draw[rho] (variable1) to (left); 
		\draw[multx] (left1) to node[right]{\tiny $j$} (root); 
	\end{tikzpicture}
	\;+\;
	\begin{tikzpicture}[scale=0.35,baseline=0.3cm]
		\node at (0,-1)  [root] (root) {};
		\node at (-1,1)  [dot] (left) {};
		\node at (1,1)  [dot] (left1) {};
		\node at (-1,3) [var] (variable1) {};
		\node at (1,3) [var] (variable2) {};
		
		\draw[testfcn] (left) to  (root);
		
		\draw[rho] (variable2) to (left1); 
		\draw[rho] (variable1) to (left); 
		\draw[multx, bend left=45] (left1) to node[right]{\tiny $j$} (root); 
		\draw[ddkernel] (left1) to (root);
	\end{tikzpicture}
	\;+\;
	\begin{tikzpicture}[scale=0.35,baseline=0.3cm]
		\node at (0,-1)  [root] (root) {};
		\node at (-1,1)  [dot] (left) {};
		\node at (1,1)  [dot] (left1) {};
		\node at (-1,3) [var] (variable1) {};
		\node at (1,3) [var] (variable2) {};
		
		\draw[testfcn] (left) to  (root);
		
		\draw[dkernel1] (left1) to (left);
		\draw[rho] (variable2) to (left1); 
		\draw[rho] (variable1) to (left);
	\end{tikzpicture}
	\;\un_{j = 1} -\;
	\begin{tikzpicture}[scale=0.35,baseline=0.3cm]
		\node at (0,-1)  [root] (root) {};
		\node at (-1,1)  [dot] (left) {};
		\node at (1,1)  [dot] (left1) {};
		\node at (-1,3) [var] (variable1) {};
		\node at (1,3) [var] (variable2) {};
		
		\draw[dtestfcn] (left) to  (root);
		
		\draw[kernel2] (left1) to (left);
		\draw[drho] (variable2) to (left1); 
		\draw[rho] (variable1) to (left); 
		\draw[multx] (left1) to node[right]{\tiny $j$} (root); 
	\end{tikzpicture}\;.
\end{equ}
For the rightmost two stochastic graphs, it is easy to check that, after taking its second moment and distributing $\eps^{1-\bar\kappa}$ to each edge corresponding to $\rho_\eps^{*2}$ and $\partial_1^2\rho_\eps^{*2}$, the resulted labelled graphs vanish in the limit again by Theorem \ref{thm:power_counting}.
For the second stochastic graph, we can label faithfully the edge
$\tikz[baseline=-0.1cm] {
	\draw[multx, bend left=30] (0,0) to (1,0); 
	\draw[ddkernel] (0,0) to (1,0);} = (x_j-y_j) \partial_1^2 K(y-x)$
by $(1, 0)$ or $(1+\bar\kappa, 0)$ with $\bar\kappa>0$, leading to
\begin{equ}
	\E\left( \eps\;
	\begin{tikzpicture}[scale=0.35,baseline=0.3cm]
		\node at (0,-1)  [root] (root) {};
		\node at (-1,1)  [dot] (left) {};
		\node at (1,1)  [dot] (left1) {};
		\node at (-1,3) [var] (variable1) {};
		\node at (1,3) [var] (variable2) {};
		
		\draw[testfcn] (left) to  (root);
		
		\draw[rho] (variable2) to (left1); 
		\draw[rho] (variable1) to (left); 
		\draw[multx, bend left=45] (left1) to node[right]{\tiny $j$} (root); 
		\draw[ddkernel] (left1) to (root);
	\end{tikzpicture}\right)^2 
	\lesssim \eps^{\bar\kappa}\;
	\begin{tikzpicture}[scale=0.35,baseline=0cm]
		\node at (0,0)  [root] (root) {};
		\node at (-2,2)  [dot] (left2) {};
		\node at (-2,-2)  [dot] (left1) {};
		\node at (2,2) [dot] (right2) {};
		\node at (2,-2) [dot] (right1) {};
		
		\draw[dist] (left1) to  (root);
		\draw[dist] (right1) to  (root);
		
		\draw[generic] (right1) to node[labl,pos=0.5] {\tiny $\bar\kappa$,0} (left1); 
		\draw[generic] (right2) to node[labl,pos=0.5] {\tiny 2,-1} (left2); 
		\draw[generic] (left2) to node[labl,pos=0.5] {\tiny 1+$\bar\kappa$,0}(root);
		\draw[generic] (right2) to node[labl,pos=0.5] {\tiny 1,0}(root);
	\end{tikzpicture}
	\;+ \eps^{2\bar\kappa}\;
	\begin{tikzpicture}[scale=0.35,baseline=0cm]
		\node at (0,0)  [root] (root) {};
		\node at (-2,2)  [dot] (left2) {};
		\node at (-2,-2)  [dot] (left1) {};
		\node at (2,2) [dot] (right2) {};
		\node at (2,-2) [dot] (right1) {};
		
		\draw[dist] (left1) to (root);
		\draw[dist] (right2) to (root);
		
		\draw[generic] (right1) to node[labl,pos=0.5] {\tiny 1+$\bar\kappa$,0} (left1); 
		\draw[generic] (right2) to node[labl,pos=0.5] {\tiny 1+$\bar\kappa$,0} (left2); 
		\draw[generic] (left2) to node[labl,pos=0.5] {\tiny 1,0}(root);
		\draw[generic] (right1) to node[labl,pos=0.5] {\tiny 1,0}(root);
	\end{tikzpicture}
	\lesssim \eps^{\bar \kappa} \lambda^{-2\bar\kappa} + \eps^{2\bar\kappa} \lambda^{-2\bar\kappa}
\end{equ}
where we verified the assumptions of Theorem \ref{thm:power_counting} by the following $\eps$-distribution: for the first labelled graph, we distributed $\eps^{2-\bar\kappa}$ to the lower mollifier edge and label one of
$\tikz[baseline=-0.1cm] {
	\draw[multx, bend left=30] (0,0) to (1,0); 
	\draw[ddkernel] (0,0) to (1,0);}$
by $(1+\bar\kappa, 0)$ so that the condition 4 is satisfied, while in the second labelled graph we distributed $\eps^{1-\bar\kappa}$ to each of the mollifiers.

For the remaining stochastic graph, first we note that since  $x_j = (x_j-y_j) + y_j$, the  $\multx$ edge can be split in two:
$\tikz[baseline=0cm] {\node at (0, 0) [root] (root) {};
	\node at (1, 0) [dot] (node1) {};
	\node at (2, 0) [dot] (node2) {};
	\draw[testfcn] (node1) to (root);
	\draw[ddkernel] (node2) to (node1);
	\draw[multx, bend right=30] (node2) to (root);}
=
\tikz[baseline=-0.1cm] {\node at (0, 0) [root] (root) {};
	\node at (1, 0) [dot] (node1) {};
	\node at (2, 0) [dot] (node2) {};
	\draw[testfcnx] (node1) to (root);
	\draw[ddkernel] (node2) to (node1);}
+
\tikz[baseline=-0.1cm] {\node at (0, 0) [root] (root) {};
	\node at (1, 0) [dot] (node1) {};
	\node at (2, 0) [dot] (node2) {};
	\draw[testfcn] (node1) to (root);
	\draw[ddkernel] (node2) to (node1);
	\draw[multx, bend right=30] (node2) to (node1);}$.
%Moreover, the resulting kernel $(x_j-y_j) \partial_1^2 K(y-x)$ can be faithfully labelled $(1,0)$ or $(1+\bar\kappa,0)$ for $\bar \kappa>0$. 
Therefore, we get
\begin{equs}
	\E \left(\eps\;
	\begin{tikzpicture}[scale=0.35,baseline=0.3cm]
		\node at (0,-1)  [root] (root) {};
		\node at (-1,1)  [dot] (left) {};
		\node at (1,1)  [dot] (left1) {};
		\node at (-1,3) [var] (variable1) {};
		\node at (1,3) [var] (variable2) {};
		
		\draw[testfcn] (left) to  (root);
		
		\draw[ddkernel] (left1) to (left);
		\draw[rho] (variable2) to (left1); 
		\draw[rho] (variable1) to (left); 
		\draw[multx] (left1) to node[right]{\tiny\color{black} $j$} (root); 
	\end{tikzpicture}\;\right)^2
	&=\eps^2 \left(
	\begin{tikzpicture}[scale=0.35,baseline=0.3cm]
		\node at (0,-1)  [root] (root) {};
		\node at (-2,1)  [dot] (left) {};
		\node at (-2,3)  [dot] (left1) {};
		\node at (2,1)  [dot] (right) {};
		\node at (2,3)  [dot] (right1) {};
		
		\draw[testfcn] (left) to  (root);
		\draw[testfcn] (right) to  (root);
		\draw[multx, bend right=90] (left1) to node[left]{\tiny\color{black} $j$} (root);
		\draw[multx, bend left=90] (right1) to node[right]{\tiny\color{black} $j$} (root);
		
		\draw[ddkernel] (left1) to (left);
		\draw[rho] (left1) to (right1);
		\draw[rho] (left) to (right); 
		\draw[ddkernel] (right1) to (right);
	\end{tikzpicture}
	\;+\;
	\begin{tikzpicture}[scale=0.35,baseline=0.3cm]
		\node at (0,-1)  [root] (root) {};
		\node at (-2,1)  [dot] (left) {};
		\node at (-2,3)  [dot] (left1) {};
		\node at (2,1)  [dot] (right) {};
		\node at (2,3)  [dot] (right1) {};
		
		\draw[testfcn] (left) to  (root);
		\draw[testfcn] (right) to  (root);
		\draw[multx, bend right=90] (left1) to node[left]{\tiny\color{black} $j$} (root);
		\draw[multx, bend left=90] (right1) to node[right]{\tiny\color{black} $j$} (root);
		
		\draw[ddkernel] (left1) to (left);
		\draw[rho] (left1) to (right);
		\draw[rho] (left) to (right1); 
		\draw[ddkernel] (right1) to (right);
	\end{tikzpicture}\right)\\
	&= \eps^2 \left(
	\begin{tikzpicture}[scale=0.35,baseline=0.3cm]
		\node at (0,-1)  [root] (root) {};
		\node at (-2,1)  [dot] (left) {};
		\node at (-2,3)  [dot] (left1) {};
		\node at (2,1)  [dot] (right) {};
		\node at (2,3)  [dot] (right1) {};
		
		\draw[testfcnx] (left) to node[below]{\tiny\color{black}$j$} (root);
		\draw[testfcnx] (right) to node[below]{\tiny\color{black}$j$}  (root);
		
		\draw[ddkernel] (left1) to (left);
		\draw[rho] (left1) to (right1);
		\draw[rho] (left) to (right); 
		\draw[ddkernel] (right1) to (right);
	\end{tikzpicture}
	\;+\;
	\begin{tikzpicture}[scale=0.35,baseline=0.3cm]
		\node at (0,-1)  [root] (root) {};
		\node at (-2,1)  [dot] (left) {};
		\node at (-2,3)  [dot] (left1) {};
		\node at (2,1)  [dot] (right) {};
		\node at (2,3)  [dot] (right1) {};
		
		\draw[testfcnx] (left) to node[below]{\tiny\color{black}$j$} (root);
		\draw[testfcnx] (right) to node[below]{\tiny\color{black}$j$} (root);
		
		\draw[ddkernel] (left1) to (left);
		\draw[rho] (left1) to (right);
		\draw[rho] (left) to (right1); 
		\draw[ddkernel] (right1) to (right);
	\end{tikzpicture}\right) + E_{\eps, \lambda}\;, \label{eq:XiIXXi_variance}
\end{equs}
where the remainder $E_{\eps, \lambda}$ can be bounded by the following sum of labelled graphs (using also that $\testfcnx = \lambda \testfcn$)
\begin{equs}
	\eps^{\bar\kappa}
	\begin{tikzpicture}[scale=0.35,baseline=0.3cm]
		\node at (0,-1)  [root] (root) {};
		\node at (-2,1)  [dot] (left) {};
		\node at (-2,3)  [dot] (left1) {};
		\node at (2,1)  [dot] (right) {};
		\node at (2,3)  [dot] (right1) {};
		
		\draw[dist] (left) to  (root);
		\draw[dist] (right) to  (root);
		
		\draw[generic] (left1) to node[labl,pos=0.5]{\tiny 1+$\bar\kappa$,0} (left);
		\draw[generic] (left1) to node[labl,pos=0.5]{\tiny 2,-1} (right1);
		\draw[generic] (left) to node[labl,pos=0.5]{\tiny $\bar\kappa$,0} (right); 
		\draw[generic] (right1) to node[labl,pos=0.5]{\tiny 1,0} (right);
	\end{tikzpicture}
	\;+\eps^{2\bar\kappa}\;
	\begin{tikzpicture}[scale=0.35,baseline=0.3cm]
		\node at (0,-1)  [root] (root) {};
		\node at (-2,1)  [dot] (left) {};
		\node at (-2,3)  [dot] (left1) {};
		\node at (2,1)  [dot] (right) {};
		\node at (2,3)  [dot] (right1) {};
		
		\draw[dist] (left) to (root);
		\draw[dist] (right) to (root);
		
		\draw[generic] (left1) to node[labl,pos=0.5]{\tiny 1,0} (left);
		\draw[generic] (left1) to node[labl,pos=0.35, above]{\tiny 1+$\bar\kappa$,0} (right);
		\draw[generic] (left) to node[labl,pos=0.35, below]{\tiny 1+$\bar\kappa$,0} (right1); 
		\draw[generic] (right1) to node[labl,pos=0.5]{\tiny 1,0} (right);
	\end{tikzpicture}
	\;+2\eps^{2\bar\kappa} \lambda \left(\;
	\begin{tikzpicture}[scale=0.35,baseline=0.3cm]
		\node at (0,-1)  [root] (root) {};
		\node at (-2,1)  [dot] (left) {};
		\node at (-2,3)  [dot] (left1) {};
		\node at (2,1)  [dot] (right) {};
		\node at (2,3)  [dot] (right1) {};
		
		\draw[dist] (left) to  (root);
		\draw[dist] (right) to  (root);
		
		\draw[generic] (left1) to node[labl,pos=0.5]{\tiny 1,0} (left);
		\draw[generic] (left1) to node[labl,pos=0.5]{\tiny 1+$\bar\kappa$,0} (right1);
		\draw[generic] (left) to node[labl,pos=0.5]{\tiny 1+$\bar\kappa$,0} (right); 
		\draw[generic] (right1) to node[labl,pos=0.5]{\tiny 2,-1} (right);
	\end{tikzpicture}
	\;+\;
	\begin{tikzpicture}[scale=0.35,baseline=0.3cm]
		\node at (0,-1)  [root] (root) {};
		\node at (-2,1)  [dot] (left) {};
		\node at (-2,3)  [dot] (left1) {};
		\node at (2,1)  [dot] (right) {};
		\node at (2,3)  [dot] (right1) {};
		
		\draw[dist] (left) to (root);
		\draw[dist] (right) to (root);
		
		\draw[generic] (left1) to node[labl,pos=0.5]{\tiny 1,0} (left);
		\draw[generic] (left1) to node[labl,pos=0.35, above]{\tiny 1+$\bar\kappa$,0} (right);
		\draw[generic] (left) to node[labl,pos=0.35, below]{\tiny 1+$\bar\kappa$,0} (right1); 
		\draw[generic] (right1) to node[labl,pos=0.5]{\tiny 2,-1} (right);
	\end{tikzpicture}\right).
\end{equs}
Here for the first labelled graph we distribute $\eps^{2-\bar\kappa}$ to the bottom mollifier edge and label one of
$\tikz[baseline=-0.1cm] {
	\draw[multx, bend left=30] (0,0) to (1,0); 
	\draw[ddkernel] (0,0) to (1,0);}$
so that the condition 4 of Theorem \ref{thm:power_counting} is verified, while for the other labelled graphs we just distributed $\eps^{1-\bar\kappa}$ equally to each mollifier edge. 
It then follows that $E_{\eps, \lambda}$ is of order $\eps^{\bar\kappa} \lambda^{-2\bar\kappa}$.

Additionally, note that the dominant term is nothing but the integral
\begin{equ}
	\iint G_\eps(x - y) x_j y_j \varphi^\lambda(x) \varphi^\lambda(y) \dd x \dd y\;.
\end{equ}
The tightness bounds \eqref{eq:tightness_model} required for the tree $\<XiIXXi>$ then follows from \eqref{eq:G_bound}. For $\lambda = 1$ fixed, we also deduce by Lemma \ref{lem:G} the limit
\begin{equ}
	\lim_{\eps \to 0} \E(\hat \Pi^\eps_0 \<XiIXXi>_j(\varphi))^2 = c_\rho^2\int (x_j \varphi(x))^2 \dd x\;.
\end{equ}

At this point, the calculation suggests that $\hat\Pi^{\eps}_0\<XiIXXi>_j$ converges to $c_\rho x_j \tilde\eta$ for some white noise $\tilde \eta$. However, it is a priori unclear whether $\tilde \eta$ coincides with $\eta$, the limiting white noise from the tree $\<XiIXi>$. In fact, a similar (strictly easier, due to one less \tikz[baseline=-0.1cm] \draw[multx] (1,0) to node[below]{\tiny $j$} (0,0); edge) calculation shows
\begin{equ}
	\lim_{\eps \to 0} \E\big(\hat \Pi^\eps_0 \<XiIXXi>_j(\varphi_1)\hat \Pi^\eps_0\<XiIXi>(\varphi_2)\big) = c_\rho^2 \int x_j \varphi_1(x)\varphi_2(x) \dd x.
\end{equ}
This will be sufficient to characterise the limiting law, provided that we prove the tightness bound for every tree in $V^\RHS$. Indeed, the consequence of Lemma \ref{lem:tightness} would imply $(\hat\Pi^\eps_0 \<XiIXi>, \hat\Pi^\eps_0 \<XiIXXi>_j)$ converge jointly in law up to subsequences. Denoting by $(c_\rho \eta, B)$ a subsequential limit of the random fields $(\hat\Pi^\eps_0 \<XiIXi>, \hat\Pi^\eps_0 \<XiIXXi>_j)$, we have seen that the law of $\eta$ coincides with a white noise and from above we have that $\E\big(c_\rho\eta(\varphi_1) B(\varphi_2)\big) = c_\rho^2 \int x_j \varphi_1(x) \varphi_2 (x) \dd x = \E\big(c_\rho\eta(\varphi_1) \cdot c_\rho\eta(x_j\varphi_2)\big)$ for any test functions $\varphi_1, \varphi_2 \in \Phi$. In particular, one deduces $\E\big(B (\varphi_2) | \sigma(\eta)\big) = c_\rho \eta(x_j \varphi_2)$ a.s., where $\sigma(\eta)$ denotes the $\sigma$-algebra generated by the random variables $\eta(\varphi), \varphi \in \Phi$. This combined with the covariance $\E\big(B(\varphi_2)^2\big) = c_\rho^2 \E\big(\eta(x_j \varphi_2)^2\big)$ from above, we conclude that $B = c_\rho x_j \eta$ a.s.. This kind of argument will be later generalised to a setting where the limiting fields take values in higher order homogeneous Wiener chaos (see Lemma \ref{lem:criterion_second_chaos} below) and will be useful in the sequel.

\subsection{Completely vanishing trees}\label{sec:vanishing_trees}
We now make a quick detour to treat the trees that vanish in the limit, which are $\<Xi>$, $\<XiX>$, $\<XiIXiIXi>$, $\<XiIXi2>$, $\<XiIXiIXi2>$, and $\<XiIXi3>$.
This argument does not need any further input from the graphical calculations. It is clear that $\hat \Pi_x^\eps \<Xi>$ converges to $0$ in $L^2(\Omega,\cC^{\alpha}(\T^2))$ for $\alpha < -3/2$. This also gives the desired bound \eqref{eq:second_moment_model} and the vanishing limit for $\<XiX>$. For the remaining trees we use a small trick to reduce the question to the regularity structure black box.
Consider two smooth noises $\zeta_1$, $\zeta_2$ and the system of equations
\begin{equs}
	\partial_t w_1&=\Delta w_1+g(w_2)\zeta_1,
	\\
	\partial_t w_2&=\Delta w_2+g(w_1)\zeta_2.
\end{equs}
Consider the regularity structure $\ccT''$ corresponding to this equation, when assigning homogeneities $\alpha_1=-3/2-\kappa$ and $\alpha_2=-3/2+\kappa+\kappa^2$ corresponding to $\zeta_1$ and $\zeta_2$, respectively. In the graphical notation the corresponding noise symbols will be denoted by $\<Zeta1>$  and $\<Zeta2>$, respectively. This regularity structure satisfies all conditions of \cite{CH}, including \cite[Def.~2.28]{CH}, that $\ccT$ violates.
Indeed, the lowest homogeneity symbols with more than one noise in them satisfy $|\<XiIZeta2>|=|\<Zeta2IXi>|=-1+\kappa^2>-1$.
Therefore, by the main results of \cite{CH},
if $\zeta_1^\eps$ and $\zeta_2^\eps$ are a sequence of smooth centred Gaussian random fields with covariance obeying the bounds
\begin{equ}
	\big|\partial^k\big(\E\zeta_i^\eps(x)\zeta_j^\eps(y)\big)\big|\lesssim |x-y|^{\alpha_i+\alpha_j+\kappa^2-|k|},
\end{equ}
uniformly in $\eps\in(0,1]$, for all multiindex $k\in\N^2\times \N^2$ with $|k|\leq 12$, and $i,j\in\{1,2\}$,
then the BPHZ models are uniformly bounded in $L^2(\Omega)$, i.e. they satisfy the uniform bound \eqref{eq:tightness_model}.
Let us take $\zeta_1^\eps=\eps^{1/2-\kappa+\kappa^2}\partial_{x_1}\xi^\eps$, $\zeta_2^\eps=\eps^{1/2+\kappa+2\kappa^2}\partial_{x_1}\xi^\eps$, which satisfy the above conditions, and $(\tilde \Pi^\eps,\tilde\Gamma^\eps)$ be the BPHZ lift of them. One has
\begin{equs}
	\hat\PPi^\eps \<XiIXiIXi>
	&=\eps^{\kappa-4\kappa^2}\tilde\PPi^\eps \<XiIXiIXiNEW>,\qquad
	&\hat \PPi^\eps\<XiIXi2>&=\eps^{\kappa-4\kappa^2}\tilde \PPi^\eps\<XiIXi2NEW>,
	\\
	\hat \PPi^\eps\<XiIXiIXi2>&=\eps^{2\kappa-5\kappa^2}\tilde \PPi^\eps\<XiIXiIXi2NEW>,\qquad
	&\hat \PPi^\eps\<XiIXi3>&=\eps^{2\kappa-5\kappa^2}\tilde \PPi^\eps\<XiIXi3NEW>.
\end{equs}
Here we have also used that the renormalisation constant of the symbols $\<XXiIXi>_i$ and $\<XiIXXi>_j$ are $0$ by symmetry in the model $(\hat \Pi^\eps,\Gamma^\eps)$, and so the fact that their ``counterpart'' symbols in $\cT''$ are of positive homogeneity and are therefore not renormalised does not affect the above identities.
For the last three symbols the recentering in the two models is analogous, and so for example
\begin{equ}
	\Big(\E\big(\hat \Pi_0^\eps \<XiIXi2>\big)(\varphi_0^\lambda)^2\Big)^{1/2}=\eps^{\kappa-4\kappa^2}
\Big(\E\big(\hat \Pi_0^\eps \<XiIXi2NEW>\big)(\varphi_0^\lambda)^2\Big)^{1/2}	
	\lesssim \eps^{\kappa-4\kappa^2}\lambda^{-1/2-\kappa+\kappa^2}.
\end{equ}
Therefore \eqref{eq:second_moment_model} is satisfied with the choice 
$\zeta=\hat \Pi_0^\eps \<XiIXi2>$ and $\beta=|\<XiIXi2>|+2\kappa+\kappa^2$. Moreover, any subsequential limiting model vanishes on $\<XiIXi2>$, and similarly for $\<XiIXiIXi2>$ and $\<XiIXi3>$.
As for the last remaining symbol, the recentering in the two models differ, since $|\,\<IXiIXi>|<1<|\,\<IXiIXiNEW>|$.
One thus has
\begin{equ}
	\big(\hat \Pi_0^\eps\<XiIXiIXi>\big)(y)=\eps^{\kappa-4\kappa^2}\Big(\big(\tilde \Pi_x^\eps  \<XiIXiIXiNEW>\big)(y)-(\nabla K\ast\tilde \Pi_0^\eps\<Zeta2IXi>)(0)\cdot y\eps^{1/2-\kappa+\kappa^2}\partial_{x_1}\xi^\eps(y)\Big),
\end{equ}
where the second term appears due to the different recentering. 
The first term is treated as above.
For the second term, we recall from Remark \ref{rem:more_eps_power} (with $\kappa^2$ in place $\kappa$ and $2\kappa^2$ in place of $\kappa'$) that
\begin{equ}
\big\|(\tilde \Pi_0^\eps\<Zeta2IXi>)(\varphi_0^\lambda)\big\|_{L^4(\Omega)}
\lesssim \eps^{\kappa^2}\lambda^{-1+\kappa^2}.
\end{equ}
Recalling that $\nabla K=\sum_{n\in\N}2^{-n}\big(\varphi^{(n)}\big)^{2^{-n}}_0$ with a sequence of test functions $(\varphi^{(n)})_{n\in\N}\subset\Phi$, we get that $\|(\nabla K\ast\tilde \Pi_0^\eps\<Zeta2IXi>)(0)\|_{L^4(\Omega)}\lesssim \eps^{\kappa^2}$.
Therefore, 
\begin{equs}
\Big\|&\int(\partial_i K\ast\tilde \Pi_0^\eps\<Zeta2IXi>)(0)y_i\eps^{1/2-\kappa+\kappa^2}\partial_{x_1}\xi^\eps(y)\varphi_0^\lambda(y)\,dy\Big\|_{L^2(\Omega)}
\\
&=\lambda\Big\|\int(\partial_i K\ast\tilde \Pi_0^\eps\<Zeta2IXi>)(0)\eps^{1/2-\kappa+\kappa^2}\partial_{x_1}\xi^\eps(y)\tilde\varphi_0^\lambda(y)\,dy\Big\|_{L^2(\Omega)}
\\
&\lesssim \lambda\|(\nabla K\ast\tilde \Pi_0^\eps\<Zeta2IXi>)(0)\|_{L^4(\Omega)}\big\|(\eps^{1/2-\kappa+\kappa^2}\partial_{x_1}\xi^\eps)(\tilde \varphi_0^\lambda)
\big\|_{L^4(\Omega)}\lesssim \eps^{\kappa^2}\lambda^{-1/2-\kappa+\kappa^2}.
\end{equs}

\begin{remark}
Although this trick only saved us from dealing $4$ out of the $6$ symbols with $3$ noises or more, for more singular problems the vanishing symbols become the vast majority. For example, in the $3$-dimensional variant of the problem more than $200$ different trees would appear, with only a handful giving nonzero contribution in the limit.
\end{remark}

\subsection{Generalising the graphical argument}\label{sec:graph_argument}
We observe in Section \ref{sec:new_noise} a general structure of proof. For each $\tau \in V^\RHS$, we wrote $\hat \Pi_0^\eps \tau$ as a sum of stochastic graphs, squared each term of them to form Feynman graphs. By distributing $\eps^{1-\bar\kappa}$ to each of the mollifier edge, we showed that the majority of Feynman graphs verify every assumption of Theorem \ref{thm:power_counting}. For those Feynman graphs where Theorem \ref{thm:power_counting} applies, the remaining $\eps$-prefactor always sends the graph to zero in the limit. We then single out those Feynman graphs which fail the assumptions and identify those among them which give rise to non-trivial limit.

The purpose of this section is to demonstrate that this pattern holds in some generality for the remaining Feynman graphs.

Let us now fix the family $\cG$ of Feynman diagrams in question, constructed as follows. Since in the previous sections we have dealt with most of the trees in the set $V^\RHS$, we will now focus on the remaining ones, 
which are $\<lastone>$ and $\<XiIXiIXiIXi>$. Let us collect the stochastic graphs generated by the BPHZ model $\hat \Pi_0^\eps \tau(\varphi^\lambda)$ for each $\tau \in \{\<lastone>, \<XiIXiIXiIXi>\}$.
One has
\begin{equs}[eq:lastone]  
	\;&\bigl(\hat \Pi_0^{(\eps)} \<lastone>\bigr)(\varphi^\lambda) = \eps^2 \left( \begin{tikzpicture}[scale=0.35,baseline=0.4cm]
		\node at (0,-1)  [root] (root) {};
		\node at (0,1)  [dot] (root2) {};
		\node at (-1.5,2.5)  [dot] (left) {};
		\node at (1.5,2.5)  [dot] (right) {};
		\node at (0,2.5) [var] (variable) {};
		\node at (-1.5,4) [var] (variablel) {};
		\node at (1.5,4) [var] (variabler) {};
		\node at (0,4) [dot] (top) {};
		\node at (0,5.5) [var] (variablet) {};
		
		\draw[testfcn] (root2) to  (root);
		
		\draw[kernel1] (left) to (root2);
		\draw[kernel1] (right) to (root2);
		\draw[kernel1] (top) to (left);
		\draw[drho] (variable) to (root2); 
		\draw[drho] (variablel) to (left); 
		\draw[drho] (variabler) to (right); 
		\draw[drho] (variablet) to (top); 
	\end{tikzpicture}
	\;-\;
	\begin{tikzpicture}[scale=0.35,baseline=0.4cm]
		\node at (0,-1)  [root] (root) {};
		\node at (0,1)  [dot] (root2) {};
		\node at (-1.5,2.5)  [dot] (left) {};
		\node at (1.5,2.5)  [dot] (right) {};
		\node at (0,4) [dot] (top) {};
		
		\draw[testfcn] (root2) to  (root);
		
		\draw[kernel,bend right=20] (left) to (root);
		\draw[kernel1] (right) to (root2);
		\draw[kernel1] (top) to (left);
		\draw[ddrho] (left) to (root2); 
		\draw[ddrho] (top) to (right);
	\end{tikzpicture}
	\;-\;
	\begin{tikzpicture}[scale=0.35,baseline=0.4cm]
		\node at (0,-1)  [root] (root) {};
		\node at (1.5, 1) [dot] (right1) {};
		\node at (-1.5, 1) [dot] (left1) {};
		\node at (-1.5,3)  [dot] (left2) {};
		\node at (1.5,3)  [dot] (right2) {};
		
		\draw[testfcn] (right1) to  (root);
		
		\draw[kernel1] (right2) to (right1);
		\draw[kernel] (left2) to (left1);
		\draw[kernel] (left1) to (root); 
		\draw[ddrho-shift] (left2) to (right1); 
		\draw[ddrho-shift] (right2) to (left1); 
	\end{tikzpicture}
	\;+\;
	\begin{tikzpicture}[scale=0.35,baseline=0.4cm]
		\node at (0,-1)  [root] (root) {};
		\node at (0,1)  [dot] (root2) {};
		\node at (-1.5,2.5)  [dot] (left) {};
		\node at (1.5,2.5)  [dot] (right) {};
		\node at (-1.5,4) [var] (variablel) {};
		\node at (1.5,4) [var] (variabler) {};
		\node at (0,4) [dot] (top) {};
		
		\draw[testfcn] (root2) to  (root);
		
		\draw[kernel1] (left) to (root2);
		\draw[kernel1] (right) to (root2);
		\draw[kernel1] (top) to (left);
		\draw[ddrho] (top) to (root2); 
		\draw[drho] (variablel) to (left); 
		\draw[drho] (variabler) to (right);
	\end{tikzpicture}
	\right.\\&
	\left.\;+\;
	\begin{tikzpicture}[scale=0.35,baseline=0.4cm]
		\node at (0,-1)  [root] (root) {};
		\node at (0,1)  [dot] (root2) {};
		\node at (-1.5,2.5)  [dot] (left) {};
		\node at (1.5,2.5)  [dot] (right) {};
		\node at (-1.5,1) [var] (variable) {};
		\node at (0,4) [dot] (top) {};
		\node at (0,5.5) [var] (variablet) {};
		
		\draw[testfcn] (root2) to  (root);
		
		\draw[kernel1] (left) to (root2);
		\draw[kernel1] (right) to (root2);
		\draw[kernel1] (top) to (left);
		\draw[drho] (variable) to (root2); 
		\draw[drho] (variablet) to (top);
		\draw[ddrho] (left) to (right); 
	\end{tikzpicture}
	\;+\;
	\begin{tikzpicture}[scale=0.35,baseline=0.4cm]
		\node at (0,-1)  [root] (root) {};
		\node at (0,1)  [dot] (root2) {};
		\node at (-1.5,2.5)  [dot] (left) {};
		\node at (1.5,2.5)  [dot] (right) {};
		\node at (0,2.5) [var] (variable) {};
		\node at (-1.5,4) [var] (variablel) {};
		\node at (0,4) [dot] (top) {};
		
		\draw[testfcn] (root2) to  (root);
		
		\draw[kernel1] (left) to (root2);
		\draw[kernel1] (right) to (root2);
		\draw[kernel1] (top) to (left);
		\draw[drho] (variable) to (root2); 
		\draw[drho] (variablel) to (left); 
		\draw[ddrho] (top) to (right);
	\end{tikzpicture}
	\;-\;
	\begin{tikzpicture}[scale=0.35,baseline=0.4cm]
		\node at (0,-1)  [root] (root) {};
		\node at (0,1)  [dot] (root2) {};
		\node at (-1.5,2.5)  [dot] (left) {};
		\node at (1.5,2.5)  [dot] (right) {};
		\node at (-1.5,4) [var] (variablel) {};
		\node at (0,4) [dot] (top) {};
		\node at (0,5.5) [var] (variablet) {};
		
		\draw[testfcn] (root2) to  (root);
		
		\draw[kernel1] (left) to (root2);
		\draw[kernel,bend left=20] (right) to (root);
		\draw[kernel1] (top) to (left);
		\draw[ddrho] (right) to (root2); 
		\draw[drho] (variablel) to (left);
		\draw[drho] (variablet) to (top); 
	\end{tikzpicture}
	\;-\;
	\begin{tikzpicture}[scale=0.35,baseline=0.4cm]
		\node at (0,-1)  [root] (root) {};
		\node at (0,1)  [dot] (root2) {};
		\node at (-1.5,2.5)  [dot] (left) {};
		\node at (1.5,2.5)  [dot] (right) {};
		\node at (0,2.5) [var] (variable) {};
		\node at (1.5,4) [var] (variabler) {};
		\node at (-1.5,4.5) [dot] (top) {};
		
		\draw[testfcn] (root2) to  (root);
		
		\draw[kernel1] (left) to (root2);
		\draw[kernel1] (right) to (root2);
		\draw[kernel,bend right = 40] (top) to (root);
		\draw[drho] (variable) to (root2); 
		\draw[ddrho] (top) to (left); 
		\draw[drho] (variabler) to (right);
	\end{tikzpicture}
	\;+\;
	\begin{tikzpicture}[scale=0.35,baseline=0.4cm]
		\node at (0,-1)  [root] (root) {};
		\node at (0,1)  [dot] (root2) {};
		\node at (-1.5,2.5)  [dot] (left) {};
		\node at (1.5,2.5)  [dot] (right) {};
		\node at (1.5,4) [var] (variabler) {};
		\node at (0,4) [dot] (top) {};
		\node at (0,5.5) [var] (variablet) {};
		
		\draw[testfcn] (root2) to  (root);
		
		\draw[kernelBig] (left) to (root2);
		\draw[kernel1] (right) to (root2);
		\draw[kernel1] (top) to (left);
		\draw[drho] (variablet) to (top); 
		\draw[drho] (variabler) to (right); 
	\end{tikzpicture}
	\right.\\&
	\left.\;-\;
	\begin{tikzpicture}[scale=0.35,baseline=0.4cm]
		\node at (0,-1)  [root] (root) {};
		\node at (0,1)  [dot] (root2) {};
		\node at (-1.5,2.5)  [dot] (left) {};
		\node at (1.5,2.5)  [dot] (right) {};
		\node at (1.5,4) [var] (variabler) {};
		\node at (0,4) [dot] (top) {};
		\node at (0,5.5) [var] (variablet) {};
		
		\draw[testfcn] (root2) to  (root);
		
		\draw[ddrho] (left) to (root2);
		\draw[kernel1] (right) to (root2);
		\draw[kernel1] (top) to (left);
		\draw[drho] (variablet) to (top); 
		\draw[drho] (variabler) to (right); 
		\draw[kernel,bend right=20] (left) to (root);
	\end{tikzpicture}
	\;+\;
	\begin{tikzpicture}[scale=0.35,baseline=0.4cm]
		\node at (0,-1)  [root] (root) {};
		\node at (0,1)  [dot] (root2) {};
		\node at (0,3)  [dot] (left) {};
		\node at (1.5,2.5)  [dot] (right) {};
		\node at (-1,3)  [dot] (ll) {};
		
		\draw[testfcn] (root2) to  (root);
		
		\draw[kernel1] (left) to (root2);
		\draw[kernel,bend left=20] (right) to (root);
		
		\draw[kernel,bend right=20] (ll) to (root);
		\draw[ddrho] (left) to (ll); 
		\draw[ddrho] (root2) to (right); 
	\end{tikzpicture}
	\;-\;
	\begin{tikzpicture}[scale=0.35,baseline=0.4cm]
		\node at (0,-1)  [root] (root) {};
		\node at (0,1)  [dot] (root2) {};
		\node at (-1.5,2.5)  [dot] (left) {};
		\node at (1.5,2.5)  [dot] (right) {};
		\node at (0,4) [dot] (top) {};
		
		\draw[testfcn] (root2) to  (root);
		
		\draw[kernel] (right) to (root);
		\draw[kernel] (top) to (left);
		\draw[kernelBig] (left) to (root2); 
		\draw[ddrho] (top) to (right);
	\end{tikzpicture}
	\;-\;
	\begin{tikzpicture}[scale=0.35,baseline=0.4cm]
		\node at (0,-1)  [root] (root) {};
		\node at (0,1)  [dot] (root2) {};
		\node at (-1.5,2.5)  [dot] (left) {};
		\node at (1.5,2.5)  [dot] (right) {};
		\node at (0,4) [dot] (top) {};
		
		\draw[testfcn] (root2) to  (root);
		
		\draw[ddrho-shift] (top) to (root2);
		
		\draw[kernel] (right) to (root);
		\draw[kernel] (top) to (left);
		\draw[kernel] (left) to (root2); 
		\draw[ddrho-shift] (right) to (left);
	\end{tikzpicture}
	\;-\;
	\begin{tikzpicture}[scale=0.35,baseline=0.4cm]
		\node at (0,-1)  [root] (root) {};
		\node at (0,1)  [dot] (root2) {};
		\node at (-1.5,2.5)  [dot] (left) {};
		\node at (1.5,2.5)  [dot] (right) {};
		\node at (0,4) [dot] (top) {};
		
		\draw[testfcn] (root2) to  (root);
		
		\draw[ddrho-shift] (top) to (root2);
		\draw[ddrho-shift] (left) to (right);
		
		\draw[kernel1] (right) to (root2);
		\draw[kernel] (top) to [out=0,in=100] (2, 2.5) to [out = 280, in=30] (root);
		\draw[kernel1] (left) to (root2);
	\end{tikzpicture}
	\right)
\end{equs}
as well as
\begin{equs}[eq:XiIXiIXiIXi]
	\;&\bigl(\hat \Pi_0^{(\eps)} \<XiIXiIXiIXi>\bigr)(\varphi^\lambda) = \eps^2 \left(\;
	\begin{tikzpicture}[scale=0.35,baseline=0.9cm]
		\node at (0,-1)  [root] (root) {};
		\node at (-2,1)  [dot] (left) {};
		\node at (-2,3)  [dot] (left1) {};
		\node at (-2,5)  [dot] (left2) {};
		\node at (-2,7)  [dot] (left3) {};
		\node at (0,1) [var] (variable1) {};
		\node at (0,3) [var] (variable2) {};
		\node at (0,5) [var] (variable3) {};
		\node at (0,7) [var] (variable4) {};
		
		\draw[testfcn] (left) to  (root);
		
		\draw[kernel2] (left1) to (left);
		\draw[kernel1] (left2) to (left1);
		\draw[kernel1] (left3) to (left2);
		\draw[drho] (variable4) to (left3);
		\draw[drho] (variable3) to (left2); 
		\draw[drho] (variable2) to (left1); 
		\draw[drho] (variable1) to (left); 
	\end{tikzpicture}
	\;-\;
	\begin{tikzpicture}[scale=0.35,baseline=0.9cm]
		\node at (0,-1)  [root] (root) {};
		\node at (-2,1)  [dot] (left) {};
		\node at (-2,3)  [dot] (left1) {};
		\node at (0,5)  [dot] (right) {};
		\node at (0,7)  [dot] (right1) {};
		
		\draw[testfcn] (left) to  (root);
		
		\draw[kernel] (left1) to (left);
		\draw[kernel] (right) to (root);
		\draw[kernel1] (right1) to (right);
		\draw[ddrho] (left) to (right); 
		\draw[ddrho] (left1) to (right1); 
	\end{tikzpicture}
	\;-\;
	\begin{tikzpicture}[scale=0.35,baseline=0.9cm]
		\node at (0,-1)  [root] (root) {};
		\node at (-2,1)  [dot] (left) {};
		\node at (-2,3)  [dot] (left1) {};
		\node at (0,7)  [dot] (right1) {};
		\node at (0,5) [dot] (right) {};
		
		\draw[testfcn] (left) to  (root);
		
		\draw[kernel2] (left1) to (left);
		\draw[kernel1] (right1) to (right);
		\draw[kernel] (right) to (root);
		\draw[ddrho-shift] (left1) to (right);
		\draw[ddrho-shift] (right1) to (left);
	\end{tikzpicture}
	\;+\;
	\begin{tikzpicture}[scale=0.35,baseline=0.9cm]
		\node at (0.5,-1)  [root] (root) {};
		\node at (-2,1)  [dot] (left) {};
		\node at (-2,3)  [dot] (left1) {};
		\node at (-1,5)  [dot] (left2) {};
		\node at (-2,7)  [dot] (left3) {};
		\node at (0.5,1) [var] (variable1) {};
		\node at (0.5,5) [var] (variable3) {};

		\draw[testfcn] (left) to  (root);
		
		\draw[kernel2] (left1) to (left);
		\draw[kernel1] (left2) to (left1);
		\draw[kernel1] (left3) to (left2);
		\draw[drho] (variable3) to (left2); 
		\draw[ddrho] (left3) to (left1);
		\draw[drho] (variable1) to (left); 
	\end{tikzpicture}
	\;-\;
	\begin{tikzpicture}[scale=0.35,baseline=0.9cm]
		\node at (0,-1)  [root] (root) {};
		\node at (-2,1)  [dot] (left) {};
		\node at (-2,3)  [dot] (left1) {};
		\node at (0,5)  [dot] (right2) {};
		\node at (-2,7)  [dot] (left3) {};
		\node at (-4,1) [var] (variable) {};
		\node at (0,7) [var] (variable3) {};
		
		\draw[testfcn] (left) to  (root);
		
		\draw[kernel2] (left1) to (left);
		\draw[kernel1] (left3) to (right2);
		\draw[kernel] (right2) to (root);
		\draw[ddrho] (left1) to (right2);
		\draw[drho] (variable3) to (left3); 
		\draw[drho] (variable) to (left); 
	\end{tikzpicture}
	\right.\\&
	\left.\;-\;
	\begin{tikzpicture}[scale=0.35,baseline=0.9cm]
		\node at (0,-1)  [root] (root) {};
		\node at (-2,1)  [dot] (left) {};
		\node at (-2,3)  [dot] (left1) {};
		\node at (-2,5)  [dot] (left2) {};
		\node at (0,5) [dot] (right) {};
		\node at (-4,1) [var] (variable) {};
		\node at (-4,3) [var] (variable1) {};
		
		\draw[testfcn] (left) to  (root);
		
		\draw[kernel2] (left1) to (left);
		\draw[kernel1] (left2) to (left1);
		\draw[ddrho] (right) to (left2);
		\draw[drho] (variable1) to (left1); 
		\draw[drho] (variable) to (left); 
		\draw[kernel] (right) to (root);
	\end{tikzpicture}
	\;+\;
	\begin{tikzpicture}[scale=0.35,baseline=0.9cm]
		\node at (0,-1)  [root] (root) {};
		\node at (-2,1)  [dot] (left) {};
		\node at (-2,3)  [dot] (left1) {};
		\node at (-2,5)  [dot] (left2) {};
		\node at (-2,7)  [dot] (left3) {};
		\node at (0,1) [var] (variable) {};
		\node at (0,7) [var] (variable3) {};
		
		\draw[testfcn] (left) to  (root);
		
		\draw[kernel2] (left1) to (left);
		\draw[kernelBig] (left2) to (left1);
		\draw[kernel1] (left3) to (left2);
		\draw[drho] (variable3) to (left3); 
		\draw[drho] (variable) to (left); 
	\end{tikzpicture}
	\;+\;
	\begin{tikzpicture}[scale=0.35,baseline=0.9cm]
		\node at (0,-1)  [root] (root) {};
		\node at (-2,1)  [dot] (left) {};
		\node at (-2,3)  [dot] (left1) {};
		\node at (-2,5)  [dot] (left2) {};
		\node at (-2,7)  [dot] (left3) {};
		\node at (0,5) [var] (variable2) {};
		\node at (0,7) [var] (variable3) {};
		
		\draw[testfcn] (left) to  (root);
		
		\draw[kernelBig] (left1) to (left);
		\draw[kernel1] (left2) to (left1);
		\draw[kernel1] (left3) to (left2);
		\draw[drho] (variable3) to (left3); 
		\draw[drho] (variable2) to (left2); 
	\end{tikzpicture}
	\;-\;
	\begin{tikzpicture}[scale=0.35,baseline=0.9cm]
		\node at (0,-1)  [root] (root) {};
		\node at (-2,1)  [dot] (left) {};
		\node at (-2,5)  [dot] (left2) {};
		\node at (-2,7)  [dot] (left3) {};
		\node at (0,3) [dot] (right1) {};
		\node at (0,5) [var] (variable2) {};
		\node at (0,7) [var] (variable3) {};
		
		\draw[testfcn] (left) to  (root);
		
		\draw[kernel1] (left2) to (right1);
		\draw[kernel1] (left3) to (left2);
		\draw[kernel] (right1) to (root);
		\draw[ddrho] (left) to (right1);
		\draw[drho] (variable3) to (left3); 
		\draw[drho] (variable2) to (left2); 
	\end{tikzpicture}
	\;-\; \sum_{j = 1, 2} \;
	\begin{tikzpicture}[scale=0.35,baseline=0.9cm]
		\node at (0,-1)  [root] (root) {};
		\node at (-2,1)  [dot] (left) {};
		\node at (-2,5)  [dot] (left2) {};
		\node at (-2,7)  [dot] (left3) {};
		\node at (0,3) [dot] (right1) {};
		\node at (0,5) [var] (variable2) {};
		\node at (0,7) [var] (variable3) {};
		
		\draw[testfcnx] (left) to node[below, pos=0.4]{\tiny\color{black} $j$} (root);
		
		\draw[kernel1] (left2) to (right1);
		\draw[kernel1] (left3) to (left2);
		\draw[dkernel] (right1) to node[right, pos=0.5]{\tiny\color{black} $j$} (root);
		\draw[ddrho] (left) to (right1);
		\draw[drho] (variable3) to (left3); 
		\draw[drho] (variable2) to (left2); 
	\end{tikzpicture}
	\right.\\&
	\left.\;+\;
	\begin{tikzpicture}[scale=0.35,baseline=0.9cm]
		\node at (0.5,-1)  [root] (root) {};
		\node at (-2,1)  [dot] (left) {};
		\node at (-1,3)  [dot] (left1) {};
		\node at (-2,5)  [dot] (left2) {};
		\node at (-2,7)  [dot] (left3) {};
		\node at (0.5,3) [var] (variable2) {};
		\node at (0.5,7) [var] (variable4) {};

		\draw[testfcn] (left) to  (root);
		
		\draw[kernel2] (left1) to (left);
		\draw[kernel1] (left2) to (left1);
		\draw[kernel1] (left3) to (left2);
		\draw[drho] (variable4) to (left3); 
		\draw[ddrho] (left2) to (left);
		\draw[drho] (variable2) to (left1); 
	\end{tikzpicture}
	\;+\;
	\begin{tikzpicture}[scale=0.35,baseline=0.9cm]
		\node at (0.5,-1)  [root] (root) {};
		\node at (-2,1)  [dot] (left) {};
		\node at (-1,3)  [dot] (left1) {};
		\node at (-1,5)  [dot] (left2) {};
		\node at (-2,7)  [dot] (left3) {};
		\node at (0.5,3) [var] (variable2) {};
		\node at (0.5,5) [var] (variable3) {};

		\draw[testfcn] (left) to  (root);
		
		\draw[kernel2] (left1) to (left);
		\draw[kernel1] (left2) to (left1);
		\draw[kernel1] (left3) to (left2);
		\draw[drho] (variable3) to (left2); 
		\draw[ddrho] (left3) to (left);
		\draw[drho] (variable2) to (left1); 
	\end{tikzpicture}\;+\;
	\begin{tikzpicture}[scale=0.35,baseline=0.9cm]
		\node at (0,-1)  [root] (root) {};
		\node at (-2,1)  [dot] (left) {};
		\node at (-2,5)  [dot] (left2) {};
		\node at (-2,7)  [dot] (left3) {};
		\node at (0,7) [dot] (variable3) {};

		\draw[testfcn] (left) to  (root);
		
		\draw[ddrho] (left) to (left2);
		\draw[kernel1] (left3) to (left2);
		\draw[ddrho] (variable3) to (left3);
		\draw[kernel] (left2) to (root);
		\draw[kernel] (variable3) to (root); 
	\end{tikzpicture}
	\;+ \sum_{j = 1,2} \;
	\begin{tikzpicture}[scale=0.35,baseline=0.9cm]
		\node at (0,-1)  [root] (root) {};
		\node at (-2,1)  [dot] (left) {};
		\node at (-2,5)  [dot] (left2) {};
		\node at (-2,7)  [dot] (left3) {};
		\node at (0,7) [dot] (variable3) {};

		\draw[testfcnx] (left) to node[below, pos=0.4]{\tiny\color{black} $j$}  (root);
		
		\draw[ddrho] (left) to (left2);
		\draw[kernel1] (left3) to (left2);
		\draw[ddrho] (variable3) to (left3);
		\draw[dkernel] (left2) to node[above, pos=0.5]{\tiny\color{black} $j$} (root);
		\draw[kernel] (variable3) to (root); 
	\end{tikzpicture}
	\;-\;
	\begin{tikzpicture}[scale=0.35,baseline=0.6cm]
		\node at (0,-1)  [root] (root) {};
		\node at (-2,1)  [dot] (left) {};
		\node at (-2,3)  [dot] (left1) {};
		\node at (-2,5)  [dot] (left2) {};
		\node at (0,5) [dot] (right3) {};

		\draw[testfcn] (left) to  (root);
		
		\draw[kernelBig] (left1) to (left);
		\draw[kernel1] (left2) to (left1);
		\draw[ddrho] (right3) to (left2);
		\draw[kernel] (right3) to (root);
	\end{tikzpicture}
	\right.\\&
	\left.\;-\;
	\begin{tikzpicture}[scale=0.35,baseline=0.6cm]
		\node at (-2,-1)  [root] (root) {};
		\node at (0,1)  [dot] (left) {};
		\node at (-2,3)  [dot] (left1) {};
		\node at (0,3)  [dot] (left2) {};
		\node at (0,5)  [dot] (left3) {};
		
		\draw[testfcn] (left) to  (root);
		
		\draw[kernel] (left1) to (root);
		\draw[kernel1] (left2) to (left1);
		\draw[kernel1] (left3) to (left2);
		\draw[ddrho] (left) to (left2);
		\draw[ddrho] (left1) to (left3);
	\end{tikzpicture}
	\;-\sum_{j = 1, 2} \; 
	\begin{tikzpicture}[scale=0.35,baseline=0.6cm]
		\node at (-2,-1)  [root] (root) {};
		\node at (0,1)  [dot] (left) {};
		\node at (-2,3)  [dot] (left1) {};
		\node at (0,3)  [dot] (left2) {};
		\node at (0,5)  [dot] (left3) {};
		
		\draw[testfcnx] (left) to node[below, pos=0.4]{\tiny\color{black} $j$}  (root);
		
		\draw[dkernel] (left1) to node[right, pos=0.5]{\tiny\color{black} $j$} (root);
		\draw[kernel1] (left2) to (left1);
		\draw[kernel1] (left3) to (left2);
		\draw[ddrho] (left) to (left2); 
		\draw[ddrho] (left1) to (left3);
	\end{tikzpicture}
	\;-\;
	\begin{tikzpicture}[scale=0.35,baseline=0.6cm]
		\node at (0,-1)  [root] (root) {};
		\node at (-2,1)  [dot] (left) {};
		\node at (-2,3)  [dot] (left1) {};
		\node at (0,5)  [dot] (right1) {};
		\node at (0,2) [dot] (right) {};
		
		\draw[testfcn] (left) to  (root);
		
		\draw[kernel] (left1) to (left);
		\draw[kernel] (right) to (left1);
		\draw[kernel,bend left=40] (right1) to (root);
		\draw[ddrho] (left1) to (right1);
		\draw[ddrho] (right) to (left); 
	\end{tikzpicture}
	\;-\;
	\begin{tikzpicture}[scale=0.35,baseline=0.6cm]
		\node at (0,-1)  [root] (root) {};
		\node at (-2,1)  [dot] (left) {};
		\node at (0,3)  [dot] (right) {};
		\node at (0,5)  [dot] (right1) {};
		\node at (-2,3) [dot] (left1) {};
		
		\draw[testfcn] (left) to  (root);
		
		\draw[kernel] (right) to (root);
		\draw[kernelBig] (right1) to (right);
		\draw[kernel] (left1) to (right1);
		\draw[ddrho] (left1) to (left);
	\end{tikzpicture}
	\;-\sum_{j = 1,2} \;
	\begin{tikzpicture}[scale=0.35,baseline=0.6cm]
		\node at (0,-1)  [root] (root) {};
		\node at (-2,1)  [dot] (left) {};
		\node at (0,3)  [dot] (right) {};
		\node at (0,5)  [dot] (right1) {};
		\node at (-2,3) [dot] (left1) {};
		
		\draw[testfcnx] (left) to node[below, pos=0.4]{\tiny\color{black}$j$}  (root);
		
		\draw[dkernel] (right) to node[right, pos=0.5]{\tiny\color{black} $j$} (root);
		\draw[kernelBig] (right1) to (right);
		\draw[kernel] (left1) to (right1);
		\draw[ddrho] (left1) to (left);
	\end{tikzpicture}\right)\;.
\end{equs}
For the stochastic graphs in \eqref{eq:lastone} and \eqref{eq:XiIXiIXiIXi} 
with at least one noise node, we compute their second moment by the Wick's product formula, which in turn results in a linear combination of Feynman graphs. We then define $\cG$ to be the collection of stochastic graphs without noise node (which are indeed Feynman graphs themselves) as well as the Feynman graphs resulted from the second moment computation.

To formalise the pairing process used in the second moment computation via Wick's product formula, let us denote by $\cS$ the collection of all stochastic graphs listed in \eqref{eq:lastone} and \eqref{eq:XiIXiIXiIXi}. Whenever needed, one can number the noise nodes of a given stochastic graph $S \in \cS$ by $\{1, 2, \dots, n\}$. Given a numbered stochastic graph and a permutation $\sigma$ on the set $\{1, 2, \dots n\}$, let us define a map $\phi_\sigma$ such that $\phi_\sigma(S)$ is the Feynman graph formed by taking two copies of $S$ and pairing the noise node numbered by $i \in \{1,\dots, n\}$ of the former copy to the noise node $\sigma(i)$ of the latter copy and turning them into a single \tikz[baseline=-0.1cm] \node[dot] at (0,0) {}; node, and finally identifying the two $\origin$ nodes of the two copies. Define then $\phi$ to be the multivalued function such that the values of $\phi(S)$ are $\phi_\sigma(S)$ with $\sigma$ running over all permutation of $n$ elements. When $S$ has no noise node, $\phi(S)$ is simply $S$ itself. With this notation, we set $\mathcal{G}=\{\phi(S): S \in \cS\}$.

We will also use the notation $\phi_{i, j}(S)=\{\phi_\sigma(S):\,\sigma(i)=j\}$ and similarly $\phi_{ij, kl}(S)=\{\phi_\sigma(S):\,\sigma(\{i, j\}) = \{k, l\}\}$.\\

Let us categorise the available edge types into the following classes
\begin{equs}[eq:edge-classes]
	E_\bM &= \left\{ \quad
	\tikz[baseline=-0.6cm] \draw[rho] (0, 0) to (0,-1); \;,\quad
	\tikz[baseline=-0.6cm] \draw[drho] (0, 0) to (0,-1); \;,\quad
	\tikz[baseline=-0.6cm] \draw[ddrho] (0, 0) to (0,-1); \;, \quad
	\tikz[baseline=-0.6cm] \draw[kernelBig] (0,0) -- (0,-1); \quad\right\}, \quad
	E_* = \left\{\tikz[baseline=-0.6cm] \draw[testfcn] (0,-1) -- (0,0);\;,\quad
	\tikz[baseline=-0.6cm] \draw[testfcnx] (0,-1) -- (0,0);, \quad
	\tikz[baseline=-0.6cm] \draw[dtestfcn] (0,-1) -- (0,0);\right\},\\
	E_\bK &= \left\{ \quad
	\tikz[baseline=-0.6cm] \draw[kernel] (0,0) -- (0,-1); \;,\quad
	\tikz[baseline=-0.6cm] \draw[kernel1] (0,0) -- (0,-1); \;,\quad
	\tikz[baseline=-0.6cm] \draw[kernel2] (0,0) -- (0,-1); \;,\quad
	\tikz[baseline=-0.6cm] \draw[dkernel] (0,0) -- (0,-1); \;,\quad
	\tikz[baseline=-0.6cm] \draw[dkernel1] (0,0) -- (0,-1); \;,\quad
	\tikz[baseline=-0.6cm] \draw[dkernel2] (0,0) -- (0,-1); \;,\quad
	\tikz[baseline=-0.6cm] \draw[ddkernel] (0,0) -- (0,-1); \;,\quad
	\tikz[baseline=-0.6cm] \draw[kernelBig] (0,0) -- (0,-1); \quad\right\}\;.
\end{equs}
Note that $E_\bM$ and $E_\bK$ intersect at $\kernelBig$. In plain words, $E_\bM$ is the set of edge types derived from $\rho_\eps^{\ast 2}$, $E_*$ includes those from the test functions while $E_\bK$ those from $K$. Note that many of these edges do not appear in graphs in $\cG$, but can appear after performing IBP.

Let us collect some useful facts about the graphs in $\cG$.
\begin{fact}\label{fact:graph}
	Every graph $(V, E)$ in $\cG$ satisfies the following properties:
	\begin{enumerate}
		\item $E$ does not contain the following edges
		
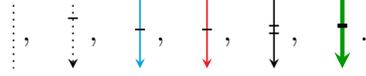
\begin{equ}
			\tikz[baseline=-0.6cm] \draw[rho] (0,0) -- (0,-1); \;,\quad
			\tikz[baseline=-0.6cm] \draw[drho] (0, 0) to (0,-1); \;,\quad
			\tikz[baseline=-0.6cm] \draw[dkernel1] (0,0) -- (0,-1); \;,\quad
			\tikz[baseline=-0.6cm] \draw[dkernel2] (0,0) -- (0,-1); \;,\quad
			\tikz[baseline=-0.6cm] \draw[ddkernel] (0,0) -- (0,-1); \;,\quad
			\tikz[baseline=-0.6cm] \draw[dtestfcn] (0,-1) -- (0,0); \;.
		\end{equ}
		\item $E_\bM$ forms a perfect matching of $V \setminus \{\origin\}$. That is, $V \setminus \{\origin\}$ has even number of elements and can be partitioned into pairs $\{v, w\}$ such that $v$ and $w$ are joined by an $E_\bM$ edge if and only if $\{v, w\}$ is a pair.
		\item Vertices in $V_\ast$ have no outgoing $E_\bK$ edge, and every vertex in $V \setminus V_\ast$ has exactly one outgoing $E_\bK$ edge.
		\item Edges with $r_e \neq 0$ never touch $\origin$, and at most one renormalised edge (i.e., $r_e < 0$) emerges from any given vertex.
		\item $E$ contains at most two edges of type $\{\tikz[baseline=-0.1cm] \draw[kernelBig] (0,0) -- (1,0);, \tikz[baseline=-0.1cm] \draw[kernel2] (0,0) -- (1,0);, \tikz[baseline=-0.1cm] \draw[dkernel] (0,0) -- (1,0);\}$.
		\item Edges of type \tikz[baseline=-0.1cm] \draw[dkernel] (0,0) -- (1,0); must be attached to the origin.
		\item Edges of type $\kernelrr$ always point to a vertex in $V_\ast\setminus\{\origin\}$.
	\end{enumerate}
\end{fact}

We will start by working with the \emph{canonical assignment} of the edge degrees.
\begin{definition}\label{def:canonical_label}
	The \emph{canonical labelling} $L^\can = (a_e, r_e)$ of
	a graph $\eps^{\frac{|V|-1}{2}} G$, $G\in\cG$, is 
	given by \eqref{eq:labels} with $\bar\kappa = 0$ for $e \notin E_\bM$,
	and by distributing $\eps$ and applying
	 \eqref{eq:labels_eps} with $\bar \kappa = 0, \gamma = 1$ for $e \in E_\bM$.
\end{definition}
As an example, the Feynman graphs in \eqref{eq:dumbbell_var} would be labelled by
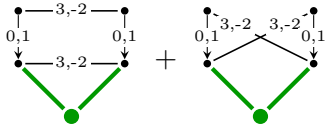
\begin{equs}
	\begin{tikzpicture}[scale=0.35,baseline=0.3cm]
		\node at (0,-1)  [root] (root) {};
		\node at (-2,1)  [dot] (left) {};
		\node at (-2,3)  [dot] (left1) {};
		\node at (2,1)  [dot] (right) {};
		\node at (2,3)  [dot] (right1) {};
		
		\draw[dist] (left) to  (root);
		\draw[dist] (right) to  (root);
		
		\draw[->] (left1) to node[labl]{\tiny 0,1} (left);
		\draw[generic] (left1) to node[labl]{\tiny 3,-2} (right1);
		\draw[generic] (left) to node[labl]{\tiny 3,-2} (right); 
		\draw[->] (right1) to node[labl]{\tiny 0,1} (right);
	\end{tikzpicture}\;
	+ \;
	\begin{tikzpicture}[scale=0.35,baseline=0.3cm]
		\node at (0,-1)  [root] (root) {};
		\node at (-2,1)  [dot] (left) {};
		\node at (-2,3)  [dot] (left1) {};
		\node at (2,1)  [dot] (right) {};
		\node at (2,3)  [dot] (right1) {};
		
		\draw[dist] (left) to  (root);
		\draw[dist] (right) to  (root);
		
		\draw[->] (left1) to node[labl]{\tiny 0,1} (left);
		\draw[generic] (left1) to node[labl, pos=0.25]{\tiny 3,-2} (right);
		\draw[generic] (right1) to node[labl, pos=0.25]{\tiny 3,-2} (left); 
		\draw[->] (right1) to node[labl]{\tiny 0,1} (right);
	\end{tikzpicture}
\end{equs}
under $L^\can$.

An advantage of the canonical labelling is that it successfully predicts the necessary $\lambda$ powers needed for the tightness of models.
Indeed, it can be checked that if $\tau \in V^\RHS$ and $\eps^{\frac{|V|-1}{2}}G$ is a Feynman graph generated by $\E[(\hat \Pi^\eps_0 \tau (\varphi^\lambda))^2]$, then the $\lambda$ exponent $\alpha$ predicted by Theorem \ref{thm:power_counting} under canonical labelling differs from the required exponent $2|\tau|$ only by a multiple of $\kappa$.
Consequently, the desired tightness bound \eqref{eq:tightness_model} for the BPHZ model $(\hat \Pi^\eps, \hat \Gamma^\eps)$ can be obtained if one can show for every $G \in \cG$ and all $\kappa' > 0$
\begin{equ}[eq:tightness_graph]
	\left|\bI\big(\eps^{\frac{|V|-1}{2}}G, \varphi^\lambda\big)\right| \lesssim \lambda^{\alpha-\kappa'}
\end{equ}
uniformly for $\eps > 0$, with $\alpha = 2|V\setminus V_\ast| - \sum_{e\in E} a_e$ determined by nothing but the canonical assignment of edge degree $a_e$ in $L^\can$.

The problem of the canonical labelling $L^\can$ is that it is neither faithful to $\eps^{\frac{|V|-1}{2}} G$ in the sense of Definition \ref{def:faithful} (the kernel $K$ explodes like a logarithm near $0$ rather than stays bounded) nor satisfies assumptions of Theorem \ref{thm:power_counting} for any Feynman graph in $\cG$.
Moreover, we have exhausted all of the $\eps$ powers, so this would not enable us to show convergence to $0$.
However, working with $L^\can$ simplifies the expression of \eqref{eq:power_counting-2}, \eqref{eq:power_counting-3}, \eqref{eq:power_counting-4} and gives insight on how one fixes the aforementioned problems. Below we will build an adjusted singularity assignment $a^\adj_e$ by tweaking $a_e$ so that the new label is faithful to respective kernels,  allows us to verify the assumptions of Theorem \ref{thm:power_counting} for a large class of Feynman graphs, and requires slightly less $\eps$ power.

Note that under the canonical labelling the condition 1 of Theorem \ref{thm:power_counting} is automatically verified since $L^\can$ is chosen such that $a_e + r_e \wedge 0 \leq 1$ for every edge $e$. Furthermore, condition 0 can also be checked, as no kernel edge with $r_e > 0$ connects vertices in $V_\ast$ due to Fact \ref{fact:graph} item 3; no coloured kernel edges ($r_e > 0$) or renormalised edges ($r_e < 0$) are attached to $\origin$ (Fact \ref{fact:graph} item 4) and no vertex is attached to more than one renormalised edge (Fact \ref{fact:graph} item 4). Subsequently, we will focus on the remaining conditions.\\

For the reason of presentation, let us further introduce subclasses of $E_\bM$, $E_{\bK}$ of edges appearing in $\cG$ according to to the singularity index. Namely,
\begin{equs}
	E_\bM^3 &= \left\{ \quad
	\tikz[baseline=-0.6cm] {\draw[ddrho] (0, 0) to (0,-1);} \;, \quad
	\tikz[baseline=-0.6cm] \draw[kernelBig] (0,0) -- (0,-1); \quad\right\}\;,
	\quad
	E_\bM^1 = \left\{ \quad
	\tikz[baseline=-0.6cm] {\draw[rho] (0, 0) to (0,-1);} \; \quad\right\}\;,\\
	E_\bK^0 &= \left\{ \quad
	\tikz[baseline=-0.6cm] \draw[kernel] (0,0) -- (0,-1); \;,\quad
	\tikz[baseline=-0.6cm] \draw[kernel1] (0,0) -- (0,-1); \;,\quad
	\tikz[baseline=-0.6cm] \draw[kernel2] (0,0) -- (0,-1); \;\quad
	\right\}\;,
	\quad
	E_\bK^1 = \left\{ \quad
	\tikz[baseline=-0.6cm] \draw[dkernel] (0,0) -- (0,-1); \;,\quad
	\tikz[baseline=-0.6cm] \draw[dkernel1] (0,0) -- (0,-1); \;,\quad
	\tikz[baseline=-0.6cm] \draw[dkernel2] (0,0) -- (0,-1); \quad\right\}\;.
\end{equs}
In particular, if $e \in E_\bK^i$ or $E_\bM^i$, then one has $a_e = i$.

\subsubsection{Condition 2}

We first look at the condition 2, that is, \eqref{eq:power_counting-2}. 
Here, only the labels $a_e$ play a role, so we omit the recentering of edges for now. Furthermore, since we only care about subsets $\bar V \subset V \setminus \{0\}$, by Fact \ref{fact:graph} items 1 and 6 and that $E_*$-edges are always attached to $\origin$, only the edge sets $E_\bM$ and $E_{\bK}^0$ need to be considered.

\begin{definition}\label{def:critical}
	We call the following types of subgraphs the \emph{critical subgraphs}:
	\begin{enumerate}
		\item[(i)] Critical pair: two disjoint edges from $E_\bM^3$.
		\item[(ii)] Critical segment: a path of length $3$ with two edges from $E_\bM^3$ and one edge from $E_\bK^0$.
		\item[(iii)] Critical block: a cycle of length $4$ with two edges from $E_\bM^3$ and two edges from $E_{\bK}^0$.
	\end{enumerate}
\end{definition}

Given any subgraph $\bar G = (\bar V, \bar E)$ of $G \in \cG$ where $\bar E = E_0(\bar V)$, let us define
\begin{equ}[eq:deg2]
	\deg_2(\bar G) = 2(|\bar V|-1) - \sum_{e \in \bar E} a_e\;,
\end{equ}
using the canonical labelling. That is, condition \eqref{eq:power_counting-2} can be rewritten as $\deg_2(\bar G)>0$ for all relevant $\bar G$.
Note that all critical subgraphs $\bar G$ have $\deg_2(\bar G)=0$.

\begin{lemma}\label{lem:cond2}
	All subgraphs $\bar G$ of ${\color{black}\eps^{\frac{|V|-1}{2}} G} \in \cG$ relevant for condition $2$ satisfy $\deg_2(\bar G) \ge 0$, and the equality occurs exactly for critical subgraphs.
\end{lemma}
\begin{proof}
	Note that since the $\dkernel$ edges must be attached to $\origin$ (Fact \ref{fact:graph} item 6) and the other differentiated kernels are absent in $\cG$ (Fact \ref{fact:graph} item 1), only $E_\bM^3$ edges contribute to the sum $\sum_{e \in \bar E} a_e$ for all relevant $\bar G$.
	As a consequence of Fact \ref{fact:graph} item 2, the sum $\sum_{e \in \bar E} a_e$ coincides with the number of pairs of vertices in $\bar G$ matched by $E_\bM^3$ multiplied by $3$.
	Denote by $u(\bar G)$ the number of vertices in $\bar V$ not paired by $E_\bM$-edges in $E_0(\bar V)$. Then
	\begin{equ}[eq:deg2_Gbar]
		\deg_2(\bar G) = 2|\bar V| - 2 - \frac32(|\bar V| - u(\bar G)) = \frac12 |\bar V| + \frac32 u(\bar G) - 2.
	\end{equ}
	Recalling that for odd $|\bar V|$ one has $u(\bar G)\geq 1$, it is immediate that for $|\bar V|=3$ and $|\bar V|\geq5$ one always has $\deg_2(\bar G) \geq 1$. Therefore $\deg_2(\bar G) = 0$ can only occur when $|\bar V|=4$, and that the three possibilities are the ones described in Definition \ref{def:critical} follow from a simple case distinction.
\end{proof}

\subsubsection{Adjusting the labelling}
Notice that for a graph in $G \in \cG$, every mollifier edge $e \in E_\bM$ is contained by various critical subgraphs in the sense of Definition \ref{def:critical}. However, with a closer look one can observe that some of the mollifer edges are only covered by critical pairs or segments, but not by a critical block. If we pick such a mollifier edge $e^*$ of type $\ddrho$ and perform integration by parts at both of its endpoints to move the derivatives away, then one can notice the subgraphs covering $e^*$ which originally have $\deg_2 = 0$ turn into having positive degree. Therefore, the $\eps$ factor originally attributed to $e^*$ under the canonical assignment can be decreased and one can then redistribute the excessive amount of $\eps$ to other mollifier edges, eliminating all critical subgraphs.

To make this argument precise, let us formally define the IBP operations. Fixing an edge $e^* = (e^*_-, e^*_+)$, we can perform IBP at the vertices $e^*_-, e^*_+$ and turn $G$ into a linear combination of Feynman graphs $G_I$, where the index $I$ runs over all IBP maps at $e^*$.
\begin{definition}[IBP at an edge $e^*$]\label{def:IBP}
	We call $I: \{e^*_-, e^*_+\} \to E$ an IBP map at $e^*$ if $I(v)$ is an edge attached to the vertex $v$ different from $e^*$.\\
	Given an IBP map $I$ at $e^*$, the graph associated $G_I = (V, E^I)$ is defined by replacing the type of each edge $e = (e_-, e_+) \in E$ according to the following rule.
	\begin{itemize}
		\item The type of $e^*$ is changed from $\ddrho$ to $\erho$; for all $e \in (E\cap E_\bM) \setminus \{e_*\}$, the type of $e$ is unchanged.
		\item For $e \in E \cap (E_\bK \cup E_*)$:
		\begin{itemize}
			\item If no vertex $v \in \{e_-, e_+\}$ is such that $I(v) = e$, then the edge type of $e$ stays unchanged.
			\item Otherwise, there exists exactly one $v \in \{e_-, e_+\}$ such that $I(v) = e$, in which case we use the following \emph{replacement rules}:
			\begin{itemize}
				\item if $e = \kernel$, then replaced by \dkernel.
				\item if $e = \kernelr$ and $I(e_+) = e$, then replaced by $\dkernel$.
				\item if $e = \kernelr$ and $I(e_-) = e$, then replaced by $\dkernelr$.
				\item if $e = \kernelrr$ and $I(e_+) = e$, then replaced by $\dkernelr$.
				\item if $e = \kernelrr$ and $I(e_-) = e$, then replaced by $\dkernelrr$.
				\item if $e = \testfcn$, then replaced by $\dtestfcn$.
				\item if $e = \testfcnx$, then replaced by $\testfcn$.
			\end{itemize}
		\end{itemize}
	\end{itemize}
\end{definition}

By the choice below we never perform IBP at an edge adjacent to an $\dkernel$ edge.

\begin{proposition}\label{prop:adj}
	For a Feynman graph $G = (V, E) \in \cG$, suppose that $G$ contains a $\ddrho$ edge, denoted by $e^*$, not contained in a critical block and is not attached to a $\dkernel$ edge.
Take an IBP map $I$ at $e^*$ and the resulting Feynman graph $G_I$.
Define $c: E^I \cap E_\bM \to \R$ by $c_{e_\ast} = \sqrt{\bar \kappa}$ and $c_e=-\bar \kappa$ for $e \neq e_\ast$ and set $a^{\adj} : E^I \to \R$ by
	\begin{equ}
		a_e^{\adj}:=
		\begin{cases}
			\bar \kappa^2,	&\text{ if } e\in E_\bK^0,\\
			a_e + c_e,	&\text{ if } e \in E_\bM,\\
			a_e,		&\text{ otherwise.}
		\end{cases}
	\end{equ}
		Then, for sufficiently small $\bar \kappa > 0$, one has $c(\bar\kappa):=\sum_e c_e > 0$, and moreover the label assignment $(a_e^\adj, r_e)_{e \in E}$ is faithful to $\eps^{\frac{|V|-1}{2}-c(\bar\kappa)}G_I$ and satisfies the conditions 0, 1 and 2 of Theorem \ref{thm:power_counting}.
\end{proposition}
\begin{proof}
	For each such IBP map $I$ at $e^*$, it is obvious that condition $1$ is preserved for $G_I$. For condition $0$, we need to make sure that no renormalised kernel in $G_I$ (i.e., $\ddkernel$ and $\kernelBig$) is attached to $\origin$ after IBP: this can only happen if $e^*$ is attached to a $\dkernel$ edge in $G$, a case which we have excluded by assumption.
	
	Take a subgraph $\bar G = (\bar V, E_0(\bar V))$ of $G$ and let $\bar G_I = (\bar V, E_0^I(\bar V))$ be its counterpart after IBP.
	We aim to show that $\deg_2(\bar G_I)\geq 0$ and the inequality is strict if $e^*\in E_0(\bar V)$.	 
	 Let us estimate $\deg_2(\bar G_I)$.
	\begin{itemize}
		\item If $e^* \in E_0(\bar V)$, then it is obvious that
		\begin{equ}
			\deg_2(\bar G_I) \ge \deg_2(\bar G) \ge 0\;.
		\end{equ}
		Furthermore, the second $\geq$ is $=$ in exactly the critical subgraphs in the sense of Definition \ref{def:critical}. By assumption $e^*$ is not a part of a critical block, and it is clear that for a critical segment containing $e^*$, the first inequality above is strict.
		\item If $e^* \in E(\bar V) \setminus E_0(\bar V)$, then $I$ can move one additional derivative into $E_0^I(\bar V)$. However, this situation happens only when there is an unpaired vertex by $E_\bM$ edges. This leads to the estimate
		\begin{equ}
			\deg_2(\bar G_I) \ge \deg_2(\bar G) - \un_{u(\bar G) \ge 1} \ge \frac12|\bar V| - \frac32,
		\end{equ}
		where we have used \eqref{eq:deg2_Gbar}. This shows $\deg_2(\bar G_I) \ge 0$ for all $|\bar V| \ge 3$.
		\item Otherwise, $\deg_2(\bar G_I) = \deg_2(\bar G)\geq 0$.
	\end{itemize}
	Note that, due to the $\bar \kappa^2 > 0$ assigned to the $E_\bK^0$ edges, now the assignment $a^\adj_e$ is faithful to every edge $e$ in the sense of Definition \ref{def:faithful}. Let us set $\deg_2^\adj$ to be defined by \eqref{eq:deg2} but with $a_e$ replaced by $a_e^\adj$. Now, observe that every subgraph $\bar G_I$ containing $e^*$ has $\deg_2(\bar G_I) > 0$, since $e^*$ was not part of any critical block before IBP. Therefore, one can reduce the $\eps$ powers attributed to $e^*$ and redistribute them to the other mollifier edges. The adjusted assignment $a^\adj_e$ is designed so that, provided $\bar \kappa$ small enough, all subgraphs containing $e^*$ preserve positive degree while all mollifier edges other than $e^*$ obtain a positive amount of $\eps$ powers. As all relevant subgraphs with $\deg_2 = 0$ had at least one mollifier edge, it follows that
	\begin{equ}
		\deg_2^\adj (\bar G_I) > 0
	\end{equ}
	for all relevant subgraph $\bar G_I$ of $G_I$. This concludes the proof.
\end{proof}

Let us now single out the set of Feynman graphs to which Proposition \ref{prop:adj} does not apply.
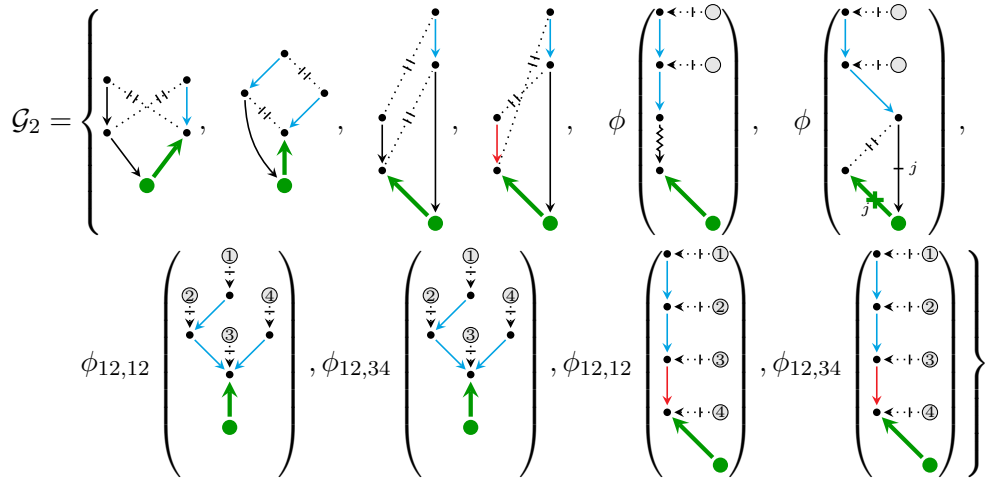
\begin{equs}
	\cG_2 =&\left\{
	\begin{tikzpicture}[scale=0.35,baseline=0.4cm]
		\node at (0,-1)  [root] (root) {};
		\node at (1.5, 1) [dot] (right1) {};
		\node at (-1.5, 1) [dot] (left1) {};
		\node at (-1.5,3)  [dot] (left2) {};
		\node at (1.5,3)  [dot] (right2) {};
		
		\draw[testfcn] (right1) to  (root);
		
		\draw[kernel1] (right2) to (right1);
		\draw[kernel] (left2) to (left1);
		\draw[kernel] (left1) to (root); 
		\draw[ddrho-shift] (left2) to (right1); 
		\draw[ddrho-shift] (right2) to (left1); 
	\end{tikzpicture}\;,
	\quad
	\begin{tikzpicture}[scale=0.35,baseline=0.4cm]
		\node at (0,-1)  [root] (root) {};
		\node at (0,1)  [dot] (root2) {};
		\node at (-1.5,2.5)  [dot] (left) {};
		\node at (1.5,2.5)  [dot] (right) {};
		\node at (0,4) [dot] (top) {};
		
		\draw[testfcn] (root2) to  (root);
		
		\draw[kernel,bend right=20] (left) to (root);
		\draw[kernel1] (right) to (root2);
		\draw[kernel1] (top) to (left);
		\draw[ddrho] (left) to (root2); 
		\draw[ddrho] (top) to (right);
	\end{tikzpicture}\;,
	\quad
	\begin{tikzpicture}[scale=0.35,baseline=0.9cm]
		\node at (0,-1)  [root] (root) {};
		\node at (-2,1)  [dot] (left) {};
		\node at (-2,3)  [dot] (left1) {};
		\node at (0,5)  [dot] (right) {};
		\node at (0,7)  [dot] (right1) {};
		
		\draw[testfcn] (left) to  (root);
		
		\draw[kernel] (left1) to (left);
		\draw[kernel] (right) to (root);
		\draw[kernel1] (right1) to (right);
		\draw[ddrho] (left) to (right); 
		\draw[ddrho] (left1) to (right1); 
	\end{tikzpicture}\;,
	\quad
	\begin{tikzpicture}[scale=0.35,baseline=0.9cm]
		\node at (0,-1)  [root] (root) {};
		\node at (-2,1)  [dot] (left) {};
		\node at (-2,3)  [dot] (left1) {};
		\node at (0,7)  [dot] (right1) {};
		\node at (0,5) [dot] (right) {};
		
		\draw[testfcn] (left) to  (root);
		
		\draw[kernel2] (left1) to (left);
		\draw[kernel1] (right1) to (right);
		\draw[kernel] (right) to (root);
		\draw[ddrho-shift] (left1) to (right);
		\draw[ddrho-shift] (right1) to (left);
	\end{tikzpicture}\;,
	\quad
	\phi\left(
	\begin{tikzpicture}[scale=0.35,baseline=0.9cm]
		\node at (0,-1)  [root] (root) {};
		\node at (-2,1)  [dot] (left) {};
		\node at (-2,3)  [dot] (left1) {};
		\node at (-2,5)  [dot] (left2) {};
		\node at (-2,7)  [dot] (left3) {};
		\node at (0,5) [var] (variable2) {};
		\node at (0,7) [var] (variable3) {};
		
		\draw[testfcn] (left) to  (root);
		
		\draw[kernelBig] (left1) to (left);
		\draw[kernel1] (left2) to (left1);
		\draw[kernel1] (left3) to (left2);
		\draw[drho] (variable3) to (left3); 
		\draw[drho] (variable2) to (left2); 
	\end{tikzpicture}
	\right),
	\quad
	\phi \left(
	\begin{tikzpicture}[scale=0.35,baseline=0.9cm]
		\node at (0,-1)  [root] (root) {};
		\node at (-2,1)  [dot] (left) {};
		\node at (-2,5)  [dot] (left2) {};
		\node at (-2,7)  [dot] (left3) {};
		\node at (0,3) [dot] (right1) {};
		\node at (0,5) [var] (variable2) {};
		\node at (0,7) [var] (variable3) {};
		
		\draw[testfcnx] (left) to node[below, pos=0.4]{\tiny\color{black} $j$} (root);
		
		\draw[kernel1] (left2) to (right1);
		\draw[kernel1] (left3) to (left2);
		\draw[dkernel] (right1) to node[right, pos=0.5]{\tiny\color{black} $j$} (root);
		\draw[ddrho] (left) to (right1);
		\draw[drho] (variable3) to (left3); 
		\draw[drho] (variable2) to (left2); 
	\end{tikzpicture}\right),
	\right.\\
	&\left.
	\phi_{12, 12}\left(
	\begin{tikzpicture}[scale=0.35,baseline=0.4cm]
		\node at (0,-1)  [root] (root) {};
		\node at (0,1)  [dot] (root2) {};
		\node at (-1.5,2.5)  [dot] (left) {};
		\node at (1.5,2.5)  [dot] (right) {};
		\node at (0,2.5) [var] (variable) {\tiny 3};
		\node at (-1.5,4) [var] (variablel) {\tiny 2};
		\node at (1.5,4) [var] (variabler) {\tiny 4};
		\node at (0,4) [dot] (top) {};
		\node at (0,5.5) [var] (variablet) {\tiny 1};
		
		\draw[testfcn] (root2) to  (root);
		
		\draw[kernel1] (left) to (root2);
		\draw[kernel1] (right) to (root2);
		\draw[kernel1] (top) to (left);
		\draw[drho] (variable) to (root2); 
		\draw[drho] (variablel) to (left); 
		\draw[drho] (variabler) to (right); 
		\draw[drho] (variablet) to (top); 
	\end{tikzpicture}\right),
	\phi_{12, 34}\left(
	\begin{tikzpicture}[scale=0.35,baseline=0.4cm]
		\node at (0,-1)  [root] (root) {};
		\node at (0,1)  [dot] (root2) {};
		\node at (-1.5,2.5)  [dot] (left) {};
		\node at (1.5,2.5)  [dot] (right) {};
		\node at (0,2.5) [var] (variable) {\tiny 3};
		\node at (-1.5,4) [var] (variablel) {\tiny 2};
		\node at (1.5,4) [var] (variabler) {\tiny 4};
		\node at (0,4) [dot] (top) {};
		\node at (0,5.5) [var] (variablet) {\tiny 1};
		
		\draw[testfcn] (root2) to  (root);
		
		\draw[kernel1] (left) to (root2);
		\draw[kernel1] (right) to (root2);
		\draw[kernel1] (top) to (left);
		\draw[drho] (variable) to (root2); 
		\draw[drho] (variablel) to (left); 
		\draw[drho] (variabler) to (right); 
		\draw[drho] (variablet) to (top); 
	\end{tikzpicture}\right),
	\phi_{12, 12} \left(
	\begin{tikzpicture}[scale=0.35,baseline=0.9cm]
		\node at (0,-1)  [root] (root) {};
		\node at (-2,1)  [dot] (left) {};
		\node at (-2,3)  [dot] (left1) {};
		\node at (-2,5)  [dot] (left2) {};
		\node at (-2,7)  [dot] (left3) {};
		\node at (0,1) [var] (variable1) {\tiny 4};
		\node at (0,3) [var] (variable2) {\tiny 3};
		\node at (0,5) [var] (variable3) {\tiny 2};
		\node at (0,7) [var] (variable4) {\tiny 1};
		
		\draw[testfcn] (left) to  (root);
		
		\draw[kernel2] (left1) to (left);
		\draw[kernel1] (left2) to (left1);
		\draw[kernel1] (left3) to (left2);
		\draw[drho] (variable4) to (left3);
		\draw[drho] (variable3) to (left2); 
		\draw[drho] (variable2) to (left1); 
		\draw[drho] (variable1) to (left); 
	\end{tikzpicture}\right),
	\phi_{12, 34} \left(
	\begin{tikzpicture}[scale=0.35,baseline=0.9cm]
		\node at (0,-1)  [root] (root) {};
		\node at (-2,1)  [dot] (left) {};
		\node at (-2,3)  [dot] (left1) {};
		\node at (-2,5)  [dot] (left2) {};
		\node at (-2,7)  [dot] (left3) {};
		\node at (0,1) [var] (variable1) {\tiny 4};
		\node at (0,3) [var] (variable2) {\tiny 3};
		\node at (0,5) [var] (variable3) {\tiny 2};
		\node at (0,7) [var] (variable4) {\tiny 1};
		
		\draw[testfcn] (left) to  (root);
		
		\draw[kernel2] (left1) to (left);
		\draw[kernel1] (left2) to (left1);
		\draw[kernel1] (left3) to (left2);
		\draw[drho] (variable4) to (left3);
		\draw[drho] (variable3) to (left2); 
		\draw[drho] (variable2) to (left1); 
		\draw[drho] (variable1) to (left); 
	\end{tikzpicture}\right)
	\right\}
\end{equs}

\subsubsection{Condition 3}
We now turn to condition 3, that is, \eqref{eq:power_counting-3}. The goal is to show that for all $G \in \cG \setminus \cG_2$, the IBP Feynman graphs $G_I$ obtained in Proposition \ref{prop:adj} satisfies the condition 3 under the labelling $a^\adj$, with only one exception.

Let $\bar V \subset V$ be a subset containing $\origin$ with $|\bar V| \ge 2$, we will consider the associated subgraph $\bar G = (\bar V, \bar E_0)$, where $\bar E_0 = E_0(\bar V)$. Let us also define the shorthand $\bar E^\uparrow = E^\uparrow(\bar V)$ and $\bar E^\downarrow = E^\downarrow(\bar V)$. Under the canonical labelling, define\footnote{Notice the slight abuse of notation: $\deg_3$ not only is a function of the subgraph $\bar G$ but also depends on the sets of edges $\bar E^\uparrow, \bar E^\downarrow$ which are not part of the subgraph. Since the notation does not create any confusion in the current work, we have suppressed this dependence.}
\begin{equ}[eq:deg3]
	\deg_3(\bar G) = 2(|\bar V|-1) - \sum_{e \in \bar E_0} a_e - \sum_{\bar E^\uparrow} (a_e + r_e - 1) + \sum_{\bar E^\downarrow} r_e\;,
\end{equ}
so that the condition \eqref{eq:power_counting-3} can be rewritten as $\deg_3(\bar G)>0$ for all relevant $\bar G$.
Analogously, we also define $\deg_3^\adj$ by replacing $a_e$ with the adjusted degree $a_e^\adj$ in \eqref{eq:deg3}.

Let us introduce the shorthand
\begin{equs}[eq:num_kernels]
	n^{\color{Cerulean}\uparrow}(\bar G) &= \left|\{e \in \bar E^\uparrow: e = \kernelr\}\right|\;, \quad
	n^{\color{red}\uparrow}(\bar G) &= \left|\{e \in \bar E^\uparrow: e = \kernelrr\}\right|\;,\\
	n^{\color{Cerulean}\downarrow}(\bar G) &= \left|\{e \in \bar E^\downarrow: e = \kernelr\}\right|\;, \quad
	n^{\color{red}\downarrow}(\bar G) &= \left|\{e \in \bar E^\downarrow: e = \kernelrr\}\right|\;,\\
	m(\bar G) &= \left|\{e \in \bar E_0: e = \dkernel\}\right|\;.
\end{equs}
When no confusion can happen, we will sometimes drop the argument $\bar G$ in the notation \eqref{eq:num_kernels}.
Set furthermore $\tilde V = \bar V \setminus \{\origin\}$ and $\tilde E_0 = \{e \in \bar E_0: \origin \notin e\}$.
Since the edges of type $\dkernel$ are always attached to $\origin$ (Fact \ref{fact:graph} item 6) and the red or blue edges never touch $\origin$ (Fact \ref{fact:graph} item 4), the notation \eqref{eq:num_kernels} simplifies the expression \eqref{eq:deg3} to
\begin{equ}[eq:deg3_simplified]
	\deg_3(\bar G) = 2|\tilde V| - \sum_{e \in \tilde E_0} a_e - m - n^{\color{red} \uparrow} + n^{\color{Cerulean} \downarrow} + 2n^{\color{red} \downarrow}\;.
\end{equ}
Note that the sum over $\bar E^\uparrow$ reduces to $n^{\color{red} \uparrow}$ since the blue edges do not contribute to the quantity $a_e+r_e-1$.

Let us set
\begin{equ}
	\cG_3 = \left\{
	\begin{tikzpicture}[scale=0.35,baseline=0.6cm]
		\node at (0,-1)  [root] (root) {};
		\node at (-2,1)  [dot] (left) {};
		\node at (0,3)  [dot] (right) {};
		\node at (0,5)  [dot] (right1) {};
		\node at (-2,3) [dot] (left1) {};
		
		\draw[testfcnx] (left) to node[below, pos=0.4]{\tiny\color{black}$j$}  (root);
		
		\draw[dkernel] (right) to node[right, pos=0.5]{\tiny\color{black} $j$} (root);
		\draw[kernelBig] (right1) to (right);
		\draw[kernel] (left1) to (right1);
		\draw[ddrho] (left1) to (left);
	\end{tikzpicture}
	\right\}\;.
\end{equ}

\begin{lemma}\label{lem:cond3}
	All Feynman graphs ${\color{black}\eps^{\frac{|V|-1}{2}} G}$ with $G \in \cG \setminus \cG_3$ verify the condition 3 of Theorem \ref{thm:power_counting} under the canonical labelling.
\end{lemma}
\begin{proof}
	Similar to the proof of Lemma \ref{lem:cond2}, we apply the matching property (Fact \ref{fact:graph} item 2) to the sum $\sum_{e \in \tilde E_0} a_e$ in \eqref{eq:deg3_simplified}, which leads to
	\begin{equs}
		\deg_3(\bar G) = \frac12 |\tilde V| + \frac32 u(\tilde G) - m - n^{\color{red} \uparrow} + n^{\color{Cerulean} \downarrow} + 2n^{\color{red} \downarrow}\;,
		\label{eq:deg3_Gbar}
	\end{equs}
	where, as before, $u(\tilde G)$ is the number of vertices in $\tilde V$ unpaired by $E_\bM$ edges in $\tilde E_0$. 
	It immediately follows that all graphs with $m = n^{\color{red} \uparrow} = 0$ satisfy condition 3, which includes in particular all graphs arising from the expansion of $\<lastone>$. We therefore concentrate our hunt on problematic graphs in the expansions of the tree $\<XiIXiIXiIXi>$.
	
	Recall first that due to Fact \ref{fact:graph} item 5, $m+n^{\color{red} \uparrow} \leq 2$. Also, whenever $|\tilde V|$ is an odd integer, one must have $u(\tilde G) \ge 1$. Next, we claim that
	\begin{equ}[eq:deg3_lower-bound]
		\deg_3(\bar G) \ge 1 \quad \text{whenever $|\tilde V| \ge 6$ or $u(\tilde G) > 0$}.
	\end{equ}
	Indeed, whenever $|\tilde V| \ge 6$, one has $\deg_3(\bar G) \ge 3 - 2 = 1$. 
	For $u(\tilde G)>0$ (and thus $\geq 1$) we discuss each possibility of $|\tilde V|\in[1,5]$.
	When $|\tilde V| \geq 3$, $\deg_3(\bar G) \ge \frac32 + \frac32 - 2 = 1$.
	When $|\tilde V|=2$, then $u(\tilde G)>0$ implies $u(\tilde G)=2$, thus $\deg_3(\bar G) \ge 1 + 3 - 2 = 2$.
	When $|\tilde V| = 1$, one has $u(\tilde G) = 1$ and $m+n^{\color{red} \uparrow} \leq 1$ by Fact \ref{fact:graph} item 3, thus $\deg_3(\bar G) \ge \frac12 + \frac32 - 1 = 1$.
	
	It remains to consider the cases not covered by \eqref{eq:deg3_lower-bound}, that is, $|\tilde V| = 2, 4$ with all vertices perfectly matched by $E_\bM$ edges. 
	The cases for which $\deg_3(\bar G) \leq 0$ are then easily characterised:
	\begin{enumerate}
		\item ($\deg_3(\bar G) = -1$) $|\tilde V| = 2$, $u(\tilde G) = 0$, $m + n^{\color{red} \uparrow} = 2$, $n^{\color{Cerulean} \downarrow} = n^{\color{red} \downarrow} = 0$.
		\item ($\deg_3(\bar G) = 0$) $|\tilde V| = 2$, $u(\tilde G) = 0$, $m + n^{\color{red} \uparrow} = 2$, $n^{\color{Cerulean} \downarrow} = 1$, $n^{\color{red} \downarrow} = 0$.
		\item ($\deg_3(\bar G) = 0$) $|\tilde V| = 2$, $u(\tilde G) = 0$, $m + n^{\color{red} \uparrow} = 1$, $n^{\color{Cerulean} \downarrow} = n^{\color{red} \downarrow} = 0$.
		\item ($\deg_3(\bar G) = 0$) $|\tilde V| = 4$, $u(\tilde G) = 0$, $m + n^{\color{red} \uparrow} = 2$, $ n^{\color{Cerulean} \downarrow} = n^{\color{red} \downarrow} = 0$.
	\end{enumerate}
	By inspecting the stochastic graphs in \eqref{eq:XiIXiIXiIXi}, it is not hard to see that whenever $|\tilde V| = 2, 4$, $m + n^{\color{red} \uparrow} = 2$ and $u(\tilde G) = 0$, one must have $n^{\color{Cerulean} \downarrow} \ge 2$. Therefore, only the third case can occur and violate \eqref{eq:power_counting-3}, which occurs only in the graph in $\cG_3$.
\end{proof}

\begin{proposition}\label{prop:cond3}
	Let $G \in \cG\setminus (\cG_2 \cup \cG_3)$ and let $G_I$ be as in in Proposition \ref{prop:adj}. Then for sufficiently small $\bar \kappa>0$, ${\color{black}\eps^{\frac{|V|-1}{2} - c(\bar\kappa)} G_I}$ satisfies the condition 3 of Theorem \ref{thm:power_counting} under the adjusted labelling $(a_e^\adj, r_e)$, with $c(\bar\kappa)$ given in Proposition \ref{prop:adj}.
\end{proposition}
\begin{proof}
	Let $e^*$ be the $\ddrho$ edge verifying the assumption of Proposition \ref{prop:adj}. Now let us fix $\bar V$ containing $\origin$ with $|\bar V| \ge 2$ and distinguish three cases:
	\begin{itemize}
		\item $e^* \in E_0(\bar V)$. It is obvious that $\deg_3(\bar G_I) \ge \deg_3(\bar G)$ and the inequality is strict when the derivative hits an edge not in $E_0(\bar V)$ without recentering.
		\item $e^* \in E(\bar V) \setminus E_0(\bar V)$. Similarly to the case of $\deg_2$, this situation can decrease $\deg_3$ by $1$. Therefore,
		\begin{equs}
			\deg_3(\bar G_I) &\ge \deg_3(\bar G) - \un_{e^* \in E(\bar V) \setminus E_0(\bar V)}\;.
		\end{equs}
		Note that $e^* \in E(\bar V) \setminus E_0(\bar V)$ happens only when $u(\tilde G) \ge 1$. Therefore, by \eqref{eq:deg3_lower-bound}, it follows that $\deg_3(\bar G_I) \ge 0$ in this case.
		\item $e^* \notin E(\bar V)$ but there is an edge $e' \in E(\bar V)$ with $r_{e'} > 0$ such that $e' \cap e^* \neq \emptyset$. In this case, one has $\deg_3(\bar G_I) = \deg_3(\bar G)$: indeed, one either has $e' \in E^\downarrow$ or $e' \in E^\uparrow$, in the former case IBP will not change $\deg_3$, in the latter case the IBP decreases $r_{e'}$ by $1$ but increases $a_{e'}$ by $1$ at the same time, hence $\deg_3$ stays invariant in both cases. 
	\end{itemize}
	Since we have required that $G \in \cG\setminus\cG_3$ and thus
$\deg_3(\bar G)>0$, 	
	 we deduce that $\deg_3(\bar G_I) \geq 0$ and equality only occurs when $e^* \in E(\bar V) \setminus E_0(\bar V)$ (and thus $u(\tilde G)>0$) and $\deg_3(\bar G) = 1$. By the discussion following \eqref{eq:deg3_lower-bound}, this happens only for the following two situations:
	\begin{enumerate}
		\item[(a)] $|\bar V| = 2$ (i.e., $|\tilde V| = 1$), $e^* \in E(\bar V) \setminus E_0(\bar V)$, $m + n^{\color{red} \uparrow} = 1$, $n^{\color{Cerulean} \downarrow} = n^{\color{red} \downarrow} = 0$.
		\item[(b)] $|\bar V| = 4$ (i.e., $|\tilde V| = 3$), $e^* \in E(\bar V) \setminus E_0(\bar V)$, $m + n^{\color{red} \uparrow} = 2$, $n^{\color{Cerulean} \downarrow} = n^{\color{red} \downarrow} = 0$.
	\end{enumerate}
	Under our assumption, we only need to check the Feynman graphs resulted from \eqref{eq:XiIXiIXiIXi} which are in neither $\cG_2$ nor $\cG_3$. Notice that the situation (b) never happens, while the situation (a) appears only in the Feynman graphs generated by the pairings
	\begin{equ}
		\phi \left(
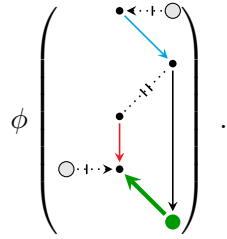
		\begin{tikzpicture}[scale=0.35,baseline=0.9cm]
			\node at (0,-1)  [root] (root) {};
			\node at (-2,1)  [dot] (left) {};
			\node at (-2,3)  [dot] (left1) {};
			\node at (0,5)  [dot] (right2) {};
			\node at (-2,7)  [dot] (left3) {};
			\node at (-4,1) [var] (variable) {};
			\node at (0,7) [var] (variable3) {};
			
			\draw[testfcn] (left) to  (root);
			
			\draw[kernel2] (left1) to (left);
			\draw[kernel1] (left3) to (right2);
			\draw[kernel] (right2) to (root);
			\draw[ddrho] (left1) to (right2);
			\draw[drho] (variable3) to (left3); 
			\draw[drho] (variable) to (left); 
		\end{tikzpicture}\right)\;.
	\end{equ}
	These graphs are however easy to treat since every mollifier edge in these graphs can be a choice for $e^*$: we only need to avoid choosing the mollifier edge attached to the tail of $\kernelrr$. Consequently, we have shown that $\deg_3(\bar G_I) > 0$ for all IBP map $I$ at $e^*$. By choosing sufficiently small $\bar \kappa$, one also has $\deg_3^\adj(\bar G_I) > 0$.
\end{proof}

\subsubsection{Condition 4}
We will now turn our attention to condition 4, that is, \eqref{eq:power_counting-4}. For a subgraph $\bar G = (\bar V, \bar E_0)$, define\footnote{The same abuse of notation as in $\deg_3$ happens here: we have suppressed this dependence of $\deg_4$ on $\bar E^\uparrow, \bar E^\downarrow$.}
\begin{equ}[eq:deg4]
	\deg_4(\bar G) = \sum_{E(\bar V) \setminus E^\downarrow} a_e + \sum_{\bar E^\uparrow} r_e - \sum_{\bar E^\downarrow} (r_e - 1) - 2|\bar V|\;,
\end{equ}
so that the condition \eqref{eq:power_counting-4} can be rewritten as $\deg_4(\bar G)>0$ for all relevant $\bar G$.
Analogously, we also define $\deg_4^\adj$ by \eqref{eq:deg4} but with $a_e$ replaced by $a_e^\adj$.

As before, we collect the exceptional graphs when it comes to condition 4 in the set
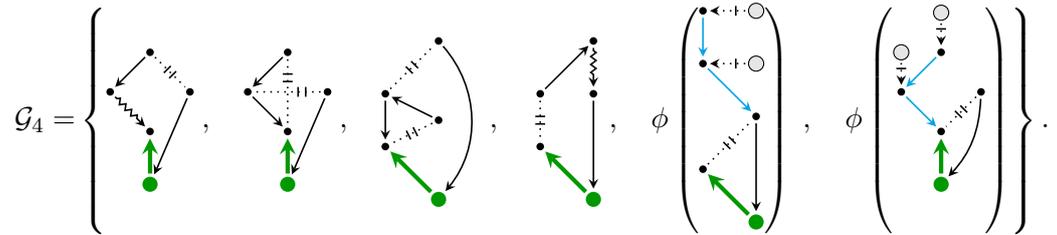
\begin{equs}
	\cG_4 = &\left\{
	\begin{tikzpicture}[scale=0.35,baseline=0.4cm]
		\node at (0,-1)  [root] (root) {};
		\node at (0,1)  [dot] (root2) {};
		\node at (-1.5,2.5)  [dot] (left) {};
		\node at (1.5,2.5)  [dot] (right) {};
		\node at (0,4) [dot] (top) {};
		
		\draw[testfcn] (root2) to  (root);
		
		\draw[kernel] (right) to (root);
		\draw[kernel] (top) to (left);
		\draw[kernelBig] (left) to (root2); 
		\draw[ddrho] (top) to (right);
	\end{tikzpicture}\;,
	\quad
	\begin{tikzpicture}[scale=0.35,baseline=0.4cm]
		\node at (0,-1)  [root] (root) {};
		\node at (0,1)  [dot] (root2) {};
		\node at (-1.5,2.5)  [dot] (left) {};
		\node at (1.5,2.5)  [dot] (right) {};
		\node at (0,4) [dot] (top) {};
		
		\draw[testfcn] (root2) to  (root);
		
		\draw[ddrho-shift] (top) to (root2);
		
		\draw[kernel] (right) to (root);
		\draw[kernel] (top) to (left);
		\draw[kernel] (left) to (root2); 
		\draw[ddrho-shift] (right) to (left);
	\end{tikzpicture}\;,
	\quad
	\begin{tikzpicture}[scale=0.35,baseline=0.6cm]
		\node at (0,-1)  [root] (root) {};
		\node at (-2,1)  [dot] (left) {};
		\node at (-2,3)  [dot] (left1) {};
		\node at (0,5)  [dot] (right1) {};
		\node at (0,2) [dot] (right) {};
		
		\draw[testfcn] (left) to  (root);
		
		\draw[kernel] (left1) to (left);
		\draw[kernel] (right) to (left1);
		\draw[kernel,bend left=40] (right1) to (root);
		\draw[ddrho] (left1) to (right1);
		\draw[ddrho] (right) to (left); 
	\end{tikzpicture}\;,
	\quad
	\begin{tikzpicture}[scale=0.35,baseline=0.6cm]
		\node at (0,-1)  [root] (root) {};
		\node at (-2,1)  [dot] (left) {};
		\node at (0,3)  [dot] (right) {};
		\node at (0,5)  [dot] (right1) {};
		\node at (-2,3) [dot] (left1) {};
		
		\draw[testfcn] (left) to  (root);
		
		\draw[kernel] (right) to (root);
		\draw[kernelBig] (right1) to (right);
		\draw[kernel] (left1) to (right1);
		\draw[ddrho] (left1) to (left);
	\end{tikzpicture}\;,
	\quad
	\phi
	\left(\begin{tikzpicture}[scale=0.35,baseline=0.9cm]
		\node at (0,-1)  [root] (root) {};
		\node at (-2,1)  [dot] (left) {};
		\node at (-2,5)  [dot] (left2) {};
		\node at (-2,7)  [dot] (left3) {};
		\node at (0,3) [dot] (right1) {};
		\node at (0,5) [var] (variable2) {};
		\node at (0,7) [var] (variable3) {};
		
		\draw[testfcn] (left) to  (root);
		
		\draw[kernel1] (left2) to (right1);
		\draw[kernel1] (left3) to (left2);
		\draw[kernel] (right1) to (root);
		\draw[ddrho] (left) to (right1);
		\draw[drho] (variable3) to (left3); 
		\draw[drho] (variable2) to (left2); 
	\end{tikzpicture}\right)\;,
	\quad
	\phi
	\left(\begin{tikzpicture}[scale=0.35,baseline=0.4cm]
		\node at (0,-1)  [root] (root) {};
		\node at (0,1)  [dot] (root2) {};
		\node at (-1.5,2.5)  [dot] (left) {};
		\node at (1.5,2.5)  [dot] (right) {};
		\node at (-1.5,4) [var] (variablel) {};
		\node at (0,4) [dot] (top) {};
		\node at (0,5.5) [var] (variablet) {};
		
		\draw[testfcn] (root2) to  (root);
		
		\draw[kernel1] (left) to (root2);
		\draw[kernel,bend left=20] (right) to (root);
		\draw[kernel1] (top) to (left);
		\draw[ddrho] (right) to (root2); 
		\draw[drho] (variablel) to (left);
		\draw[drho] (variablet) to (top); 
	\end{tikzpicture}\right)
	\right\}.
\end{equs}

\begin{proposition}\label{prop:cond4}
	Let $G \in \cG\setminus(\cG_2 \cup \cG_3 \cup \cG_4)$ and let $G_I$ be as in Proposition \ref{prop:adj}. Then for sufficiently small $\bar \kappa>0$, ${\color{black}\eps^{\frac{|V|-1}{2}-c(\bar\kappa)} G_I}$ satisfies the condition 4 of Theorem \ref{thm:power_counting} under the adjusted labelling $(a_e^\adj,r_e)$, with $c(\bar\kappa)$ given in Proposition \ref{prop:adj}.
\end{proposition}
\begin{proof}
	Fix a subgraph $\bar G$ of $G$ with $\bar V \subset V\setminus V_\ast$. Using the notation \eqref{eq:num_kernels}, one can rewrite $\deg_4(\bar G)$ as
	\begin{equ}
		\deg_4(\bar G) = \sum_{E(V) \setminus \bar E^\downarrow} a_e + n^{\color{Cerulean}\uparrow} (\bar G) +  2n^{\color{red}\uparrow} (\bar G) - n^{\color{red}\downarrow} (\bar G) - 2|\bar V|\;.
	\end{equ}
	Note that $n^{\color{red}\downarrow} (\bar G) = 0$ for all $\bar V \subset V \setminus V_\ast$ by Fact \ref{fact:graph} item 7. Let us recall that $u(\bar G)$ denotes the number of unpaired vertices by $E_\bM$ in $\bar V$, and define $\tilde m(\bar G)$ to be the number of $\dkernel$ edges in $E(\bar V)$ (note the nuance with $m(\bar G)$, which only counts the number of $\dkernel$ in $E_0(\bar V)$). Under canonical labelling, it holds that
	\begin{equs}
		\sum_{E(\bar V) \setminus E^\downarrow} a_e &= 3 |\{\text{Pairs of vertices in $\bar V$ connected by $E_\bM$ edge}\}| + 3 u(\bar G) + \tilde m(\bar G)\\
		&= \frac32 (|\bar V| - u(\bar G)) + 3 u(\bar G) + \tilde m(\bar G) = \frac32 (|\bar V| + u(\bar G)) + \tilde m(\bar G)\;.
	\end{equs}
	As a consequence, one has
	\begin{equ}[eq:deg4_simplified]
		\deg_4(\bar G) =  \frac32 u(\bar G) - \frac12 |\bar V| + n^{\color{Cerulean}\uparrow} (\bar G) +  2n^{\color{red}\uparrow} (\bar G) + \tilde m(\bar G)\;.
	\end{equ}
	
	Now consider the effect introduced by the IBP. Let us fix a $\ddrho$ edge $e^*$ verifying the assumption of Proposition \ref{prop:adj} and distinguish three cases
	\begin{itemize}
		\item $e^* \in E_0(\bar V)$. In this case, the degree is unchanged, $\deg_4(\bar G_I) = \deg_4(\bar G)$. To see this, one notices that if an edge $e \in E(\bar V) \setminus E^\uparrow(\bar V)$ is attached to $e^*$, then any derivatives attributed to $e$ after IBP would change neither the sum over $E(\bar V) \setminus E^\uparrow(\bar V)$ nor that over $E^\uparrow(\bar V)$; if an edge $e = \kernelrr \in E^\downarrow(\bar V)$ and a derivative is distributed to $e$, then $e$ is turned into $\dkernelr$, which would decreases both the sum over $E(\bar V) \setminus E^\uparrow(\bar V)$ and that over $E^{\downarrow}(\bar V)$ by $1$; if an edge $e = \kernelr \in \bar E^\downarrow(\bar V)$ and a derivative is distributed to $e$, then $e$ would be rendered $\dkernel$ which becomes an element in $E(\bar V) \setminus E^\downarrow(\bar V)$ and thus both the sum over $E((\bar V) \setminus E^\uparrow(\bar V)$ and that over $E^{\downarrow}(\bar V)$ are invariant. Consequently, one does have $\deg_4(\bar G_I) = \deg_4(\bar G)$ in all cases.
		\item $e^* \in E(\bar V) \setminus E_0(\bar V)$. In this case, $\deg_4$ can be decreased by $1$. One thus deduces the estimate
		\begin{equ}
			\deg_4(\bar G_I) \ge \deg_4(\bar G) - \un_{e^* \in E(\bar V) \setminus E_0(\bar V)}\;.
		\end{equ}
		\item $e^* \notin E(\bar V)$ but there is an edge $e' \in E(\bar V)$ such that $e' \cap e^* \neq \emptyset$. In this case, it is obvious that $\deg_4(\bar G_I) \ge \deg_4(\bar G)$ ($>$ when the derivative hits $e'$ and $e'$ is without recentering).
	\end{itemize}
	Summarising the above cases and using the inequality
	\begin{equ}
		\deg_4^\adj(\bar G_I) \ge \deg_4(\bar G_I) + \sqrt{\bar \kappa} \un_{e^* \in E(\bar V)} - \bar \kappa |E|\;,
	\end{equ}
	we deduce that $\deg_4^\adj (\bar G_I) > 0$ if
	\begin{equ}[eq:deg4_adj_GI]
		\deg_4(\bar G) - \un_{e^* \in E(\bar V) \setminus E_0(\bar V)} + \sqrt{\bar \kappa} \un_{e^* \in E(\bar V)} - \bar \kappa |E| > 0\;.
	\end{equ}
	Given that $\bar \kappa$ is sufficiently small so that $\sqrt{\bar \kappa} > \bar \kappa |E|$, the criterion \eqref{eq:deg4_adj_GI} implies that one indeed has $\deg_4^\adj (\bar G_I) > 0$ provided $\bar G$ falls in one of the following cases:
	\begin{itemize}
		\item[(a)] $\deg_4(\bar G) \ge 0$ and $e^* \in E_0(\bar V)$.
		\item[(b)] $\deg_4(\bar G) \ge 1$ and $e^* \in E(\bar V) \setminus E_0(\bar V)$. (In particular, $u(\bar G) \ge 1$.)
		\item[(c)] $\deg_4(\bar G) > 0$ and $e^* \notin E(\bar V)$.
	\end{itemize}
	
	Note that the Feynman graphs generated from \eqref{eq:lastone} and \eqref{eq:XiIXiIXiIXi} have at most $6$ vertices outside of $V_\ast$. Hence, we only need to consider subgraphs with $1 \leq |\bar V| \leq 6$. For this range, let us now discuss the cases where $u(\bar G) = 0$ and $u(\bar G) > 0$ separately using \eqref{eq:deg4_simplified}. We will focus on the Feynman graphs in $\cG\setminus (\cG_2 \cup \cG_3)$.
	
	We start by considering Feynman graphs containing subgraphs with $u(\bar G) = 0$ and $\deg_4(\bar G) \leq 0$. These are the subgraphs which cannot be worsen by IBP; when $\deg_4(\bar G) = 0$ for some subgraph $\bar G$, its important that $e^*$ is contained in $\bar G$ so that the case (a) holds and $\deg_4^\adj(\bar G_I) > 0$. In the following, we will mark the choice of $e^*$ which saves us. 
	\begin{enumerate}
		\item $|\bar V| = 2$ and $u(\bar G) = n^{\color{Cerulean}\uparrow} (\bar G) = n^{\color{red}\uparrow} (\bar G) = \tilde m(\bar G) = 0$. This case happens only for the graphs
		
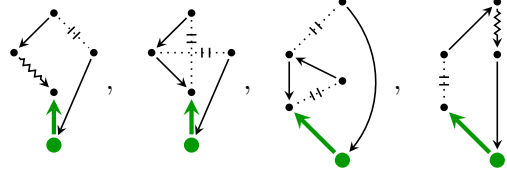
\begin{equ}
			\begin{tikzpicture}[scale=0.35,baseline=0.4cm]
				\node at (0,-1)  [root] (root) {};
				\node at (0,1)  [dot] (root2) {};
				\node at (-1.5,2.5)  [dot] (left) {};
				\node at (1.5,2.5)  [dot] (right) {};
				\node at (0,4) [dot] (top) {};
				
				\draw[testfcn] (root2) to  (root);
				
				\draw[kernel] (right) to (root);
				\draw[kernel] (top) to (left);
				\draw[kernelBig] (left) to (root2); 
				\draw[ddrho] (top) to (right);
			\end{tikzpicture}\;,
			\quad
			\begin{tikzpicture}[scale=0.35,baseline=0.4cm]
				\node at (0,-1)  [root] (root) {};
				\node at (0,1)  [dot] (root2) {};
				\node at (-1.5,2.5)  [dot] (left) {};
				\node at (1.5,2.5)  [dot] (right) {};
				\node at (0,4) [dot] (top) {};
				
				\draw[testfcn] (root2) to  (root);
				
				\draw[ddrho-shift] (top) to (root2);
				
				\draw[kernel] (right) to (root);
				\draw[kernel] (top) to (left);
				\draw[kernel] (left) to (root2); 
				\draw[ddrho-shift] (right) to (left);
			\end{tikzpicture}\;,
			\quad
			\begin{tikzpicture}[scale=0.35,baseline=0.6cm]
				\node at (0,-1)  [root] (root) {};
				\node at (-2,1)  [dot] (left) {};
				\node at (-2,3)  [dot] (left1) {};
				\node at (0,5)  [dot] (right1) {};
				\node at (0,2) [dot] (right) {};
				
				\draw[testfcn] (left) to  (root);
				
				\draw[kernel] (left1) to (left);
				\draw[kernel] (right) to (left1);
				\draw[kernel,bend left=40] (right1) to (root);
				\draw[ddrho] (left1) to (right1);
				\draw[ddrho] (right) to (left); 
			\end{tikzpicture}\;,
			\quad
			\begin{tikzpicture}[scale=0.35,baseline=0.6cm]
				\node at (0,-1)  [root] (root) {};
				\node at (-2,1)  [dot] (left) {};
				\node at (0,3)  [dot] (right) {};
				\node at (0,5)  [dot] (right1) {};
				\node at (-2,3) [dot] (left1) {};
				
				\draw[testfcn] (left) to  (root);
				
				\draw[kernel] (right) to (root);
				\draw[kernelBig] (right1) to (right);
				\draw[kernel] (left1) to (right1);
				\draw[ddrho] (left1) to (left);
			\end{tikzpicture}
		\end{equ}
		No choice of $e^*$ can save us here. These graphs fail the condition 4 under the adjusted labelling and are thus included in $\cG_4$.
		\item $|\bar V| = 2$, $u(\bar G) = n^{\color{red}\uparrow} (\bar G) = 0$, and $n^{\color{Cerulean}\uparrow} (\bar G) + \tilde m(\bar G) = 1$. This case happens for all $G \in \cG$ containing
		
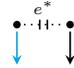
\begin{equ}
			\begin{tikzpicture}[scale=0.35,baseline=0.4cm]
				\node at (-1,1) [dot] (left) {};
				\node at (1, 1) [dot] (right) {};
				\node at (-1,-1) [] (left1) {};
				\node at (1,-1) [] (right1) {};
				
				\draw[ddrho] (left) to node[above]{\tiny$e^*$} (right);
				\draw[kernel1] (left) to (left1);
				\draw[kernel] (right) to (right1);
			\end{tikzpicture}
		\end{equ}
		which is a subgraph with $\deg_4 = 0$. However, by inspection one can see that in the Feynman graphs in $\cG\setminus \cG_2$ containing this subgraph, one can always choose its $\ddrho$ to be $e^*$, bringing us to the situation (a).
		\item $|\bar V| = 4$, $u(\bar G) = 0$, $n^{\color{Cerulean}\uparrow} (\bar G) + 2n^{\color{red}\uparrow} (\bar G) + \tilde m(\bar G) \leq 2$. In $\cG \setminus (\cG_2 \cup \cG_3)$, this situation happens only for
		\begin{equ}
			\phi
			\left(
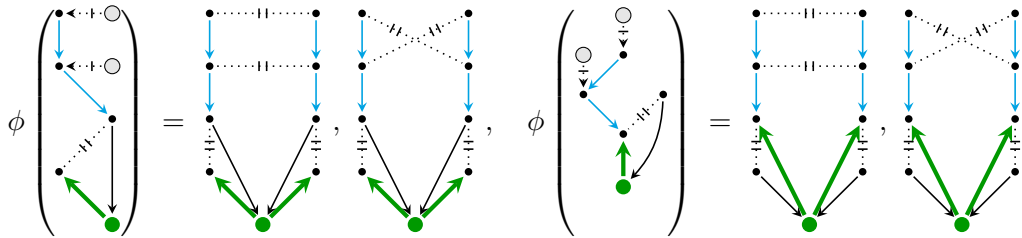
\begin{tikzpicture}[scale=0.35,baseline=0.9cm]
				\node at (0,-1)  [root] (root) {};
				\node at (-2,1)  [dot] (left) {};
				\node at (-2,5)  [dot] (left2) {};
				\node at (-2,7)  [dot] (left3) {};
				\node at (0,3) [dot] (right1) {};
				\node at (0,5) [var] (variable2) {};
				\node at (0,7) [var] (variable3) {};
				
				\draw[testfcn] (left) to  (root);
				
				\draw[kernel1] (left2) to (right1);
				\draw[kernel1] (left3) to (left2);
				\draw[kernel] (right1) to (root);
				\draw[ddrho] (left) to (right1);
				\draw[drho] (variable3) to (left3); 
				\draw[drho] (variable2) to (left2); 
			\end{tikzpicture}\right)
			\;=\;
			\begin{tikzpicture}[scale=0.35,baseline=0.9cm]
				\node at (0,-1)  [root] (root) {};
				\node at (-2,1)  [dot] (left) {};
				\node at (-2,3)  [dot] (left1) {};
				\node at (-2,5)  [dot] (left2) {};
				\node at (-2,7)  [dot] (left3) {};
				\node at (2,1)  [dot] (right) {};
				\node at (2,3)  [dot] (right1) {};
				\node at (2,5)  [dot] (right2) {};
				\node at (2,7)  [dot] (right3) {};
				
				\draw[testfcn] (left) to  (root);
				\draw[testfcn] (right) to  (root);
				
				\draw[kernel1] (left2) to (left1);
				\draw[kernel1] (left3) to (left2);
				\draw[kernel] (left1) to (root);
				\draw[ddrho] (left) to (left1);
				\draw[kernel1] (right2) to (right1);
				\draw[kernel1] (right3) to (right2);
				\draw[kernel] (right1) to (root);
				\draw[ddrho] (right) to (right1);
				
				\draw[ddrho] (right3) to (left3); 
				\draw[ddrho] (right2) to (left2); 
			\end{tikzpicture}
			\;,\;
			\begin{tikzpicture}[scale=0.35,baseline=0.9cm]
				\node at (0,-1)  [root] (root) {};
				\node at (-2,1)  [dot] (left) {};
				\node at (-2,3)  [dot] (left1) {};
				\node at (-2,5)  [dot] (left2) {};
				\node at (-2,7)  [dot] (left3) {};
				\node at (2,1)  [dot] (right) {};
				\node at (2,3)  [dot] (right1) {};
				\node at (2,5)  [dot] (right2) {};
				\node at (2,7)  [dot] (right3) {};
				
				\draw[testfcn] (left) to  (root);
				\draw[testfcn] (right) to  (root);
				
				\draw[kernel1] (left2) to (left1);
				\draw[kernel1] (left3) to (left2);
				\draw[kernel] (left1) to (root);
				\draw[ddrho] (left) to (left1);
				\draw[kernel1] (right2) to (right1);
				\draw[kernel1] (right3) to (right2);
				\draw[kernel] (right1) to (root);
				\draw[ddrho] (right) to (right1);
				
				\draw[ddrho-shift] (right3) to (left2); 
				\draw[ddrho-shift] (left3) to (right2); 
			\end{tikzpicture}
			\;,\quad
			\phi
			\left(\begin{tikzpicture}[scale=0.35,baseline=0.4cm]
				\node at (0,-1)  [root] (root) {};
				\node at (0,1)  [dot] (root2) {};
				\node at (-1.5,2.5)  [dot] (left) {};
				\node at (1.5,2.5)  [dot] (right) {};
				\node at (-1.5,4) [var] (variablel) {};
				\node at (0,4) [dot] (top) {};
				\node at (0,5.5) [var] (variablet) {};
				
				\draw[testfcn] (root2) to  (root);
				
				\draw[kernel1] (left) to (root2);
				\draw[kernel,bend left=20] (right) to (root);
				\draw[kernel1] (top) to (left);
				\draw[ddrho] (right) to (root2); 
				\draw[drho] (variablel) to (left);
				\draw[drho] (variablet) to (top); 
			\end{tikzpicture}\right)
			\;=\;
			\begin{tikzpicture}[scale=0.35,baseline=0.9cm]
				\node at (0,-1)  [root] (root) {};
				\node at (-2,1)  [dot] (left) {};
				\node at (-2,3)  [dot] (left1) {};
				\node at (-2,5)  [dot] (left2) {};
				\node at (-2,7)  [dot] (left3) {};
				\node at (2,1)  [dot] (right) {};
				\node at (2,3)  [dot] (right1) {};
				\node at (2,5)  [dot] (right2) {};
				\node at (2,7)  [dot] (right3) {};
				
				\draw[testfcn] (left1) to  (root);
				\draw[testfcn] (right1) to  (root);
				
				\draw[kernel1] (left2) to (left1);
				\draw[kernel1] (left3) to (left2);
				\draw[kernel] (left) to (root);
				\draw[ddrho] (left) to (left1);
				\draw[kernel1] (right2) to (right1);
				\draw[kernel1] (right3) to (right2);
				\draw[kernel] (right) to (root);
				\draw[ddrho] (right) to (right1);
				
				\draw[ddrho] (right3) to (left3); 
				\draw[ddrho] (right2) to (left2); 
			\end{tikzpicture}
			\;,\;
			\begin{tikzpicture}[scale=0.35,baseline=0.9cm]
				\node at (0,-1)  [root] (root) {};
				\node at (-2,1)  [dot] (left) {};
				\node at (-2,3)  [dot] (left1) {};
				\node at (-2,5)  [dot] (left2) {};
				\node at (-2,7)  [dot] (left3) {};
				\node at (2,1)  [dot] (right) {};
				\node at (2,3)  [dot] (right1) {};
				\node at (2,5)  [dot] (right2) {};
				\node at (2,7)  [dot] (right3) {};
				
				\draw[testfcn] (left1) to  (root);
				\draw[testfcn] (right1) to  (root);
				
				\draw[kernel1] (left2) to (left1);
				\draw[kernel1] (left3) to (left2);
				\draw[kernel] (left) to (root);
				\draw[ddrho] (left) to (left1);
				\draw[kernel1] (right2) to (right1);
				\draw[kernel1] (right3) to (right2);
				\draw[kernel] (right) to (root);
				\draw[ddrho] (right) to (right1);
				
				\draw[ddrho-shift] (right3) to (left2); 
				\draw[ddrho-shift] (left3) to (right2); 
			\end{tikzpicture}
		\end{equ}
		For the four Feynman graphs, the problematic subgraph are formed by the top four vertices, with edge set consisting of two $\ddrho$ and two $\kernelr$ in $E_0(\bar V)$ as well as two $\kernelr$ in $E^\uparrow(\bar V)$. Furthermore, the two mollifiers in the subgraph could not be chosen as $e^*$ since they are part of a critical block, falling just short of the cases (a) and (c). Therefore,  these Feynman graphs fail condition 4 under the adjusted labelling and are thus included in $\cG_4$.
		\item $|\bar V| = 6$, $u(\bar G) = 0$ and $n^{\color{Cerulean}\uparrow} (\bar G) + 2n^{\color{red}\uparrow} (\bar G) + \tilde m(\bar G) \leq 3$. This case does not happen in $\cG$ since whenever $|\bar V| = 6$ and $u(\bar G) = 0$, one must have $n^{\color{Cerulean}\uparrow} (\bar G) + 2n^{\color{red}\uparrow} (\bar G) + \tilde m(\bar G) = 4$.
	\end{enumerate}
	Let us now consider the graphs containing subgraphs with $u(\bar G) > 0$ and $\deg_4(\bar G) \leq 1$: these are the subgraphs which can have $\deg_4 = 0$ after IBP at $e^*$.
	\begin{enumerate}
		\setcounter{enumi}{4}
		\item $|\bar V| = 3$, $u(\bar G) = 1$, $n^{\color{Cerulean}\uparrow} (\bar G) + 2n^{\color{red}\uparrow} (\bar G) + \tilde m(\bar G) \leq 1$. In addition to the Feynman graphs in case 1, this case only happens in
		\begin{equs}
			\phi
			\left(
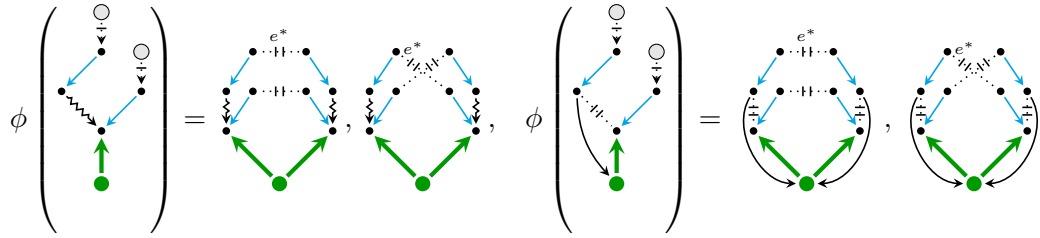
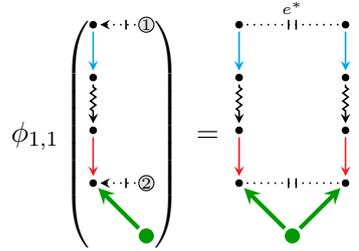
\begin{tikzpicture}[scale=0.35,baseline=0.4cm]
				\node at (0,-1)  [root] (root) {};
				\node at (0,1)  [dot] (root2) {};
				\node at (-1.5,2.5)  [dot] (left) {};
				\node at (1.5,2.5)  [dot] (right) {};
				\node at (1.5,4) [var] (variabler) {};
				\node at (0,4) [dot] (top) {};
				\node at (0,5.5) [var] (variablet) {};
				
				\draw[testfcn] (root2) to  (root);
				
				\draw[kernelBig] (left) to (root2);
				\draw[kernel1] (right) to (root2);
				\draw[kernel1] (top) to (left);
				\draw[drho] (variablet) to (top); 
				\draw[drho] (variabler) to (right); 
			\end{tikzpicture}\right)
			=\;
			\begin{tikzpicture}[scale=0.35,baseline=0.4cm]
				\node at (0,-1)  [root] (root) {};
				\node at (-2,1)  [dot] (left) {};
				\node at (-2,2.5)  [dot] (left1) {};
				\node at (-1,2.5)  [dot] (var_left) {};
				\node at (-1,4) [dot] (var_left1) {};
				\node at (2,1)  [dot] (right) {};
				\node at (2,2.5)  [dot] (right1) {};
				\node at (1,2.5)  [dot] (var_right) {};
				\node at (1,4) [dot] (var_right1) {};
				
				\draw[testfcn] (left) to (root);
				\draw[testfcn] (right) to (root);
				
				\draw[kernelBig] (left1) to (left);
				\draw[kernel1] (var_left) to (left);
				\draw[kernel1] (var_left1) to (left1);
				
				\draw[kernelBig] (right1) to (right);
				\draw[kernel1] (var_right) to (right);
				\draw[kernel1] (var_right1) to (right1);
				
				\draw[ddrho] (var_right1) to node[above]{\tiny$e^*$} (var_left1); 
				\draw[ddrho] (var_right) to (var_left); 
			\end{tikzpicture}
			\;,\;
			\begin{tikzpicture}[scale=0.35,baseline=0.4cm]
				\node at (0,-1)  [root] (root) {};
				\node at (-2,1)  [dot] (left) {};
				\node at (-2,2.5)  [dot] (left1) {};
				\node at (-1,2.5)  [dot] (var_left) {};
				\node at (-1,4) [dot] (var_left1) {};
				\node at (2,1)  [dot] (right) {};
				\node at (2,2.5)  [dot] (right1) {};
				\node at (1,2.5)  [dot] (var_right) {};
				\node at (1,4) [dot] (var_right1) {};
				
				\draw[testfcn] (left) to (root);
				\draw[testfcn] (right) to (root);
				
				\draw[kernelBig] (left1) to (left);
				\draw[kernel1] (var_left) to (left);
				\draw[kernel1] (var_left1) to (left1);
				
				\draw[kernelBig] (right1) to (right);
				\draw[kernel1] (var_right) to (right);
				\draw[kernel1] (var_right1) to (right1);
				
				\draw[ddrho-shift] (var_right1) to (var_left); 
				\draw[ddrho-shift] (var_left1) to node[above, pos=0.3]{\tiny$e^*$} (var_right); 
			\end{tikzpicture}\;,
			\quad
			&\phi
			\left(\begin{tikzpicture}[scale=0.35,baseline=0.4cm]
				\node at (0,-1)  [root] (root) {};
				\node at (0,1)  [dot] (root2) {};
				\node at (-1.5,2.5)  [dot] (left) {};
				\node at (1.5,2.5)  [dot] (right) {};
				\node at (1.5,4) [var] (variabler) {};
				\node at (0,4) [dot] (top) {};
				\node at (0,5.5) [var] (variablet) {};
				
				\draw[testfcn] (root2) to  (root);
				
				\draw[ddrho] (left) to (root2);
				\draw[kernel1] (right) to (root2);
				\draw[kernel1] (top) to (left);
				\draw[drho] (variablet) to (top); 
				\draw[drho] (variabler) to (right); 
				\draw[kernel,bend right=20] (left) to (root);
			\end{tikzpicture}\right)
			=
			\begin{tikzpicture}[scale=0.35,baseline=0.4cm]
				\node at (0,-1)  [root] (root) {};
				\node at (-2,1)  [dot] (left) {};
				\node at (-2,2.5)  [dot] (left1) {};
				\node at (-1,2.5)  [dot] (var_left) {};
				\node at (-1,4) [dot] (var_left1) {};
				\node at (2,1)  [dot] (right) {};
				\node at (2,2.5)  [dot] (right1) {};
				\node at (1,2.5)  [dot] (var_right) {};
				\node at (1,4) [dot] (var_right1) {};
				
				\draw[testfcn] (left) to (root);
				\draw[testfcn] (right) to (root);
				\draw[kernel, bend right = 60] (left1) to (root);
				\draw[kernel, bend left = 60] (right1) to (root);
				
				\draw[ddrho] (left1) to (left);
				\draw[kernel1] (var_left) to (left);
				\draw[kernel1] (var_left1) to (left1);
				
				\draw[ddrho] (right1) to (right);
				\draw[kernel1] (var_right) to (right);
				\draw[kernel1] (var_right1) to (right1);
				
				\draw[ddrho] (var_right1) to node[above]{\tiny$e^*$} (var_left1); 
				\draw[ddrho] (var_right) to (var_left); 
			\end{tikzpicture}
			,
			\begin{tikzpicture}[scale=0.35,baseline=0.4cm]
				\node at (0,-1)  [root] (root) {};
				\node at (-2,1)  [dot] (left) {};
				\node at (-2,2.5)  [dot] (left1) {};
				\node at (-1,2.5)  [dot] (var_left) {};
				\node at (-1,4) [dot] (var_left1) {};
				\node at (2,1)  [dot] (right) {};
				\node at (2,2.5)  [dot] (right1) {};
				\node at (1,2.5)  [dot] (var_right) {};
				\node at (1,4) [dot] (var_right1) {};
				
				\draw[testfcn] (left) to (root);
				\draw[testfcn] (right) to (root);
				\draw[kernel, bend right = 60] (left1) to (root);
				\draw[kernel, bend left = 60] (right1) to (root);
				
				\draw[ddrho] (left1) to (left);
				\draw[kernel1] (var_left) to (left);
				\draw[kernel1] (var_left1) to (left1);
				
				\draw[ddrho] (right1) to (right);
				\draw[kernel1] (var_right) to (right);
				\draw[kernel1] (var_right1) to (right1);
				
				\draw[ddrho-shift] (var_right1) to (var_left); 
				\draw[ddrho-shift] (var_left1) to node[above, pos=0.3]{\tiny$e^*$} (var_right); 
			\end{tikzpicture}\\
			\intertext{\hspace{2em} as well as}
			\phi_{1,1}
			\left(\begin{tikzpicture}[scale=0.35,baseline=0.9cm]
				\node at (0,-1)  [root] (root) {};
				\node at (-2,1)  [dot] (left) {};
				\node at (-2,3)  [dot] (left1) {};
				\node at (-2,5)  [dot] (left2) {};
				\node at (-2,7)  [dot] (left3) {};
				\node at (0,1) [var] (variable) {\tiny $2$};
				\node at (0,7) [var] (variable3) {\tiny $1$};
				
				\draw[testfcn] (left) to  (root);
				
				\draw[kernel2] (left1) to (left);
				\draw[kernelBig] (left2) to (left1);
				\draw[kernel1] (left3) to (left2);
				\draw[drho] (variable3) to (left3); 
				\draw[drho] (variable) to (left); 
			\end{tikzpicture}\right)
			&\;=\;
			\begin{tikzpicture}[scale=0.35,baseline=0.9cm]
				\node at (0,-1)  [root] (root) {};
				\node at (-2,1)  [dot] (left) {};
				\node at (-2,3)  [dot] (left1) {};
				\node at (-2,5)  [dot] (left2) {};
				\node at (-2,7)  [dot] (left3) {};
				\node at (2,1)  [dot] (right) {};
				\node at (2,3)  [dot] (right1) {};
				\node at (2,5)  [dot] (right2) {};
				\node at (2,7)  [dot] (right3) {};
				
				\draw[testfcn] (left) to  (root);
				\draw[testfcn] (right) to  (root);
				
				\draw[kernel2] (left1) to (left);
				\draw[kernelBig] (left2) to (left1);
				\draw[kernel1] (left3) to (left2);
				\draw[kernel2] (right1) to (right);
				\draw[kernelBig] (right2) to (right1);
				\draw[kernel1] (right3) to (right2);
				
				\draw[ddrho] (right3) to node[above]{\tiny$e^*$} (left3); 
				\draw[ddrho] (right) to (left); 
			\end{tikzpicture}\;.
		\end{equs}
		These Feynman graphs all contain relevant subgraphs $\bar G$ of $3$ vertices, with one $\kernelr$ and one $\ddrho$ edge in $E_0(\bar V)$ as well as one $\kernelr$ in $E^\uparrow(\bar V)$.  Such subgraphs have $\deg_4(\bar G) = 1$ and could lose one degree if $e^*$ is chosen to be the mollifier connected to the unpaired vertex. 
With the choices of $e^*$ indicated above, this does not happen and all such $\bar G$ ends up in either case (a) or (c).
		\item $|\bar V| = 4$, $u(\bar G) = 2$, $n^{\color{Cerulean}\uparrow} (\bar G) = n^{\color{red}\uparrow} (\bar G) = \tilde m(\bar G) = 0$. In $\cG \setminus (\cG_2 \cup \cG_3)$, this case happens only in
		\begin{equ}
			\phi_{1,1}
			\left(
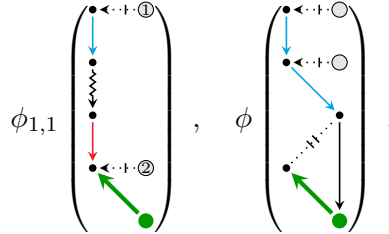
\begin{tikzpicture}[scale=0.35,baseline=0.9cm]
				\node at (0,-1)  [root] (root) {};
				\node at (-2,1)  [dot] (left) {};
				\node at (-2,3)  [dot] (left1) {};
				\node at (-2,5)  [dot] (left2) {};
				\node at (-2,7)  [dot] (left3) {};
				\node at (0,1) [var] (variable) {\tiny $2$};
				\node at (0,7) [var] (variable3) {\tiny $1$};
				
				\draw[testfcn] (left) to  (root);
				
				\draw[kernel2] (left1) to (left);
				\draw[kernelBig] (left2) to (left1);
				\draw[kernel1] (left3) to (left2);
				\draw[drho] (variable3) to (left3); 
				\draw[drho] (variable) to (left); 
			\end{tikzpicture}\right)\;,
			\quad
			\phi
			\left(\begin{tikzpicture}[scale=0.35,baseline=0.9cm]
				\node at (0,-1)  [root] (root) {};
				\node at (-2,1)  [dot] (left) {};
				\node at (-2,5)  [dot] (left2) {};
				\node at (-2,7)  [dot] (left3) {};
				\node at (0,3) [dot] (right1) {};
				\node at (0,5) [var] (variable2) {};
				\node at (0,7) [var] (variable3) {};
				
				\draw[testfcn] (left) to  (root);
				
				\draw[kernel1] (left2) to (right1);
				\draw[kernel1] (left3) to (left2);
				\draw[kernel] (right1) to (root);
				\draw[ddrho] (left) to (right1);
				\draw[drho] (variable3) to (left3); 
				\draw[drho] (variable2) to (left2); 
			\end{tikzpicture}\right)\;.
		\end{equ}
		The first graph falls in the case 5 and the same choice of $e^*$ brings the problematic subgraph to the situation (a); however, the second graph falls also in the case 3 and has to be considered independently.
		\item $|\bar V| = 5$, $u(\bar G) = 1$ and $n^{\color{Cerulean}\uparrow} (\bar G) + 2n^{\color{red}\uparrow} (\bar G) + \tilde m(\bar G) \leq 2$. In $\cG \setminus (\cG_2 \cup \cG_3)$, this case only happens in 
		\begin{equ}
			\phi
			\left(
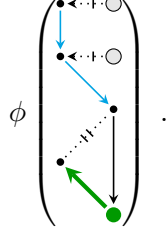
\begin{tikzpicture}[scale=0.35,baseline=0.9cm]
				\node at (0,-1)  [root] (root) {};
				\node at (-2,1)  [dot] (left) {};
				\node at (-2,5)  [dot] (left2) {};
				\node at (-2,7)  [dot] (left3) {};
				\node at (0,3) [dot] (right1) {};
				\node at (0,5) [var] (variable2) {};
				\node at (0,7) [var] (variable3) {};
				
				\draw[testfcn] (left) to  (root);
				
				\draw[kernel1] (left2) to (right1);
				\draw[kernel1] (left3) to (left2);
				\draw[kernel] (right1) to (root);
				\draw[ddrho] (left) to (right1);
				\draw[drho] (variable3) to (left3); 
				\draw[drho] (variable2) to (left2); 
			\end{tikzpicture}\right)\;.
		\end{equ}
		This graph cannot be improved by choosing smartly $e^*$ and already appeared in previous cases, hence is included in $\cG_4$. 
		\item $|\bar V| = 6$, $u(\bar G) = 2$ and $n^{\color{Cerulean}\uparrow} (\bar G) + 2n^{\color{red}\uparrow} (\bar G) + \tilde m(\bar G) \leq 1$. The same happens here as in the case 7.
	\end{enumerate}
	
	We therefore conclude that all the remaining Feynman graphs in $\cG\setminus (\cG_2 \cup \cG_3 \cup \cG_4)$ satisfy conditions 4 after the IBP process described in Proposition \ref{prop:adj} and under the adjusted labelling.
\end{proof}

We conclude this section by a corollary which follows immediately from Propositions \ref{prop:adj}, \ref{prop:cond3} and \ref{prop:cond4}.
\begin{corollary}\label{cor:graph}
	For every $G = (V,E) \in \cG\setminus (\cG_2 \cup \cG_3 \cup \cG_4)$, $G$ equals a linear combination of graphs $G_I$ such that for sufficiently small $\bar\kappa>0$, $\eps^{\frac{|V|-1}{2}-c(\bar \kappa)} G_I$ satisfies the assumption of Theorem \ref{thm:power_counting} under the adjusted labelling $(a^\adj_e, r_e)$, with $c(\bar\kappa)$ given by Proposition \ref{prop:adj}. In particular, with $c'(\bar \kappa)= c(\bar\kappa) + \bar\kappa^2 |E^0_{\bK} \cap E|$ and $\alpha = 2|V\setminus V_\ast| - \sum_e a_e$, where $a_e$ is given by the canonical labelling,  one has the bound
	\begin{equ}
		\left| \bI\big(\eps^{\frac{|V|-1}{2}}G, \varphi^\lambda\big)\right| \lesssim \eps^{c(\bar\kappa)} \lambda^{\alpha - c'(\bar\kappa)}\;.
	\end{equ}
\end{corollary}

\subsection{The critical Feynman graphs}\label{sec:critical-graphs}
We will now show the tightness bounds as well as the convergence of Feynman graphs in $\cG_2 \cup \cG_3 \cup \cG_4$.

Set
\begin{equs}[eq:Gcrit]
	\cG_{\mathrm{crit}} = &\left\{
	\begin{tikzpicture}[scale=0.35,baseline=0.4cm]
		\node at (0,-1)  [root] (root) {};
		\node at (0,1)  [dot] (root2) {};
		\node at (-1.5,2.5)  [dot] (left) {};
		\node at (1.5,2.5)  [dot] (right) {};
		\node at (0,4) [dot] (top) {};
		
		\draw[testfcn] (root2) to  (root);
		
		\draw[kernel,bend right=20] (left) to (root);
		\draw[kernel1] (right) to (root2);
		\draw[kernel1] (top) to (left);
		\draw[ddrho] (left) to (root2); 
		\draw[ddrho] (top) to (right);
	\end{tikzpicture}\;,
	\quad
	\begin{tikzpicture}[scale=0.35,baseline=0.4cm]
		\node at (0,-1)  [root] (root) {};
		\node at (1.5, 1) [dot] (right1) {};
		\node at (-1.5, 1) [dot] (left1) {};
		\node at (-1.5,3)  [dot] (left2) {};
		\node at (1.5,3)  [dot] (right2) {};
		
		\draw[testfcn] (right1) to  (root);
		
		\draw[kernel1] (right2) to (right1);
		\draw[kernel] (left2) to (left1);
		\draw[kernel] (left1) to (root); 
		\draw[ddrho-shift] (left2) to (right1); 
		\draw[ddrho-shift] (right2) to (left1); 
	\end{tikzpicture}\;,
	\quad
	\phi_{12, 12}\left(
	\begin{tikzpicture}[scale=0.35,baseline=0.4cm]
		\node at (0,-1)  [root] (root) {};
		\node at (0,1)  [dot] (root2) {};
		\node at (-1.5,2.5)  [dot] (left) {};
		\node at (1.5,2.5)  [dot] (right) {};
		\node at (0,2.5) [var] (variable) {\tiny 3};
		\node at (-1.5,4) [var] (variablel) {\tiny 2};
		\node at (1.5,4) [var] (variabler) {\tiny 4};
		\node at (0,4) [dot] (top) {};
		\node at (0,5.5) [var] (variablet) {\tiny 1};
		
		\draw[testfcn] (root2) to  (root);
		
		\draw[kernel1] (left) to (root2);
		\draw[kernel1] (right) to (root2);
		\draw[kernel1] (top) to (left);
		\draw[drho] (variable) to (root2); 
		\draw[drho] (variablel) to (left); 
		\draw[drho] (variabler) to (right); 
		\draw[drho] (variablet) to (top); 
	\end{tikzpicture}\right),
	\quad
	\phi_{12, 34}\left(
	\begin{tikzpicture}[scale=0.35,baseline=0.4cm]
		\node at (0,-1)  [root] (root) {};
		\node at (0,1)  [dot] (root2) {};
		\node at (-1.5,2.5)  [dot] (left) {};
		\node at (1.5,2.5)  [dot] (right) {};
		\node at (0,2.5) [var] (variable) {\tiny 3};
		\node at (-1.5,4) [var] (variablel) {\tiny 2};
		\node at (1.5,4) [var] (variabler) {\tiny 4};
		\node at (0,4) [dot] (top) {};
		\node at (0,5.5) [var] (variablet) {\tiny 1};
		
		\draw[testfcn] (root2) to  (root);
		
		\draw[kernel1] (left) to (root2);
		\draw[kernel1] (right) to (root2);
		\draw[kernel1] (top) to (left);
		\draw[drho] (variable) to (root2); 
		\draw[drho] (variablel) to (left); 
		\draw[drho] (variabler) to (right); 
		\draw[drho] (variablet) to (top); 
	\end{tikzpicture}\right),\right.\\
	&\left.
	\begin{tikzpicture}[scale=0.35,baseline=0.9cm]
		\node at (0,-1)  [root] (root) {};
		\node at (-2,1)  [dot] (left) {};
		\node at (-2,3)  [dot] (left1) {};
		\node at (0,5)  [dot] (right) {};
		\node at (0,7)  [dot] (right1) {};
		
		\draw[testfcn] (left) to  (root);
		
		\draw[kernel] (left1) to (left);
		\draw[kernel] (right) to (root);
		\draw[kernel1] (right1) to (right);
		\draw[ddrho] (left) to (right); 
		\draw[ddrho] (left1) to (right1); 
	\end{tikzpicture}\;,
	\quad
	\begin{tikzpicture}[scale=0.35,baseline=0.9cm]
		\node at (0,-1)  [root] (root) {};
		\node at (-2,1)  [dot] (left) {};
		\node at (-2,3)  [dot] (left1) {};
		\node at (0,7)  [dot] (right1) {};
		\node at (0,5) [dot] (right) {};
		
		\draw[testfcn] (left) to  (root);
		
		\draw[kernel2] (left1) to (left);
		\draw[kernel1] (right1) to (right);
		\draw[kernel] (right) to (root);
		\draw[ddrho-shift] (left1) to (right);
		\draw[ddrho-shift] (right1) to (left);
	\end{tikzpicture}\;,
	\quad
	\phi_{12, 12} \left(
	\begin{tikzpicture}[scale=0.35,baseline=0.9cm]
		\node at (0,-1)  [root] (root) {};
		\node at (-2,1)  [dot] (left) {};
		\node at (-2,3)  [dot] (left1) {};
		\node at (-2,5)  [dot] (left2) {};
		\node at (-2,7)  [dot] (left3) {};
		\node at (0,1) [var] (variable1) {\tiny 4};
		\node at (0,3) [var] (variable2) {\tiny 3};
		\node at (0,5) [var] (variable3) {\tiny 2};
		\node at (0,7) [var] (variable4) {\tiny 1};
		
		\draw[testfcn] (left) to  (root);
		
		\draw[kernel2] (left1) to (left);
		\draw[kernel1] (left2) to (left1);
		\draw[kernel1] (left3) to (left2);
		\draw[drho] (variable4) to (left3);
		\draw[drho] (variable3) to (left2); 
		\draw[drho] (variable2) to (left1); 
		\draw[drho] (variable1) to (left); 
	\end{tikzpicture}\right),
	\quad
	\phi_{12, 34} \left(
	\begin{tikzpicture}[scale=0.35,baseline=0.9cm]
		\node at (0,-1)  [root] (root) {};
		\node at (-2,1)  [dot] (left) {};
		\node at (-2,3)  [dot] (left1) {};
		\node at (-2,5)  [dot] (left2) {};
		\node at (-2,7)  [dot] (left3) {};
		\node at (0,1) [var] (variable1) {\tiny 4};
		\node at (0,3) [var] (variable2) {\tiny 3};
		\node at (0,5) [var] (variable3) {\tiny 2};
		\node at (0,7) [var] (variable4) {\tiny 1};
		
		\draw[testfcn] (left) to  (root);
		
		\draw[kernel2] (left1) to (left);
		\draw[kernel1] (left2) to (left1);
		\draw[kernel1] (left3) to (left2);
		\draw[drho] (variable4) to (left3);
		\draw[drho] (variable3) to (left2); 
		\draw[drho] (variable2) to (left1); 
		\draw[drho] (variable1) to (left); 
	\end{tikzpicture}\right)\right\}
\end{equs}
as well as
\begin{equs}
	\cG_{\mathrm{van}} &= (\cG_2 \cup \cG_3 \cup \cG_4) \setminus \cG_{\mathrm{crit}}\\
	&= \left\{
	\begin{tikzpicture}[scale=0.35,baseline=0.4cm]
		\node at (0,-1)  [root] (root) {};
		\node at (0,1)  [dot] (root2) {};
		\node at (-1.5,2.5)  [dot] (left) {};
		\node at (1.5,2.5)  [dot] (right) {};
		\node at (0,4) [dot] (top) {};
		
		\draw[testfcn] (root2) to  (root);
		
		\draw[kernel] (right) to (root);
		\draw[kernel] (top) to (left);
		\draw[kernelBig] (left) to (root2); 
		\draw[ddrho] (top) to (right);
	\end{tikzpicture}\;,
	\quad
	\begin{tikzpicture}[scale=0.35,baseline=0.4cm]
		\node at (0,-1)  [root] (root) {};
		\node at (0,1)  [dot] (root2) {};
		\node at (-1.5,2.5)  [dot] (left) {};
		\node at (1.5,2.5)  [dot] (right) {};
		\node at (0,4) [dot] (top) {};
		
		\draw[testfcn] (root2) to  (root);
		
		\draw[ddrho-shift] (top) to (root2);
		
		\draw[kernel] (right) to (root);
		\draw[kernel] (top) to (left);
		\draw[kernel] (left) to (root2); 
		\draw[ddrho-shift] (right) to (left);
	\end{tikzpicture}\;,
	\quad
	\begin{tikzpicture}[scale=0.35,baseline=0.6cm]
		\node at (0,-1)  [root] (root) {};
		\node at (-2,1)  [dot] (left) {};
		\node at (-2,3)  [dot] (left1) {};
		\node at (0,5)  [dot] (right1) {};
		\node at (0,2) [dot] (right) {};
		
		\draw[testfcn] (left) to  (root);
		
		\draw[kernel] (left1) to (left);
		\draw[kernel] (right) to (left1);
		\draw[kernel,bend left=40] (right1) to (root);
		\draw[ddrho] (left1) to (right1);
		\draw[ddrho] (right) to (left); 
	\end{tikzpicture}\;,
	\quad
	\begin{tikzpicture}[scale=0.35,baseline=0.6cm]
		\node at (0,-1)  [root] (root) {};
		\node at (-2,1)  [dot] (left) {};
		\node at (0,3)  [dot] (right) {};
		\node at (0,5)  [dot] (right1) {};
		\node at (-2,3) [dot] (left1) {};
		
		\draw[testfcn] (left) to  (root);
		
		\draw[kernel] (right) to (root);
		\draw[kernelBig] (right1) to (right);
		\draw[kernel] (left1) to (right1);
		\draw[ddrho] (left1) to (left);
	\end{tikzpicture}\;,
	\quad
	\begin{tikzpicture}[scale=0.35,baseline=0.6cm]
		\node at (0,-1)  [root] (root) {};
		\node at (-2,1)  [dot] (left) {};
		\node at (0,3)  [dot] (right) {};
		\node at (0,5)  [dot] (right1) {};
		\node at (-2,3) [dot] (left1) {};
		
		\draw[testfcnx] (left) to node[below, pos=0.4]{\tiny\color{black}$j$}  (root);
		
		\draw[dkernel] (right) to node[right, pos=0.5]{\tiny\color{black} $j$} (root);
		\draw[kernelBig] (right1) to (right);
		\draw[kernel] (left1) to (right1);
		\draw[ddrho] (left1) to (left);
	\end{tikzpicture}\;,
	\right.\\
	&\left.
	\phi
	\left(\begin{tikzpicture}[scale=0.35,baseline=0.9cm]
		\node at (0,-1)  [root] (root) {};
		\node at (-2,1)  [dot] (left) {};
		\node at (-2,5)  [dot] (left2) {};
		\node at (-2,7)  [dot] (left3) {};
		\node at (0,3) [dot] (right1) {};
		\node at (0,5) [var] (variable2) {};
		\node at (0,7) [var] (variable3) {};
		
		\draw[testfcn] (left) to  (root);
		
		\draw[kernel1] (left2) to (right1);
		\draw[kernel1] (left3) to (left2);
		\draw[kernel] (right1) to (root);
		\draw[ddrho] (left) to (right1);
		\draw[drho] (variable3) to (left3); 
		\draw[drho] (variable2) to (left2); 
	\end{tikzpicture}\right)\;,
	\quad
	\phi
	\left(\begin{tikzpicture}[scale=0.35,baseline=0.4cm]
		\node at (0,-1)  [root] (root) {};
		\node at (0,1)  [dot] (root2) {};
		\node at (-1.5,2.5)  [dot] (left) {};
		\node at (1.5,2.5)  [dot] (right) {};
		\node at (-1.5,4) [var] (variablel) {};
		\node at (0,4) [dot] (top) {};
		\node at (0,5.5) [var] (variablet) {};
		
		\draw[testfcn] (root2) to  (root);
		
		\draw[kernel1] (left) to (root2);
		\draw[kernel,bend left=20] (right) to (root);
		\draw[kernel1] (top) to (left);
		\draw[ddrho] (right) to (root2); 
		\draw[drho] (variablel) to (left);
		\draw[drho] (variablet) to (top); 
	\end{tikzpicture}\right)\;,
	\quad
	\phi \left(
	\begin{tikzpicture}[scale=0.35,baseline=0.9cm]
		\node at (0,-1)  [root] (root) {};
		\node at (-2,1)  [dot] (left) {};
		\node at (-2,5)  [dot] (left2) {};
		\node at (-2,7)  [dot] (left3) {};
		\node at (0,3) [dot] (right1) {};
		\node at (0,5) [var] (variable2) {};
		\node at (0,7) [var] (variable3) {};
		
		\draw[testfcnx] (left) to node[below, pos=0.4]{\tiny\color{black} $j$} (root);
		
		\draw[kernel1] (left2) to (right1);
		\draw[kernel1] (left3) to (left2);
		\draw[dkernel] (right1) to node[right, pos=0.5]{\tiny\color{black} $j$} (root);
		\draw[ddrho] (left) to (right1);
		\draw[drho] (variable3) to (left3); 
		\draw[drho] (variable2) to (left2); 
	\end{tikzpicture}\right)\;,
	\quad
	\phi\left(
	\begin{tikzpicture}[scale=0.35,baseline=0.9cm]
		\node at (0,-1)  [root] (root) {};
		\node at (-2,1)  [dot] (left) {};
		\node at (-2,3)  [dot] (left1) {};
		\node at (-2,5)  [dot] (left2) {};
		\node at (-2,7)  [dot] (left3) {};
		\node at (0,5) [var] (variable2) {};
		\node at (0,7) [var] (variable3) {};
		
		\draw[testfcn] (left) to  (root);
		
		\draw[kernelBig] (left1) to (left);
		\draw[kernel1] (left2) to (left1);
		\draw[kernel1] (left3) to (left2);
		\draw[drho] (variable3) to (left3); 
		\draw[drho] (variable2) to (left2); 
	\end{tikzpicture}
	\right).
	\right\}
\end{equs}

\begin{proposition}\label{prop:vanishing}
	All Feynman graphs in $\cG_{\mathrm{van}}$ satisfy the tightness bounds \eqref{eq:tightness_graph} and converge to zero as $\eps \to 0$.
\end{proposition}
\begin{proof}
	Note that the integral associated to the three graphs
	
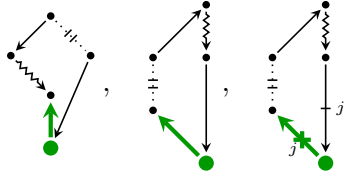
\begin{equ}
		\begin{tikzpicture}[scale=0.35,baseline=0.4cm]
			\node at (0,-1)  [root] (root) {};
			\node at (0,1)  [dot] (root2) {};
			\node at (-1.5,2.5)  [dot] (left) {};
			\node at (1.5,2.5)  [dot] (right) {};
			\node at (0,4) [dot] (top) {};
			
			\draw[testfcn] (root2) to  (root);
			
			\draw[kernel] (right) to (root);
			\draw[kernel] (top) to (left);
			\draw[kernelBig] (left) to (root2); 
			\draw[ddrho] (top) to (right);
		\end{tikzpicture}\;,
		\quad
		\begin{tikzpicture}[scale=0.35,baseline=0.6cm]
			\node at (0,-1)  [root] (root) {};
			\node at (-2,1)  [dot] (left) {};
			\node at (0,3)  [dot] (right) {};
			\node at (0,5)  [dot] (right1) {};
			\node at (-2,3) [dot] (left1) {};
			
			\draw[testfcn] (left) to  (root);
			
			\draw[kernel] (right) to (root);
			\draw[kernelBig] (right1) to (right);
			\draw[kernel] (left1) to (right1);
			\draw[ddrho] (left1) to (left);
		\end{tikzpicture}\;,
		\quad
		\begin{tikzpicture}[scale=0.35,baseline=0.6cm]
			\node at (0,-1)  [root] (root) {};
			\node at (-2,1)  [dot] (left) {};
			\node at (0,3)  [dot] (right) {};
			\node at (0,5)  [dot] (right1) {};
			\node at (-2,3) [dot] (left1) {};
			
			\draw[testfcnx] (left) to node[below, pos=0.4]{\tiny\color{black}$j$}  (root);
			
			\draw[dkernel] (right) to node[right, pos=0.5]{\tiny\color{black} $j$} (root);
			\draw[kernelBig] (right1) to (right);
			\draw[kernel] (left1) to (right1);
			\draw[ddrho] (left1) to (left);
		\end{tikzpicture}
	\end{equ}
	multiplied by the factor $\eps^2$, can all be written as $\eps^{\bar\kappa}\int (\rho_\eps^{\ast 2} \ast \partial_1 K \ast \eps^{2-\bar\kappa} \ccR R_\eps \ast \partial_j K) (x) \varphi^\lambda(x) \dd x$, where $\ccR R_\eps$ denotes the renormalised kernel defined by \eqref{eq:kernel_renorm}, with $j \in \{1, 2\}$.
	Since the function $\partial_1 K$ has singularity of order $1$ and the renormalised distribution $\eps^{2-\bar\kappa }\ccR R_\eps$ has singularity of order $2+2\bar\kappa$ uniformly for $\eps$, it follows from \cite[Lem.~10.14,~10.16,~10.17]{H0} that the function $\rho_\eps^{\ast 2} \ast \partial_1 K \ast \eps^{2-2\bar\kappa} \ccR R_\eps \ast \partial_j K$ has a singularity of order $2\bar\kappa$.
	Consequently, these Feynman graphs are of order $\eps^{\bar\kappa} \lambda^{-2\bar\kappa}$, thus satisfying the desired bound \eqref{eq:tightness_graph} for sufficiently small $\bar \kappa$, and moreover converges to $0$.
	
	Similarly, the graphs
	
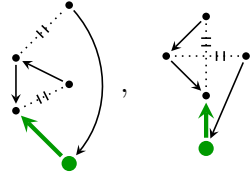
\begin{equ}
		\begin{tikzpicture}[scale=0.35,baseline=0.6cm]
			\node at (0,-1)  [root] (root) {};
			\node at (-2,1)  [dot] (left) {};
			\node at (-2,3)  [dot] (left1) {};
			\node at (0,5)  [dot] (right1) {};
			\node at (0,2) [dot] (right) {};
			
			\draw[testfcn] (left) to  (root);
			
			\draw[kernel] (left1) to (left);
			\draw[kernel] (right) to (left1);
			\draw[kernel,bend left=40] (right1) to (root);
			\draw[ddrho] (left1) to (right1);
			\draw[ddrho] (right) to (left); 
		\end{tikzpicture}\;,
		\quad
		\begin{tikzpicture}[scale=0.35,baseline=0.4cm]
			\node at (0,-1)  [root] (root) {};
			\node at (0,1)  [dot] (root2) {};
			\node at (-1.5,2.5)  [dot] (left) {};
			\node at (1.5,2.5)  [dot] (right) {};
			\node at (0,4) [dot] (top) {};
			
			\draw[testfcn] (root2) to  (root);
			
			\draw[ddrho-shift] (top) to (root2);
			
			\draw[kernel] (right) to (root);
			\draw[kernel] (top) to (left);
			\draw[kernel] (left) to (root2); 
			\draw[ddrho-shift] (right) to (left);
		\end{tikzpicture}
	\end{equ}
	with prefactor $\eps^2$ can both be represented by the integral $\eps^{\bar\kappa} \int (\eps^{1-\bar\kappa}\partial_1^2 K_\eps) \ast (K \cdot \eps \partial_1^2 K_\eps) (x) \varphi^\lambda(x) \dd x$,
	where $K_\eps = \rho_\eps^{\ast 2} \ast K$. Again by the fact that both the functions $\eps^{1-\bar\kappa}\partial_1^2 K_\eps$ and $K \cdot \eps \partial_1^2 K_\eps$ have a singularity of order $1+\bar\kappa$, we deduce by \cite[Lem.~ 10.14]{H0} that the quantity is of order $\eps^{\bar \kappa} \lambda^{-2\bar\kappa}$.
	
	Let us now consider the last term in the list $\cG_{\mathrm{van}}$. By IBP, one has
	\begin{equs}
		\E\left(\eps^2\;
		\begin{tikzpicture}[scale=0.35,baseline=0.9cm]
			\node at (0,-1)  [root] (root) {};
			\node at (-2,1)  [dot] (left) {};
			\node at (-2,3)  [dot] (left1) {};
			\node at (-2,5)  [dot] (left2) {};
			\node at (-2,7)  [dot] (left3) {};
			\node at (0,5) [var] (variable2) {};
			\node at (0,7) [var] (variable3) {};
			
			\draw[testfcn] (left) to  (root);
			
			\draw[kernelBig] (left1) to (left);
			\draw[kernel1] (left2) to (left1);
			\draw[kernel1] (left3) to (left2);
			\draw[drho] (variable3) to (left3); 
			\draw[drho] (variable2) to (left2); 
		\end{tikzpicture}\right)^2
		\;= \eps^2 \; 
		\begin{tikzpicture}[scale=0.35,baseline=0.7cm]
			\node at (0,-1)  [root] (root) {};
			\node at (-2,1)  [dot] (left) {};
			\node at (-2,3)  [dot] (left1) {};
			\node at (-2,5)  [dot] (left2) {};
			\node at (2,1)  [dot] (right) {};
			\node at (2,3)  [dot] (right1) {};
			\node at (2,5)  [dot] (right2) {};
			
			\draw[testfcn] (left) to  (root);
			\draw[testfcn] (right) to  (root);
			
			\draw[kernel1] (left2) to (left1);
			\draw[BigG] (right2) to (left2);
			\draw[kernel1] (right2) to (right1);
			\draw[kernelBig] (left1) to (left);
			\draw[kernelBig] (right1) to (right);
		\end{tikzpicture}
		+ E_{\eps, \lambda}\label{eq:van-graph}
	\end{equs}
	where the edge $\BigG$ represents the kernel $G_\eps$ defined in \eqref{eq:G} and the error $E_{\eps, \lambda}$ consists of the sum
	\begin{equ}
		2\eps^4\;\left(
		\begin{tikzpicture}[scale=0.35,baseline=0.9cm]
			\node at (0,-1)  [root] (root) {};
			\node at (-2,1)  [dot] (left) {};
			\node at (-2,3)  [dot] (left1) {};
			\node at (-2,5)  [dot] (left2) {};
			\node at (-2,7)  [dot] (left3) {};
			\node at (2,1)  [dot] (right) {};
			\node at (2,3)  [dot] (right1) {};
			\node at (2,5)  [dot] (right2) {};
			\node at (2,7)  [dot] (right3) {};
			
			\draw[testfcn] (left) to  (root);
			\draw[testfcn] (right) to  (root);
			
			\draw[kernelBig] (left1) to (left);
			\draw[kernel1] (left2) to (left1);
			\draw[ddkernel] (left3) to (left2);
			\draw[kernelBig] (right1) to (right);
			\draw[dkernel1] (right2) to (right1);
			\draw[dkernel1] (right3) to (right2);
			\draw[rho] (right3) to (left3); 
			\draw[rho] (right2) to (left2); 
		\end{tikzpicture}
		\;+\;
		\begin{tikzpicture}[scale=0.35,baseline=0.9cm]
			\node at (0,-1)  [root] (root) {};
			\node at (-2,1)  [dot] (left) {};
			\node at (-2,3)  [dot] (left1) {};
			\node at (-2,5)  [dot] (left2) {};
			\node at (-2,7)  [dot] (left3) {};
			\node at (2,1)  [dot] (right) {};
			\node at (2,3)  [dot] (right1) {};
			\node at (2,5)  [dot] (right2) {};
			\node at (2,7)  [dot] (right3) {};
			
			\draw[testfcn] (left) to  (root);
			\draw[testfcn] (right) to  (root);
			
			\draw[kernelBig] (left1) to (left);
			\draw[kernel1] (left2) to (left1);
			\draw[ddkernel] (left3) to (left2);
			\draw[kernelBig] (right1) to (right);
			\draw[dkernel1] (right2) to (right1);
			\draw[dkernel1] (right3) to (right2);
			\draw[rho] (right3) to (left2); 
			\draw[rho] (right2) to (left3); 
		\end{tikzpicture}
		\right) + \eps^4\left(\;
		\begin{tikzpicture}[scale=0.35,baseline=0.9cm]
			\node at (0,-1)  [root] (root) {};
			\node at (-2,1)  [dot] (left) {};
			\node at (-2,3)  [dot] (left1) {};
			\node at (-2,5)  [dot] (left2) {};
			\node at (-2,7)  [dot] (left3) {};
			\node at (2,1)  [dot] (right) {};
			\node at (2,3)  [dot] (right1) {};
			\node at (2,5)  [dot] (right2) {};
			\node at (2,7)  [dot] (right3) {};
			
			\draw[testfcn] (left) to  (root);
			\draw[testfcn] (right) to  (root);
			
			\draw[kernelBig] (left1) to (left);
			\draw[dkernel1] (left2) to (left1);
			\draw[dkernel1] (left3) to (left2);
			\draw[kernelBig] (right1) to (right);
			\draw[dkernel1] (right2) to (right1);
			\draw[dkernel1] (right3) to (right2);
			\draw[rho] (right3) to (left3); 
			\draw[rho] (right2) to (left2); 
		\end{tikzpicture}
		\;+\;
		\begin{tikzpicture}[scale=0.35,baseline=0.9cm]
			\node at (0,-1)  [root] (root) {};
			\node at (-2,1)  [dot] (left) {};
			\node at (-2,3)  [dot] (left1) {};
			\node at (-2,5)  [dot] (left2) {};
			\node at (-2,7)  [dot] (left3) {};
			\node at (2,1)  [dot] (right) {};
			\node at (2,3)  [dot] (right1) {};
			\node at (2,5)  [dot] (right2) {};
			\node at (2,7)  [dot] (right3) {};
			
			\draw[testfcn] (left) to  (root);
			\draw[testfcn] (right) to  (root);
			
			\draw[kernelBig] (left1) to (left);
			\draw[dkernel1] (left2) to (left1);
			\draw[dkernel1] (left3) to (left2);
			\draw[kernelBig] (right1) to (right);
			\draw[dkernel1] (right2) to (right1);
			\draw[dkernel1] (right3) to (right2);
			\draw[rho] (right3) to (left2); 
			\draw[rho] (right2) to (left3); 
		\end{tikzpicture}
		\right)\;.
	\end{equ}
	To bound these graphs, it suffices to distribute $\eps^{1-2\bar\kappa}$ to every $\erho$ and $\eps^{1+\bar\kappa}$ to every $\kernelBig$. In particular, $E_{\eps, \lambda}$ can be rewritten as the labelled graphs
	\begin{equ}
		2\eps^{2\bar\kappa}
		\begin{tikzpicture}[scale=0.35,baseline=0.9cm]
			\node at (0,-1)  [root] (root) {};
			\node at (-2,1)  [dot] (left) {};
			\node at (-2,3)  [dot] (left1) {};
			\node at (-2,5)  [dot] (left2) {};
			\node at (-2,7)  [dot] (left3) {};
			\node at (2,1)  [dot] (right) {};
			\node at (2,3)  [dot] (right1) {};
			\node at (2,5)  [dot] (right2) {};
			\node at (2,7)  [dot] (right3) {};
			
			\draw[dist] (left) to  (root);
			\draw[dist] (right) to  (root);
			
			\draw[generic] (left1) to node[labl]{\tiny3-$\bar\kappa$,-1} (left);
			\draw[->] (left2) to node[labl]{\tiny $\bar\kappa$,1} (left1);
			\draw[generic] (left3) to node[labl]{\tiny 2,-1} (left2);
			\draw[generic] (right1) to node[labl]{\tiny3-$\bar\kappa$,-1} (right);
			\draw[->] (right2) to node[labl]{\tiny 1,1} (right1);
			\draw[->] (right3) to node[labl]{\tiny 1,1} (right2);
			\draw[generic] (right3) to node[labl]{\tiny 1+2$\bar\kappa$,0} (left3); 
			\draw[generic] (right2) to node[labl]{\tiny 1+2$\bar\kappa$,0} (left2); 
		\end{tikzpicture}
		\;+2\eps^{2\bar\kappa}\;
		\begin{tikzpicture}[scale=0.35,baseline=0.9cm]
			\node at (0,-1)  [root] (root) {};
			\node at (-2,1)  [dot] (left) {};
			\node at (-2,3)  [dot] (left1) {};
			\node at (-2,5)  [dot] (left2) {};
			\node at (-2,7)  [dot] (left3) {};
			\node at (2,1)  [dot] (right) {};
			\node at (2,3)  [dot] (right1) {};
			\node at (2,5)  [dot] (right2) {};
			\node at (2,7)  [dot] (right3) {};
			
			\draw[dist] (left) to  (root);
			\draw[dist] (right) to  (root);
			
			\draw[generic] (left1) to node[labl]{\tiny3-$\bar\kappa$,-1} (left);
			\draw[->] (left2) to node[labl]{\tiny $\bar\kappa$,1} (left1);
			\draw[generic] (left3) to node[labl]{\tiny 2,-1} (left2);
			\draw[generic] (right1) to node[labl]{\tiny3-$\bar\kappa$,-1} (right);
			\draw[->] (right2) to node[labl]{\tiny 1,1} (right1);
			\draw[->] (right3) to node[labl]{\tiny 1,1} (right2);
			\draw[generic] (right3) to node[labl, pos=0.45, above]{\tiny1+2$\bar\kappa$,0} (left2); 
			\draw[generic] (right2) to node[labl, pos=0.45, below]{\tiny1+2$\bar\kappa$,0} (left3); 
		\end{tikzpicture}
		+ \eps^{2\bar\kappa}\;
		\begin{tikzpicture}[scale=0.35,baseline=0.9cm]
			\node at (0,-1)  [root] (root) {};
			\node at (-2,1)  [dot] (left) {};
			\node at (-2,3)  [dot] (left1) {};
			\node at (-2,5)  [dot] (left2) {};
			\node at (-2,7)  [dot] (left3) {};
			\node at (2,1)  [dot] (right) {};
			\node at (2,3)  [dot] (right1) {};
			\node at (2,5)  [dot] (right2) {};
			\node at (2,7)  [dot] (right3) {};
			
			\draw[dist] (left) to  (root);
			\draw[dist] (right) to  (root);
			
			\draw[generic] (left1) to node[labl]{\tiny 3-$\bar\kappa$,-1} (left);
			\draw[->] (left2) to node[labl]{\tiny 1,1} (left1);
			\draw[->] (left3) to node[labl]{\tiny 1,1} (left2);
			\draw[generic] (right1) to node[labl]{\tiny 3-$\bar\kappa$,-1} (right);
			\draw[->] (right2) to node[labl]{\tiny 1,1} (right1);
			\draw[->] (right3) to node[labl]{\tiny 1,1} (right2);
			\draw[generic] (right3) to node[labl]{\tiny 1+2$\bar\kappa$,0} (left3); 
			\draw[generic] (right2) to node[labl]{\tiny 1+2$\bar\kappa$,0} (left2); 
		\end{tikzpicture}
		\;+\eps^{2\bar\kappa}\;
		\begin{tikzpicture}[scale=0.35,baseline=0.9cm]
			\node at (0,-1)  [root] (root) {};
			\node at (-2,1)  [dot] (left) {};
			\node at (-2,3)  [dot] (left1) {};
			\node at (-2,5)  [dot] (left2) {};
			\node at (-2,7)  [dot] (left3) {};
			\node at (2,1)  [dot] (right) {};
			\node at (2,3)  [dot] (right1) {};
			\node at (2,5)  [dot] (right2) {};
			\node at (2,7)  [dot] (right3) {};
			
			\draw[dist] (left) to  (root);
			\draw[dist] (right) to  (root);
			
			\draw[generic] (left1) to node[labl]{\tiny 3-$\bar\kappa$,-1} (left);
			\draw[->] (left2) to node[labl]{\tiny 1,1} (left1);
			\draw[->] (left3) to node[labl]{\tiny 1,1} (left2);
			\draw[generic] (right1) to node[labl]{\tiny 3-$\bar\kappa$,-1} (right);
			\draw[->] (right2) to node[labl]{\tiny 1,1} (right1);
			\draw[->] (right3) to node[labl]{\tiny 1,1} (right2);
			\draw[generic] (right3) to node[labl, pos=0.45, above]{\tiny1+2$\bar\kappa$,0} (left2); 
			\draw[generic] (right2) to node[labl, pos=0.45, below]{\tiny1+2$\bar\kappa$,0} (left3); 
		\end{tikzpicture}\;.
	\end{equ}
	This yields the bound $E_{\eps, \lambda} \lesssim \eps^{2\bar\kappa} \lambda^{-3\bar\kappa}$ by Theorem \ref{thm:power_counting}.
	As for the first term on the right-hand side of \eqref{eq:van-graph}, first note that since the renormalised kernel $\kernelBig$ has vanishing integral,	both $\kernelr$ edges can be replaced by $\kernel$.
	By distributing $\eps^{1-\bar\kappa}$ to each $\kernelBig$, the resulting graph is nothing but
	$\int (\ccR K_1\ast K_2\ast \ccR K_3\ast K_4\ast \ccR K_5)(x-y)\varphi^\lambda(x)\varphi^\lambda(y)\dd x\dd y$, with $\|K_1\|_{3+\bar \kappa,m},\|K_5\|_{3+\bar \kappa,m},\|K_2\|_{\bar \kappa,m}, \|K_4\|_{\bar\kappa,m},\|K_3\|_{2, m}\lesssim 1$.
	By repeatedly applying \cite[Lem.~10.16]{H0}, one deduces the term is of order $\eps^{2\bar\kappa} \lambda^{-4\bar\kappa}$.
	
	Notice that the remaining terms in the list $\cG_{\mathrm{van}}$
	\begin{equ}
		\phi
		\left(\begin{tikzpicture}[scale=0.35,baseline=0.9cm]
			\node at (0,-1)  [root] (root) {};
			\node at (-2,1)  [dot] (left) {};
			\node at (-2,5)  [dot] (left2) {};
			\node at (-2,7)  [dot] (left3) {};
			\node at (0,3) [dot] (right1) {};
			\node at (0,5) [var] (variable2) {};
			\node at (0,7) [var] (variable3) {};
			
			\draw[testfcn] (left) to  (root);
			
			\draw[kernel1] (left2) to (right1);
			\draw[kernel1] (left3) to (left2);
			\draw[kernel] (right1) to (root);
			\draw[ddrho] (left) to (right1);
			\draw[drho] (variable3) to (left3); 
			\draw[drho] (variable2) to (left2); 
		\end{tikzpicture}\right)\;,
		\quad
		\phi
		\left(\begin{tikzpicture}[scale=0.35,baseline=0.4cm]
			\node at (0,-1)  [root] (root) {};
			\node at (0,1)  [dot] (root2) {};
			\node at (-1.5,2.5)  [dot] (left) {};
			\node at (1.5,2.5)  [dot] (right) {};
			\node at (-1.5,4) [var] (variablel) {};
			\node at (0,4) [dot] (top) {};
			\node at (0,5.5) [var] (variablet) {};
			
			\draw[testfcn] (root2) to  (root);
			
			\draw[kernel1] (left) to (root2);
			\draw[kernel,bend left=20] (right) to (root);
			\draw[kernel1] (top) to (left);
			\draw[ddrho] (right) to (root2); 
			\draw[drho] (variablel) to (left);
			\draw[drho] (variablet) to (top); 
		\end{tikzpicture}\right)\;,
		\quad
		\phi \left(
		\begin{tikzpicture}[scale=0.35,baseline=0.9cm]
			\node at (0,-1)  [root] (root) {};
			\node at (-2,1)  [dot] (left) {};
			\node at (-2,5)  [dot] (left2) {};
			\node at (-2,7)  [dot] (left3) {};
			\node at (0,3) [dot] (right1) {};
			\node at (0,5) [var] (variable2) {};
			\node at (0,7) [var] (variable3) {};
			
			\draw[testfcnx] (left) to node[below, pos=0.4]{\tiny\color{black} $j$} (root);
			
			\draw[kernel1] (left2) to (right1);
			\draw[kernel1] (left3) to (left2);
			\draw[dkernel] (right1) to node[right, pos=0.5]{\tiny\color{black} $j$} (root);
			\draw[ddrho] (left) to (right1);
			\draw[drho] (variable3) to (left3); 
			\draw[drho] (variable2) to (left2); 
		\end{tikzpicture}\right)
	\end{equ}
	all have a similar structure. By arguing verbatim with the above IBP procedure and by absorbing some derivatives to the test functions, one can reduce these graphs (with the prefactor $\eps^4$) to
	\begin{equ}[eq:last_vanishing]
		\eps^2 \lambda^{-2} \; 
		\begin{tikzpicture}[scale=0.35,baseline=0.9cm]
			\node at (0,1)  [root] (root) {};
			\node at (-2,1)  [dot] (left) {};
			\node at (-2,3)  [dot] (left1) {};
			\node at (-2,5)  [dot] (left2) {};
			\node at (2,1)  [dot] (right) {};
			\node at (2,3)  [dot] (right1) {};
			\node at (2,5)  [dot] (right2) {};
			
			\draw[testfcn] (left) to  (root);
			\draw[testfcn] (right) to  (root);
			
			\draw[kernel1] (left2) to (left1);
			\draw[BigG] (right2) to (left2);
			\draw[kernel1] (right2) to (right1);
			\draw[rho] (left1) to (left);
			\draw[kernel] (left1) to (root);
			\draw[rho] (right1) to (right);
			\draw[kernel] (right1) to (root);
		\end{tikzpicture}\;,
		\quad
		\eps^2 \; 
		\begin{tikzpicture}[scale=0.35,baseline=0.9cm]
			\node at (0,1)  [root] (root) {};
			\node at (-2,1)  [dot] (left) {};
			\node at (-2,3)  [dot] (left1) {};
			\node at (-2,5)  [dot] (left2) {};
			\node at (2,1)  [dot] (right) {};
			\node at (2,3)  [dot] (right1) {};
			\node at (2,5)  [dot] (right2) {};
			
			\draw[testfcn] (left1) to  (root);
			\draw[testfcn] (right1) to  (root);
			
			\draw[kernel1] (left2) to (left1);
			\draw[BigG] (right2) to (left2);
			\draw[kernel1] (right2) to (right1);
			\draw[drho] (left) to (left1);
			\draw[dkernel] (left) to (root);
			\draw[drho] (right) to (right1);
			\draw[dkernel] (right) to (root);
		\end{tikzpicture}\;,
		\quad
		\eps^2 \lambda^{-1} \; 
		\begin{tikzpicture}[scale=0.35,baseline=0.9cm]
			\node at (0,1)  [root] (root) {};
			\node at (-2,1)  [dot] (left) {};
			\node at (-2,3)  [dot] (left1) {};
			\node at (-2,5)  [dot] (left2) {};
			\node at (2,1)  [dot] (right) {};
			\node at (2,3)  [dot] (right1) {};
			\node at (2,5)  [dot] (right2) {};
			
			\draw[testfcn] (left) to  (root);
			\draw[testfcn] (right) to  (root);
			
			\draw[kernel1] (left2) to (left1);
			\draw[BigG] (right2) to (left2);
			\draw[kernel1] (right2) to (right1);
			\draw[rho] (left1) to (left);
			\draw[dkernel] (left1) to node[above]{\tiny$j$} (root);
			\draw[rho] (right1) to (right);
			\draw[dkernel] (right1) to node[above]{\tiny$j$} (root);
		\end{tikzpicture}
	\end{equ}
	at the cost of some remainder quantities of order $\eps^{c(\bar\kappa)} \lambda^{-c'(\bar\kappa)}$, with some $c(\bar\kappa), c'(\bar\kappa) \to 0$. The graphs in \eqref{eq:last_vanishing} are easy to bound: one only needs to distribute $\eps^{1-\bar\kappa}$ to each $\erho$ or $\drho$ and see the associated faithful labelling satisfies the assumption of Theorem \ref{thm:power_counting}. This yields the bound $\eps^{2\bar\kappa} \lambda^{-4\bar\kappa}$, $\eps^{2\bar\kappa} \lambda^{-2\bar\kappa}$, $\eps^{2\bar\kappa} \lambda^{-2\bar\kappa}$ to each of these terms, respectively, and thus concludes the proof.
\end{proof}

Combining the results in Section \ref{sec:vanishing_trees}, Corollary \ref{cor:graph} and Proposition \ref{prop:vanishing}, we have shown that the all Feynman graphs in $\cG \setminus \cG_{\mathrm{crit}}$ satisfy the tightness bound \eqref{eq:tightness_graph} and converge to zero. It the remains to study the set $\cG_{\mathrm{crit}}$ of critical Feynman graphs.\\

Let us now prove a useful criterion for the convergence in law of a sequence of random variables whose limit is expected to live in a second-order homogeneous Wiener chaos. 
To recall the basic definitions for Wiener chaoses, for $d,n \in \N$, let us denote by $L^2_{\mathrm{sym}}((\T^d)^n)$ the space of square integrable symmetric functions $f: (\T^d)^n \to \R$ (i.e., $f(x_1, \dots, x_n) = f(x_{\sigma(1)}, \dots, x_{\sigma(n)})$ for all $x_i \in \T^d$, $i = 1,\dots, n$ and all permutation $\sigma$ on $\{1, \dots, n\}$) equipped with the norm
\begin{equ}
	\norm{f}_{L^2_{\mathrm{sym}}}^2 = n! \int_{(\T^d)^n} |f(x)|^2 \dd x\;.
\end{equ}
For a function $f$ not necessarily symmetric, define
\begin{equs}
	\norm{f}_{L^2_{\mathrm{sym}}}^2 &= \|\tilde f\|_{L^2_{\mathrm{sym}}}^2,
	\end{equs}
where the symmetrisation $\tilde f$ is given by
\begin{equs}
	\tilde f(x_1, \dots, x_n) &= \frac{1}{n!} \sum_{\sigma} f(x_{\sigma(1)}, \dots, x_{\sigma(n)}),
\end{equs}
with the sum running over all permutations $\sigma$ of $\{1, \dots, n\}$.
We denote by $\langle\cdot,\cdot\rangle_{L^2_{\mathrm{sym}}}$ the induced inner product.
Given a white noise $\xi$ and a function $f \in L^2(\T^2 \times \T^2)$, let us denote by $I_2(f)$ the iterated stochastic integral
\begin{equ}
	I_2(f) = \iint f(x, y) \xi(\dd x) \xi(\dd y)\;.
\end{equ}
In particular, one has the Itô isometry
\begin{equ}
	\E\big(I_2(f)^2\big) = \norm{f}_{L^2_{\mathrm{sym}}}^2\;.
\end{equ}

\begin{lemma}\label{lem:criterion_second_chaos}
	Let $\alpha,\beta\in\R$, $\kappa>0$.
	Let $(\Omega, \cF, \P)$ be a probability space and $A_\eps, B_\eps: \Omega \to \cD'(\T^d)$ be centred random distributions indexed by $\eps \in(0,1]$.
	Let furthermore  $\bB(x,y)$ be a family of distributions on $\T^d$, indexed by $(x,y)\in\T^d\times \T^d$ such that
	for every $\varphi \in C^\infty$,
	$\bB_\varphi(x,y):=\big(\bB(x,y)\big)(\varphi)	$ is square integrable.
	Assume the following
	\begin{enumerate}
	\item[(i)] $A_\eps$ converges in $\cC^{\alpha+\kappa}(\T^d)$ in law to a white noise $A$ ;
	\item[(ii)] For some $p\in[1,\infty)$ and $\gamma\in\R$ satisfying $(\gamma-d)/p>\beta$, one has the bound
		\begin{equ}
			\E |B_\eps(\varphi^\lambda)|^p\lesssim \lambda^{\gamma}
		\end{equ}
	uniformly in $\eps\in(0,1]$, $\lambda\in(0,1]$, and $\varphi\in C^\infty_c$.
	\item[(iii)] For any $\varphi\in C^\infty$, the $3+\kappa$-th moments of $A_\eps(\varphi)$ and $B_\eps(\varphi)$ are uniformly bounded in $\eps>0$;
		\item[(iv)] For any $\varphi \in C^\infty$, one has as $\eps\to 0$
		\begin{equ}
			\E\Big(B_\eps(\varphi) A_\eps(\varphi)A_\eps( \varphi)\Big) \to \langle \bB_{\varphi},\varphi \otimes \varphi\rangle_{L^2_{\mathrm{sym}}(\T^d\times\T^d)};	
		\end{equ}
	\item[(v)] For any $\varphi\in C^\infty$, one has as $\eps\to 0$
		\begin{equ}
			\E\big(B_\eps(\varphi)\big)^2 \to \|\bB_{\varphi}\|_{L^2_{\mathrm{sym}}(\T^d\times\T^d)}^2;
		\end{equ}
	\end{enumerate}
	Then, $(A_\eps,B_\eps)$ converges jointly in law in $\cC^{\alpha}(\T^d)\times\cC^\beta(\T^d)$ to $(A,B)$, where $B$ satisfies
		\begin{equ}
			B(\varphi)=I_2(\bB_\varphi)
		\end{equ}
	for any $\varphi \in C^\infty$ almost surely, with $I_2$ taken with respect to $A$.
\end{lemma}
\begin{remark}
If $A_\eps$ and $B_\eps$ belong to the Wiener chaos of a fixed order with respect to a reference Gaussian measure, then (ii) and (iii) may be replaced by 
\begin{enumerate}
\item [(ii')] As (ii) above, but with $p=2$ and any $\gamma>2\beta$.
\end{enumerate}
\end{remark}
\begin{proof}
	Take a subsequence (still denoted by $(A_\eps,B_\eps)$).
	The conditions (i) and (ii) together imply that the family of the laws of $(A_\eps,B_\eps)$ on the closure of smooth functions in $\cC^{\alpha}(\T^d)\times\cC^\beta(\T^d)$ is tight.
	By Prokhorov's theorem there exists a further subsequence (still denoted by $(A_\eps,B_\eps)$) that converges in law to a limit $(A,B)$. 
	By Skorohod's representation theorem there exists a probability space $(\bar \Omega,\bar\cF,\bar\P)$, and on it, random variables $(\bar A_\eps,\bar B_\eps)$, $(\bar A,\bar B)$, such that $(\bar A_\eps,\bar B_\eps)\overset{\mathrm{law}}{=}( A_\eps,B_\eps)$, $(\bar A,\bar B)\overset{\mathrm{law}}{=}( A,B)$, and that $(\bar A_\eps,\bar B_\eps)\to( \bar A,\bar B)$ for all $\bar\omega\in\bar \Omega$.

	By (iii), the random variables in the expectation in the left-hand side of  (iv) are uniformly integrable. This also holds in the new probability space, so one can take the $\eps\to 0 $ limit inside the expectation and get, with $\varphi_1=\varphi_2=\varphi_3=\varphi$,
	\begin{equ}
		\langle \bB_{\varphi_1},\varphi_2 \otimes \varphi_3\rangle_{L^2_{\mathrm{sym}}(\T^d\times\T^d)}=\bar\E\Big(\bar B(\varphi_1) \bar A(\varphi_2)\bar A( \varphi_3)\Big) = \bar\E\Big(\bar B(\varphi_1)\bar I_2(\varphi_2\otimes\varphi_3)\Big),
	\end{equ}
	where $\bar I_2$ is taken with respect to $\bar A$ and where the second equality follows from the assumption that $B_\eps$ are centred and therefore so is $B$.
	By polarisation, this holds for any choices of smooth $\varphi_1,\varphi_2,\varphi_3$.
	One obviously also has
\begin{equ}
	\bar\E\Big(\bar I_2(\bB_{\varphi_1})\bar I_2(\varphi_2\otimes\varphi_3)\Big) = \langle \bB_{\varphi_1},\varphi_2 \otimes \varphi_3\rangle_{L^2_{\mathrm{sym}}(\T^d\times\T^d)},
\end{equ}
Therefore denoting by $\bar\cF^2$ the $\sigma$-algebra generated by the second homogeneous chaos of $\bar A$, one has $\bar \P$-a.s. $\bar \E(\bar B(\varphi_1)|\bar \cF^2)=\bar I_2(\bB_{\varphi_1})$. Since by (v), $\bar \E(\bar B(\varphi_1))^2=\bar \E\big(\bar I_2(\bB_{\varphi_1})\big)^2$, it follows that $\bar \P$-a.s. $\bar B(\varphi_1)=\bar I_2(\bB_{\varphi_1})$. Since the sub-subsequential limit is the same along any subsequence, this concludes the proof.
\end{proof}

Let us begin with the tree $\<lastone>$.
\begin{lemma}\label{lem:IBP_lastone}
	Let $\{i, j\} \in \{\{1,2\},\{3, 4\}\}$. There exist constants $c= c(\bar \kappa), c' = c'(
	\bar\kappa)$ such that $c(\bar \kappa), c'(\bar\kappa) \to 0$ as $\bar\kappa \to 0$ and that the following identities hold:
	\begin{equs}
		\label{eq:IBP_lastone-1}
		\eps^4\;\phi_{12, ij}\left(\begin{tikzpicture}[scale=0.35,baseline=0.4cm]
			\node at (0,-1)  [root] (root) {};
			\node at (0,1)  [dot] (root2) {};
			\node at (-1.5,2.5)  [dot] (left) {};
			\node at (1.5,2.5)  [dot] (right) {};
			\node at (0,2.5) [var] (variable) {\tiny $3$};
			\node at (-1.5,4) [var] (variablel) {\tiny $2$};
			\node at (1.5,4) [var] (variabler) {\tiny $4$};
			\node at (0,4) [dot] (top) {};
			\node at (0,5.5) [var] (variablet) {\tiny $1$};
			
			\draw[testfcn] (root2) to  (root);
			
			\draw[kernel1] (left) to (root2);
			\draw[kernel1] (right) to (root2);
			\draw[kernel1] (top) to (left);
			\draw[drho] (variable) to (root2); 
			\draw[drho] (variablel) to (left); 
			\draw[drho] (variabler) to (right); 
			\draw[drho] (variablet) to (top); 
		\end{tikzpicture}\right)
		&\;=\;
		\eps^4\;\phi_{12, ij}\left(\begin{tikzpicture}[scale=0.35,baseline=0.4cm]
			\node at (0,-1)  [root] (root) {};
			\node at (0,1)  [dot] (root2) {};
			\node at (-1.5,2.5)  [dot] (left) {};
			\node at (1.5,2.5)  [dot] (right) {};
			\node at (0,2.5) [var] (variable) {\tiny $3$};
			\node at (-1.5,4) [var] (variablel) {\tiny $2$};
			\node at (1.5,4) [var] (variabler) {\tiny $4$};
			\node at (0,4) [dot] (top) {};
			\node at (0,5.5) [var] (variablet) {\tiny $1$};
			
			\draw[testfcn] (root2) to  (root);
			
			\draw[kernel1] (left) to (root2);
			\draw[ddkernel] (right) to (root2);
			\draw[ddkernel] (top) to (left);
			\draw[rho] (variable) to (root2); 
			\draw[rho] (variablel) to (left); 
			\draw[rho] (variabler) to (right); 
			\draw[rho] (variablet) to (top); 
		\end{tikzpicture}\right)
		+ O(\eps^{c(\bar\kappa)} \lambda^{-c'(\bar\kappa)})\\
		\label{eq:IBP_lastone-2}
		\eps^2 \left(
		\begin{tikzpicture}[scale=0.35,baseline=0.4cm]
			\node at (0,-1)  [root] (root) {};
			\node at (0,1)  [dot] (root2) {};
			\node at (-1.5,2.5)  [dot] (left) {};
			\node at (1.5,2.5)  [dot] (right) {};
			\node at (0,4) [dot] (top) {};
			
			\draw[testfcn] (root2) to  (root);
			
			\draw[kernel,bend right=20] (left) to (root);
			\draw[kernel1] (right) to (root2);
			\draw[kernel1] (top) to (left);
			\draw[ddrho] (left) to (root2); 
			\draw[ddrho] (top) to (right);
		\end{tikzpicture}
		\;+\;
		\begin{tikzpicture}[scale=0.35,baseline=0.4cm]
			\node at (0,-1)  [root] (root) {};
			\node at (1.5, 1) [dot] (right1) {};
			\node at (-1.5, 1) [dot] (left1) {};
			\node at (-1.5,3)  [dot] (left2) {};
			\node at (1.5,3)  [dot] (right2) {};
			
			\draw[testfcn] (right1) to  (root);
			
			\draw[kernel1] (right2) to (right1);
			\draw[kernel] (left2) to (left1);
			\draw[kernel] (left1) to (root); 
			\draw[ddrho-shift] (left2) to (right1); 
			\draw[ddrho-shift] (right2) to (left1); 
		\end{tikzpicture}\right)
		&\;=\;
		\eps^2 \left(
		\begin{tikzpicture}[scale=0.35,baseline=0.4cm]
			\node at (0,-1)  [root] (root) {};
			\node at (0,1)  [dot] (root2) {};
			\node at (-1.5,2.5)  [dot] (left) {};
			\node at (1.5,2.5)  [dot] (right) {};
			\node at (0,4) [dot] (top) {};
			
			\draw[testfcn] (root2) to  (root);
			
			\draw[kernel,bend right=20] (left) to (root);
			\draw[ddkernel] (right) to (root2);
			\draw[ddkernel] (top) to (left);
			\draw[rho] (left) to (root2); 
			\draw[rho] (top) to (right);
		\end{tikzpicture}
		\;+\;
		\begin{tikzpicture}[scale=0.35,baseline=0.4cm]
			\node at (0,-1)  [root] (root) {};
			\node at (1.5, 1) [dot] (right1) {};
			\node at (-1.5, 1) [dot] (left1) {};
			\node at (-1.5,3)  [dot] (left2) {};
			\node at (1.5,3)  [dot] (right2) {};
			
			\draw[testfcn] (right1) to  (root);
			
			\draw[ddkernel] (right2) to (right1);
			\draw[ddkernel] (left2) to (left1);
			\draw[kernel] (left1) to (root); 
			\draw[rho] (left2) to (right1); 
			\draw[rho] (right2) to (left1); 
		\end{tikzpicture}\right)
		+ O(\eps^{c(\bar\kappa)} \lambda^{-c'(\bar\kappa)})
	\end{equs}
\end{lemma}
\begin{proof}
	Note that, by performing IBP at every vertex, the Feynman graphs on the left-hand sides of \eqref{eq:IBP_lastone-1}, \eqref{eq:IBP_lastone-2} can be expressed as a linear combinations of graphs without derivatives on the mollifiers. We claim that all such Feynman graphs except for the dominant terms on the right-hand sides of \eqref{eq:IBP_lastone-1}, \eqref{eq:IBP_lastone-2}, have integrals of order $\eps^{c(\bar\kappa)} \lambda^{-c'(\bar\kappa)}$.

	Let us now adapt the IBP argument used in Propositions \ref{prop:cond3} and \ref{prop:cond4}. For a given Feynman graph $G$ appearing on the left-hand sides of \eqref{eq:IBP_lastone-1}-\eqref{eq:IBP_lastone-2}, we call $I: V \setminus\{\origin\} \to E$ a \emph{full IBP map} of $G$ if $I(v)$ is an edge attached to $v$ not of $E_\bM$-type. Given such an IBP map of $G$, define the IBP graph $G_I$ by replacing every $\ddrho$ with $\erho$, replacing every edge $e$ with $|\{v \in V: I(v) = e\}| = 1$ according to the replacement rules of Definition \ref{def:IBP}. Finally, $|\{v \in V: I(v) = e\}| = 2$ can only happen for an edge of type $\kernel$ or $\kernelr$, then this edge is replaced by $\ddkernel$.
	For each subgraph $\bar G$ of $G$ and IBP map $I$, we denote by $\bar G_I$ the counterpart of $\bar G$ in the graph $G_I$.
	
	Now fix a Feynman graph $G$ appearing on the left-hand sides of \eqref{eq:IBP_lastone-1}-\eqref{eq:IBP_lastone-2}. Note that the edge types $\kernelrr$ and $\dkernel$ are absent from these graphs. In the proofs of Propositions \ref{prop:adj}, \ref{prop:cond3} and \ref{prop:cond4}, we noticed that IBP can decrease the degrees $\deg_j(\bar G)$, $j \in \{2, 3, 4\}$, of a subgraph $\bar G$ only if $\bar G$ contains unpaired vertices by $E_\bM$ edges. We therefore deduce the lower bounds
	\begin{equs}[eq:IBP_lower_bounds]
		\deg_2(\bar G_I) &\ge \deg_2(\bar G) - u(\bar G) = \frac12 |\bar V| + \frac12 u(\bar G) - 2\;,\\
		\deg_3(\bar G_I) &\ge \deg_3(\bar G) - u(\tilde G) = \frac12 |\tilde V| + \frac12 u(\tilde G) + n^{\color{Cerulean} \downarrow}\;,\\
		\deg_4(\bar G_I) &\ge \deg_4(\bar G) - u(\bar G) = \frac12 u(\bar G) - \frac12 |\bar V| + n^{\color{Cerulean} \uparrow},
	\end{equs}
	where in each line a relevant subgraph $\bar G$ for the given condition is chosen (recall that in $\deg_3$ we used the notation $\tilde V = \bar V \setminus \{\origin\}$ and $\tilde G$ is the subgraph associated to $\tilde V$). It follows immediately from \eqref{eq:IBP_lower_bounds} that $\deg_3(\bar G_I) > 0$ and $\deg_2(\bar G_I) \ge 0$.
	Moreover \eqref{eq:IBP_lower_bounds} allows $\deg_2(\bar G_I) = 0$ only if $|\bar V|=3$ and $u(\bar G)=1$ or $|\bar V|=4$ and $u(\bar G)=0$.
	In the graphs $G$ considered here, a relevant $\bar G_I$ with $|\bar V|=3$ has $2$ edges, and $u(\bar G)=1$ enforces one of them to be $\erho$. Since the other one has singularity $a_e\leq 2$, one has $\deg_2(\bar G_I)\geq 1$ in this case.
	A subgraph $\bar G_I$ with $|\bar V|=4$ and $u(\bar G)=0$ can have at most $4$ edges, $2$ of which are $\erho$, and therefore $\deg_2(\bar G_I)=0$ if and only if the other two edges are $\ddkernel$ (corresponding to a critical block where all IBP derivatives stay inside the block).

	For IBP graphs $G_I$ not in the form of the dominant terms of \eqref{eq:IBP_lastone-1}, \eqref{eq:IBP_lastone-2}, one can always find at least one $\erho$ edge not covered by this latter case. In fact, due to the structure of Feynman graphs in the statement, there are either $2$ or $4$ such $\erho$ edges. Denote by $\cE^*$ the set of such $\erho$ edges and notice that all subgraphs containing an edge in $\cE^*$ have $\deg_2 > 0$. Now consider the adjusted labelling $a_e^\adj$ given by
	\begin{equ}[eq:a_adj]
		a_e^{\adj}:=
		\begin{cases}
			\bar \kappa^2,	&\text{ if } e\in E_\bK^0,\\
			a_e + \sqrt{\bar \kappa},	&\text{ if } e \in \cE^*,\\
			a_e - \bar \kappa,	&\text{ if } e \in E_\bM \setminus\cE^*,\\
			a_e,		&\text{ otherwise.}
		\end{cases}
	\end{equ}
	For $\bar \kappa$ sufficiently small, one infers that $\deg_2^\adj(\bar G_I) > 0$ and $\deg_3^\adj(\bar G_I) > 0$ for all relevant subgraphs $\bar G_I$.
	
	Finally, it remains to check that $\deg_4^\adj(\bar G_I) > 0$ for all relevant subgraphs $\bar G_I$. By $\deg_4^\adj(\bar G_I) \ge \deg_4(\bar G_I) + \sqrt{\bar \kappa} |\cE^* \cap E(\bar V)| - \bar \kappa |E(\bar V)\setminus \cE^*|$, it is sufficient to check that all subgraphs $\bar G$ of $G$ is in one of the following situations:
	\begin{itemize}
		\item[(a)] $\cE^* \cap E(\bar V) \neq \emptyset$ and $\frac12 u(\bar G) - \frac12 |\bar V| + n^{\color{Cerulean} \uparrow} \ge 0$.
		\item[(b)] $\cE^* \cap E(\bar V) = \emptyset$ and $\frac12 u(\bar G) - \frac12 |\bar V| + n^{\color{Cerulean} \uparrow} > 0$.
	\end{itemize}
	For both Feynman graphs on the left-hand side of \eqref{eq:IBP_lastone-2}, it is straightforward to check that all their subgraphs fall into the case (a): indeed, all of their mollifiers belong to $\cE^*$ and one has
	\begin{itemize}
		\item If $|\bar V| = 1$, then $\frac12 u(\bar G) - \frac12 |\bar V| +  n^{\color{Cerulean} \uparrow} = n^{\color{Cerulean}\uparrow} \ge 0$.
		\item If $|\bar V| = 2$, then either $u(\bar G) = 2$ or $n^{\color{Cerulean} \uparrow} \ge 1$, implying $\frac12 u(\bar G) - \frac12 |\bar V| +  n^{\color{Cerulean} \uparrow} \ge 0$.
		\item If $|\bar V| = 3$, then $u(\bar G) = 1$ and $n^{\color{Cerulean} \uparrow} = 1$, implying $\frac12 u(\bar G) - \frac12 |\bar V| +  n^{\color{Cerulean} \uparrow} = 0$.
	\end{itemize}
	For Feynman graphs on the left-hand side of \eqref{eq:IBP_lastone-1}, one can check
	\begin{itemize}
		\item If $|\bar V| = 1$, then $u(\bar G) = n^{\color{Cerulean} \uparrow} = 1$ and thus $\frac12 u(\bar G) - \frac12 |\bar V| +  n^{\color{Cerulean} \uparrow} = 1$.
		\item If $|\bar V| = 2$, then either $u(\bar G) = 2$, $n^{\color{Cerulean} \uparrow} \ge 1$ or $u(\bar G) = 0$, $n^{\color{Cerulean} \uparrow} = 2$, implying $\frac12 u(\bar G) - \frac12 |\bar V| +  n^{\color{Cerulean} \uparrow} \ge 1$.
		\item If $|\bar V| = 3$, then $u(\bar G) \ge 1$ and $n^{\color{Cerulean} \uparrow} \ge 2$, implying $\frac12 u(\bar G) - \frac12 |\bar V| +  n^{\color{Cerulean} \uparrow} \ge 1$.
		\item If $|\bar V| = 5, 6$, then one must have $\cE^* \cap E(\bar V) \neq \emptyset$ and $n^{\color{Cerulean} \uparrow} \ge 3$, leading to $\frac12 u(\bar G) - \frac12 |\bar V| +  n^{\color{Cerulean} \uparrow} \ge 0$. Thus the case (a) holds.
	\end{itemize}
	The remaining potentially problematic situation is when $|\bar V| = 4$. Indeed, although  $n^{\color{Cerulean} \uparrow} \ge 2$ in this case, it can happen that $u(\bar G) = 0$, $n^{\color{Cerulean} \uparrow} = 2$ and $\cE^* \cap E(\bar V) = \emptyset$, failing thus both the case (a) and (b). However, this occurs only if $\bar G$ is a critical block with two outgoing $\kernelr$: after IBP, the corresponding $\bar G_I$ coincides with one of the following
	\begin{equ}[eq:bad_case_deg4]
		\begin{tikzpicture}[scale=0.35,baseline=0.4cm]
			\node at (1.5, 1) [dot] (right1) {};
			\node at (-1.5, 1) [dot] (left1) {};
			\node at (-1.5,3)  [dot] (left2) {};
			\node at (1.5,3)  [dot] (right2) {};
			\node at (1.5, -1) [] (right) {};
			\node at (-1.5, -1) [] (left) {};
			
			\draw[ddkernel] (right2) to (right1);
			\draw[ddkernel] (left2) to (left1);
			\draw[rho] (left2) to (right2); 
			\draw[rho] (right1) to (left1); 
			\draw[->] (right1) to node[labl,pos=0.45] {\tiny $a_2$, $r_2$} (right);
			\draw[->] (left1) to node[labl,pos=0.45] {\tiny $a_1$, $r_1$} (left);
		\end{tikzpicture}
		\quad \text{or} \quad
		\begin{tikzpicture}[scale=0.35,baseline=0.4cm]
			\node at (1.5, 1) [dot] (right1) {};
			\node at (-1.5, 1) [dot] (left1) {};
			\node at (-1.5,3)  [dot] (left2) {};
			\node at (1.5,3)  [dot] (right2) {};
			\node at (1.5, -1) [] (right) {};
			\node at (-1.5, -1) [] (left) {};
			
			\draw[ddkernel] (right2) to (right1);
			\draw[ddkernel] (left2) to (left1);
			\draw[rho] (left2) to (right1); 
			\draw[rho] (right2) to (left1); 
			\draw[->] (right1) to node[labl,pos=0.45] {\tiny $a_4$, $r_4$} (right);
			\draw[->] (left1) to node[labl,pos=0.45] {\tiny $a_3$, $r_3$} (left);
		\end{tikzpicture}
	\end{equ}
	where, under the assignment $a^\adj_e$, $a_i + r_i$ is either $1$ (if it is a $\dkernel$ edge) or $1+\bar \kappa^2$ (if it is a $\kernelr$ edge), for $i = 1, 2, 3, 4$. It is not hard to check that in all IBP Feynman graphs obtained in the expansion of \eqref{eq:IBP_lastone-1} which contain subgraphs of type \eqref{eq:bad_case_deg4}, its \tikz[baseline=-0.1cm] \draw[->] (0, 0) to node[labl,pos=0.45] {\tiny $a_i$,$r_i$} (1,0); edges can only be part of subgraphs with $\deg_2, \deg_3 \geq 1$. Hence, the remedy to problem is simple: one simply increases artificially each $a_i$ by $2\bar \kappa$. This new assignment is still faithful in the sense of Definition \ref{def:faithful}, and for sufficiently small $\bar \kappa > 0$, all subgraphs $\bar G_I$ have now positive degrees under the new assignment, satisfying the assumptions of Theorem \ref{thm:power_counting}. This therefore completes the proof.
\end{proof}

\begin{lemma}\label{lem:G_conv}
	For all $\varphi \in \Phi$ and all compactly supported function $f \in L^2$ such that $x \mapsto \int f(x - y) f(-y) \dd y$ is continuous, it holds that
	\begin{equs}
		\int &\varphi(x) (G_\eps*f)(x) \dd x \to c_\rho^2 \int f(x) \varphi(x) \dd x \;,\\
		\iint &\varphi(x) \varphi(y) (G_\eps * f)(x - y)^2 \dd x \dd y \to c_\rho^4 \iint \varphi(x) \varphi(y) f(x - y)^2 \dd x \dd y\;,\\
		\iiint &\varphi(x) \varphi(y) G_\eps(x - y) (G_\eps*f)(x - u) f(y - u) \dd u \dd x \dd y \to c_\rho^4 \iint \varphi(x)^2  f(x - u)^2 \dd u \dd x\;\;.
	\end{equs}
\end{lemma}
\begin{proof}
	Note that the integrals on the left-hand sides are well-defined. Since $f \in L^2$, by Lemma \ref{lem:G} one has $G_\eps * f$ converges to $c_\rho^2 f$ in $L^2$. Therefore, the first and the second convergences follow immediately.
	
	To prove the final convergence, let us denote the left-hand side by $I_\eps$ and introduce
	\begin{equ}
		J_\eps(x - y) = \int G_\eps *f (x - u) f(y - u) \dd u\;, \quad
		J_0(x - y) = c_\rho^2 \int f (x - u) f(y - u) \dd u\;.
	\end{equ}
	Note that $\norm{J_\eps - J_0}_\infty \leq \norm{G_\eps *f - c_\rho^2 f}_{L^2} \norm{f}_{L^2}$. For every fixed $\delta > 0$, we can write
	\begin{equs}
		I_\eps &= \iint \varphi(x) \varphi(y) G_\eps(x - y) (J_\eps - J_0)(x - y) \dd x \dd y \\
		&+ \int_x \int_{|y - x| > \delta} \varphi(x) \varphi(y) G_\eps(x - y) J_0(x - y) \dd x \dd y\\
		&+ \int_x \int_{|y - x| \leq \delta} \varphi(x) \varphi(y) G_\eps(x - y) J_0(x - y) \dd x \dd y\;.
	\end{equs}
	By the uniform convergence of $J_\eps - J_0$ as well as Properties 1 and 3 of Lemma \ref{lem:G}, the first and second term of above vanish as $\eps \to 0$. Furthermore, notice that the function $(x, y) \mapsto \varphi(x) \varphi(y) J_0(x - y)$ is uniformly continuous. Therefore, by fixing arbitrarily $\eta > 0$, one can find $\delta > 0$ sufficiently small such that $|\varphi(x) \varphi(y) J_0(x - y) - \varphi(x)^2 J_0(0)| \leq \eta$ for all $|y-x| \leq \delta$. The third term can then be written as
	\begin{equ}
		\int_x \varphi(x)^2 J_0(0)\dd x \int_{|z| \leq \delta} G_\eps(z) \dd z + O(\eta)\;.
	\end{equ}
	By sending $\eps \to 0$ with Lemma \ref{lem:G} Properties 2 and 3, and subsequently sending $\eta \to 0$, the claim is then proved.
\end{proof}

In the sequel, we pose
\begin{equ}[eq:A]
	A_\eps (\varphi) := \eps\;
	\begin{tikzpicture}[scale=0.35,baseline=0.3cm]
		\node at (0,-1)  [root] (root) {};
		\node at (-2,3)  [dot] (left2) {};
		\node at (-2,1)  [dot] (left1) {};
		\node at (0,1) [var] (variable1) {};
		\node at (0,3) [var] (variable2) {};
		
		\draw[testfcn] (left1) to  (root);
		
		\draw[rho] (variable2) to (left2);
		\draw[rho] (variable1) to (left1);
		\draw[ddkernel] (left2) to (left1);
	\end{tikzpicture}
\end{equ}
which converges in law to a white noise tested against the test function $\varphi \in \Phi$, as we have seen in Section \ref{sec:new_noise}.

\begin{proposition}\label{prop:lastone}
	The random field $\hat\Pi^{\eps}_0 \<lastone>$ satisfies the bound \eqref{eq:second_moment_model}, that is,
	\begin{equ}
		\E\left(\frac{\crochet{\hat\Pi^{\eps}_0 \<lastone>, \varphi^\lambda}}{\lambda^{-4\kappa + \kappa'}}\right)^2 \lesssim 1\;
	\end{equ}
	for some $\kappa' > 0$ and uniformly for $\varphi \in \Phi$, $\lambda \in (0, 1]$ and $\eps > 0$. Moreover, it holds that 
	\begin{equs}
		\E\left(\hat\Pi^{\eps}_0 \<lastone> (\varphi)\right) &\to -c_\rho^2 \begin{tikzpicture}[scale=0.35,baseline=0.9cm]
			\node at (0,1)  [root] (root) {};
			\node at (0,4)  [dot] (top) {};
			
			\draw[testfcn,bend right = 60] (top) to  (root);
			
			\draw[kernel,bend left = 60] (top) to (root);
		\end{tikzpicture}\;,\label{eq:conv_lastone1}
		\\
		\E\left(\hat\Pi^{\eps}_0 \<lastone> (\varphi) - \E\left(\hat\Pi^{\eps}_0 \<lastone> (\varphi)\right)\right)^2 &\to c_\rho^4 \left(\;
		\begin{tikzpicture}[scale=0.35,baseline=0.9cm]
			\node at (0,1)  [root] (root) {};
			\node at (0,3)  [dot] (down) {};
			\node at (0,5)  [dot] (top) {};
			
			\draw[testfcn, bend right = 30] (down) to  (root);
			\draw[testfcn, bend left = 30] (down) to  (root);
			
			\draw[kernel1, bend left = 30] (top) to (down);
			\draw[kernel1, bend right = 30] (top) to (down);
		\end{tikzpicture}
		\;+\;
		\begin{tikzpicture}[scale=0.35,baseline=0.9cm]
			\node at (0,1)  [root] (root) {};
			\node at (-1,4)  [dot] (left) {};
			\node at (1,4)  [dot] (right) {};
			
			\draw[testfcn] (left) to  (root);
			\draw[testfcn] (right) to  (root);
			
			\draw[kernel1, bend left = 30] (left) to (right);
			\draw[kernel1, bend left = 30] (right) to (left);
		\end{tikzpicture}\right)\;,\label{eq:conv_lastone}
		\\
		\E\left(\hat\Pi^{\eps}_0 \<lastone> (\varphi) A_\eps (\varphi) A_\eps (\varphi)\right) &\to c_\rho^4 \left(2\;
		\begin{tikzpicture}[scale=0.35,baseline=0.9cm]
			\node at (0,1)  [root] (root) {};
			\node at (0,3)  [dot] (down) {};
			\node at (0,5)  [dot] (top) {};
			
			\draw[testfcn, bend right = 30] (down) to  (root);
			\draw[testfcn, bend left = 30] (down) to  (root);
			\draw[testfcn, bend left = 60] (top) to  (root);
			
			\draw[kernel1] (top) to (down);
		\end{tikzpicture}
		\;-\;
		\begin{tikzpicture}[scale=0.35,baseline=0.9cm]
			\node at (0,1)  [root] (root) {};
			\node at (0,4)  [dot] (top) {};
			
			\draw[testfcn,bend right = 60] (top) to  (root);
			
			\draw[kernel,bend left = 60] (top) to (root);
		\end{tikzpicture}
		\begin{tikzpicture}[scale=0.35,baseline=0.9cm]
			\node at (0,1)  [root] (root) {};
			\node at (0,4)  [dot] (top) {};
			
			\draw[testfcn,bend right = 60] (top) to  (root);
			\draw[testfcn,bend left = 60] (top) to  (root);
		\end{tikzpicture}
		\right).\label{eq:conv_lastone3}
	\end{equs}
\end{proposition}
\begin{proof}
	We focus on the tightness and the convergences \eqref{eq:conv_lastone1} and \eqref{eq:conv_lastone}. For \eqref{eq:conv_lastone3}, one can express the left-hand side as $\E\big((I^{(4)}+ I^{(2)} + I^{(0)})(\varphi^\lambda) A_\eps(\varphi^\lambda) A_\eps(\varphi^\lambda)\big)$, where $I^{(j)}(\varphi^\lambda)$ denotes the projection of $\hat\Pi^\eps_0 \<lastone> (\varphi^\lambda)$ onto the $j$-th homogeneous chaos.
The convergence of 	$\E\big(I^{(0)}(\varphi^\lambda) A_\eps(\varphi^\lambda) A_\eps(\varphi^\lambda)\big)$ to the second term on the right-hand side of \eqref{eq:conv_lastone3} follows from \eqref{eq:conv_lastone1} and Section \ref{sec:new_noise}.
The variance of $I^{(2)}(\varphi^\lambda)$ vanishes in the limit by Corollary \ref{cor:graph}, Proposition \ref{prop:vanishing}.
As for the convergence of $\E\big(I^{(4)}(\varphi^\lambda) A_\eps(\varphi^\lambda) A_\eps(\varphi^\lambda)\big)$
to the first term on the right-hand side of \eqref{eq:conv_lastone3}, it follows from a similar (and simpler) calculation as the proof of \eqref{eq:conv_lastone}, and the latter we detail below.
	
	By \eqref{eq:lastone}, Corollary \ref{cor:graph}, Proposition \ref{prop:vanishing} and Lemma \ref{lem:IBP_lastone}, there exist $c(\bar\kappa), c'(\bar \kappa) \in (0, 1)$ such that $c(\bar\kappa), c'(\bar\kappa) \to 0$ as $\bar \kappa \to 0$ and that
	\begin{equs}
		\E\left(\hat\Pi^{\eps}_0 \<lastone> (\varphi^\lambda)\right) &= - \eps^2 \left(
		\begin{tikzpicture}[scale=0.35,baseline=0.4cm]
			\node at (0,-1)  [root] (root) {};
			\node at (0,1)  [dot] (root2) {};
			\node at (-1.5,2.5)  [dot] (left) {};
			\node at (1.5,2.5)  [dot] (right) {};
			\node at (0,4) [dot] (top) {};
			
			\draw[testfcn] (root2) to  (root);
			
			\draw[kernel,bend right=20] (left) to (root);
			\draw[ddkernel] (right) to (root2);
			\draw[ddkernel] (top) to (left);
			\draw[rho] (left) to (root2); 
			\draw[rho] (top) to (right);
		\end{tikzpicture}
		\;+\;
		\begin{tikzpicture}[scale=0.35,baseline=0.4cm]
			\node at (0,-1)  [root] (root) {};
			\node at (1.5, 1) [dot] (right1) {};
			\node at (-1.5, 1) [dot] (left1) {};
			\node at (-1.5,3)  [dot] (left2) {};
			\node at (1.5,3)  [dot] (right2) {};
			
			\draw[testfcn] (right1) to  (root);
			
			\draw[ddkernel] (right2) to (right1);
			\draw[ddkernel] (left2) to (left1);
			\draw[kernel] (left1) to (root); 
			\draw[rho] (left2) to (right1); 
			\draw[rho] (right2) to (left1); 
		\end{tikzpicture}\right) + O(\eps^{c(\bar\kappa)}\lambda^{-c'(\bar\kappa)})\\
		\E\left(\hat\Pi^{\eps}_0 \<lastone> (\varphi^\lambda) - \E\left(\hat\Pi^{\eps}_0 \<lastone> (\varphi^\lambda)\right)\right)^2 &= \eps^4 \sum_{\sigma} \phi_\sigma\left(
		\begin{tikzpicture}[scale=0.35,baseline=0.4cm]
			\node at (0,-1)  [root] (root) {};
			\node at (0,1)  [dot] (root2) {};
			\node at (-1.5,2.5)  [dot] (left) {};
			\node at (1.5,2.5)  [dot] (right) {};
			\node at (0,2.5) [var] (variable) {\tiny 3};
			\node at (-1.5,4) [var] (variablel) {\tiny 2};
			\node at (1.5,4) [var] (variabler) {\tiny 4};
			\node at (0,4) [dot] (top) {};
			\node at (0,5.5) [var] (variablet) {\tiny 1};
			
			\draw[testfcn] (root2) to  (root);
			
			\draw[kernel1] (left) to (root2);
			\draw[ddkernel] (right) to (root2);
			\draw[ddkernel] (top) to (left);
			\draw[rho] (variable) to (root2); 
			\draw[rho] (variablel) to (left); 
			\draw[rho] (variabler) to (right); 
			\draw[rho] (variablet) to (top); 
		\end{tikzpicture}\right)
		+ O(\eps^{c(\bar\kappa)}\lambda^{-c'(\bar\kappa)})\;,
	\end{equs}
	where the sum runs over the four possible permutations $\sigma$ of $\{1, 2, 3, 4\}$ such that $\sigma(\{1, 2\}) = \{1, 2\}$ or $\{3, 4\}$. 
	
	Notice that the dominant terms in the expectation sum up to
	\[\iint G_\eps (x - y) K(-y) \varphi^\lambda(x) \dd x \dd y \lesssim \lambda^{-\bar \kappa}\]
	where $G_\eps$ is defined in \eqref{eq:G} and we used the fact $|G_\eps(x)| \lesssim \eps^2 (|x| + \eps)^{-4}$ uniformly for $\eps$. By Lemma \ref{lem:G_conv}, it follows that this integral converges to
	\begin{equ}
		c_\rho^2 \int K(-x)\varphi(x) \dd x
	\end{equ}
	as $\eps \to 0$, which coincides with the desired limit.
	
	On the other hand, the dominant terms in the variance sum up to
	
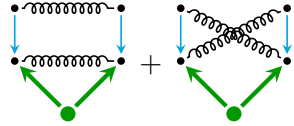
\begin{equ}
		\begin{tikzpicture}[scale=0.35,baseline=0.9cm]
			\node at (0,1)  [root] (root) {};
			\node at (-2,3)  [dot] (left1) {};
			\node at (-2,5)  [dot] (left2) {};
			\node at (2,3)  [dot] (right1) {};
			\node at (2,5)  [dot] (right2) {};
			
			\draw[testfcn] (left1) to  (root);
			\draw[testfcn] (right1) to  (root);
			
			\draw[kernel1] (left2) to (left1);
			\draw[kernel1] (right2) to (right1);
			
			\draw[BigG] (left2) to (right2);
			\draw[BigG] (left1) to (right1);
		\end{tikzpicture}
		\;+\;
		\begin{tikzpicture}[scale=0.35,baseline=0.9cm]
			\node at (0,1)  [root] (root) {};
			\node at (-2,3)  [dot] (left1) {};
			\node at (-2,5)  [dot] (left2) {};
			\node at (2,3)  [dot] (right1) {};
			\node at (2,5)  [dot] (right2) {};
			
			\draw[testfcn] (left1) to  (root);
			\draw[testfcn] (right1) to  (root);
			
			\draw[kernel1] (left2) to (left1);
			\draw[kernel1] (right2) to (right1);
			
			\draw[BigG] (left2) to (right1);
			\draw[BigG] (left1) to (right2);
		\end{tikzpicture}\;
	\end{equ}
	where $\BigG$ represents the kernel $G_\eps$ \eqref{eq:G}. Note that one can label the kernel $G_\eps$ by $(2, -1)$ with $I_{e, (0, 0)} = \int G_\eps$. Therefore, one gets
	
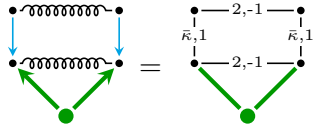
\begin{equ}
		\begin{tikzpicture}[scale=0.35,baseline=0.9cm]
			\node at (0,1)  [root] (root) {};
			\node at (-2,3)  [dot] (left1) {};
			\node at (-2,5)  [dot] (left2) {};
			\node at (2,3)  [dot] (right1) {};
			\node at (2,5)  [dot] (right2) {};
			
			\draw[testfcn] (left1) to  (root);
			\draw[testfcn] (right1) to  (root);
			
			\draw[kernel1] (left2) to (left1);
			\draw[kernel1] (right2) to (right1);
			
			\draw[BigG] (left2) to (right2);
			\draw[BigG] (left1) to (right1);
		\end{tikzpicture}
		\;=\;
		\begin{tikzpicture}[scale=0.35,baseline=0.9cm]
			\node at (0,1)  [root] (root) {};
			\node at (-2,3)  [dot] (left1) {};
			\node at (-2,5)  [dot] (left2) {};
			\node at (2,3)  [dot] (right1) {};
			\node at (2,5)  [dot] (right2) {};
			
			\draw[dist] (left1) to  (root);
			\draw[dist] (right1) to  (root);
			
			\draw[generic] (left2) to node[labl,pos=0.5] {\tiny $\bar \kappa$,1} (left1);
			\draw[generic] (right2) to node[labl,pos=0.5] {\tiny $\bar \kappa$,1} (right1);
			
			\draw[generic] (left2) to node[labl,pos=0.5] {\tiny 2,-1} (right2);
			\draw[generic] (left1) to node[labl,pos=0.5] {\tiny 2,-1} (right1);
		\end{tikzpicture}
	\end{equ}
	and deduces by Theorem \ref{thm:power_counting} that the integral associated to the labelled graph is uniformly of order $\lambda^{-2\bar\kappa}$. The second Feynman graph in the dominant term can be estimated in exactly the same manner.
	Finally, we see by Lemma \ref{lem:G_conv} that the sum of the two dominant graphs converges to
	\begin{equ}
		c_\rho^4 \left(\iint K^{(1)}(y, x) K^{(1)}(x, y) \varphi(x) \varphi(y) + \iint K^{(1)}(y, x) K^{(1)}(y, x) \varphi(y)^2 \right)\;.
	\end{equ}
	Summing the contribution of each term, we get the desired limit.
\end{proof}

Let us now turn to the tree $\<XiIXiIXiIXi>$.
\begin{lemma}\label{lem:IBP_XiIXiIXiIXi}
	Let $\{i, j\} \in \{\{1,2\},\{3, 4\}\}$. There exist constants $c= c(\bar \kappa), c' = c'(
	\bar\kappa)$ such that $c(\bar \kappa), c'(\bar\kappa) \to 0$ as $\bar\kappa \to 0$ and that the following identities hold:
	\begin{equs}
		\label{eq:IBP_XiIXiIXiIXi-1}
		\eps^4\; \phi_{12,ij} \left(\begin{tikzpicture}[scale=0.35,baseline=0.9cm]
			\node at (0,-1)  [root] (root) {};
			\node at (-2,1)  [dot] (left) {};
			\node at (-2,3)  [dot] (left1) {};
			\node at (-2,5)  [dot] (left2) {};
			\node at (-2,7)  [dot] (left3) {};
			\node at (0,1) [var] (variable1) {\tiny$4$};
			\node at (0,3) [var] (variable2) {\tiny$3$};
			\node at (0,5) [var] (variable3) {\tiny$2$};
			\node at (0,7) [var] (variable4) {\tiny$1$};
			
			\draw[testfcn] (left) to  (root);
			
			\draw[kernel2] (left1) to (left);
			\draw[kernel1] (left2) to (left1);
			\draw[kernel1] (left3) to (left2);
			\draw[drho] (variable4) to (left3);
			\draw[drho] (variable3) to (left2); 
			\draw[drho] (variable2) to (left1); 
			\draw[drho] (variable1) to (left); 
		\end{tikzpicture}\right)
		&\;=\;
		\eps^4\; \phi_{12,ij} \left(\begin{tikzpicture}[scale=0.35,baseline=0.9cm]
			\node at (0,-1)  [root] (root) {};
			\node at (-2,1)  [dot] (left) {};
			\node at (-2,3)  [dot] (left1) {};
			\node at (-2,5)  [dot] (left2) {};
			\node at (-2,7)  [dot] (left3) {};
			\node at (0,1) [var] (variable1) {\tiny$4$};
			\node at (0,3) [var] (variable2) {\tiny$3$};
			\node at (0,5) [var] (variable3) {\tiny$2$};
			\node at (0,7) [var] (variable4) {\tiny$1$};
			
			\draw[testfcn] (left) to  (root);
			
			\draw[ddkernel] (left1) to (left);
			\draw[kernel1] (left2) to (left1);
			\draw[ddkernel] (left3) to (left2);
			\draw[rho] (variable4) to (left3);
			\draw[rho] (variable3) to (left2); 
			\draw[rho] (variable2) to (left1); 
			\draw[rho] (variable1) to (left); 
		\end{tikzpicture}\right)
		+ O(\eps^{c(\bar\kappa)} \lambda^{-c'(\bar\kappa)})\\
		\label{eq:IBP_XiIXiIXiIXi-2}
		\eps^2\;\left(\;
		\begin{tikzpicture}[scale=0.35,baseline=0.9cm]
			\node at (0,-1)  [root] (root) {};
			\node at (-2,1)  [dot] (left) {};
			\node at (-2,3)  [dot] (left1) {};
			\node at (0,7)  [dot] (right1) {};
			\node at (0,5) [dot] (right) {};
			
			\draw[testfcn] (left) to  (root);
			
			\draw[kernel] (left1) to (left);
			\draw[kernel1] (right1) to (right);
			\draw[kernel] (right) to (root);
			\draw[ddrho] (left1) to (right1);
			\draw[ddrho] (right) to (left);
		\end{tikzpicture}
		\;+\;
		\begin{tikzpicture}[scale=0.35,baseline=0.9cm]
			\node at (0,-1)  [root] (root) {};
			\node at (-2,1)  [dot] (left) {};
			\node at (-2,3)  [dot] (left1) {};
			\node at (0,7)  [dot] (right1) {};
			\node at (0,5) [dot] (right) {};
			
			\draw[testfcn] (left) to  (root);
			
			\draw[kernel2] (left1) to (left);
			\draw[kernel1] (right1) to (right);
			\draw[kernel] (right) to (root);
			\draw[ddrho-shift] (left1) to (right);
			\draw[ddrho-shift] (right1) to (left);
		\end{tikzpicture}\right)
		&\;=\;
		\eps^2\;\left(\;
		\begin{tikzpicture}[scale=0.35,baseline=0.9cm]
			\node at (0,-1)  [root] (root) {};
			\node at (-2,1)  [dot] (left) {};
			\node at (-2,3)  [dot] (left1) {};
			\node at (0,7)  [dot] (right1) {};
			\node at (0,5) [dot] (right) {};
			
			\draw[testfcn] (left) to  (root);
			
			\draw[ddkernel] (left1) to (left);
			\draw[ddkernel] (right1) to (right);
			\draw[kernel] (right) to (root);
			\draw[rho] (left1) to (right1);
			\draw[rho] (right) to (left);
		\end{tikzpicture}
		\;+\;
		\begin{tikzpicture}[scale=0.35,baseline=0.9cm]
			\node at (0,-1)  [root] (root) {};
			\node at (-2,1)  [dot] (left) {};
			\node at (-2,3)  [dot] (left1) {};
			\node at (0,7)  [dot] (right1) {};
			\node at (0,5) [dot] (right) {};
			
			\draw[testfcn] (left) to  (root);
			
			\draw[ddkernel] (left1) to (left);
			\draw[ddkernel] (right1) to (right);
			\draw[kernel] (right) to (root);
			\draw[rho] (left1) to (right);
			\draw[rho] (right1) to (left);
		\end{tikzpicture}\right)
		+ O(\eps^{c(\bar\kappa)} \lambda^{-c'(\bar\kappa)})
	\end{equs}
\end{lemma}

\begin{proof}
	For \eqref{eq:IBP_XiIXiIXiIXi-2}, let us write
	\begin{equs}
		\begin{tikzpicture}[scale=0.35,baseline=0.9cm]
			\node at (0,-1)  [root] (root) {};
			\node at (-2,1)  [dot] (left) {};
			\node at (-2,3)  [dot] (left1) {};
			\node at (0,7)  [dot] (right1) {};
			\node at (0,5) [dot] (right) {};
			
			\draw[testfcn] (left) to  (root);
			
			\draw[kernel] (left1) to (left);
			\draw[kernel1] (right1) to (right);
			\draw[kernel] (right) to (root);
			\draw[ddrho] (left1) to (right1);
			\draw[ddrho] (right) to (left);
		\end{tikzpicture}
		&\;=\;
		\begin{tikzpicture}[scale=0.35,baseline=0.9cm]
			\node at (0,-1)  [root] (root) {};
			\node at (-2,1)  [dot] (left) {};
			\node at (-2,3)  [dot] (left1) {};
			\node at (0,7)  [dot] (right1) {};
			\node at (0,5) [dot] (right) {};
			
			\draw[testfcn] (left) to  (root);
			
			\draw[kernel1] (left1) to (left);
			\draw[kernel1] (right1) to (right);
			\draw[kernel] (right) to (root);
			\draw[ddrho] (left1) to (right1);
			\draw[ddrho] (right) to (left);
		\end{tikzpicture}
		\;+\;
		\begin{tikzpicture}[scale=0.35,baseline=0.9cm]
			\node at (0,-1)  [root] (root) {};
			\node at (-2,1)  [dot] (left) {};
			\node at (-2,3)  [dot] (left1) {};
			\node at (0,7)  [dot] (right1) {};
			\node at (0,5) [dot] (right) {};
			
			\draw[testfcn] (left) to  (root);
			
			\draw[kernel] (left1) to (root);
			\draw[kernel1] (right1) to (right);
			\draw[kernel] (right) to (root);
			\draw[ddrho] (left1) to (right1);
			\draw[ddrho] (right) to (left);
		\end{tikzpicture}\;, \qquad
		\begin{tikzpicture}[scale=0.35,baseline=0.9cm]
			\node at (0,-1)  [root] (root) {};
			\node at (-2,1)  [dot] (left) {};
			\node at (-2,3)  [dot] (left1) {};
			\node at (0,7)  [dot] (right1) {};
			\node at (0,5) [dot] (right) {};
			
			\draw[testfcn] (left) to  (root);
			
			\draw[kernel2] (left1) to (left);
			\draw[kernel1] (right1) to (right);
			\draw[kernel] (right) to (root);
			\draw[ddrho-shift] (left1) to (right);
			\draw[ddrho-shift] (right1) to (left);
		\end{tikzpicture}
		\;=\;
		\begin{tikzpicture}[scale=0.35,baseline=0.9cm]
			\node at (0,-1)  [root] (root) {};
			\node at (-2,1)  [dot] (left) {};
			\node at (-2,3)  [dot] (left1) {};
			\node at (0,7)  [dot] (right1) {};
			\node at (0,5) [dot] (right) {};
			
			\draw[testfcn] (left) to  (root);
			
			\draw[kernel1] (left1) to (left);
			\draw[kernel1] (right1) to (right);
			\draw[kernel] (right) to (root);
			\draw[ddrho-shift] (left1) to (right);
			\draw[ddrho-shift] (right1) to (left);
		\end{tikzpicture}
		\;- \sum_{j = 1, 2}\;
		\begin{tikzpicture}[scale=0.35,baseline=0.9cm]
			\node at (0,-1)  [root] (root) {};
			\node at (-2,1)  [dot] (left) {};
			\node at (-2,3)  [dot] (left1) {};
			\node at (0,7)  [dot] (right1) {};
			\node at (0,5) [dot] (right) {};
			
			\draw[testfcnx] (left) to node[below]{\tiny\color{black}$j$} (root);
			
			\draw[dkernel] (left1) to node[right]{\tiny\color{black}$j$} (root);
			\draw[kernel1] (right1) to (right);
			\draw[kernel] (right) to (root);
			\draw[ddrho-shift] (left1) to (right);
			\draw[ddrho-shift] (right1) to (left);
		\end{tikzpicture}\;.
	\end{equs}
	Notice that the first term of the both equalities can be analysed in exactly the same way as in Lemma \ref{lem:IBP_lastone}. We only comment on the fact that it is important for the kernel edge attached to $\testfcn$ to be recentred. As for the second term of the both equalities, it follows easily from \cite[Lem.~10.14]{H0} that, after being multiplied by the prefactor $\eps^2$, they are of order $O(\eps^{c(\bar\kappa)} \lambda^{-c'(\bar\kappa)})$ for some $c(\bar\kappa), c'(\bar\kappa) \to 0$ as $\bar \kappa \to 0$.
	
	For \eqref{eq:IBP_XiIXiIXiIXi-1}, one can check
	\begin{equs}[eq:XiIXiIXiIXi_expand]
		\begin{tikzpicture}[scale=0.35,baseline=1cm]
			\node at (0,-1)  [root] (root) {};
			\node at (-2,1)  [dot] (left) {};
			\node at (-2,3)  [dot] (left1) {};
			\node at (-2,5)  [dot] (left2) {};
			\node at (-2,7)  [dot] (left3) {};
			\node at (0,1) [var] (variable1) {\tiny$4$};
			\node at (0,3) [var] (variable2) {\tiny$3$};
			\node at (0,5) [var] (variable3) {\tiny$2$};
			\node at (0,7) [var] (variable4) {\tiny$1$};
			
			\draw[testfcn] (left) to  (root);
			
			\draw[kernel2] (left1) to (left);
			\draw[kernel1] (left2) to (left1);
			\draw[kernel1] (left3) to (left2);
			\draw[drho] (variable4) to (left3);
			\draw[drho] (variable3) to (left2); 
			\draw[drho] (variable2) to (left1); 
			\draw[drho] (variable1) to (left);  
		\end{tikzpicture}
		&\;=\;
		\begin{tikzpicture}[scale=0.35,baseline=1cm]
			\node at (0,-1)  [root] (root) {};
			\node at (-2,1)  [dot] (left) {};
			\node at (-2,3)  [dot] (left1) {};
			\node at (-2,5)  [dot] (left2) {};
			\node at (-2,7)  [dot] (left3) {};
			\node at (0,1) [var] (variable1) {\tiny$4$};
			\node at (0,3) [var] (variable2) {\tiny$3$};
			\node at (0,5) [var] (variable3) {\tiny$2$};
			\node at (0,7) [var] (variable4) {\tiny$1$};
			
			\draw[testfcn] (left) to  (root);
			
			\draw[ddkernel] (left1) to (left);
			\draw[kernel1] (left2) to (left1);
			\draw[ddkernel] (left3) to (left2);
			\draw[rho] (variable4) to (left3);
			\draw[rho] (variable3) to (left2); 
			\draw[rho] (variable2) to (left1); 
			\draw[rho] (variable1) to (left); 
		\end{tikzpicture}
		\;+\;
		\begin{tikzpicture}[scale=0.35,baseline=1cm]
			\node at (-2,0)  [root] (root) {};
			\node at (0,1.5) 	 [dot] (left) {};
			\node at (-2,3)  [dot] (left1) {};
			\node at (-2,5)  [dot] (left2) {};
			\node at (-2,7)  [dot] (left3) {};
			\node at (2,1.5) [var] (variable1) {\tiny$4$};
			\node at (0,3) [var] (variable2) {\tiny$3$};
			\node at (0,5) [var] (variable3) {\tiny$2$};
			\node at (0,7) [var] (variable4) {\tiny$1$};
			
			\draw[testfcn] (left) to (root);
			
			\draw[dkernel] (left1) to (root);
			\draw[dkernel] (left2) to (left1);
			\draw[ddkernel] (left3) to (left2);
			\draw[rho] (variable4) to (left3);
			\draw[rho] (variable3) to (left2); 
			\draw[rho] (variable2) to (left1); 
			\draw[rho] (variable1) to (left); 
		\end{tikzpicture}
		\;-\;
		\begin{tikzpicture}[scale=0.35,baseline=1cm]
			\node at (-2,0)  [root] (root) {};
			\node at (0,1.5) [dot] (left) {};
			\node at (-2,3)  [dot] (left1) {};
			\node at (-2,5)  [dot] (left2) {};
			\node at (-2,7)  [dot] (left3) {};
			\node at (2,1.5) [var] (variable1) {\tiny$4$};
			\node at (0,3) [var] (variable2) {\tiny$3$};
			\node at (0,5) [var] (variable3) {\tiny$2$};
			\node at (0,7) [var] (variable4) {\tiny$1$};
			
			\draw[testfcn] (left) to (root);
			
			\draw[dkernel] (left1) to (root);
			\draw[kernel1] (left2) to (left1);
			\draw[ddkernel] (left3) to (left2);
			\draw[rho] (variable4) to (left3);
			\draw[rho] (variable3) to (left2); 
			\draw[drho] (variable2) to (left1); 
			\draw[rho] (variable1) to (left); 
		\end{tikzpicture}
		\;+\;
		\begin{tikzpicture}[scale=0.35,baseline=1cm]
			\node at (0,-1)  [root] (root) {};
			\node at (-2,1)  [dot] (left) {};
			\node at (-2,3)  [dot] (left1) {};
			\node at (-2,5)  [dot] (left2) {};
			\node at (-2,7)  [dot] (left3) {};
			\node at (0,1) [var] (variable1) {\tiny$4$};
			\node at (0,3) [var] (variable2) {\tiny$3$};
			\node at (0,5) [var] (variable3) {\tiny$2$};
			\node at (0,7) [var] (variable4) {\tiny$1$};
			
			\draw[dtestfcn] (left) to (root);
			
			\draw[dkernel2] (left1) to (left);
			\draw[kernel1] (left2) to (left1);
			\draw[ddkernel] (left3) to (left2);
			\draw[rho] (variable4) to (left3);
			\draw[rho] (variable3) to (left2); 
			\draw[rho] (variable2) to (left1); 
			\draw[rho] (variable1) to (left); 
		\end{tikzpicture}
		\;+\;
		\begin{tikzpicture}[scale=0.35,baseline=1cm]
			\node at (0,-1)  [root] (root) {};
			\node at (-2,1)  [dot] (left) {};
			\node at (-2,3)  [dot] (left1) {};
			\node at (-2,5)  [dot] (left2) {};
			\node at (-2,7)  [dot] (left3) {};
			\node at (0,1) [var] (variable1) {\tiny$4$};
			\node at (0,3) [var] (variable2) {\tiny$3$};
			\node at (0,5) [var] (variable3) {\tiny$2$};
			\node at (0,7) [var] (variable4) {\tiny$1$};
			
			\draw[testfcn] (left) to  (root);
			
			\draw[kernel2] (left1) to (left);
			\draw[dkernel] (left2) to (left1);
			\draw[ddkernel] (left3) to (left2);
			\draw[rho] (variable4) to (left3);
			\draw[rho] (variable3) to (left2); 
			\draw[rho] (variable2) to (left1); 
			\draw[drho] (variable1) to (left); 
		\end{tikzpicture}
		\;+\;
		\begin{tikzpicture}[scale=0.35,baseline=1cm]
			\node at (0,-1)  [root] (root) {};
			\node at (-2,1)  [dot] (left) {};
			\node at (-2,3)  [dot] (left1) {};
			\node at (-2,5)  [dot] (left2) {};
			\node at (-2,7)  [dot] (left3) {};
			\node at (0,1) [var] (variable1) {\tiny$4$};
			\node at (0,3) [var] (variable2) {\tiny$3$};
			\node at (0,5) [var] (variable3) {\tiny$2$};
			\node at (0,7) [var] (variable4) {\tiny$1$};
			
			\draw[testfcn] (left) to  (root);
			
			\draw[kernel2] (left1) to (left);
			\draw[dkernel1] (left2) to (left1);
			\draw[dkernel1] (left3) to (left2);
			\draw[rho] (variable4) to (left3);
			\draw[rho] (variable3) to (left2); 
			\draw[drho] (variable2) to (left1); 
			\draw[drho] (variable1) to (left); 
		\end{tikzpicture}\\
		&= S_1 + S_2 + S_3 + S_4 + S_5 + S_6\;,
	\end{equs}
	where we introduced the shorthand $S_k$ by numbering the stochastic graphs in the right-hand side of \eqref{eq:XiIXiIXiIXi_expand} from left to right. For $S_k, S_\ell \in \{S_1, \dots, S_6\}$ and any permutations $\sigma$ on $\{1, 2, 3, 4\}$, we define the Feynman graph $\phi_\sigma(S_k, S_\ell)$ by pairing the $n$-th noise node of $S_k$ to $\sigma(n)$-th noise node of $S_\ell$. We are in particular interested in permutation $\sigma$ verifying $\sigma(\{1, 2\}) = \{i, j\}$. We claim that, except for $\phi_{\sigma}(S_1, S_1)$, all other pairings satisfy $\bI(\eps^4 \phi_\sigma(S_k, S_\ell), \varphi^\lambda) \lesssim \eps^{c(\bar\kappa)} \lambda^{-c'(\bar \kappa)}$, as required by the statement. Notice that $\phi_{\sigma}(S_1, S_1)$ produces the dominant term in \eqref{eq:IBP_XiIXiIXiIXi-1}. From now on, let us consider the graphs formed by pairing $S_k$ and $S_\ell$ with $k, \ell$ not both equal to $1$.
	
	Note that the all Feynman graphs $\phi_\sigma(S_k, S_\ell)$ described above can be obtained by performing IBP at some of the vertices of Feynman graphs resulted from pairing two (not necessarily distinct) copies of
	\begin{equ}\label{eq:ladder-graphs}
		\begin{tikzpicture}[scale=0.35,baseline=1cm]
			\node at (0,-1)  [root] (root) {};
			\node at (-2,1)  [dot] (left) {};
			\node at (-2,3)  [dot] (left1) {};
			\node at (-2,5)  [dot] (left2) {};
			\node at (-2,7)  [dot] (left3) {};
			\node at (0,1) [var] (variable1) {\tiny$4$};
			\node at (0,3) [var] (variable2) {\tiny$3$};
			\node at (0,5) [var] (variable3) {\tiny$2$};
			\node at (0,7) [var] (variable4) {\tiny$1$};
			
			\draw[testfcn] (left) to  (root);
			
			\draw[kernel2] (left1) to (left);
			\draw[kernel1] (left2) to (left1);
			\draw[kernel1] (left3) to (left2);
			\draw[drho] (variable4) to (left3);
			\draw[drho] (variable3) to (left2); 
			\draw[drho] (variable2) to (left1); 
			\draw[drho] (variable1) to (left);  
		\end{tikzpicture}\;,
		\qquad
		\begin{tikzpicture}[scale=0.35,baseline=1cm]
			\node at (-2,0)  [root] (root) {};
			\node at (0,1.5) 	 [dot] (left) {};
			\node at (-2,3)  [dot] (left1) {};
			\node at (-2,5)  [dot] (left2) {};
			\node at (-2,7)  [dot] (left3) {};
			\node at (2,1.5) [var] (variable1) {\tiny$4$};
			\node at (0,3) [var] (variable2) {\tiny$3$};
			\node at (0,5) [var] (variable3) {\tiny$2$};
			\node at (0,7) [var] (variable4) {\tiny$1$};
			
			\draw[testfcn] (left) to (root);
			
			\draw[dkernel] (left1) to (root);
			\draw[kernel1] (left2) to (left1);
			\draw[kernel1] (left3) to (left2);
			\draw[drho] (variable4) to (left3);
			\draw[drho] (variable3) to (left2); 
			\draw[drho] (variable2) to (left1); 
			\draw[rho] (variable1) to (left); 
		\end{tikzpicture}\;.
	\end{equ}
	Note also that the Feynman graphs obtained by pairing two copies among \eqref{eq:ladder-graphs} satisfy Fact \ref{fact:graph}.
	
	Here, let us call $I: V \to E$ a \emph{partial IBP map} of $G=(V, E)$ if $I(v)$ is an edge attached to $v$ and if there is at least one (but not necessarily all) $v \in V$ for which $I(v)$ is an $E_\bK$ or $E_*$ edge. We will only consider partial IBP maps $I$ such that
	\begin{equ}[eq:reds_dont_get_two_derivatives]
		|\{v \in V: I(v) = e\}| \leq 1 \quad \text{whenever } e = \kernelrr\;.
	\end{equ}
	Given a Feynman graph $G$ formed by the parings of \eqref{eq:ladder-graphs} and a partial IBP map of $G$ satisfying \eqref{eq:reds_dont_get_two_derivatives}, one can define the associated partial IBP graph $G_I$ by the rules: if an edge $e \in E_\bK \cup E_*$ gets one derivative, then the edge type of $e$ in $G_I$ is determined by Definition \ref{def:IBP}; if $e \in E_\bK \setminus \{\kernelrr\}$ gets two derivatives, then the edge type of $e$ in $G_I$ is $\ddkernel$; for an edge $e = \ddrho$, the edge type of $e$ becomes $\erho$ (resp., $\drho, \ddrho$) if $e$ gets no (resp., one, two) derivatives; the edge type $\kernelBig$ is invariable under IBP. Note that all $\phi_\sigma(S_k, S_\ell)$ take the form $G_I$ for some $G$ formed by the parings of \eqref{eq:ladder-graphs} and some partial IBP map of $G$ satisfying \eqref{eq:reds_dont_get_two_derivatives}.
	
	Similar to Lemma \ref{lem:IBP_lastone}, the subgraphs $\bar G_I$ of $G_I$ still satisfy the degree lower bounds
	\begin{equs}[eq:IBP_lower_bounds-2]
		\deg_2(\bar G_I) &\ge \deg_2(\bar G) - u(\bar G) = \frac12 |\bar V| + \frac12 u(\bar G) - 2\;,\\
		\deg_3(\bar G_I) &\ge \deg_3(\bar G) - u(\tilde G) = \frac12 |\tilde V| + \frac12 u(\tilde G) - m - n^{\color{red} \uparrow}  + n^{\color{Cerulean} \downarrow}\;,\\
		\deg_4(\bar G_I) &\ge \deg_4(\bar G) - u(\bar G) = \frac12 u(\bar G) - \frac12 |\bar V| + n^{\color{Cerulean} \uparrow} +  2n^{\color{red}\uparrow} (\bar G) + \tilde m(\bar G)\;,
	\end{equs}
	where $\bar G$ is the counterpart subgraph of $\bar G_I$ in $G$ (recall that in $\deg_3$ we used the notation $\tilde V = \bar V \setminus \{\origin\}$ and $\tilde G$ is the subgraph associated to $\tilde V$). The same argument as in Lemma \ref{lem:IBP_lastone} shows that $\deg_2(\bar G_I) \ge 0$ for all relevant subgraphs $\bar G_I$ and characterises the cases where $\deg_2(\bar G_I) = 0$ only if $\bar G_I$ is a critical block. Under the premise that $S_1$ cannot be paired with itself, we see that there must be a mollifier edge not covered by a critical block. One can define the adjusted labelling $a^\adj_e$ by \eqref{eq:a_adj} (recall that the set $\cE^*$ is the set of mollifiers which can be increased in singularity in exchange for $\eps$-prefactors), under which all relevant subgraphs have $\deg_2^\adj > 0$. Concerning $\deg_3$, recall that in the proof of Lemma \ref{lem:cond3} 
	the $\deg_3(\bar G)\leq 0$ case was fully characterised, and none of those problematic subgraphs appear in a $G$ formed by pairing the two graphs in \eqref{eq:ladder-graphs}. Thus $\deg_3(\bar G) \geq 1$ for all relevant subgraph $\bar G$ of $G$. This immediately shows that $\deg_3^\adj(\bar G_I) > 0$ whenever $u(\tilde G) = 0$. For $u(\tilde G) > 0$, we can distinguish cases in terms of number of vertices as before:
	\begin{itemize}
		\item If $|\tilde V| = 1$, we observe $n^{\color{Cerulean} \downarrow}- m - n^{\color{red} \uparrow} \geq 0$, hence $\deg_3(\bar G_I) \ge 1$.
		\item If $|\tilde V| = 2$ and $u(\tilde G) = 2$, we observe $n^{\color{Cerulean} \downarrow}- m - n^{\color{red} \uparrow} \ge 0$, hence $\deg_3(\bar G_I) \ge 2$.
		\item If $|\tilde V| = 3$, then either $u(\tilde G) = 3$ or $u(\tilde G) = 1, n^{\color{Cerulean} \downarrow}- m - n^{\color{red} \uparrow} \ge 0$. In both cases $\deg_3(\bar G_I) \ge 1$.
		\item If $|\tilde V| = 4$ and $u(\tilde G) \ge 2$, one has $\deg_3(\bar G_I) \ge 3 -2 = 1$ since $m + n^{\color{red} \uparrow} \leq 2$.
		\item If $|\tilde V| = 5$ and $u(\tilde G) \ge 1$, one has $\deg_3(\bar G_I) \ge 3 -2 = 1$ since $m + n^{\color{red} \uparrow} \leq 2$.
		\item If $|\tilde V| \ge 6$, one has $\deg_3(\bar G_I) \ge 3 -2 = 1$ since $m + n^{\color{red} \uparrow} \leq 2$.
	\end{itemize}
	It follows that $\deg^\adj_3(\bar G_I) > 0$ for all relevant subgraph $\bar G_I$.
	
	The situation for the condition 4 is more involved.
	We claim that there exists a further adjusted degree assignment $\bar a^\adj_e$ such that, with $\overline \deg_j^\adj$ denoting the degrees defined in the obvious way with $a^\adj_e$ replaced by $\bar a^\adj_e$,
	one has
	$\overline\deg_2^\adj(\bar G_I), \overline\deg_3^\adj(\bar G_I) > 0$ and
	$\overline\deg_4^\adj(\bar G_I) \ge \deg_4^\adj(\bar G_I)$
	for all relevant subgraphs, and that the following holds:
	\begin{enumerate}
		\item[(1)] For $\sigma$ such that $\sigma(4) \neq 4$, all relevant subgraphs of $\phi_\sigma(S_k, S_\ell)$, where $k, \ell$ are not both $1$, have $\overline{\deg}_4^\adj > 0$. This includes in particular all cases associated to $\phi_{12, 34}$.
		\item[(2)] For $\sigma$ such that $\sigma(\{1, 2\}) = \{1, 2\}$ and $\sigma(4) = 4$, all relevant subgraphs of $\phi_\sigma(S_k, S_\ell)$, with $k, \ell$ not both in $\{2,3\}$ and not both $1$, have $\overline{\deg}_4^\adj > 0$.
		\item[(3)] For $\sigma$ such that $\sigma(\{1, 2\}) = \{1, 2\}$ and $\sigma(4) = 4$, the Feynman graphs $\phi_\sigma(S_k, S_\ell)$ with $k, \ell \in \{2,3\}$ fail the condition 4 of Theorem \ref{thm:power_counting}, but can be bounded by ad hoc $\eps$-distribution and labelling.
	\end{enumerate}
	To see this, we can utilise the lower bounds \eqref{eq:IBP_lower_bounds-2} and distinguish cases in terms of $|\bar V| \in [1, 6]$ as in Lemma \ref{lem:IBP_lastone}. Note that the cases $1 \leq |\bar V| \leq 4$ are essentially the same as in Lemma \ref{lem:IBP_lastone}. Let us just recall that the only place the things can go wrong here is when $|\bar V| = 4$: more precisely, when $\bar G_I$ is a subgraph of the form \eqref{eq:bad_case_deg4}. To remedy this, we need to artificially increase by $2\bar \kappa$ the singularity of some specific $E_\bK$ edges. Such amendment gives rise to the new degree assignment $\bar a_e^\adj$, which preserves the already verified conditions 2 and 3 since all subgraphs containing such edges with artificially worsen singularity had already $\deg_2, \deg_3 \ge 1$ under canonical labelling. Therefore, one indeed has $\overline{\deg}_j^\adj(\bar G_I) > 0$ for $j \in \{2, 3\}$ and $\overline{\deg}_4^\adj(\bar G_I) \ge \deg_4^\adj(\bar G_I)$, for all relevant subgraphs.
	And by construction, one also has $\overline \deg_4^\adj(\bar G_I) > 0$ for all subgraphs with up to 4 vertices.\\
	On the other hand, for $|\bar V| = 5, 6$, one notices that
	\begin{itemize}
		\item if $\sigma(4) \neq 4$, then one must have $u(\bar G) \ge 2$ and $2n^{\color{red} \uparrow} + \tilde m \ge 2$, whence $\deg_4(\bar G_I) \ge 0$;
		\item if $\sigma(4) = 4$ and consider Feynman graphs $\phi_\sigma(S_k, S_\ell)$ with $k, \ell$ not both in $\{2, 3\}$, then one must have
		\begin{equ}
			\begin{cases}
				u(\bar G) = 1 \text{ and } n^{\color{blue} \uparrow} + 2n^{\color{red} \uparrow} + \tilde m \ge 2, & \text{if } |\bar V| = 5\\
				2n^{\color{red} \uparrow} + \tilde m \ge 3, & \text{if } |\bar V| = 6\;,
			\end{cases}
		\end{equ}
		which implies $\deg_4(\bar G_I) \ge 0$. 
	\end{itemize}
	Since moreover $E(\bar V) \cap \cE^*$ is non-empty in the above two cases, one deduces $\overline{\deg}_4^\adj(\bar G_I) \ge \deg_4^\adj(\bar G_I) > 0$. This proves the claims (1) and (2).
	
	We are left with the claim (3), i.e., bounding the Feynman graphs corresponding to $\phi_\sigma(S_k, S_\ell)$ with $k, \ell \in \{2,3\}$ and $\sigma$ such that $\sigma(\{1, 2\}) = \{1, 2\}$ and $\sigma(4) = 4$. These Feynman graphs contain subgraphs $\bar G_I$ with $\deg_4(\bar G_I) = -1$ under canonical labelling and need ad hoc distribution of $\eps$ as well as singularity assignment so that Theorem \ref{thm:power_counting} is applicable. Note that the sum of these graphs multiplied by $\eps^4$ can be expressed as
	\begin{equ}
		\sum_{k,\ell \in \{2, 3\}}\sum_{\substack{\sigma(\{1, 2\}) = \{1, 2\}\\ \sigma(4) = 4}} \eps^4 \phi_\sigma (S_k, S_\ell)
		=\eps^2\left(\;
		\begin{tikzpicture}[scale=0.35,baseline=0.3cm]
			\node at (0, 0) [root] (root) {};
			\node at (-2, 1) [dot] (left1) {};
			\node at (-2, 3) [dot] (left2) {};
			\node at (2, 1) [dot] (right1) {};
			\node at (2, 3) [dot] (right2) {};
			\node at (-2, -1) [dot] (left0) {};
			\node at (2, -1) [dot] (right0) {};
			
			\draw[rho] (left0) to (right0);
			\draw[testfcn]  (left0) to (root);
			\draw[testfcn] (right0) to (root);
			
			\draw[BigG] (left2) to (right2);
			\draw[kernel1] (left2) to (left1);
			\draw[dkernel] (left1) to (root);
			\draw[kernel1] (right2) to (right1);
			\draw[dkernel] (right1) to (root);
			\draw[ddrho] (left1) to (right1);
		\end{tikzpicture}
		\;+\;
		\begin{tikzpicture}[scale=0.35,baseline=0.3cm]
			\node at (0, 0) [root] (root) {};
			\node at (-2, 1) [dot] (left1) {};
			\node at (-2, 3) [dot] (left2) {};
			\node at (2, 1) [dot] (right1) {};
			\node at (2, 3) [dot] (right2) {};
			\node at (-2, -1) [dot] (left0) {};
			\node at (2, -1) [dot] (right0) {};
			
			\draw[rho] (left0) to (right0);
			\draw[testfcn]  (left0) to (root);
			\draw[testfcn] (right0) to (root);
			
			\draw[BigG] (left2) to (right2);
			\draw[dkernel] (left2) to (left1);
			\draw[dkernel] (left1) to (root);
			\draw[dkernel] (right2) to (right1);
			\draw[dkernel] (right1) to (root);
			\draw[rho] (left1) to (right1);
		\end{tikzpicture}
		\;+2\;
		\begin{tikzpicture}[scale=0.35,baseline=0.3cm]
			\node at (0, 0) [root] (root) {};
			\node at (-2, 1) [dot] (left1) {};
			\node at (-2, 3) [dot] (left2) {};
			\node at (2, 1) [dot] (right1) {};
			\node at (2, 3) [dot] (right2) {};
			\node at (-2, -1) [dot] (left0) {};
			\node at (2, -1) [dot] (right0) {};
			
			\draw[rho] (left0) to (right0);
			\draw[testfcn]  (left0) to (root);
			\draw[testfcn] (right0) to (root);
			
			\draw[BigG] (left2) to (right2);
			\draw[kernel1] (left2) to (left1);
			\draw[dkernel] (left1) to (root);
			\draw[dkernel] (right2) to (right1);
			\draw[dkernel] (right1) to (root);
			\draw[drho] (left1) to (right1);
		\end{tikzpicture}
		\right)
	\end{equ}
	which can be bounded by the labelled graphs
	\begin{equ}
		\eps^{\bar\kappa}\;
		\begin{tikzpicture}[scale=0.5,baseline=0.3cm]
			\node at (0, 0) [root] (root) {};
			\node at (-2, 1) [dot] (left1) {};
			\node at (-2, 3) [dot] (left2) {};
			\node at (2, 1) [dot] (right1) {};
			\node at (2, 3) [dot] (right2) {};
			\node at (-2, -1) [dot] (left0) {};
			\node at (2, -1) [dot] (right0) {};
			
			\draw[generic] (left0) to node[labl,pos=0.5]{\tiny 4$\bar\kappa$,0} (right0);
			\draw[testfcn]  (left0) to (root);
			\draw[testfcn] (right0) to (root);
			
			\draw[generic] (left2) to node[labl,pos=0.5]{\tiny 2,-1} (right2);
			\draw[->] (left2) to node[labl,pos=0.5]{\tiny $\bar\kappa$,1} (left1);
			\draw[generic] (left1) to node[labl,pos=0.5]{\tiny 1+$\bar\kappa$,0} (root);
			\draw[->] (right2) to node[labl,pos=0.5]{\tiny $\bar\kappa$,1} (right1);
			\draw[generic] (right1) to node[labl,pos=0.5]{\tiny 1+$\bar\kappa$,0} (root);
			\draw[generic] (left1) to node[labl,pos=0.5]{\tiny 4-3$\bar\kappa$,-2} (right1);
		\end{tikzpicture}
		\;+\eps^{\bar\kappa}\;
		\begin{tikzpicture}[scale=0.5,baseline=0.3cm]
			\node at (0, 0) [root] (root) {};
			\node at (-2, 1) [dot] (left1) {};
			\node at (-2, 3) [dot] (left2) {};
			\node at (2, 1) [dot] (right1) {};
			\node at (2, 3) [dot] (right2) {};
			\node at (-2, -1) [dot] (left0) {};
			\node at (2, -1) [dot] (right0) {};
			
			\draw[generic] (left0) to node[labl,pos=0.5]{\tiny 4$\bar\kappa$,0} (right0);
			\draw[testfcn]  (left0) to (root);
			\draw[testfcn] (right0) to (root);
			
			\draw[generic] (left2) to node[labl,pos=0.5]{\tiny 2,-1} (right2);
			\draw[generic] (left2) to node[labl,pos=0.5]{\tiny 1+$\bar\kappa$,0} (left1);
			\draw[generic] (left1) to node[labl,pos=0.5]{\tiny 1+$\bar\kappa$,0} (root);
			\draw[generic] (right2) to node[labl,pos=0.5]{\tiny 1+$\bar\kappa$,0} (right1);
			\draw[generic] (right1) to node[labl,pos=0.5]{\tiny 1+$\bar\kappa$,0} (root);
			\draw[generic] (left1) to node[labl,pos=0.5]{\tiny 2-3$\bar\kappa$,0} (right1);
		\end{tikzpicture}
		\;+2\eps^{\bar\kappa}\;
		\begin{tikzpicture}[scale=0.5,baseline=0.3cm]
			\node at (0, 0) [root] (root) {};
			\node at (-2, 1) [dot] (left1) {};
			\node at (-2, 3) [dot] (left2) {};
			\node at (2, 1) [dot] (right1) {};
			\node at (2, 3) [dot] (right2) {};
			\node at (-2, -1) [dot] (left0) {};
			\node at (2, -1) [dot] (right0) {};
			
			\draw[generic] (left0) to node[labl,pos=0.5]{\tiny 4$\bar\kappa$,0} (right0);
			\draw[testfcn]  (left0) to (root);
			\draw[testfcn] (right0) to (root);
			
			\draw[generic] (left2) to node[labl,pos=0.5]{\tiny 2,-1} (right2);
			\draw[generic] (left2) to node[labl,pos=0.5]{\tiny $\bar\kappa$,1} (left1);
			\draw[generic] (left1) to node[labl,pos=0.5]{\tiny 1+$\bar\kappa$,0} (root);
			\draw[generic] (right2) to node[labl,pos=0.5]{\tiny 1+$\bar\kappa$,0} (right1);
			\draw[generic] (right1) to node[labl,pos=0.5]{\tiny 1+$\bar\kappa$,0} (root);
			\draw[generic] (left1) to node[labl,pos=0.5]{\tiny 3-3$\bar\kappa$,-1} (right1);
		\end{tikzpicture}\;.
	\end{equ}
	One can verify that the assumptions of Theorem \ref{thm:power_counting} are checked for these labelled graphs, yielding the bound $\eps^{\bar\kappa} \lambda^{-5\bar\kappa}$.
	
	Combining the above, we see that all graphs in question except the terms $\phi_{\sigma}(S_1, S_1)$ are indeed of order $\eps^{c(\bar\kappa)} \lambda^{-c'(\bar\kappa)}$ for some $c, c'$ vanishing with $\bar\kappa$, terminating the proof. 
\end{proof}

\begin{proposition}\label{prop:XiIXiIXiIXi}
	The random field $\hat\Pi^{\eps}_0 \<XiIXiIXiIXi>$ satisfies the bound \eqref{eq:second_moment_model}, that is,
	\begin{equ}
		\E\left(\frac{\hat\Pi^{\eps}_0 \<XiIXiIXiIXi>( \varphi^\lambda)}{\lambda^{-4\kappa+\kappa'}}\right)^2 \lesssim 1\;
	\end{equ}
	for some $\kappa' > 0$ and uniformly for $\varphi \in \Phi$, $\lambda \in (0, 1]$ and $\eps > 0$. Moreover, it holds that
	\begin{equs}
		\E\left(\hat\Pi^{\eps}_0 \<XiIXiIXiIXi> (\varphi)\right) &\to -c_\rho^2 \begin{tikzpicture}[scale=0.35,baseline=0.9cm]
			\node at (0,1)  [root] (root) {};
			\node at (0,4)  [dot] (top) {};
			
			\draw[testfcn,bend right = 60] (top) to  (root);
			
			\draw[kernel,bend left = 60] (top) to (root);
		\end{tikzpicture}\;,\label{eq:conv_XiIXiIXiIXi1}
		\\
		\E\left(\hat\Pi^{\eps}_0 \<XiIXiIXiIXi> (\varphi) - \E\left(\hat\Pi^{\eps}_0 \<XiIXiIXiIXi> (\varphi)\right)\right)^2 &\to c_\rho^4 \left(\;
		\begin{tikzpicture}[scale=0.35,baseline=0.9cm]
			\node at (0,1)  [root] (root) {};
			\node at (0,3)  [dot] (down) {};
			\node at (0,5)  [dot] (top) {};
			
			\draw[testfcn, bend right = 30] (down) to  (root);
			\draw[testfcn, bend left = 30] (down) to  (root);
			
			\draw[kernel1, bend left = 30] (top) to (down);
			\draw[kernel1, bend right = 30] (top) to (down);
		\end{tikzpicture}
		\;+\;
		\begin{tikzpicture}[scale=0.35,baseline=0.9cm]
			\node at (0,1)  [root] (root) {};
			\node at (-1,4)  [dot] (left) {};
			\node at (1,4)  [dot] (right) {};
			
			\draw[testfcn] (left) to  (root);
			\draw[testfcn] (right) to  (root);
			
			\draw[kernel1, bend left = 30] (left) to (right);
			\draw[kernel1, bend left = 30] (right) to (left);
		\end{tikzpicture}\right)\label{eq:conv_XiIXiIXiIXi}
		\\
		\E\left(\hat\Pi^{\eps}_0 \<XiIXiIXiIXi> (\varphi) A_\eps (\varphi) A_\eps (\varphi)\right) &\to c_\rho^4 \left(2\;
		\begin{tikzpicture}[scale=0.35,baseline=0.9cm]
			\node at (0,1)  [root] (root) {};
			\node at (0,3)  [dot] (down) {};
			\node at (0,5)  [dot] (top) {};
			
			\draw[testfcn, bend right = 30] (down) to  (root);
			\draw[testfcn, bend left = 30] (down) to  (root);
			\draw[testfcn, bend left = 60] (top) to  (root);
			
			\draw[kernel1] (top) to (down);
		\end{tikzpicture}
		\;-\;
		\begin{tikzpicture}[scale=0.35,baseline=0.9cm]
			\node at (0,1)  [root] (root) {};
			\node at (0,4)  [dot] (top) {};
			
			\draw[testfcn,bend right = 60] (top) to  (root);
			
			\draw[kernel,bend left = 60] (top) to (root);
		\end{tikzpicture}
		\begin{tikzpicture}[scale=0.35,baseline=0.9cm]
			\node at (0,1)  [root] (root) {};
			\node at (0,4)  [dot] (top) {};
			
			\draw[testfcn,bend right = 60] (top) to  (root);
			\draw[testfcn,bend left = 60] (top) to  (root);
		\end{tikzpicture}
		\right).\label{eq:conv_XiIXiIXiIXi3}
	\end{equs}
\end{proposition}
\begin{proof}
	As in Proposition \ref{prop:lastone} we detail the proof of tightness and \eqref{eq:conv_XiIXiIXiIXi1} and \eqref{eq:conv_XiIXiIXiIXi}. The convergence \eqref{eq:conv_XiIXiIXiIXi3} follows from a similar (and simpler) calculation.
	
	By \eqref{eq:XiIXiIXiIXi}, Corollary \ref{cor:graph}, Proposition \ref{prop:vanishing} and Lemma \ref{lem:IBP_XiIXiIXiIXi}, there exist $c(\bar\kappa), c'(\bar \kappa) \in (0, 1)$ such that $c(\bar\kappa), c'(\bar\kappa) \to 0$ as $\bar \kappa \to 0$ and that
	\begin{equs}
		\E\left(\hat\Pi^{\eps}_0 \<XiIXiIXiIXi> (\varphi^\lambda)\right) &=
		-\eps^2\left(\;
		\begin{tikzpicture}[scale=0.35,baseline=0.9cm]
			\node at (0,-1)  [root] (root) {};
			\node at (-2,1)  [dot] (left) {};
			\node at (-2,3)  [dot] (left1) {};
			\node at (0,7)  [dot] (right1) {};
			\node at (0,5) [dot] (right) {};
			
			\draw[testfcn] (left) to  (root);
			
			\draw[ddkernel] (left1) to (left);
			\draw[ddkernel] (right1) to (right);
			\draw[kernel] (right) to (root);
			\draw[rho] (left1) to (right1);
			\draw[rho] (right) to (left);
		\end{tikzpicture}
		\;+\;
		\begin{tikzpicture}[scale=0.35,baseline=0.9cm]
			\node at (0,-1)  [root] (root) {};
			\node at (-2,1)  [dot] (left) {};
			\node at (-2,3)  [dot] (left1) {};
			\node at (0,7)  [dot] (right1) {};
			\node at (0,5) [dot] (right) {};
			
			\draw[testfcn] (left) to  (root);
			
			\draw[ddkernel] (left1) to (left);
			\draw[ddkernel] (right1) to (right);
			\draw[kernel] (right) to (root);
			\draw[rho] (left1) to (right);
			\draw[rho] (right1) to (left);
		\end{tikzpicture}\right)
		+ O(\eps^{c(\bar\kappa)}\lambda^{-c'(\bar\kappa)})\;,\\
		\E\left(\hat\Pi^{\eps}_0 \<XiIXiIXiIXi> (\varphi^\lambda) - \E\left(\hat\Pi^{\eps}_0 \<XiIXiIXiIXi> (\varphi^\lambda)\right)\right)^2 &=
		\eps^4 \sum_\sigma \phi_{\sigma} \left(\begin{tikzpicture}[scale=0.35,baseline=0.9cm]
			\node at (0,-1)  [root] (root) {};
			\node at (-2,1)  [dot] (left) {};
			\node at (-2,3)  [dot] (left1) {};
			\node at (-2,5)  [dot] (left2) {};
			\node at (-2,7)  [dot] (left3) {};
			\node at (0,1) [var] (variable1) {};
			\node at (0,3) [var] (variable2) {};
			\node at (0,5) [var] (variable3) {};
			\node at (0,7) [var] (variable4) {};
			
			\draw[testfcn] (left) to  (root);
			
			\draw[ddkernel] (left1) to (left);
			\draw[kernel1] (left2) to (left1);
			\draw[ddkernel] (left3) to (left2);
			\draw[rho] (variable4) to (left3);
			\draw[rho] (variable3) to (left2); 
			\draw[rho] (variable2) to (left1); 
			\draw[rho] (variable1) to (left); 
		\end{tikzpicture}\right)
		+ O(\eps^{c(\bar\kappa)}\lambda^{-c'(\bar\kappa)})\;,
	\end{equs}
	where the sum runs over all permutation $\sigma$ of $\{1, 2, 3, 4\}$ such that $\sigma(\{1, 2\}) = \{1, 2\}$ or $\{3, 4\}$. The contribution of the expectation term $\E\left(\hat\Pi^{\eps}_0 \<XiIXiIXiIXi> (\varphi^\lambda)\right)^2 \lesssim \lambda^{-c'(\bar\kappa)}$ to the required tightness bound as well as the convergence \eqref{eq:conv_XiIXiIXiIXi1} then follow immediately from \eqref{eq:G_bound} and Lemma \ref{lem:G_conv}.
	
	Let us now turn to the variance term, where we consider the pairings $\phi_{12, 12}$ and $\phi_{12, 34}$ separately.\\
	
	\noindent\textbf{The pairings $\phi_{12, 12}$.}
	The resulted $4$ pairings sum up to
	\begin{equs}[eq:type1_pairings_XiIXiIXiIXi]
		\eps^{2}\left(\;
		\begin{tikzpicture}[scale=0.35,baseline=0.7cm]
			\node at (0,-1)  [root] (root) {};
			\node at (-2,1)  [dot] (left) {};
			\node at (-2,3)  [dot] (left1) {};
			\node at (-2,5)  [dot] (left2) {};
			\node at (2,1)  [dot] (right) {};
			\node at (2,3)  [dot] (right1) {};
			\node at (2,5)  [dot] (right2) {};
			
			\draw[testfcn] (left) to  (root);
			\draw[testfcn] (right) to  (root);
			
			\draw[ddkernel] (left1) to (left);
			\draw[kernel1] (left2) to (left1);
			
			\draw[ddkernel] (right1) to (right);
			\draw[kernel1] (right2) to (right1);
			\draw[rho] (left1) to (right1); 
			\draw[rho] (left) to (right);
			
			\draw[BigG] (left2) to (right2);
		\end{tikzpicture}
		\;+\;
		\begin{tikzpicture}[scale=0.35,baseline=0.7cm]
			\node at (0,-1)  [root] (root) {};
			\node at (-2,1)  [dot] (left) {};
			\node at (-2,3)  [dot] (left1) {};
			\node at (-2,5)  [dot] (left2) {};
			\node at (2,1)  [dot] (right) {};
			\node at (2,3)  [dot] (right1) {};
			\node at (2,5)  [dot] (right2) {};
			
			\draw[testfcn] (left) to  (root);
			\draw[testfcn] (right) to  (root);
			
			\draw[ddkernel] (left1) to (left);
			\draw[kernel1] (left2) to (left1);
			
			\draw[ddkernel] (right1) to (right);
			\draw[kernel1] (right2) to (right1);
			\draw[rho] (left1) to (right); 
			\draw[rho] (left) to (right1);
			
			\draw[BigG] (left2) to (right2);
		\end{tikzpicture}
		\;\right)
	\end{equs}
	where the edge \tikz[baseline=-0.1cm] \draw[BigG] (0,0) to (1,0); represents the kernel $G_\eps$. We claim that \eqref{eq:type1_pairings_XiIXiIXiIXi} satisfies the required tightness bound, and that with $\lambda = 1$ fixed it converges to
	\begin{equ}
		c_\rho^4 \iint K(y, x) K(y, x) \varphi(y) \varphi(y) \dd x \dd y\;.
	\end{equ}
	
	To prove this, let us consider the cases where $\eps^\frac12 > \lambda$ and $\eps^\frac12 \leq \lambda$ separately. For the case $\eps^\frac12 > \lambda$, it suffices to remark that
	\begin{equ}
		|D^k \rho_\eps(x)| \lesssim \lambda^{-2\gamma} (|x| + \eps)^{-2 + \gamma -|k|}
	\end{equ}
	for all $\gamma \ge 0$ and $\eps \in (\lambda^2,1]$. That is, one can label the mollifier edge $\rho_\eps^{*2}$ by $(2-\gamma, 0)$ instead of $(2, -1)$ by paying a price of $\lambda^{-2\gamma}$. Choosing a $\gamma$ sufficiently small, it follows easily that \eqref{eq:type1_pairings_XiIXiIXiIXi} satisfies the tightness bound.
	
	For the case $\eps^\frac12 \le \lambda$, let us introduce the smooth cutoff $\chi_\eps$ such that $\chi_\eps = 1$ on $\{x \in \R^2: |x| \leq \eps^\frac12/2\}$ and supported in $\{x \in \R^2: |x| \leq \eps^\frac12\}$. Let the edge
	$\tikz[baseline=-0.1cm] \draw[ddkernel] (0,0) to node[below]{\tiny $\leq$} (1,0);$
	and
	$\tikz[baseline=-0.1cm] \draw[ddkernel] (0,0) to node[below]{\tiny$>$} (1,0);$
	represent the kernels $(\partial_1^2 K \chi_\eps)(x-y)$ and $(\partial_1^2 K(1-\chi_\eps)) (x-y)$, respectively. Provided $\eps^\frac12 \le \lambda$, one has
	\begin{equ}
		(\varphi^\lambda(x) - \varphi^\lambda(y))\chi_\eps (x - y) = c \eps^{\frac12}\lambda^{-1} \psi^\lambda(y) \chi_\eps (x - y)
	\end{equ}
	for some $\psi \in \Phi$ and some constant $c$ independent of $x, y$ and $\eps$.
	As a consequence, \eqref{eq:type1_pairings_XiIXiIXiIXi} has the dominant term
	\begin{equ}[eq:type1_dominant_XiIXiIXiIXi]
		\eps^2I:=\eps^{2}\left(\;
		\begin{tikzpicture}[scale=0.35,baseline=0.7cm]
			\node at (0,-1)  [root] (root) {};
			\node at (-2,1)  [dot] (left) {};
			\node at (-2,3)  [dot] (left1) {};
			\node at (-2,5)  [dot] (left2) {};
			\node at (2,1)  [dot] (right) {};
			\node at (2,3)  [dot] (right1) {};
			\node at (2,5)  [dot] (right2) {};
			
			\draw[testfcn, bend right = 70] (left1) to (root);
			\draw[testfcn, bend left = 70] (right1) to  (root);
			
			\draw[ddkernel] (left1) to node[right]{\tiny $\leq$} (left);
			\draw[kernel1] (left2) to (left1);
			
			\draw[ddkernel] (right1) to node[right]{\tiny $\leq$} (right);
			\draw[kernel1] (right2) to (right1);
			\draw[rho] (left1) to (right1); 
			\draw[rho] (left) to (right);
			
			\draw[BigG] (left2) to (right2);
		\end{tikzpicture}
		\;+\;
		\begin{tikzpicture}[scale=0.35,baseline=0.7cm]
			\node at (0,-1)  [root] (root) {};
			\node at (-2,1)  [dot] (left) {};
			\node at (-2,3)  [dot] (left1) {};
			\node at (-2,5)  [dot] (left2) {};
			\node at (2,1)  [dot] (right) {};
			\node at (2,3)  [dot] (right1) {};
			\node at (2,5)  [dot] (right2) {};
			
			\draw[testfcn, bend right = 70] (left1) to (root);
			\draw[testfcn, bend left = 70] (right1) to (root);
			
			\draw[ddkernel] (left1) to node[right]{\tiny $\leq$} (left);
			\draw[kernel1] (left2) to (left1);
			
			\draw[ddkernel] (right1) to node[right]{\tiny $\leq$} (right);
			\draw[kernel1] (right2) to (right1);
			\draw[rho] (left1) to (right); 
			\draw[rho] (left) to (right1);
			
			\draw[BigG] (left2) to (right2);
		\end{tikzpicture}
		\;\right)
	\end{equ}
	with the difference between \eqref{eq:type1_pairings_XiIXiIXiIXi} and \eqref{eq:type1_dominant_XiIXiIXiIXi} controlled by
	\begin{equs}[eq:difference_XiIXiIXiIXi]
		\;\eps^2(2\eps^{\frac12} \lambda^{-1}+\eps\lambda^{-2})I 
		\;+\eps^2\sum_{\bullet,\circ} \left(\;
		\begin{tikzpicture}[scale=0.35,baseline=0.7cm]
			\node at (0,-1)  [root] (root) {};
			\node at (-2,1)  [dot] (left) {};
			\node at (-2,3)  [dot] (left1) {};
			\node at (-2,5)  [dot] (left2) {};
			\node at (2,1)  [dot] (right) {};
			\node at (2,3)  [dot] (right1) {};
			\node at (2,5)  [dot] (right2) {};
			
			\draw[testfcn] (left) to  (root);
			\draw[testfcn] (right) to  (root);
			
			\draw[ddkernel] (left1) to node[right]{$\bullet$} (left);
			\draw[kernel1] (left2) to (left1);
			
			\draw[ddkernel] (right1) to node[right]{ $\circ$} (right);
			\draw[kernel1] (right2) to (right1);
			\draw[rho] (left1) to (right1); 
			\draw[rho] (left) to (right);
			
			\draw[BigG] (left2) to (right2);
		\end{tikzpicture}
		\;+\;
		\begin{tikzpicture}[scale=0.35,baseline=0.7cm]
			\node at (0,-1)  [root] (root) {};
			\node at (-2,1)  [dot] (left) {};
			\node at (-2,3)  [dot] (left1) {};
			\node at (-2,5)  [dot] (left2) {};
			\node at (2,1)  [dot] (right) {};
			\node at (2,3)  [dot] (right1) {};
			\node at (2,5)  [dot] (right2) {};
			
			\draw[testfcn] (left) to  (root);
			\draw[testfcn] (right) to  (root);
			
			\draw[ddkernel] (left1) to node[right]{ $\bullet$} (left);
			\draw[kernel1] (left2) to (left1);
			
			\draw[ddkernel] (right1) to node[right]{$\circ$} (right);
			\draw[kernel1] (right2) to (right1);
			\draw[rho] (left1) to (right); 
			\draw[rho] (left) to (right1);
			
			\draw[BigG] (left2) to (right2);
		\end{tikzpicture}\right),
	\end{equs}
	where the sum runs over pairs $\{\bullet,\circ\}\in\{>,\leq\}^2$ such that at least one of them is $>$.
	
	To prove the desired tightness bound, it is sufficient to show that $\eps^2I\lesssim \lambda^{-c'(\bar\kappa)}$ and that each summand in \eqref{eq:difference_XiIXiIXiIXi} is of order $\eps^{c(\bar\kappa)} \lambda^{-c'(\bar\kappa)}$ for some $c(\bar\kappa), c'(\bar\kappa) \to 0$ as $\bar\kappa\to 0$.
	
	For the latter, one simply notes that $\partial_1^2 K (1-\chi_\eps)$ is nothing but a smooth function with compact support satisfying the bound
	\begin{equ}[eq:ddK_>]
		\norm{\eps^{\frac{\kappa'}{2}}\partial_1^2 K (1-\chi_\eps)}_{2-\kappa', m} = \sup_{|k| < m} \sup_{0< |x| \leq 1} \eps^{\frac{\kappa'}{2}} |x|^{2-\kappa' + k} |D^k[\partial_1^2 K (1-\chi_\eps)](x)| \lesssim 1
	\end{equ}
	for all $\kappa' \in (0, 2)$ and all integer $m > 0$. This is why we have chosen the cut-off at order $\eps^{\frac12}$: we only need to pay $\eps^{-\kappa'/2}$ to improve the singularity assignment of $\tikz[baseline=-0.1cm] \draw[ddkernel] (0,0) to node[below]{\tiny$>$} (1,0);$ by $\kappa'$. One can thus write, for example, with $\circ \in \{\leq, >\}$,
	\begin{equ}
		\eps^2\;
		\begin{tikzpicture}[scale=0.35,baseline=0.7cm]
			\node at (0,-1)  [root] (root) {};
			\node at (-2,1)  [dot] (left) {};
			\node at (-2,3)  [dot] (left1) {};
			\node at (-2,5)  [dot] (left2) {};
			\node at (2,1)  [dot] (right) {};
			\node at (2,3)  [dot] (right1) {};
			\node at (2,5)  [dot] (right2) {};
			
			\draw[testfcn] (left) to  (root);
			\draw[testfcn] (right) to  (root);
			
			\draw[ddkernel] (left1) to node[right]{\tiny $>$} (left);
			\draw[kernel1] (left2) to (left1);
			
			\draw[ddkernel] (right1) to node[right]{$\circ$} (right);
			\draw[kernel1] (right2) to (right1);
			\draw[rho] (left1) to (right1); 
			\draw[rho] (left) to (right);
			
			\draw[BigG] (left2) to (right2);
		\end{tikzpicture}
		\;=\eps^{\bar\kappa}\;
		\begin{tikzpicture}[scale=0.35,baseline=0.7cm]
			\node at (0,-1)  [root] (root) {};
			\node at (-2,1)  [dot] (left) {};
			\node at (-2,3)  [dot] (left1) {};
			\node at (-2,5)  [dot] (left2) {};
			\node at (2,1)  [dot] (right) {};
			\node at (2,3)  [dot] (right1) {};
			\node at (2,5)  [dot] (right2) {};
			
			\draw[dist] (left) to  (root);
			\draw[dist] (right) to  (root);
			
			\draw[generic] (left1) to node[labl,pos=0.5] {\tiny 2-4$\bar\kappa$,0} (left);
			\draw[->] (left2) to node[labl,pos=0.45] {\tiny 3$\bar\kappa$,1} (left1);
			
			\draw[generic] (right1) to node[labl,pos=0.5] {\tiny 2,-1} (right);
			\draw[->] (right2) to node[labl,pos=0.5] {\tiny 3$\bar\kappa$,1} (right1);
			\draw[generic] (left1) to node[labl,pos=0.5] {\tiny 2-$\bar\kappa$,0} (right1); 
			\draw[generic] (left) to node[labl,pos=0.5] {\tiny 4$\bar\kappa$,0} (right);
			
			\draw[generic] (left2) to node[labl,pos=0.5] {\tiny 2,-1} (right2);
		\end{tikzpicture}
		\lesssim \eps^{\bar\kappa} \lambda^{-5\bar\kappa}
	\end{equ}
	where one can check the labelled graph indeed satisfies the assumptions of Theorem \ref{thm:power_counting}. The remaining terms in \eqref{eq:difference_XiIXiIXiIXi} can all be treated in the same way.
	
	One the other hand, \eqref{eq:type1_dominant_XiIXiIXiIXi} can be written as
	\begin{equ}
		\begin{tikzpicture}[scale=0.35,baseline=0.9cm]
			\node at (0,1)  [root] (root) {};
			\node at (-2,3)  [dot] (left1) {};
			\node at (-2,5)  [dot] (left2) {};
			\node at (2,3)  [dot] (right1) {};
			\node at (2,5)  [dot] (right2) {};
			
			\draw[testfcn] (left1) to  (root);
			\draw[testfcn] (right1) to  (root);
			
			\draw[kernel1] (left2) to (left1);
			\draw[kernel1] (right2) to (right1);
			
			\draw[BigG] (left2) to (right2);
			\draw[BigG] (left1) to (right1);
		\end{tikzpicture}
		+ \eps^2\sum_{\bullet,\circ}\left(
			\begin{tikzpicture}[scale=0.35,baseline=0.7cm]
				\node at (0,-1)  [root] (root) {};
				\node at (-2,1)  [dot] (left) {};
				\node at (-2,3)  [dot] (left1) {};
				\node at (-2,5)  [dot] (left2) {};
				\node at (2,1)  [dot] (right) {};
				\node at (2,3)  [dot] (right1) {};
				\node at (2,5)  [dot] (right2) {};
				
				\draw[testfcn, bend right = 70] (left1) to (root);
				\draw[testfcn, bend left = 70] (right1) to  (root);
				
				\draw[ddkernel] (left1) to node[right]{$\bullet$} (left);
				\draw[kernel1] (left2) to (left1);
				
				\draw[ddkernel] (right1) to node[left]{$\circ$} (right);
				\draw[kernel1] (right2) to (right1);
				\draw[rho] (left1) to (right1); 
				\draw[rho] (left) to (right);
				
				\draw[BigG] (left2) to (right2);
			\end{tikzpicture}
			\;+\;
			\begin{tikzpicture}[scale=0.35,baseline=0.7cm]
				\node at (0,-1)  [root] (root) {};
				\node at (-2,1)  [dot] (left) {};
				\node at (-2,3)  [dot] (left1) {};
				\node at (-2,5)  [dot] (left2) {};
				\node at (2,1)  [dot] (right) {};
				\node at (2,3)  [dot] (right1) {};
				\node at (2,5)  [dot] (right2) {};
				
				\draw[testfcn, bend right = 70] (left1) to (root);
				\draw[testfcn, bend left = 70] (right1) to (root);
				
				\draw[ddkernel] (left1) to node[right]{$\bullet$} (left);
				\draw[kernel1] (left2) to (left1);
				
				\draw[ddkernel] (right1) to node[left]{$\circ$} (right);
				\draw[kernel1] (right2) to (right1);
				\draw[rho] (left1) to (right); 
				\draw[rho] (left) to (right1);
				
				\draw[BigG] (left2) to (right2);
			\end{tikzpicture}\right)\;,
	\end{equ}
	with the sum as above.
	That this sum is of order $\eps^{c(\bar\kappa)} \lambda^{-c'(\bar\kappa)}$ can be proved similarly to above by using \eqref{eq:ddK_>}. Note that the first Feynman diagram already appeared in the proof of Proposition \ref{prop:lastone}, where its associated integral is shown to be of order $\lambda^{-2\bar\kappa}$. By the fact that $\eps^{\frac12} \lambda^{-1} \leq \eps^{\bar \kappa} \lambda^{-2\bar \kappa}$, it follows that \eqref{eq:difference_XiIXiIXiIXi} is of order $\eps^{c(\bar\kappa)} \lambda^{-c'(\bar\kappa)}$. Finally, the desired limit then follows easily from Lemma \ref{lem:G_conv}.\\
		
	\noindent\textbf{The pairing $\phi_{12, 34}$.}
	The resulted $4$ pairings sum up to
	\begin{equs}[eq:type2_pairings_XiIXiIXiIXi]
		\eps^{4}\left(\;
		\begin{tikzpicture}[scale=0.35,baseline=0.9cm]
			\node at (0,-1)  [root] (root) {};
			\node at (-2,1)  [dot] (left) {};
			\node at (-2,3)  [dot] (left1) {};
			\node at (-2,5)  [dot] (left2) {};
			\node at (-2,7)  [dot] (left3) {};
			\node at (2,1)  [dot] (right) {};
			\node at (2,3)  [dot] (right1) {};
			\node at (2,5)  [dot] (right2) {};
			\node at (2,7)  [dot] (right3) {};
			
			\draw[testfcn] (left) to  (root);
			\draw[testfcn] (right) to  (root);
			
			\draw[ddkernel] (left1) to (left);
			\draw[ddkernel] (right1) to (right);
			
			\draw[rho] (left1) to (right3); 
			\draw[rho] (left) to (right2);
			
			\draw[kernel1] (left2) to (left1);
			\draw[kernel1] (right2) to (right1);
			
			\draw[ddkernel] (left3) to (left2);
			\draw[ddkernel] (right3) to (right2);
			
			\draw[rho] (left3) to (right1); 
			\draw[rho] (left2) to (right);
		\end{tikzpicture}
		\;+\;
		\begin{tikzpicture}[scale=0.35,baseline=0.9cm]
			\node at (0,-1)  [root] (root) {};
			\node at (-2,1)  [dot] (left) {};
			\node at (-2,3)  [dot] (left1) {};
			\node at (-2,5)  [dot] (left2) {};
			\node at (-2,7)  [dot] (left3) {};
			\node at (2,1)  [dot] (right) {};
			\node at (2,3)  [dot] (right1) {};
			\node at (2,5)  [dot] (right2) {};
			\node at (2,7)  [dot] (right3) {};
			
			\draw[testfcn] (left) to  (root);
			\draw[testfcn] (right) to  (root);
			
			\draw[ddkernel] (left1) to (left);
			\draw[ddkernel] (right1) to (right);
			
			\draw[rho] (left1) to (right2); 
			\draw[rho] (left) to (right3);
			
			\draw[kernel1] (left2) to (left1);
			\draw[kernel1] (right2) to (right1);
			
			\draw[ddkernel] (left3) to (left2);
			\draw[ddkernel] (right3) to (right2);
			
			\draw[rho] (left3) to (right1); 
			\draw[rho] (left2) to (right);
		\end{tikzpicture}
		\;+\;
		\begin{tikzpicture}[scale=0.35,baseline=0.9cm]
			\node at (0,-1)  [root] (root) {};
			\node at (-2,1)  [dot] (left) {};
			\node at (-2,3)  [dot] (left1) {};
			\node at (-2,5)  [dot] (left2) {};
			\node at (-2,7)  [dot] (left3) {};
			\node at (2,1)  [dot] (right) {};
			\node at (2,3)  [dot] (right1) {};
			\node at (2,5)  [dot] (right2) {};
			\node at (2,7)  [dot] (right3) {};
			
			\draw[testfcn] (left) to  (root);
			\draw[testfcn] (right) to  (root);
			
			\draw[ddkernel] (left1) to (left);
			\draw[ddkernel] (right1) to (right);
			
			\draw[rho] (left1) to (right3); 
			\draw[rho] (left) to (right2);
			
			\draw[kernel1] (left2) to (left1);
			\draw[kernel1] (right2) to (right1);
			
			\draw[ddkernel] (left3) to (left2);
			\draw[ddkernel] (right3) to (right2);
			
			\draw[rho] (left3) to (right); 
			\draw[rho] (left2) to (right1);
		\end{tikzpicture}
		\;+\;
		\begin{tikzpicture}[scale=0.35,baseline=0.9cm]
			\node at (0,-1)  [root] (root) {};
			\node at (-2,1)  [dot] (left) {};
			\node at (-2,3)  [dot] (left1) {};
			\node at (-2,5)  [dot] (left2) {};
			\node at (-2,7)  [dot] (left3) {};
			\node at (2,1)  [dot] (right) {};
			\node at (2,3)  [dot] (right1) {};
			\node at (2,5)  [dot] (right2) {};
			\node at (2,7)  [dot] (right3) {};
			
			\draw[testfcn] (left) to  (root);
			\draw[testfcn] (right) to  (root);
			
			\draw[ddkernel] (left1) to (left);
			\draw[ddkernel] (right1) to (right);
			
			\draw[rho] (left1) to (right2); 
			\draw[rho] (left) to (right3);
			
			\draw[kernel1] (left2) to (left1);
			\draw[kernel1] (right2) to (right1);
			
			\draw[ddkernel] (left3) to (left2);
			\draw[ddkernel] (right3) to (right2);
			
			\draw[rho] (left3) to (right); 
			\draw[rho] (left2) to (right1);
		\end{tikzpicture}
		\;\right)\;.
	\end{equs}
	We claim that \eqref{eq:type2_pairings_XiIXiIXiIXi} is bounded by a quantity of order $1$ uniformly for all $\eps, \lambda \in (0, 1]$ as required by tightness, and that it converges as $\eps \to 0$ to
	\begin{equ}
		c_\rho^4 \iint K(y, x) K(x, y) \varphi(x) \varphi(y)\;.
	\end{equ}
	Again consider the cases $\eps^{\frac12} > \lambda$ and $\eps^{\frac12} \leq \lambda$. For the former case the same reasoning as previously applies in the exact same way. For the latter case, again using the decomposition $\tikz[baseline=-0.1cm] \draw[ddkernel] (0,0) to (1,0); = \tikz[baseline=-0.1cm] \draw[ddkernel] (0,0) to node[below]{\tiny\color{black} $\leq$} (1,0); + \tikz[baseline=-0.1cm] \draw[ddkernel] (0,0) to node[below]{\tiny\color{black} $>$} (1,0);$ and the same procedure as in the previous case, we see that \eqref{eq:type2_pairings_XiIXiIXiIXi} has the dominant term
	\begin{equs}
		\eps^{4}&\left(\;
		\begin{tikzpicture}[scale=0.3,baseline=0.9cm]
			\node at (0,-1)  [root] (root) {};
			\node at (-2,1)  [dot] (left) {};
			\node at (-2,3)  [dot] (left1) {};
			\node at (-2,5)  [dot] (left2) {};
			\node at (-2,7)  [dot] (left3) {};
			\node at (2,1)  [dot] (right) {};
			\node at (2,3)  [dot] (right1) {};
			\node at (2,5)  [dot] (right2) {};
			\node at (2,7)  [dot] (right3) {};
			
			\draw[testfcn, bend right = 70] (left1) to  (root);
			\draw[testfcn, bend left = 70] (right1) to  (root);
			
			\draw[ddkernel] (left1) to node[right]{\tiny $\leq$} (left);
			\draw[ddkernel] (right1) to node[right]{\tiny $\leq$} (right);
			
			\draw[rho] (left1) to (right3); 
			\draw[rho] (left) to (right2);
			
			\draw[kernel1] (left2) to (left1);
			\draw[kernel1] (right2) to (right1);
			
			\draw[ddkernel] (left3) to (left2);
			\draw[ddkernel] (right3) to (right2);
			
			\draw[rho] (left3) to (right1); 
			\draw[rho] (left2) to (right);
		\end{tikzpicture}
		+
		\begin{tikzpicture}[scale=0.3,baseline=0.9cm]
			\node at (0,-1)  [root] (root) {};
			\node at (-2,1)  [dot] (left) {};
			\node at (-2,3)  [dot] (left1) {};
			\node at (-2,5)  [dot] (left2) {};
			\node at (-2,7)  [dot] (left3) {};
			\node at (2,1)  [dot] (right) {};
			\node at (2,3)  [dot] (right1) {};
			\node at (2,5)  [dot] (right2) {};
			\node at (2,7)  [dot] (right3) {};
			
			\draw[testfcn, bend right = 70] (left1) to  (root);
			\draw[testfcn, bend left = 70] (right1) to  (root);
			
			\draw[ddkernel] (left1) to node[right]{\tiny $\leq$} (left);
			\draw[ddkernel] (right1) to node[right]{\tiny $\leq$} (right);
			
			\draw[rho] (left1) to (right2); 
			\draw[rho] (left) to (right3);
			
			\draw[kernel1] (left2) to (left1);
			\draw[kernel1] (right2) to (right1);
			
			\draw[ddkernel] (left3) to (left2);
			\draw[ddkernel] (right3) to (right2);
			
			\draw[rho] (left3) to (right1); 
			\draw[rho] (left2) to (right);
		\end{tikzpicture}
		+
		\begin{tikzpicture}[scale=0.3,baseline=0.9cm]
			\node at (0,-1)  [root] (root) {};
			\node at (-2,1)  [dot] (left) {};
			\node at (-2,3)  [dot] (left1) {};
			\node at (-2,5)  [dot] (left2) {};
			\node at (-2,7)  [dot] (left3) {};
			\node at (2,1)  [dot] (right) {};
			\node at (2,3)  [dot] (right1) {};
			\node at (2,5)  [dot] (right2) {};
			\node at (2,7)  [dot] (right3) {};
			
			\draw[testfcn, bend right = 70] (left1) to  (root);
			\draw[testfcn, bend left = 70] (right1) to  (root);
			
			\draw[ddkernel] (left1) to node[right]{\tiny $\leq$} (left);
			\draw[ddkernel] (right1) to node[right]{\tiny $\leq$} (right);
			
			\draw[rho] (left1) to (right3); 
			\draw[rho] (left) to (right2);
			
			\draw[kernel1] (left2) to (left1);
			\draw[kernel1] (right2) to (right1);
			
			\draw[ddkernel] (left3) to (left2);
			\draw[ddkernel] (right3) to (right2);
			
			\draw[rho] (left3) to (right); 
			\draw[rho] (left2) to (right1);
		\end{tikzpicture}
		+
		\begin{tikzpicture}[scale=0.3,baseline=0.9cm]
			\node at (0,-1)  [root] (root) {};
			\node at (-2,1)  [dot] (left) {};
			\node at (-2,3)  [dot] (left1) {};
			\node at (-2,5)  [dot] (left2) {};
			\node at (-2,7)  [dot] (left3) {};
			\node at (2,1)  [dot] (right) {};
			\node at (2,3)  [dot] (right1) {};
			\node at (2,5)  [dot] (right2) {};
			\node at (2,7)  [dot] (right3) {};
			
			\draw[testfcn, bend right = 70] (left1) to  (root);
			\draw[testfcn, bend left = 70] (right1) to  (root);
			
			\draw[ddkernel] (left1) to node[right]{\tiny $\leq$} (left);
			\draw[ddkernel] (right1) to node[right]{\tiny $\leq$} (right);
			
			\draw[rho] (left1) to (right2); 
			\draw[rho] (left) to (right3);
			
			\draw[kernel1] (left2) to (left1);
			\draw[kernel1] (right2) to (right1);
			
			\draw[ddkernel] (left3) to (left2);
			\draw[ddkernel] (right3) to (right2);
			
			\draw[rho] (left3) to (right); 
			\draw[rho] (left2) to (right1);
		\end{tikzpicture}
		\;\right)\;.\label{eq:type2_dominant_XiIXiIXiIXi}
	\end{equs}
	By the same line of computation as previously, the difference between \eqref{eq:type2_pairings_XiIXiIXiIXi} and \eqref{eq:type2_dominant_XiIXiIXiIXi} is of order $\eps^{c(\bar\kappa)} \lambda^{-c'(\bar\kappa)}$ uniformly for $\eps^{\frac12} \leq \lambda$ and for some $c(\bar\kappa), c'(\bar\kappa) \to 0$ as $\bar\kappa \to 0$. Finally, \eqref{eq:type2_dominant_XiIXiIXiIXi} can be rewritten as
	\begin{equ}
		\begin{tikzpicture}[scale=0.35,baseline=1cm]
			\node at (0,1)  [root] (root) {};
			\node at (-2,3)  [dot] (left1) {};
			\node at (-2,5)  [dot] (left2) {};
			\node at (2,3)  [dot] (right1) {};
			\node at (2,5)  [dot] (right2) {};
			
			\draw[testfcn] (left1) to  (root);
			\draw[testfcn] (right1) to  (root);
			
			\draw[kernel1] (left2) to (left1);
			\draw[kernel1] (right2) to (right1);
			
			\draw[BigG] (left2) to (right1);
			\draw[BigG] (left1) to (right2);
		\end{tikzpicture} + O(\eps^{c(\bar\kappa)} \lambda^{-c'(\bar\kappa)})\;,
	\end{equ}
	with the error term estimated by the procedure identical to the previous case. Again, this dominant Feynman graph appeared already in Proposition \ref{prop:lastone} where the bound $\lambda^{-2\bar\kappa}$ is shown. Finally, the desired limit follows easily from Lemma \ref{lem:G_conv}. This concludes the proof
\end{proof}

\subsection{Concluding the proof}\label{sec:conclusion}
All of the ingredients above combine to proving the claimed convergence of the BPHZ models.
\begin{proof}[Proof of Theorem \ref{thm:convergence}]
	Up to now we have shown that the sequence of BPHZ models $(\hat \Pi^\eps, \hat \Gamma^\eps)$ verifies the tightness condition in Lemma \ref{lem:tightness}. It is then left to show that any subsequential limiting model $(\hat \Pi, \hat \Gamma)$ coincides in law with the one defined by \eqref{eq:area-model}. By translation invariance and admissibility it suffices to consider the case $x=0$ and trees with a noise at their node (i.e. non-planted trees).
	In Section \ref{sec:new_noise}, we have proved that the joint law of $(\hat \Pi_0\<XiIXi>,\hat \Pi_0 \<XiIXXi>_j,\hat \Pi_0 \<XXiIXi>_k)_{j, k \in \{1,2\}}$ equals to that of $(\<Eta>, \<XEta>_j, \<XEta>_k)_{j, k \in \{1,2\}}$.
	In Section  \ref{sec:vanishing_trees} we have proved that $\hat \Pi_0\tau$ vanishes on $\tau \in \{\<Xi>, \<XXi>, \<XiIXiIXi>, \<XiIXi2>, \<XiIXiIXi2>, \<XiIXi3>\}$.
	
	Finally, for $\tau \in \{\<XiIXiIXiIXi>, \<lastone>\}$ we use Lemma \ref{lem:criterion_second_chaos}, with $A_\eps$ defined in \eqref{eq:A} and $B_\eps=\hat \Pi^\eps_0\tau(\varphi) - \E[\hat \Pi^\eps_0\tau(\varphi)]$. Conditions (i)-(ii)-(iii) follow from the already  established tightness and convergence results.
	Propositions \ref{prop:lastone} and \ref{prop:XiIXiIXiIXi} verify (iv) and (v) with
	\begin{equ}
		\mathbf{B}_\varphi(x,y)=\varphi(x)(K(x-y)-K(-y)).
	\end{equ}
	In particular, notice that $I_2(\bB_\varphi) = \<EtaIEta>(\varphi) - \<IEta>(0)\diamond\<Eta>\;(\varphi)$ where $\diamond$ denotes the Wick product. Additionally, by Propositions \ref{prop:lastone} and \ref{prop:XiIXiIXiIXi} one has $\E(\hat \Pi^\eps_0\tau(\varphi)) \to -\E(\<IEta>(0)  \,\<Eta>\;(\varphi))$.
	Recall also that $A_\eps=\hat \Pi_0^\eps\<XiIXi>+R_\eps$ with $R_\eps\to 0$ in probability.
	This implies that for both choices of $\tau$,
	\begin{equ}
		\lim\,(\hat \Pi_0^\eps\<XiIXi>,\hat \Pi^\eps_0\tau)
		=\lim\, \big(A_\eps,B_\eps+\E(\hat \Pi^\eps_0\tau)\big)=(\<Eta>,\<EtaIEta> - \<IEta>(0)\,\<Eta>).
	\end{equ}
	This characterises the 
	joint law of $(\hat \Pi_0\<XiIXi>,\hat \Pi_0 \<XiIXXi>_j,\hat \Pi_0 \<XXiIXi>_k,\hat \Pi_0\<XiIXiIXiIXi>,\hat\Pi_0\<lastone>)_{j, k \in \{1,2\}}$  as that of $(\<Eta>, \<XEta>_j, \<XEta>_k,\<EtaIEta> - \<IEta>(0)\,\<Eta>,\<EtaIEta> - \<IEta>(0)\,\<Eta>)_{j, k \in \{1,2\}}$.
	The proof is then finished.
\end{proof}

\bibliography{variance_gPAM}
\bibliographystyle{Martin}

\end{document}